\documentclass[oneside,final,letterpaper]{ucr}

\usepackage{adjustbox}
\usepackage{amsfonts}
\usepackage{amssymb}
\usepackage{amsthm}
\usepackage{color}
\usepackage{float}
\usepackage{framed}
\usepackage{graphicx}
\usepackage{hyperref}
\usepackage{listings}
\usepackage{mathrsfs}
\usepackage{mathtools}
\usepackage[new]{old-arrows}
\usepackage{sidecap}
\usepackage{stmaryrd}
\usepackage{tikz}
\usepackage{tikz-cd}
\usepackage{xargs}

\usetikzlibrary{patterns,patterns.meta}
\sidecaptionvpos{figure}{c}

\definecolor{mygreen}{rgb}{0,0.6,0}
\definecolor{mygray}{rgb}{0.5,0.5,0.5}
\definecolor{mymauve}{rgb}{0.58,0,0.82}
\newcommand{\locIdeal}[2]{\tikz[baseline=.1ex]{
	\draw[gray!40, thick, fill=gray!40, domain=-45:225] plot ({cos(\x)}, {sin(\x)}) to[out=45, in=130] (0.71, -0.71);
	\draw[thick] (-0.71, -0.71) to[out=45, in=135] (0.71, -0.71);
	\node[draw, circle, inner sep=0pt, minimum size=4pt, fill=white] (p1) at (0,-0.41) {};
	\draw[thick] (p1) -- (-0.71,0.71);
	\draw[thick] (0.71,0.71) -- (p1);
	\node (s1) at (-0.6,-0.1) {$#1$};
	\node (s2) at (0.6,-0.1) {$#2$};}
}
\newcommand{\locTangle}[2]{\tikz[baseline=.1ex]{
	\draw[gray!40, thick, fill=gray!40, domain=-45:225] plot ({cos(\x)}, {sin(\x)}) to[out=45, in=130] (0.71, -0.71);
	\draw[thick] (-0.71, -0.71) to[out=45, in=135] (0.71, -0.71);
	\node[draw, circle, inner sep=0pt, minimum size=4pt, fill=white] (p1) at (0,-0.41) {};
	\draw[thick] (-0.15, -0.173) -- (-0.71,0.71);
	\draw[thick] (0.71,0.71) -- (p1);
	\node (s1) at (-0.55,-.15) {$#1$};
	\node (s2) at (0.55,-0.15) {$#2$};}
}
\newcommandx{\MarkedTorusBackground}[6][1=, 2=, 3=, 4=, 5=, 6=]{
    \draw[draw=none, fill=gray!40] (0,1) -- (2,1) -- (2,-1) -- (0,-1) -- (1, -0.15) to[out=0, in=-90] (1.15, 0) to[out=90, in=0] (1, 0.15) to[out=180, in=90] (0.85, 0) to[out=-90, in=180] (1, -0.15) -- (0, -1) -- (0,1);
    \draw[thick, ->] (0,1) -- (0.9,1);
    \draw[thick, ->] (0.85,1) -- (1.2,1);
    \draw[thick] (1.2,1) -- (2,1);
    \draw[thick, ->] (0,-1) -- (0.9,-1);
    \draw[thick, ->] (0.85,-1) -- (1.2,-1);
    \draw[thick] (1.2,-1) -- (2,-1);
    \draw[thick, ->] (0,-1) -- (0,0);
    \draw[thick] (0,-1) -- (0,1);
    \draw[thick, ->] (2,-1) -- (2,0);
    \draw[thick] (2,-1) -- (2,1);
    \draw[thick, black] (1, 0) circle (0.15);
    \node at (0.1, -1.2) {#1};
    \node at (0.45, -1.2) {#2};
    \node at (0.8, -1.2) {#3};
    \node at (1.2, -1.2) {#4};
    \node at (1.55, -1.2) {#5};
    \node at (1.9, -1.2) {#6};
}

\newcommand{\dee}{\partial}

\newcommand{\mf}[1]{\mathfrak{#1}}
\newcommand{\ms}[1]{\mathscr{#1}}

\theoremstyle{definition}
\newtheorem{theorem}{Theorem}[section]
\newtheorem{thm}[theorem]{Theorem}

\newtheorem{conj}[theorem]{Conjecture}

\newtheorem{definition}[theorem]{Definition}
\newtheorem{example}[theorem]{Example}

\newtheorem{lemma}[theorem]{Lemma}

\newtheorem{prop}[theorem]{Proposition}
\newtheorem{remark}[theorem]{Remark}

\begin{document}


\title{Stated Skein Theory and Double Affine Hecke Algebra Representations}
\author{Raymond Alexander Dzintars Matson}
\degreemonth{September}
\degreeyear{2024}
\degree{Doctor of Philosophy}
\chair{Dr. Peter Samuelson}
\othermembers{Dr. Jacob Greenstein\\
Dr. Wee Liang Gan}
\numberofmembers{3}
\field{Mathematics}
\campus{Riverside}

\maketitle

\degreesemester{Summer}

\begin{frontmatter}

\begin{acknowledgements}

First and foremost, I would like to express my deepest gratitude to my advisor, Peter Samuelson, for his immense kindness, remarkable guidance, and continuous patience throughout my doctoral journey. Thank you for being a wonderful mentor, a phenomenal mathematician, and most importantly, a great friend. I have no doubt that your future students will look up to you the way I do.

I would especially like to thank Jacob Greenstein, Chris Grossack, Melody Molander, Shane Rankin, Alex Space, Elliott Vest, and Stefano Vidussi for their incredibly helpful conversations and constructive comments over the years. I am also deeply thankful to Neima Ghandian, Will Hoffer, and Rahul Rajkumar for always engaging with my relentless algebra nonsense and their consistently helpful feedback.

I owe a tremendous thank you to my WeierstrHouse roommates for all of the unforgettable adventures and lifelong memories we created together. My time in Riverside was made so much better by your companionship, and I could not have asked for better friends.

A major thank you to Margarita Roman for her instrumental role in keeping our department running smoothly and for all the unwavering support she provides us. We all love and deeply appreciate all that you do for us. I would also like to thank all of the UCR math graduate students that I had the privilege to meet during my time here. I'll never forget the frustration, laughter, passion, tears, and growth we all experienced together. Know that I appreciate each and every one of you!
\end{acknowledgements}

\begin{abstract}

In this thesis, we explore the representation theory of double affine Hecke algebras (DAHAs) through the lens of stated skein theory. Over the past decade, there have been several works establishing robust connections between skein algebras and DAHAs. Particularly, Samuelson proved that a spherical subalgebra of the type $A_1$ DAHA can be realized as a quotient of the Kauffman bracket skein algebra of the torus with boundary, $K_q(T^2 \setminus D^2)$. Since the $A_1$ double affine Hecke algebra is Morita equivalent to its spherical subalgebra, discovering modules for $K_q(T^2 \setminus D^2)$ immediately provides us with modules for the $A_1$ DAHA.

Stated skein theory enhances traditional Kauffman bracket skein theory by incorporating the boundary components of manifolds, thereby offering additional properties such as excision that enrich the algebraic structure. Furthermore, Kauffman bracket skein algebras embed into their stated counterparts, showing that stated skein algebras are extensions of Kauffman bracket skein algebras. We use this extended framework to further develop the representation theory of the $A_1$ DAHA.

After identifying generators for the stated skein algebra of $T^2 \setminus D^2$, we embed this algebra into a quantum $6$-torus and leverage the nice representation-theoretic properties of quantum tori to construct a module of Laurent polynomials. Additionally, as  $T^2$ is the boundary of any knot complement, we discuss how to construct a more topologically-defined module from various knots and provide an explicit example for the unknot. This approach builds upon the ideas of Berest and Samuelson, who showed that there exists a natural DAHA action on the Kauffman bracket skein module of knot complements.
\end{abstract}

\tableofcontents
\listoffigures
\end{frontmatter}

\chapter{Introduction}

The inception of double affine Hecke algebras (DAHAs) is credited to Ivan Cherednik \cite{MR2133033}, who used certain properties of DAHAs to solve the Macdonald conjectures. Macdonald polynomials, which are intrinsically related to quantum physics, have led to DAHAs finding applications across numerous fields, including string theory, operator theory, homological algebra, Harish-Chandra theory, gauge theory, Fourier analysis, mirror symmetry, quantum topology, and even the Langlands program.

In 1962, Freeman Dyson conjectured that the constant term of a certain Laurent polynomial in $n$ variables is given by $$\mathrm{CT}\left( \prod_{i \neq j} (1 - x_i x_j^{-1})^k \right) = \frac{(nk)!}{(k!)^n}$$
for $k \in \mathbb{Z}_{\geq 0}$ (Conjecture B in \cite{10106311703773}). Here, ``CT'' refers to the function that extracts the constant term of the Laurent polynomial.

Ian Macdonald later generalized this conjecture in 1982 \cite{MR674768} to a root system generalization. Specifically, let $\mathfrak{g}$ be a finite-dimensional (real or complex) reductive Lie algebra, $\mathfrak{h} \subset \mathfrak{g}$ its Cartan subalgebra, and $R \subset \mathfrak{h}^\ast$ a corresponding reduced root system. Macdonald conjectured that
$$\mathrm{CT}\left( \prod_{\alpha \in R_{+}} \prod_{i=1}^k\left(1-q^{i-1} e^{-\alpha}\right)\left(1-q^i e^\alpha\right) \right) = \prod_{i=1}^l \genfrac[]{0pt}{0}{k d_i}{k},$$
where $l = \operatorname{dim}(\mathfrak{h})$, $e^\alpha$ is the formal exponential corresponding to $\alpha \in R$, $\{ d_1, d_2, \cdots, d_l \}$ are the fundamental degrees of $R$ (by this we mean the fundamental invariants of $W$, the Weyl group of $R$), and $\genfrac[]{0pt}{0}{n}{k}$ are the Gaussian binomial coefficients (see Section \ref{section:QuantumGroups} for the definition of Gaussian binomial coefficients).

Furthermore, Macdonald extended this conjecture to affine root systems by introducing an additional parameter $t \in \mathbb{C}^*$, leading to his famous Constant Term $(q, t)$-Conjecture:
$$
\frac{1}{|W|} \mathrm{CT}\left(\prod_{n \geq 0} \prod_{\alpha \in R} \frac{1-q^n e^\alpha}{1-q^n t e^\alpha}\right) = \prod_{n \geq 0} \prod_{i=1}^l \frac{\left(1-q^n t\right)\left(1-q^{n+1} t^{d_i-1}\right)}{\left(1-q^{n+1}\right)\left(1-q^n t^{d_i}\right)}.
$$
This was proved in 1995 by Ivan Cherednik \cite{MR1314036}, using the representation theory of double affine Hecke algebras.

Since then, extensive work has been done to better understand double affine Hecke algebras from various perspectives. However, much remains to be explored regarding the representation theory of DAHAs. This thesis approaches DAHAs through the lens of quantum topology, particularly further exploring the connections between the representation theory of DAHAs and the representation theory of skein algebras of tori. When viewed appropriately, the Kauffman bracket skein algebra of a torus with boundary (or simply a punctured torus) is Morita equivalent to the $A_1$ double affine Hecke algebra.

In the 1980s, Vaughan Jones discovered what is now known as the Jones polynomial, a knot invariant derived from von Neumann algebras, which was further simplified by Louis Kauffman's introduction of the Kauffman bracket. This polynomial invariant for links in $3$-dimensional space paved the way for the development of skein modules, first introduced by J\'{o}zef Przytycki in 1987 \cite{MR1194712} and independently by Vladimir Turaev in 1988 \cite{MR964255}. These modules serve as $3$-manifold invariants, capturing information about the manifold based on the kind of knot theory and algebraic structures that manifold admits. For example, the Kauffman bracket skein algebra uses links and crossing relations to forge a comprehensive algebraic framework. These algebras have become a cornerstone in quantum topology, connecting knot theory with quantum gravity and furthering our understanding of quantum field theories. For a far better explanation of the history and development of skein theory than I could ever give, see \cite{MR4297592}.

In recent years, Thang T. Q. L\^{e} and others have laid substantial groundwork for a particular generalization of Kauffman bracket skein algebras, known as stated skein algebras. These algebras coincide with Kauffman bracket skein algebras when the corresponding manifold lacks a boundary. One of the main benefits this generalization offers is an excision property that enables gluing and splitting by leveraging the boundary components. Stated skein algebras encompass their corresponding Kauffman bracket skein algebras and are, in general, significantly larger objects to work with. The primary question explored in this thesis is whether we can extract more interesting representations of DAHAs through the framework of stated skein algebras.

Chapter 2 of this thesis provides a thorough background of the setting, including a discussion on the origin of diagrammatically-defined modules, their utility, and the representation-theoretic data they offer. We also define the $A_1$ double affine Hecke algebra along with its spherical subalgebra, and examine its relationship with the Kauffman bracket skein algebra of $T^2 \setminus D^2$, as well as the corresponding stated skein algebra. While our definition of stated skein algebras differs slightly from the conventional definition found in the literature, we demonstrate that both models are naturally isomorphic.

Given that stated skein algebras are larger than Kauffman bracket skein algebras, we review common algebraic techniques to convert modules over the stated skein algebra into modules over the spherical $A_1$ DAHA. In Chapter 3, we first get our hands dirty and classify all simple closed curves based at the boundary of $T^2 \setminus D^2$. Using this classification, we then explicitly compute the generators for the main algebra we care about, the stated skein algebra of the torus with boundary and one marking. Once the generators for this algebra have been established, we then proceed to define modules over it.

Chapter 4 briefly explains and employs an embedding technique introduced by L\^{e} and Yu in \cite{MR4431131}, where we map our stated skein algebra into a quantum $6$-torus. Since the representation theory of quantum tori is well-known and behaves quite nicely, this approach allows us to compute representations for our algebra and, consequently, our double affine Hecke algebra. Notably, we uncover how our algebra acts non-trivially on the vector space of complex Laurent polynomials in four variables.

Chapter 5 expands upon a well-known module structure for Kauffman bracket skein algebras. For a 3-manifold $M$ with boundary $\partial M$, the Kauffman bracket skein module of $M$ naturally becomes a module over the Kauffman bracket skein algebra of $\partial M$. The module action is defined by gluing $\partial M \times [0,1]$ into the boundary and “pushing” any curves in $\partial M \times [0,1]$ into $M$. We begin by exploring a more geometrically-centered understanding of skein algebras, particularly in relation to knot theory and quantizations of character varieties, and discuss how knot complements induce modules over the Kauffman bracket skein algebra of the torus. We then define and explore how these structures can be upgraded to stated skein algebras and stated skein modules.

\chapter{Background}

\section{Categorical Framework}

We begin by discussing the categorical setting for our categories of interest. We will introduce a more concrete example of these definitions in section \ref{section:RibbonHopfAlg}. However, it will be beneficial to introduce the proper language of these objects first.

\begin{definition}
    A \textit{monoidal category}, $\mathcal{C}$, is a $\mathbb{K}$-linear category equipped with
    \begin{enumerate}
        \item a bifunctor $\otimes: \mathcal{C} \times \mathcal{C} \longrightarrow \mathcal{C}$ written as $(a, b) \mapsto a \otimes b$ called the monoidal product or tensor product,
        \item an object $1 \in \mathcal{C}$ called the monoidal unit,
        \item a natural isomorphism $\alpha:\left( -\otimes - \right) \otimes - \longrightarrow - \otimes\left( - \otimes - \right)$ called the associator, with components $\alpha_{A, B, C}:(A \otimes B) \otimes C \longrightarrow A \otimes(B \otimes C)$,
        \item natural isomorphisms $\lambda: 1 \otimes - \longrightarrow -$ and $\rho: - \otimes 1 \longrightarrow -$ called the left unintor and right unitor, with respective components $\lambda_A : 1 \otimes A \longrightarrow A$ and $\rho_A : A \otimes 1 \longrightarrow A$
    \end{enumerate}
    such that for all $A,B,C,D \in \mathcal{C}$, the following diagrams commute.
    \begin{center}\resizebox{0.8\width}{!}
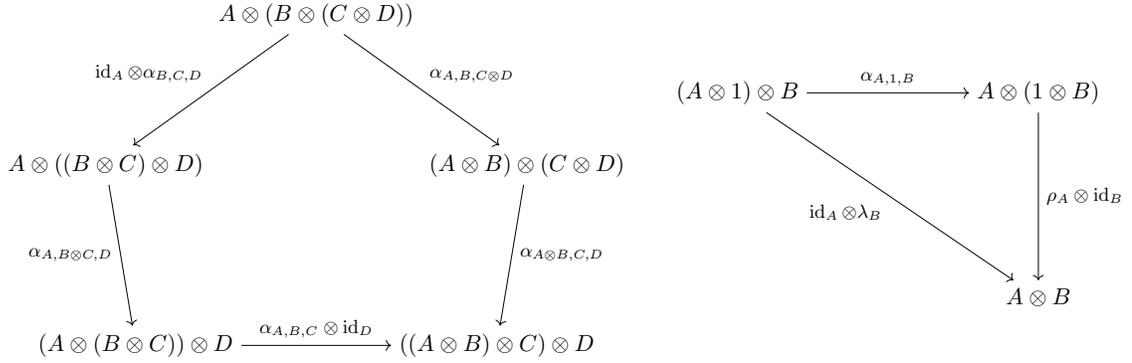
{
        \begin{tikzpicture}
            \node (PT) at (0, 2.5) {$A \otimes ( B \otimes ( C \otimes D ) )$};
            \node (ML) at (-3.5, 0) {$A \otimes ((B \otimes C) \otimes D)$};
            \node (MR) at (3.5, 0) {$(A \otimes B) \otimes (C \otimes D)$};
            \node (BL) at (-3, -3) {$(A \otimes (B \otimes C)) \otimes D$};
            \node (BR) at (3, -3) {$((A \otimes B) \otimes C) \otimes D$};
            \draw[->] (PT) -- (ML) node[midway, above left] {\footnotesize{$\operatorname{id}_{A} \otimes \alpha_{B,C,D}$}};
            \draw[->] (PT) -- (MR) node[midway, above right] {\footnotesize{$\alpha_{A,B,C \otimes D}$}};
            \draw[->] (ML) -- (BL) node[midway, left] {\footnotesize{$\alpha_{A, B \otimes C, D}$}};
            \draw[->] (MR) -- (BR) node[midway, right] {\footnotesize{$\alpha_{A \otimes B, C, D}$}};
            \draw[->] (BL) -- (BR) node[midway, above] {\footnotesize{$\alpha_{A,B,C} \otimes \operatorname{id}_{D}$}};
            \node (TTL) at (7, 1.2) {$(A \otimes 1) \otimes B$};
            \node (TTR) at (12, 1.2) {$A \otimes (1 \otimes B)$};
            \node (TBR) at (12, -2.2) {$A \otimes B$};
            \draw[->] (TTL) -- (TTR) node[midway, above] {\footnotesize{$\alpha_{A, 1, B}$}};
            \draw[->] (TTL) -- (TBR) node[midway, below left] {\footnotesize{$\operatorname{id}_A \otimes \lambda_B$}};
            \draw[->] (TTR) -- (TBR) node[midway, right] {\footnotesize{$\rho_A \otimes \operatorname{id}_{B}$}};
        \end{tikzpicture}}
    \end{center}
    If the $\alpha$, $\lambda$, and $\rho$ are all identity maps, then we say that $\mathcal{C}$ is \emph{strict}.
\end{definition}
If $\mathcal{D}$ is a category enriched over the monoidal category $\mathcal{C}$, then we require that $\otimes$ be a $\mathcal{C}$-enriched functor, and $\alpha$ and $\lambda$ be $\mathcal{C}$-enriched natural transformations. When this happens we say the monoidal structure is \emph{compatible} with the enrichment.

An effective way of to think about the following categorical definitions is by using a diagrammatic interpretation, which will be discussed in greater detail later in section \ref{section:SkeinMods}. This interpretation is typically reserved for strict categories, however, it's a useful tool to better understand the properties trying to be expressed in these definitions. For now, just imagine these morphisms as (oriented) strings connecting their corresponding sources and targets. If we want to consider a map between monoidal product of objects, we can instead consider multiple strings next to each other, possibly being weaved together in some way, and composition is attaching strings on top of each other. For example, a map $A \otimes B \to C \otimes D$ might look something like
$$
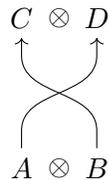
\begin{tikzpicture}
    \node (A) at (0,-1) {$A$};
    \node (B) at (1,-1) {$B$};
    \node (C) at (0, 1) {$C$};
    \node (D) at (1, 1) {$D$};
    \node at (0.5, -1) {$\otimes$};
    \node at (0.5, 1) {$\otimes$};
    \draw[->] (A) -- (0, -0.5) to[out=90, in=270] (1, 0.5) -- (D);
    \draw[->] (B) -- (1, -0.5) to[out=90, in=270] (0, 0.5) -- (C);
\end{tikzpicture}$$
I will incorporate pictures of the diagrammatic interpretations throughout this section and use the convention that morphisms move upwards.

\begin{definition}
    An object $X^*$ in $\mathcal{C}$ is said to be a \textit{left dual} of $X$ if there exist morphisms $\operatorname{ev}_X: X^* \otimes X \rightarrow 1$ and $\operatorname{coev}_X: 1 \rightarrow X \otimes X^*$, called the evaluation and coevaluation respectively, such that the following diagrams commute.
    \begin{center}
        \begin{tikzcd}
            X \arrow[rr, "\operatorname{coev} \otimes \operatorname{id}_X"] \arrow[dd, "\operatorname{id}_X"'] & & (X \otimes X^*) \otimes X \arrow[dd, "{\alpha_{X, X^*, X}}"'] & & X^* \arrow[rr, "\operatorname{id}_{X^*} \otimes \operatorname{coev}"] \arrow[dd, "\operatorname{id}_{X^*}"'] & & X^* \otimes (X \otimes X^*) \arrow[dd, "{\alpha_{X^*, X, X^*}^{-1}}"'] \\
            & & & & & & \\
            X & & X \otimes (X^* \otimes X) \arrow[ll, "\operatorname{id}_X \otimes \operatorname{ev}"'] & & X^* & & (X^* \otimes X) \otimes X^* \arrow[ll, "\operatorname{ev} \otimes \operatorname{id}_{X^*}"]
        \end{tikzcd}
    \end{center}
\end{definition}

\begin{definition}
    Similarly, an object $X^*$ in $\mathcal{C}$ is said to be a \textit{right dual} of $X$ if there exist morphisms $\operatorname{ev}_X: X \otimes X^* \rightarrow 1$ and $\operatorname{coev}_X: 1 \rightarrow X^* \otimes X$ such that the following diagrams commute.
    \begin{center}
        \begin{tikzcd}
            X \arrow[rr, "\operatorname{id}_X \otimes \operatorname{coev}"] \arrow[dd, "\operatorname{id}_X"'] & & X \otimes (X^* \otimes X) \arrow[dd, "{\alpha_{X, X^*, X}^{-1}}"'] & & X^* \arrow[rr, "\operatorname{coev} \otimes \operatorname{id}_{X^*}"] \arrow[dd, "\operatorname{id}_{X^*}"'] & & (X^* \otimes X) \otimes X^* \arrow[dd, "{\alpha_{X^*, X, X^*}}"'] \\
            & & & & & & \\
            X & & (X \otimes X^*) \otimes X \arrow[ll, "\operatorname{ev} \otimes \operatorname{id}_X"'] & & X^* & & X^* \otimes (X \otimes X^*) \arrow[ll, "\operatorname{id}_{X^*} \otimes \operatorname{ev}"]
        \end{tikzcd}
    \end{center}
\end{definition}
These commuting diagrams are often called the ``zigzag identites''.

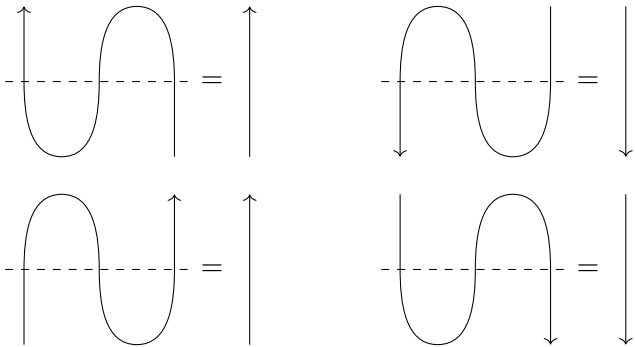
\begin{SCfigure}[][h]
    \centering
    \begin{tikzpicture}
        \draw[->] (2, 0.5) -- (2, 1.5) to[out=90, in=0] (1.5, 2.5) to[out=180, in=90] (1, 1.5) to[out=270, in=0] (0.5, 0.5) to[out=180, in=270] (0, 1.5) -- (0, 2.5);
        \draw[dashed] (-0.25, 1.5) -- (2.25, 1.5);
        \node at (2.5, 1.5) {$=$};
        \draw[->] (3, 0.5) -- (3, 2.5);
        \draw[->] (7, 2.5) -- (7, 1.5) to[out=270, in=0] (6.5, 0.5) to[out=180, in=270] (6, 1.5) to[out=90, in=0] (5.5, 2.5) to[out = 180, in=90] (5, 1.5) -- (5, 0.5);
        \draw[dashed] (4.75, 1.5) -- (7.25, 1.5);
        \node at (7.5, 1.5) {$=$};
        \draw[->] (8, 2.5) -- (8, 0.5);
        \draw[->] (0, -2) -- (0, -1) to[out=90, in=180] (0.5, 0) to[out=0, in=90] (1, -1) to[out=270, in=180] (1.5, -2) to[out=0, in=270] (2, -1) -- (2, 0);
        \draw[dashed] (-0.25, -1) -- (2.25, -1);
        \node at (2.5, -1) {$=$};
        \draw[->] (3, -2) -- (3, 0);
        \draw[->] (5, 0) -- (5, -1) to[out=270, in=180] (5.5, -2) to[out=0, in=270] (6, -1) to[out=90, in=180] (6.5, 0) to[out=0, in=90] (7, -1) -- (7,-2);
        \draw[dashed] (4.75, -1) -- (7.25, -1);
        \node at (7.5, -1) {$=$};
        \draw[->] (8, 0) -- (8, -2);
    \end{tikzpicture}
    \caption[The zigzag identities for rigid categories]{The zigzag identities can be understood as curved strings being straightened out.}
    \label{fig:ZigzagId}
\end{SCfigure}
When an object has a dual, evaluation and coevaluation maps can be pictorially represented using cups and caps, as the notion of duals corresponds to reversing the orientation of our strings. In figure \ref{fig:ZigzagId}, the top left picture corresponds to the identity $(\operatorname{id}_X \otimes \operatorname{ev}) \circ \alpha_{X, X^*, X} \circ (\operatorname{coev} \otimes \operatorname{id}_X) = \operatorname{id}_X$ for left duals, while the bottom right corresponds to $(\operatorname{id}_{X^*} \otimes \operatorname{ev}) \circ \alpha_{X^*,X,X} \circ (\operatorname{coev} \otimes \operatorname{id}_{X^*}) = \operatorname{id}_{X^*}$ for right duals.

\begin{remark}
    It's hopefully clear that if $X^*$ is a left dual of an object $X$, then $X$ is a right dual of $X^*$ and in any monoidal category, $1$ is equal to its left and right duals. Moreover, left and right duals are unique up to a unique isomorphism.
\end{remark}

\begin{remark}
    Changing the order of tensor products, when possible, switches left duals and right duals. Therefore, for any statement concerning right duals there corresponds a symmetric statement about left duals.
\end{remark}

\begin{remark}
    Some texts use the notation ${}^*X$ for right duals to distinguish between the two. However, we won't need to worry about this distinction too much in this thesis and so we will not use this notation.
\end{remark}

\begin{definition}
    An object in a monoidal category is called \textit{rigid} if it has left and right duals. A monoidal category $\mathcal{C}$ is called \textit{rigid} if every object of $\mathcal{C}$ is rigid.
\end{definition}

For those that have never seen the definition of a rigid category before, you should think of this as merely saying each object has a well-defined dual that acts exactly as we expect it to. A quick example is the category of finite dimensional complex vector spaces $\operatorname{FdVect}_{\mathbb{C}}$, where $V^\ast := \operatorname{Hom}_{\mathbb{C}}(V, \mathbb{C})$ and
\begin{align*}
    \operatorname{ev} : V \otimes V^\ast &\to \mathbb{C} & \operatorname{coev} : \mathbb{C} &\to V^\ast \otimes V\\
    (v, \varphi) &\mapsto \varphi(v) & k &\mapsto k \cdot \sum_{i \in I} \varphi_i \otimes v_i
\end{align*}
where $\{ v_i \}_{i \in I}$ is a basis for $V$ and $\{ \varphi_i \}_{i \in I}$ is the dual basis such that $\varphi_i(v_j) = \delta_{i,j}$ (the usual evaluation map that we all know and love).

\begin{definition}
    A monoidal category, $\mathcal{C}$, is \textit{braided} if for every pair of objects $X, Y \in \mathcal{C}$, there is a natural isomorphism $B_{X,Y}: X \otimes Y \to Y \otimes X$ such that for all $X,Y,Z \in \mathcal{C}$, the following hexagonal diagrams commute.
    \begin{center}
        \begin{tikzcd}
            & X \otimes (Y \otimes Z) \arrow[rr, "{B_{X, Y \otimes Z}}"] & & (Y \otimes Z) \otimes X \arrow[rd, "{\alpha_{Y,Z,X}}"] & \\
            (X \otimes Y) \otimes Z \arrow[ru, "{\alpha_{X,Y,Z}}"] \arrow[rd, "{B_{X,Y} \otimes \operatorname{id}_Z}"'] & & & & Y \otimes (Z \otimes X) \\
            & (Y \otimes X) \otimes Z \arrow[rr, "{\alpha_{Y,X,Z}}"] & & Y \otimes (X \otimes Z) \arrow[ru, "{\operatorname{id}_Y \otimes B_{X,Z}}"'] & 
        \end{tikzcd}
    \end{center}
    \begin{center}
        \begin{tikzcd}
            & (X \otimes Y) \otimes Z \arrow[rr, "{B_{X \otimes Y, Z}}"] & & Z \otimes (X \otimes Y) \arrow[rd, "{\alpha_{Z,X,Y}^{-1}}"] & \\
            X \otimes (Y \otimes Z) \arrow[ru, "{\alpha_{X,Y,Z}^{-1}}"] \arrow[rd, "{\operatorname{id}_X \otimes B_{Y,Z}}"'] & & & & (Z \otimes X) \otimes Y \\
            & X \otimes (Z \otimes Y) \arrow[rr, "{\alpha_{X,Z,Y}^{-1}}"] & & (X \otimes Z) \otimes Y \arrow[ru, "{B_{X,Z} \otimes \operatorname{id}_Y}"'] &
        \end{tikzcd}
    \end{center}
\end{definition}

\begin{figure}[H]
    \centering
    \begin{tikzpicture}
        \draw[->] (1, 0) -- (1, 1) -- (0, 2) -- (0, 3);
        \draw[->] (2, 0) -- (2, 1) -- (1, 2) -- (1, 3);
        \draw[line width=3mm, white] (0, 1) -- (2, 2);
        \draw[->] (0, 0) -- (0, 1) -- (2, 2) -- (2, 3);
        \node at (2.5, 1.5) {$=$};
        \draw[->] (4, 0) -- (4, 0.5) -- (3, 1) -- (3, 3);
        \draw[->] (5, 0) -- (5, 2) -- (4, 2.5) -- (4, 3);
        \draw[line width=3mm, white] (3, 0.5) -- (4, 1);
        \draw[line width=3mm, white] (4, 2) -- (5, 2.5);
        \draw[->] (3, 0) -- (3, 0.5) -- (4, 1) -- (4, 2) -- (5, 2.5) -- (5, 3);
        \draw[dashed] (2.8, 1.5) -- (5.2, 1.5);
        \node at (0, -0.3) {$X$};
        \node at (1, -0.3) {$Y$};
        \node at (2, -0.3) {$Z$};
        \node at (3, -0.3) {$X$};
        \node at (4, -0.3) {$Y$};
        \node at (5, -0.3) {$Z$};
        \draw[->] (9, 0) -- (9, 1) -- (7, 2) -- (7,3);
        \draw[line width=3mm, white] (8, 1) -- (9, 2);
        \draw[->] (8, 0) -- (8, 1) -- (9, 2) -- (9, 3);
        \draw[line width=3mm, white] (7, 1) -- (8, 2);
        \draw[->] (7, 0) -- (7, 1) -- (8, 2) -- (8, 3);
        \node at (9.5, 1.5) {$=$};
        \draw[->] (12, 0) -- (12, 0.5) -- (11, 1) -- (11, 2) -- (10, 2.5) -- (10, 3);
        \draw[line width=3mm, white] (11, 0.5) -- (12, 1);
        \draw[->] (11, 0) -- (11, 0.5) -- (12, 1) -- (12, 3);
        \draw[line width=3mm, white] (10, 2) -- (11, 2.5);
        \draw[->] (10, 0) -- (10, 2) -- (11, 2.5) -- (11, 3);
        \draw[dashed] (9.8, 1.5) -- (12.2, 1.5);
        \node at (7, -0.3) {$X$};
        \node at (8, -0.3) {$Y$};
        \node at (9, -0.3) {$Z$};
        \node at (10, -0.3) {$X$};
        \node at (11, -0.3) {$Y$};
        \node at (12, -0.3) {$Z$};
    \end{tikzpicture}
    \caption[A graphical interpretation of the hexagonal identities]{A graphical interpretation of the hexagonal identities.}
    \label{fig:HexId}
\end{figure}
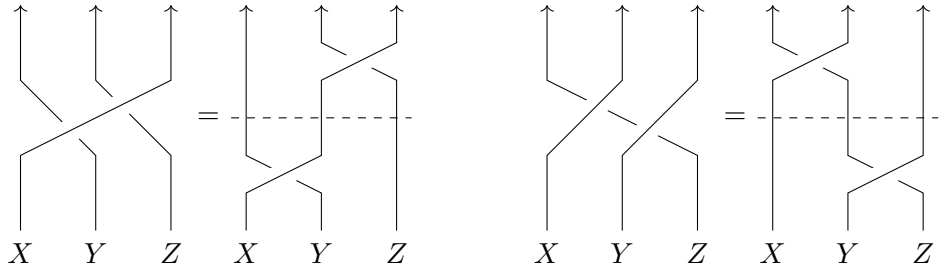

These isomorphisms are called \emph{braidings} and should be thought of as a commutivity constraint. When $B_{Y,X} \circ B_{X,Y} = \operatorname{id}_{X \otimes Y}$ for every $X,Y \in \mathcal{C}$, we call $\mathcal{C}$ a \textit{symmetric monoidal category}. When understanding a symmetric monoidal category diagrammatically, we often drop the ``depth'' to the braidings as $B_{X,Y}$ and $B_{Y,X}$ are inverse to each other.
\begin{SCfigure}[][h]
    \centering
    \begin{tikzpicture}
        \node (A1) at (0,-1) {$X$};
        \node (B1) at (1,-1) {$Y$};
        \node (A2) at (0, 1) {};
        \node (B2) at (1, 1) {};
        \node (A3) at (0, 2) {};
        \node (B3) at (1, 2) {};
        \draw (B1) -- (1, -0.5) to[out=90, in=270] (0, 0.5) -- (A2);
        \draw[line width=3mm, white] (0, -0.5) to[out=90, in=270] (1, 0.5);
        \draw (A1) -- (0, -0.5) to[out=90, in=270] (1, 0.5) -- (B2);
        \draw[->] (1, 0.8) to[out=90, in=270] (0, 2) -- (A3);
        \draw[line width=3mm, white] (0.2, 0.9) to[out=90, in=270] (1, 2);
        \draw[->] (0, 0.8) to[out=90, in=270] (1, 2) -- (B3);
        \node at (1.5, 0.5) {$=$};
        \node (A3) at (2,-1) {$X$};
        \node (B3) at (3,-1) {$Y$};
        \draw[->] (A3) -- (2, 2);
        \draw[->] (B3) -- (3, 2);
        \node at (3.5, 0.5) {$=$};
        \node (A4) at (4, -1) {$X$};
        \node (B4) at (5, -1) {$Y$};
        \draw[->] (B4) -- (5, -0.5) to[out=90, in=270] (4, 0.5) -- (4,0.8) to[out=90, in=270] (5, 1.9) -- (5,2);
        \draw[->] (A4) -- (4, -0.5) to[out=90, in=270] (5, 0.5) -- (5,0.8) to[out=90, in=270] (4, 1.9) -- (4,2);
    \end{tikzpicture}
    \caption[Symmetric braidings]{A graphical interpretation of a symmetric braiding.}
    \label{fig:SymmIdentity}
\end{SCfigure}
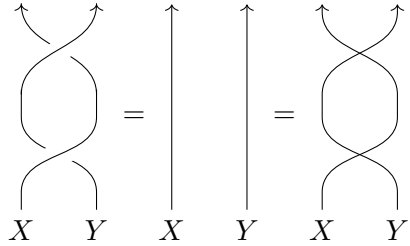\\
As the setting of this work revolves around quantum groups, you can reasonably expect our braidings to lack symmetric properties.

\begin{definition}
    A \textit{twist} on a rigid braided monoidal category is a natural isomorphism from the identity functor to itself, with components $\theta_X : X \to X$ such that for all $X, Y \in \mathcal{C}$ the following identities hold.
    \begin{align}
        \theta_{X \otimes Y} &= \left( \theta_X \otimes \theta_Y \right) \circ B_{Y,X} \circ B_{X, Y}\label{eq:Twist&BraidRel}
    \end{align}
    $$\theta_{X^*} = (\theta_X)^*$$
\end{definition}

\begin{figure}[H]
    \centering
    \begin{tikzpicture}
        \draw[->] (-0.8, -1) -- (-0.8, 1);
        \node at (0, 0) {$\mapsto$};
        \draw[thick] (1, -1) to[out=90, in=270] (0.7, -0.4) to[out=90, in=270] (1.3, 0.4) to[out=90, in=270] (1, 1);
        \draw[line width=2.5mm, white] (1.3, -0.4) to[out=90, in=270] (0.7, 0.4);
        \draw[thick] (1, -1) to[out=90, in=270] (1.3, -0.4) to[out=90, in=270] (0.7, 0.4) to[out=90, in=270] (1, 1);
        \draw[pattern=north east lines, draw=none] (1, -1) to[out=90, in=270] (1.3, -0.4) to[out=90, in=0] (1, -0.1) to[out=180, in=90] (0.7, -0.4) to[out=270, in=90] (1, -1);
        \draw[pattern=north west lines, draw=none] (1, 1) to[out=270, in=90] (1.3, 0.4) to[out=270, in=0] (1, 0.1) to[out=180, in=270] (0.7, 0.4) to[out=90, in=270] (1, 1);
        \node at (2, 0) {$=$};
        \draw[thick] (2.7, -1) to[out=90, in=180] (3, 0.2) to[out=0, in=90] (3.3, 0);
        \draw[line width=2.5mm, white] (3, -0.2) to[out=180, in=270] (2.7, 1);
        \draw[thick] (3.3, 0) to[out=270, in=0] (3, -0.2) to[out=180, in=270] (2.7, 1);
        \draw[thick] (2.9, -1) to[out=90, in=180] (3.2, 0.2) to[out=0, in=90] (3.5, 0);
        \draw[line width=2.5mm, white] (3.2, -0.2) to[out=180, in=270] (2.9, 1);
        \draw[thick] (3.5, 0) to[out=270, in=0] (3.2, -0.2) to[out=180, in=270] (2.9, 1);
    \end{tikzpicture}
    \caption[A graphical interpretation of a twist]{A graphical interpretation of a twist.}
\end{figure}
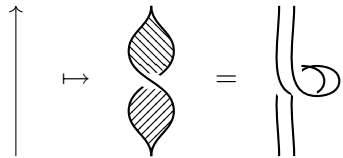
The graphical interpretation of a twist involves thickening up the string and creating a full twist in the now ``ribbon'', distinguishing it from the identity map. This is graphically equivalent to creating a loop instead.\footnote{You can verify this for yourself by grabbing the ends of a belt, giving it a full twist, slowly bringing the ends together, and watching your prop unravel into a loop.}

\begin{remark}
    Given any two objects $X,Y \in \mathcal{C}$, we also have the alternative equality
    $$\theta_{X \otimes Y} = B_{Y,X} \circ B_{X, Y} \circ \left( \theta_X \otimes \theta_Y \right).$$
    Figure \ref{fig:TwistEqualities} sketches a diagrammatic proof of how these identities are equivalent. These equalities follow from the naturality of the braidings and of the twists. For clarity, note that the first two pictures correspond to the compositions $B_{Y,X} \circ B_{X,Y} \circ \left( \theta_X \otimes \theta_Y \right)$ and $B_{Y,X} \circ \left( \operatorname{id}_Y \otimes \theta_X \right) \circ B_{X,Y} \circ \left( \operatorname{id}_X \otimes \theta_Y \right)$, respectively. The last picture corresponds to the right hand side of equation \ref{eq:Twist&BraidRel}. For more details see section $14.3.1$ in \cite{MR1321145}.
\end{remark}
\begin{figure}[H]
    \centering
    \begin{tikzpicture}
        \draw[thick] (0.7, 0) -- (0.7, 0.3) to[out=90, in=270] (1.3, 0.9);
        \draw[line width=2.5mm, white] (1.3, 0.3) to[out=90, in=270] (0.7, 0.9);
        \draw[thick] (1.3, 0) -- (1.3, 0.3) to[out=90, in=270] (0.7, 0.9) to[out=90, in=270] (1.3, 1.7) -- (1.3, 2);
        \draw[line width=2.5mm, white](1.3, 0.9) to[out=90, in=270] (0.7, 1.7);
        \draw[thick] (1.3, 0.9) to[out=90, in=270] (0.7, 1.7) -- (0.7, 2);
        \draw[pattern=north east lines, draw=none] (0.7, 0) -- (0.7, 0.3) to[out=90, in=225] (1, 0.6) to[out=315, in=90] (1.3, 0.3) -- (1.3, 0) -- (0.7, 0);
        \draw[pattern=north west lines, draw=none] (1, 0.6) to[out=135, in=270] (0.73, 0.9) to[out=90, in=225] (1, 1.3) to[out=315, in=90] (1.27, 0.9) to[out=270, in=45] (1, 0.6);
        \draw[pattern=north east lines, draw=none] (0.7, 2) -- (0.7, 1.7) to[out=270, in=135] (1, 1.3) to[out=45, in=270] (1.3, 1.7) -- (1.3, 2) -- (0.7, 2);
        \draw[thick] (1.7, 0) -- (1.7, 0.3) to[out=90, in=270] (2.3, 0.9);
        \draw[line width=2.5mm, white] (2.3, 0.3) to[out=90, in=270] (1.7, 0.9);
        \draw[thick] (2.3, 0) -- (2.3, 0.3) to[out=90, in=270] (1.7, 0.9) to[out=90, in=270] (2.3, 1.7) -- (2.3, 2);
        \draw[line width=2.5mm, white](2.3, 0.9) to[out=90, in=270] (1.7, 1.7);
        \draw[thick] (2.3, 0.9) to[out=90, in=270] (1.7, 1.7) -- (1.7, 2);
        \draw[pattern=north east lines, draw=none] (1.7, 0) -- (1.7, 0.3) to[out=90, in=225] (2, 0.6) to[out=315, in=90] (2.3, 0.3) -- (2.3, 0) -- (1.7, 0);
        \draw[pattern=north west lines, draw=none] (2, 0.6) to[out=135, in=270] (1.73, 0.9) to[out=90, in=225] (2, 1.3) to[out=315, in=90] (2.27, 0.9) to[out=270, in=45] (2, 0.6);
        \draw[pattern=north east lines, draw=none] (1.7, 2) -- (1.7, 1.7) to[out=270, in=135] (2, 1.3) to[out=45, in=270] (2.3, 1.7) -- (2.3, 2) -- (1.7, 2);
        \draw[thick] (1.7, 3) to[out=90, in=270] (0.7, 4);
        \draw[thick] (2.3, 3) to[out=90, in=270] (1.3, 4);
        \draw[pattern=north east lines, draw=none] (1.7, 3) to[out=90, in=270] (0.7, 4) -- (1.3, 4) to[out=270, in=90] (2.3, 3) -- (1.7, 3);
        \draw[line width=3.7mm, white] (1, 3) to[out=90, in=270] (2, 4);
        \draw[thick] (0.7, 3) to[out=90, in=270] (1.7, 4);
        \draw[thick] (1.3, 3) to[out=90, in=270] (2.3, 4);
        \draw[pattern=north east lines, draw=none] (0.7, 3) to[out=90, in=270] (1.7, 4) -- (2.3, 4) to[out=270, in=90] (1.3, 3) -- (0.7, 3);
        \draw[thick] (1.7, 5) to[out=90, in=270] (0.7, 6);
        \draw[thick] (2.3, 5) to[out=90, in=270] (1.3, 6);
        \draw[pattern=north east lines, draw=none] (1.7, 5) to[out=90, in=270] (0.7, 6) -- (1.3, 6) to[out=270, in=90] (2.3, 5) -- (1.7, 5);
        \draw[line width=3.7mm, white] (1, 5) to[out=90, in=270] (2, 6);
        \draw[thick] (0.7, 5) to[out=90, in=270] (1.7, 6);
        \draw[thick] (1.3, 5) to[out=90, in=270] (2.3, 6);
        \draw[pattern=north east lines, draw=none] (0.7, 5) to[out=90, in=270] (1.7, 6) -- (2.3, 6) to[out=270, in=90] (1.3, 5) -- (0.7, 5);
        \draw[thick] (0.7, 2) -- (0.7, 3);
        \draw[thick] (1.3, 2) -- (1.3, 3);
        \draw[thick] (1.7, 2) -- (1.7, 3);
        \draw[thick] (2.3, 2) -- (2.3, 3);
        \draw[thick] (0.7, 4) -- (0.7, 5);
        \draw[thick] (1.3, 4) -- (1.3, 5);
        \draw[thick] (1.7, 4) -- (1.7, 5);
        \draw[thick] (2.3, 4) -- (2.3, 5);
        \draw[pattern=north east lines, draw=none] (0.7, 2) -- (0.7, 3) -- (1.3, 3) -- (1.3, 2) -- (0.7, 2);
        \draw[pattern=north east lines, draw=none] (1.7, 2) -- (1.7, 3) -- (2.3, 3) -- (2.3, 2) -- (1.7, 2);
        \draw[pattern=north east lines, draw=none] (0.7, 4) -- (0.7, 5) -- (1.3, 5) -- (1.3, 4) -- (0.7, 4);
        \draw[pattern=north east lines, draw=none] (1.7, 4) -- (1.7, 5) -- (2.3, 5) -- (2.3, 4) -- (1.7, 4);
    \end{tikzpicture}
    \begin{tikzpicture}
        \node at (0, 3.3) {$=$};
        \draw[thick] (0.7, 0) -- (0.7, 2);
        \draw[thick] (1.3, 0) -- (1.3, 2);
        \draw[pattern=north east lines, draw=none] (0.7, 0) -- (0.7, 2) -- (1.3, 2) -- (1.3, 0) -- (0.7, 0);
        \draw[thick] (1.7, 0) -- (1.7, 0.3) to[out=90, in=270] (2.3, 0.9);
        \draw[line width=2.5mm, white] (2.3, 0.3) to[out=90, in=270] (1.7, 0.9);
        \draw[thick] (2.3, 0) -- (2.3, 0.3) to[out=90, in=270] (1.7, 0.9) to[out=90, in=270] (2.3, 1.7) -- (2.3, 2);
        \draw[line width=2.5mm, white](2.3, 0.9) to[out=90, in=270] (1.7, 1.7);
        \draw[thick] (2.3, 0.9) to[out=90, in=270] (1.7, 1.7) -- (1.7, 2);
        \draw[pattern=north east lines, draw=none] (1.7, 0) -- (1.7, 0.3) to[out=90, in=225] (2, 0.6) to[out=315, in=90] (2.3, 0.3) -- (2.3, 0) -- (1.7, 0);
        \draw[pattern=north west lines, draw=none] (2, 0.6) to[out=135, in=270] (1.73, 0.9) to[out=90, in=225] (2, 1.3) to[out=315, in=90] (2.27, 0.9) to[out=270, in=45] (2, 0.6);
        \draw[pattern=north east lines, draw=none] (1.7, 2) -- (1.7, 1.7) to[out=270, in=135] (2, 1.3) to[out=45, in=270] (2.3, 1.7) -- (2.3, 2) -- (1.7, 2);
        \draw[thick] (1.7, 2) to[out=90, in=270] (0.7, 3);
        \draw[thick] (2.3, 2) to[out=90, in=270] (1.3, 3);
        \draw[pattern=north east lines, draw=none] (1.7, 2) to[out=90, in=270] (0.7, 3) -- (1.3, 3) to[out=270, in=90] (2.3, 2) -- (1.7, 2);
        \draw[line width=3.7mm, white] (1, 2) to[out=90, in=270] (2, 3);
        \draw[thick] (0.7, 2) to[out=90, in=270] (1.7, 3);
        \draw[thick] (1.3, 2) to[out=90, in=270] (2.3, 3);
        \draw[pattern=north east lines, draw=none] (0.7, 2) to[out=90, in=270] (1.7, 3) -- (2.3, 3) to[out=270, in=90] (1.3, 2) -- (0.7, 2);
        \draw[thick] (0.7, 3) -- (0.7, 5);
        \draw[thick] (1.3, 3) -- (1.3, 5);
        \draw[pattern=north east lines, draw=none] (0.7, 3) -- (0.7, 5) -- (1.3, 5) -- (1.3, 3) -- (0.7, 3);
        \draw[thick] (1.7, 3) -- (1.7, 3.3) to[out=90, in=270] (2.3, 3.9);
        \draw[line width=2.5mm, white] (2.3, 3.3) to[out=90, in=270] (1.7, 3.9);
        \draw[thick] (2.3, 3) -- (2.3, 3.3) to[out=90, in=270] (1.7, 3.9) to[out=90, in=270] (2.3, 4.7) -- (2.3, 5);
        \draw[line width=2.5mm, white](2.3, 3.9) to[out=90, in=270] (1.7, 4.7);
        \draw[thick] (2.3, 3.9) to[out=90, in=270] (1.7, 4.7) -- (1.7, 5);
        \draw[pattern=north east lines, draw=none] (1.7, 3) -- (1.7, 3.3) to[out=90, in=225] (2, 3.6) to[out=315, in=90] (2.3, 3.3) -- (2.3, 3) -- (1.7, 3);
        \draw[pattern=north west lines, draw=none] (2, 3.6) to[out=135, in=270] (1.73, 3.9) to[out=90, in=225] (2, 4.3) to[out=315, in=90] (2.27, 3.9) to[out=270, in=45] (2, 3.6);
        \draw[pattern=north east lines, draw=none] (1.7, 5) -- (1.7, 4.7) to[out=270, in=135] (2, 4.3) to[out=45, in=270] (2.3, 4.7) -- (2.3, 5) -- (1.7, 5);
        \draw[thick] (1.7, 5) to[out=90, in=270] (0.7, 6);
        \draw[thick] (2.3, 5) to[out=90, in=270] (1.3, 6);
        \draw[pattern=north east lines, draw=none] (1.7, 5) to[out=90, in=270] (0.7, 6) -- (1.3, 6) to[out=270, in=90] (2.3, 5) -- (1.7, 5);
        \draw[line width=3.7mm, white] (1, 5) to[out=90, in=270] (2, 6);
        \draw[thick] (0.7, 5) to[out=90, in=270] (1.7, 6);
        \draw[thick] (1.3, 5) to[out=90, in=270] (2.3, 6);
        \draw[pattern=north east lines, draw=none] (0.7, 5) to[out=90, in=270] (1.7, 6) -- (2.3, 6) to[out=270, in=90] (1.3, 5) -- (0.7, 5);
    \end{tikzpicture}
    \begin{tikzpicture}
        \node at (0 , 3.1) {$=$};
        \draw[thick] (1.7, -0.2) -- (1.7, 0) to[out=90, in=270] (0.7, 1);
        \draw[thick] (2.3, -0.2) -- (2.3, 0) to[out=90, in=270] (1.3, 1);
        \draw[pattern=north east lines, draw=none] (1.7, -0.2) -- (1.7, 0) to[out=90, in=270] (0.7, 1) -- (1.3, 1) to[out=270, in=90] (2.3, 0) -- (2.3, -0.2) -- (1.7, -0.2);
        \draw[line width=3.7mm, white] (1, 0) to[out=90, in=270] (2, 1);
        \draw[thick] (0.7, -0.2) -- (0.7, 0) to[out=90, in=270] (1.7, 1);
        \draw[thick] (1.3, -0.2) -- (1.3, 0) to[out=90, in=270] (2.3, 1);
        \draw[pattern=north east lines, draw=none] (0.7, -0.2) -- (0.7, 0) to[out=90, in=270] (1.7, 1) -- (2.3, 1) to[out=270, in=90] (1.3, 0) -- (1.3, -0.2) -- (0.7, -0.2);
        \draw[thick] (0.7, 2) -- (0.7, 2.3) to[out=90, in=270] (1.3, 2.9);
        \draw[line width=2.5mm, white] (1.3, 2.3) to[out=90, in=270] (0.7, 2.9);
        \draw[thick] (1.3, 2) -- (1.3, 2.3) to[out=90, in=270] (0.7, 2.9) to[out=90, in=270] (1.3, 3.7) -- (1.3, 4);
        \draw[line width=2.5mm, white](1.3, 2.9) to[out=90, in=270] (0.7, 3.7);
        \draw[thick] (1.3, 2.9) to[out=90, in=270] (0.7, 3.7) -- (0.7, 4);
        \draw[pattern=north east lines, draw=none] (0.7, 2) -- (0.7, 2.3) to[out=90, in=225] (1, 2.6) to[out=315, in=90] (1.3, 2.3) -- (1.3, 2) -- (0.7, 2);
        \draw[pattern=north west lines, draw=none] (1, 2.6) to[out=135, in=270] (0.73, 2.9) to[out=90, in=225] (1, 3.3) to[out=315, in=90] (1.27, 2.9) to[out=270, in=45] (1, 2.6);
        \draw[pattern=north east lines, draw=none] (0.7, 4) -- (0.7, 3.7) to[out=270, in=135] (1, 3.3) to[out=45, in=270] (1.3, 3.7) -- (1.3, 4) -- (0.7, 4);
        \draw[thick] (1.7, 2) -- (1.7, 2.3) to[out=90, in=270] (2.3, 2.9);
        \draw[line width=2.5mm, white] (2.3, 2.3) to[out=90, in=270] (1.7, 2.9);
        \draw[thick] (2.3, 2) -- (2.3, 2.3) to[out=90, in=270] (1.7, 2.9) to[out=90, in=270] (2.3, 3.7) -- (2.3, 4);
        \draw[line width=2.5mm, white](2.3, 2.9) to[out=90, in=270] (1.7, 3.7);
        \draw[thick] (2.3, 2.9) to[out=90, in=270] (1.7, 3.7) -- (1.7, 4);
        \draw[pattern=north east lines, draw=none] (1.7, 2) -- (1.7, 2.3) to[out=90, in=225] (2, 2.6) to[out=315, in=90] (2.3, 2.3) -- (2.3, 2) -- (1.7, 2);
        \draw[pattern=north west lines, draw=none] (2, 2.6) to[out=135, in=270] (1.73, 2.9) to[out=90, in=225] (2, 3.3) to[out=315, in=90] (2.27, 2.9) to[out=270, in=45] (2, 2.6);
        \draw[pattern=north east lines, draw=none] (1.7, 4) -- (1.7, 3.7) to[out=270, in=135] (2, 3.3) to[out=45, in=270] (2.3, 3.7) -- (2.3, 4) -- (1.7, 4);
        \draw[thick] (1.7, 5-0.2) to[out=90, in=270] (0.7, 6-0.2);
        \draw[thick] (2.3, 5-0.2) to[out=90, in=270] (1.3, 6-0.2);
        \draw[pattern=north east lines, draw=none] (1.7, 5-0.2) to[out=90, in=270] (0.7, 6-0.2) -- (1.3, 6-0.2) to[out=270, in=90] (2.3, 5-0.2) -- (1.7, 5-0.2);
        \draw[line width=3.7mm, white] (1, 5-0.2) to[out=90, in=270] (2, 6-0.2);
        \draw[thick] (0.7, 5-0.2) to[out=90, in=270] (1.7, 6-0.2);
        \draw[thick] (1.3, 5-0.2) to[out=90, in=270] (2.3, 6-0.2);
        \draw[pattern=north east lines, draw=none] (0.7, 5-0.2) to[out=90, in=270] (1.7, 6-0.2) -- (2.3, 6-0.2) to[out=270, in=90] (1.3, 5-0.2) -- (0.7, 5-0.2);
        \draw[thick] (0.7, 1) -- (0.7, 2);
        \draw[thick] (1.3, 1) -- (1.3, 2);
        \draw[thick] (1.7, 1) -- (1.7, 2);
        \draw[thick] (2.3, 1) -- (2.3, 2);
        \draw[thick] (0.7, 4) -- (0.7, 4.8);
        \draw[thick] (1.3, 4) -- (1.3, 4.8);
        \draw[thick] (1.7, 4) -- (1.7, 4.8);
        \draw[thick] (2.3, 4) -- (2.3, 4.8);
        \draw[pattern=north east lines, draw=none] (0.7, 1) -- (0.7, 2) -- (1.3, 2) -- (1.3, 1) -- (0.7, 1);
        \draw[pattern=north east lines, draw=none] (1.7, 1) -- (1.7, 2) -- (2.3, 2) -- (2.3, 1) -- (1.7, 1);
        \draw[pattern=north east lines, draw=none] (0.7, 4) -- (0.7, 4.8) -- (1.3, 4.8) -- (1.3, 4) -- (0.7, 4);
        \draw[pattern=north east lines, draw=none] (1.7, 4) -- (1.7, 4.8) -- (2.3, 4.8) -- (2.3, 4) -- (1.7, 4);
    \end{tikzpicture}
    \begin{tikzpicture}
        \node at (0 , 3.3-0.2) {$=$};
        \draw[thick] (1.7, -0.2) -- (1.7, 0) to[out=90, in=270] (0.7, 1);
        \draw[thick] (2.3, -0.2) -- (2.3, 0) to[out=90, in=270] (1.3, 1);
        \draw[pattern=north east lines, draw=none] (1.7, -0.2) -- (1.7, 0) to[out=90, in=270] (0.7, 1) -- (1.3, 1) to[out=270, in=90] (2.3, 0) -- (2.3, -0.2) -- (1.7, -0.2);
        \draw[line width=3.7mm, white] (1, 0) to[out=90, in=270] (2, 1);
        \draw[thick] (0.7, -0.2) -- (0.7, 0) to[out=90, in=270] (1.7, 1);
        \draw[thick] (1.3, -0.2) -- (1.3, 0) to[out=90, in=270] (2.3, 1);
        \draw[pattern=north east lines, draw=none] (0.7, -0.2) -- (0.7, 0) to[out=90, in=270] (1.7, 1) -- (2.3, 1) to[out=270, in=90] (1.3, 0) -- (1.3, -0.2) -- (0.7, -0.2);
        \draw[thick] (0.7, 1) -- (0.7, 3);
        \draw[thick] (1.3, 1) -- (1.3, 3);
        \draw[pattern=north east lines, draw=none] (0.7, 1) -- (0.7, 3) -- (1.3, 3) -- (1.3, 1) -- (0.7, 1);
        \draw[thick] (1.7, 1) -- (1.7, 1.3) to[out=90, in=270] (2.3, 1.9);
        \draw[line width=2.5mm, white] (2.3, 1.3) to[out=90, in=270] (1.7, 1.9);
        \draw[thick] (2.3, 1) -- (2.3, 1.3) to[out=90, in=270] (1.7, 1.9) to[out=90, in=270] (2.3, 2.7) -- (2.3, 3);
        \draw[line width=2.5mm, white](2.3, 1.9) to[out=90, in=270] (1.7, 2.7);
        \draw[thick] (2.3, 1.9) to[out=90, in=270] (1.7, 2.7) -- (1.7, 3);
        \draw[pattern=north east lines, draw=none] (1.7, 1) -- (1.7, 1.3) to[out=90, in=225] (2, 1.6) to[out=315, in=90] (2.3, 1.3) -- (2.3, 1) -- (1.7, 1);
        \draw[pattern=north west lines, draw=none] (2, 1.6) to[out=135, in=270] (1.73, 1.9) to[out=90, in=225] (2, 2.3) to[out=315, in=90] (2.27, 1.9) to[out=270, in=45] (2, 1.6);
        \draw[pattern=north east lines, draw=none] (1.7, 3) -- (1.7, 2.7) to[out=270, in=135] (2, 2.3) to[out=45, in=270] (2.3, 2.7) -- (2.3, 3) -- (1.7, 3);
        \draw[thick] (1.7, 3) to[out=90, in=270] (0.7, 4);
        \draw[thick] (2.3, 3) to[out=90, in=270] (1.3, 4);
        \draw[pattern=north east lines, draw=none] (1.7, 3) to[out=90, in=270] (0.7, 4) -- (1.3, 4) to[out=270, in=90] (2.3, 3) -- (1.7, 3);
        \draw[line width=3.7mm, white] (1, 3) to[out=90, in=270] (2, 4);
        \draw[thick] (0.7, 3) to[out=90, in=270] (1.7, 4);
        \draw[thick] (1.3, 3) to[out=90, in=270] (2.3, 4);
        \draw[pattern=north east lines, draw=none] (0.7, 3) to[out=90, in=270] (1.7, 4) -- (2.3, 4) to[out=270, in=90] (1.3, 3) -- (0.7, 3);
        \draw[thick] (0.7, 4) -- (0.7, 6-0.2);
        \draw[thick] (1.3, 4) -- (1.3, 6-0.2);
        \draw[pattern=north east lines, draw=none] (0.7, 4) -- (0.7, 6-0.2) -- (1.3, 6-0.2) -- (1.3, 4) -- (0.7, 4);
        \draw[thick] (1.7, 4) -- (1.7, 4.3) to[out=90, in=270] (2.3, 4.9);
        \draw[line width=2.5mm, white] (2.3, 4.3) to[out=90, in=270] (1.7, 4.9);
        \draw[thick] (2.3, 4) -- (2.3, 4.3) to[out=90, in=270] (1.7, 4.9) to[out=90, in=270] (2.3, 5.7) -- (2.3, 6-0.2);
        \draw[line width=2.5mm, white](2.3, 4.9) to[out=90, in=270] (1.7, 5.7);
        \draw[thick] (2.3, 4.9) to[out=90, in=270] (1.7, 5.7) -- (1.7, 6-0.2);
        \draw[pattern=north east lines, draw=none] (1.7, 4) -- (1.7, 4.3) to[out=90, in=225] (2, 4.6) to[out=315, in=90] (2.3, 4.3) -- (2.3, 4) -- (1.7, 4);
        \draw[pattern=north west lines, draw=none] (2, 4.6) to[out=135, in=270] (1.73, 4.9) to[out=90, in=225] (2, 5.3) to[out=315, in=90] (2.27, 4.9) to[out=270, in=45] (2, 4.6);
        \draw[pattern=north east lines, draw=none] (1.7, 6-0.2) -- (1.7, 5.7) to[out=270, in=135] (2, 5.3) to[out=45, in=270] (2.3, 5.7) -- (2.3, 6-0.2) -- (1.7, 6-0.2);
    \end{tikzpicture}
    \begin{tikzpicture}
        \node at (0 , 3.1) {$=$};
        \draw[thick] (1.7, -0.2) -- (1.7, 0) to[out=90, in=270] (0.7, 1);
        \draw[thick] (2.3, -0.2) -- (2.3, 0) to[out=90, in=270] (1.3, 1);
        \draw[pattern=north east lines, draw=none] (1.7, -0.2) -- (1.7, 0) to[out=90, in=270] (0.7, 1) -- (1.3, 1) to[out=270, in=90] (2.3, 0) -- (2.3, -0.2) -- (1.7, -0.2);
        \draw[line width=3.7mm, white] (1, 0) to[out=90, in=270] (2, 1);
        \draw[thick] (0.7, -0.2) -- (0.7, 0) to[out=90, in=270] (1.7, 1);
        \draw[thick] (1.3, -0.2) -- (1.3, 0) to[out=90, in=270] (2.3, 1);
        \draw[pattern=north east lines, draw=none] (0.7, -0.2) -- (0.7, 0) to[out=90, in=270] (1.7, 1) -- (2.3, 1) to[out=270, in=90] (1.3, 0) -- (1.3, -0.2) -- (0.7, -0.2);
        \draw[thick] (1.7, 2) to[out=90, in=270] (0.7, 3);
        \draw[thick] (2.3, 2) to[out=90, in=270] (1.3, 3);
        \draw[pattern=north east lines, draw=none] (1.7, 2) to[out=90, in=270] (0.7, 3) -- (1.3, 3) to[out=270, in=90] (2.3, 2) -- (1.7, 2);
        \draw[line width=3.7mm, white] (1, 2) to[out=90, in=270] (2, 3);
        \draw[thick] (0.7, 2) to[out=90, in=270] (1.7, 3);
        \draw[thick] (1.3, 2) to[out=90, in=270] (2.3, 3);
        \draw[pattern=north east lines, draw=none] (0.7, 2) to[out=90, in=270] (1.7, 3) -- (2.3, 3) to[out=270, in=90] (1.3, 2) -- (0.7, 2);
        \draw[thick] (0.7, 2+1.8) -- (0.7, 2.3+1.8) to[out=90, in=270] (1.3, 2.9+1.8);
        \draw[line width=2.5mm, white] (1.3, 2.3+1.8) to[out=90, in=270] (0.7, 2.9+1.8);
        \draw[thick] (1.3, 2+1.8) -- (1.3, 2.3+1.8) to[out=90, in=270] (0.7, 2.9+1.8) to[out=90, in=270] (1.3, 3.7+1.8) -- (1.3, 4+1.8);
        \draw[line width=2.5mm, white](1.3, 2.9+1.8) to[out=90, in=270] (0.7, 3.7+1.8);
        \draw[thick] (1.3, 2.9+1.8) to[out=90, in=270] (0.7, 3.7+1.8) -- (0.7, 4+1.8);
        \draw[pattern=north east lines, draw=none] (0.7, 2+1.8) -- (0.7, 2.3+1.8) to[out=90, in=225] (1, 2.6+1.8) to[out=315, in=90] (1.3, 2.3+1.8) -- (1.3, 2+1.8) -- (0.7, 2+1.8);
        \draw[pattern=north west lines, draw=none] (1, 2.6+1.8) to[out=135, in=270] (0.73, 2.9+1.8) to[out=90, in=225] (1, 3.3+1.8) to[out=315, in=90] (1.27, 2.9+1.8) to[out=270, in=45] (1, 2.6+1.8);
        \draw[pattern=north east lines, draw=none] (0.7, 4+1.8) -- (0.7, 3.7+1.8) to[out=270, in=135] (1, 3.3+1.8) to[out=45, in=270] (1.3, 3.7+1.8) -- (1.3, 4+1.8) -- (0.7, 4+1.8);
        \draw[thick] (1.7, 2+1.8) -- (1.7, 2.3+1.8) to[out=90, in=270] (2.3, 2.9+1.8);
        \draw[line width=2.5mm, white] (2.3, 2.3+1.8) to[out=90, in=270] (1.7, 2.9+1.8);
        \draw[thick] (2.3, 2+1.8) -- (2.3, 2.3+1.8) to[out=90, in=270] (1.7, 2.9+1.8) to[out=90, in=270] (2.3, 3.7+1.8) -- (2.3, 4+1.8);
        \draw[line width=2.5mm, white](2.3, 2.9+1.8) to[out=90, in=270] (1.7, 3.7+1.8);
        \draw[thick] (2.3, 2.9+1.8) to[out=90, in=270] (1.7, 3.7+1.8) -- (1.7, 4+1.8);
        \draw[pattern=north east lines, draw=none] (1.7, 2+1.8) -- (1.7, 2.3+1.8) to[out=90, in=225] (2, 2.6+1.8) to[out=315, in=90] (2.3, 2.3+1.8) -- (2.3, 2+1.8) -- (1.7, 2+1.8);
        \draw[pattern=north west lines, draw=none] (2, 2.6+1.8) to[out=135, in=270] (1.73, 2.9+1.8) to[out=90, in=225] (2, 3.3+1.8) to[out=315, in=90] (2.27, 2.9+1.8) to[out=270, in=45] (2, 2.6+1.8);
        \draw[pattern=north east lines, draw=none] (1.7, 4+1.8) -- (1.7, 3.7+1.8) to[out=270, in=135] (2, 3.3+1.8) to[out=45, in=270] (2.3, 3.7+1.8) -- (2.3, 4+1.8) -- (1.7, 4+1.8);
        \draw[thick] (0.7, 1) -- (0.7, 2);
        \draw[thick] (1.3, 1) -- (1.3, 2);
        \draw[thick] (1.7, 1) -- (1.7, 2);
        \draw[thick] (2.3, 1) -- (2.3, 2);
        \draw[thick] (0.7, 3) -- (0.7, 3.8);
        \draw[thick] (1.3, 3) -- (1.3, 3.8);
        \draw[thick] (1.7, 3) -- (1.7, 3.8);
        \draw[thick] (2.3, 3) -- (2.3, 3.8);
        \draw[pattern=north east lines, draw=none] (0.7, 1) -- (0.7, 2) -- (1.3, 2) -- (1.3, 1) -- (0.7, 1);
        \draw[pattern=north east lines, draw=none] (1.7, 1) -- (1.7, 2) -- (2.3, 2) -- (2.3, 1) -- (1.7, 1);
        \draw[pattern=north east lines, draw=none] (0.7, 3) -- (0.7, 3.8) -- (1.3, 3.8) -- (1.3, 3) -- (0.7, 3);
        \draw[pattern=north east lines, draw=none] (1.7, 3) -- (1.7, 3.8) -- (2.3, 3.8) -- (2.3, 3) -- (1.7, 3);
    \end{tikzpicture}
    \caption[Equivalent identities for categorical twists]{Equivalent identities for categorical twists.}
    \label{fig:TwistEqualities}
\end{figure}
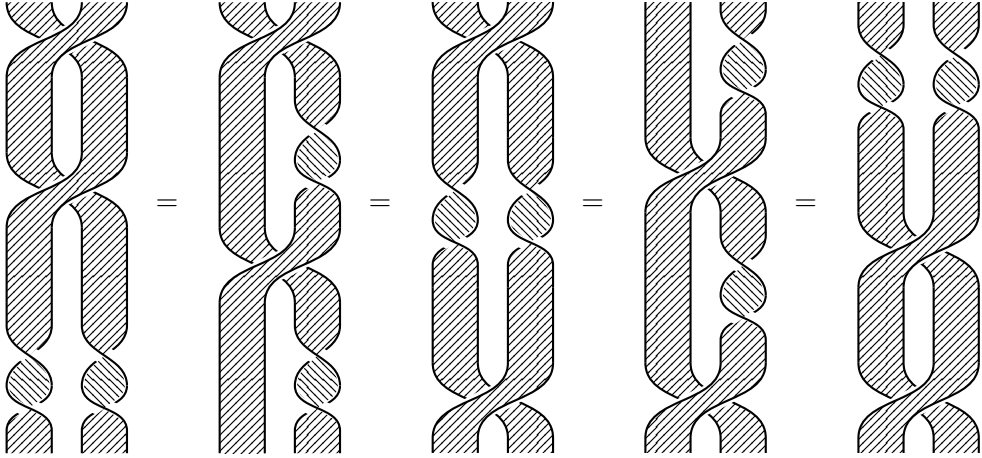

We can now provide the definition of a ribbon category.
\begin{definition}
    A \textit{ribbon category}, $\mathcal{C}$, is a rigid braided monoidal category equipped with a twist  such that for all $X \in \mathcal{C}$,
    $$\left( \theta_X \otimes \operatorname{Id}_{X^*} \right) \circ \operatorname{coev} = \left( \operatorname{Id}_X  \otimes \theta_{X^*} \right) \circ \operatorname{coev}.$$
\end{definition}
You can read more about these kinds of categories in detail in \cite{MR1300632, MR3242743, MR1321145, MR1381692}.

Lastly, we descend from the level of categories to the level of algebras in order to provide a few other useful definitions.

\begin{definition}
    A \textit{bialgebra} over a field $\mathbb{K}$ is a tuple $(A, \mu, \iota, \Delta, \varepsilon)$ such that $(A, \mu, \iota)$ is an algebra over $\mathbb{K}$, $(A, \Delta, \varepsilon)$ is a coalgebra over $\mathbb{K}$, and the following diagrams commute.
    \begin{center}
        \begin{tikzcd}
            A \otimes A \arrow[rr, "\varepsilon \otimes \varepsilon"] \arrow[dd, "\mu"'] & & \mathbb{K} \otimes \mathbb{K} \arrow[dd, "\sim" style={anchor=south, rotate=90, inner sep=.5mm}] \arrow[rr, "\iota \otimes \iota"] & & A \otimes A \\
            & & & & \\
            A \arrow[rr, "\varepsilon"] & & \mathbb{K} \arrow[rr, "\iota"] & & A \arrow[uu, "\Delta"']
        \end{tikzcd}
    \end{center}
    \begin{center}
        \begin{tikzcd}
            A \otimes A \arrow[rr, "\Delta \otimes \Delta"] \arrow[dd, "\Delta \circ \mu"'] & & A \otimes A \otimes A \otimes A \arrow[dd, "\operatorname{id} \otimes \tau \otimes \operatorname{id}"] \\
            &  & \\
            A \otimes A &  & A \otimes A \otimes A \otimes A \arrow[ll, "\mu \otimes \mu"']
        \end{tikzcd}
        \quad\quad\quad
        \begin{tikzcd}
            \mathbb{K} \arrow[rd, "\iota"] \arrow[dd, "\operatorname{id}"'] & \\
            & A \arrow[ld, "\varepsilon"] \\
            \mathbb{K} &
        \end{tikzcd}
    \end{center}
\end{definition}

\begin{definition}
    A coalgebra is said to be \textit{cocommutative} if $\tau \circ \Delta = \Delta$, where $\tau$ is the transposition map $\tau: a \otimes b \mapsto b \otimes a$.
\end{definition}

Incorporating specific additional structure upgrades a bialgebra to a Hopf algebra, which is a primary object of interest throughout this work.
\begin{definition}
    A \textit{Hopf algebra}, $A$, is a bialgebra equipped with a $\mathbb{K}$-linear map, $S: A \to A$, called the antipode map, such that the following diagrams commute.
    \begin{center}
        \begin{tikzcd}
            A \arrow[rr, "\Delta"] \arrow[dd, "\iota \circ \varepsilon"'] & & A \otimes A \arrow[dd, "\operatorname{id} \otimes S"] \\
            & & \\
            A & & A \otimes A \arrow[ll, "\mu"']  
        \end{tikzcd}\quad\quad
        \begin{tikzcd}
            A \arrow[rr, "\Delta"] \arrow[dd, "\iota \circ \varepsilon"'] & & A \otimes A \arrow[dd, "S \otimes \operatorname{id}"] \\
            & & \\
            A & & A \otimes A \arrow[ll, "\mu"']  
        \end{tikzcd}
    \end{center}
\end{definition}
Although not every bialgebra is a Hopf algebra, it's a simple and quick exercise to show that these commuting diagrams force $S$ to be unique, when it does exist.

\begin{definition}
    Let $A$ be a Hopf algebra. We define \textit{$A$-Rep} to be the category whose objects are finite dimensional $A$-modules and morphisms are $A$-module homomorphisms.
\end{definition}

\section{Quantum Groups}\label{section:QuantumGroups}

When we first start learning about multiplication in mathematics, everything we concerned ourselves with would always commute. For example $7 \cdot 3$ is the same as $3 \cdot 7$. As we progress in our mathematical journeys, we eventually learn about certain objects that don't commute at all, like matrices or free groups. When attempting to multiply matrices of different dimensions, $AB$ may not even be well-defined where as $BA$ might be. Quantum commuting lies somewhere in-between. Roughly speaking, we say $A$ and $B$ $q$-commute if $AB=qBA$ for some $q \in \mathbb{C}^\times$. The idea is as we limit $q \to 1$, we get back full commutivity.

Using this notion, we can take an algebra and throw in some $q$'s to try to give it a \emph{quantum deformation}. For example, take the free algebra in two variables $\mathbb{C}\langle X,Y \rangle$. If we want $X$ and $Y$ to commute in this algebra, we need to additionally quotient out by the ideal generated by $XY - YX$. We often refer to this algebra as a \emph{torus} to indicate that both $X$ and $Y$ are commutative and invertible. In order to ``deform'' or ``quantize'' this object, we slightly change this quotient to be $XY - qYX$ for some $q \in \mathbb{C}^\times$ instead and can now say that $X$ and $Y$ \textit{$q$-commute}.

Unfortunately, there is no universally accepted definition for the term \emph{quantum group}. However, there are some generally agreed upon properties that that these objects should contain. In particular, quantum groups should include deformations of particular objects that relate to algebraic groups. The main examples most mathematicians tend to care about are quantum Lie groups.

Unless specified otherwise, this thesis will exclusively use $\mathbb{C}$ as the underlying ring and $q \in \mathbb{C}^\times$. However, it should be noted that much of this story can be adapted for any Noetherian domain, $\mathcal{R}$, and using the ring $\mathcal{R}[q^{\pm 1}]$ or a localization of it over polynomials in $q$.

A \textit{quantum Lie group} is a non-commutative, non-cocommutative Hopf algebra, $U_q(\mathfrak{g})$, which is a deformation of the universal enveloping algebra, $U(\mathfrak{g})$. We will primarily be working with the Lie algebra $\mathfrak{g} = \mathfrak{sl}_2(\mathbb{C})$ in this thesis. Fix an element $q \in \mathbb{C}^\times$ and suppose it's not a root of unity.\footnote{Many of the results throughout this work fail when $q$ is a root of unity. Although many people study these kinds of problems at roots of unity, we will not consider this case as it is quite a different beast to work with.} The corresponding quantized enveloping algebra is defined as

$$U_q\left(\mathfrak{sl}_2\right) = \frac{\mathbb{C} \left[E, F, K^{\pm 1}\right]}{\left( \begin{gathered} KEK^{-1} = q^2E\\ KFK^{-1} = q^{-2}F\\ [E,F] = \frac{K - K^{-1}}{q - q^{-1}}\end{gathered} \right)}$$
The set of monomials $\{ F^s K^n E^r \}$ with $r,s \in \mathbb{N}$ and $n \in \mathbb{Z}$ forms a PBW basis for $U_q(\mathfrak{sl}_2)$ (see Theorem 1.5 in \cite{MR1359532} for details). As we can grade a free algebra using arbitrary degrees for each generator, we typically consider $\operatorname{deg}(E)=1$, $\operatorname{deg}(F)=-1$, and $\operatorname{deg}(K) = \operatorname{deg}(K^{-1}) = 0$, giving this algebra a $\mathbb{Z}$-grading. In particular, a monomial $F^sK^nE^r$ would have degree $r-s$.

In the theory of quantized enveloping algebras, there are many parallels to formulas found in the theory of universal enveloping algebras. In particular, we often replace the traditional binomial coefficients with their quantum counterparts, known as Gaussian binomial coefficients. For any $m \in \mathbb{N}$, we first define
$$[m] = \frac{q^{m} - q^{-m}}{q-q^{-1}} \quad\quad\quad \text{ and } \quad\quad \quad [m]! = [1] [2] \cdots [m]$$
with $[0]! = 1$. The \textit{Gaussian binomial coefficients} are then defined as:
$$\genfrac[]{0pt}{0}{m}{k} = \displaystyle\frac{[m]!}{[k]![m-k]!}.$$
This notation will be used throughout the remainder of this section.

As is the general trend for $\mathfrak{sl}_2$-modules, we can think of $F$ and $E$ as raising and lowering operators, while $K$ can be thought of, in some sense, as a diagonalizable operator since it's invertible. It turns out that when $M$ is a finite dimensional $U_q(\mathfrak{sl}_2)$-module, $E$ and $F$ \emph{must} act as raising and lowering operators and so $M$ must be fully torsion. Thus, when $M$ is finite dimensional, there are always integers $r, s > 0$ such that $E^rM = F^sM = 0$. Moreover, any finite dimensional $U_q(\mathfrak{sl}_2)$-module is a direct sum of all of its weight spaces and all the weights are of the form $\pm q^{n}$ for some $n \in \mathbb{Z}$. For further details, see \cite{MR1359532}.

\begin{theorem}[2.6 in \cite{MR1359532}]\label{theorem:Uqsl2FinModClassification}
    For every integer $n \geq 0$, there exists a simple $U_q(\mathfrak{sl}_2)$-module, $L(n,+)$, with basis $\{ m_0, \cdots, m_n \}$ such that for all $0 \leq i \leq n$
    \begin{align*}
        Km_i &= q^{n-2i}m_i, \\
        Fm_i &= \begin{cases}
            m_{i+1}, & \text{ if } i < n \\
            0,\phantom{-[i][n+1-i]m_{i-1},} & \text{ if } i = n
        \end{cases}\\
        Em_i &= \begin{cases}
            [i][n+1-i]m_{i-1},\phantom{-0,} & \text{ if } i > 0 \\
            0, & \text{ if } i = 0
        \end{cases}
    \end{align*}
    as well as a simple $U_q(\mathfrak{sl}_2)$-module, $L(n,-)$, with basis $\{ m_0^\prime, \cdots, m_n^\prime \}$ such that
    \begin{align*}
        Km_i^\prime &= -q^{n-2i}m_i^\prime, \\
        Fm_i^\prime &= \begin{cases}
            m_{i+1}^\prime, & \text{ if } i < n \\
            0,\phantom{-[i][n+1-i]m_{i-1},} & \text{ if } i = n
        \end{cases}\\ 
        Em_i^\prime &= \begin{cases}
            -[i][n+1-i]m_{i-1}^\prime,\phantom{0,} & \text{ if } i > 0 \\
            0, & \text{ if } i = 0
        \end{cases}
    \end{align*}
\end{theorem}

Every simple $U_q(\mathfrak{sl}_2)$-module of dimension $n+1$ is isomorphic to either $L(n,+)$ or $L(n,-)$, giving us a full classification of $U_q(\mathfrak{sl}_2)$-Rep (see \cite{MR1300632, MR1359532}). These two modules are isomorphic if and only if the underlying base field has characteristic $2$. Although $L(n,+)$ is not isomorphic to $L(n,-)$ in our case as we're working over $\mathbb{C}$, we will only consider $L(n,+)$ and just keep a note in the back untouched recesses of our minds that there are actually two of them. Furthermore, we will condense and abuse notation by using $L(n)$ in place of $L(n,+)$ throughout. $L(n)$ are analogues of highest-weight modules of $\mathfrak{sl}_2$.

A quick observation shows us that the smallest module where $E$, $F$, and $K$ don't act trivially is $L(1)$. The two dimensional module $L(1)$ is often called the \emph{standard representation} or \emph{fundamental representation} of $U_q(\mathfrak{sl}_2)$ as it comes from the particular two dimensional representation $\rho : U_q(\mathfrak{sl}_2) \to \operatorname{End}(V)$ defined by
\begin{align*}
    \rho(K) &= \begin{pmatrix} q & 0\\ 0 & q^{-1} \end{pmatrix} & \rho(F) &= \begin{pmatrix} 0 & 0\\ 1 & 0 \end{pmatrix} & \rho(E) &= \begin{pmatrix} 0 & 1\\ 0 & 0 \end{pmatrix},
\end{align*}
which matches up with the module structure of $L(1)$ defined above.

\begin{align*}
    Km_0 &= qm_0 & Km_1 &= q^{-1}m_1\\
    Fm_0 &= m_1 & Fm_1 &= 0\\
    Em_0 &= 0 & Em_1 &= m_0
\end{align*}
Clearly, this representation is closely related to the standard representation of $U(\mathfrak{sl}_2)$, i.e. $\mathfrak{sl}_2(\mathbb{C})$.
\begin{align*}
    \rho(h) &= \begin{pmatrix} 1 & 0\\ 0 & -1 \end{pmatrix} & \rho(f) &= \begin{pmatrix} 0 & 0\\ 1 & 0 \end{pmatrix} & \rho(e) &= \begin{pmatrix} 0 & 1\\ 0 & 0 \end{pmatrix}
\end{align*}

It follows from the fact that $U_q(\mathfrak{sl}_2)$ is a Hopf algebra that the category of finite-dimensional $U_q(\mathfrak{sl}_2)$ representations is closed under tensor products. In particular, we have that
\begin{align}\label{eq:Ln}
    L(n) \otimes L(m) \cong \bigoplus_{k=0}^{\min \{ m,n \}} L(m+n-2k).
\end{align}
By equation \ref{eq:Ln}, we get that $L(2) \subset L(2) \oplus L(0) \cong L(1) \otimes L(1)$ and $L(m+1) \subset L(m+1) \oplus L(m-1) \cong L(m) \otimes L(1)$. By induction, every finite dimensional $U_q(\mathfrak{sl}_2)$-module can be embedded inside a sufficiently high enough tensor power of $L(1)$. Specifically, for any non-negative integer $n$, $L(n) \subset L(1)^{\otimes n}$. With the appropriate projector, we can always take this tensor power and project onto the corresponding component we are interested in. For this reason, $L(1)$ is called a \textit{tensor generator} in this category. Specifically, we have particular inclusion and projection maps that intertwine the action of $U_q(\mathfrak{sl}_2)$. Recall that $\{ m_0, m_1, \cdots, m_{n} \}$ is a basis for $L(n)$. These inclusions and projections are
\begin{gather}
    \label{eq:JWP1}\iota_n : L(n) \hookrightarrow L(1)^{\otimes n} \\
    m_{k} \mapsto \genfrac[]{0pt}{0}{n}{k}^{-1} \sum_{\substack{(\varepsilon_1, \cdots, \varepsilon_{k}) \in \{ 0, 1 \}^n,\\ \sum_{\ell = 1}^{n} \varepsilon_\ell = k}} q^{\displaystyle\sum_{i<j} \delta_{\varepsilon_{i}, \varepsilon_{j}+1}} m_{\varepsilon_{1}} \otimes \cdots \otimes m_{\varepsilon_{n}} \notag
\end{gather}
\begin{gather}
    \label{eq:JWP2}\pi_n : L(1)^{\otimes n} \twoheadrightarrow L(n) \\
    (m_{\varepsilon_1} \otimes \cdots \otimes m_{\varepsilon_n}) \mapsto q^{\displaystyle\sum_{i<j} -\delta_{\varepsilon_{i}+1, \varepsilon_{j}}} m_{\sum_{k=1}^n \varepsilon_k}. \notag
\end{gather}
One can check that the operator $p_n = \iota_n \circ \pi_n$, which is often called the \textit{Jones-Wenzl projector}, satisfies $p_n \circ p_n = p_n$, and so this is indeed a projector.\footnote{The definitions provided here for $\pi_n$ and $\iota_n$ were translated from Frenkel and Khovanov's work in \cite{MR1446615} to better fit our notation.}

\section{Ribbon Hopf Algebras}\label{section:RibbonHopfAlg}

We mentioned before that $U_q(\mathfrak{sl}_2)$ is a non-commutative non-cocommutative Hopf algebra. The algebra structure is clear and the rest of its Hopf structure is as follows:
\begin{align*}
    \Delta(E) &= E \otimes 1 + K \otimes E & \varepsilon(E) &= 0 & S(E) &= K^{-1}E\\
    \Delta(F) &= F \otimes K^{-1} + 1 \otimes F & \varepsilon(F) &= 0 & S(F) &= -FK\\
    \Delta(K) &= K \otimes K & \varepsilon(K) &= 1 & S(K) &= K^{-1}
\end{align*}

Throughout this section, we will be using \emph{sumless Sweedler notation}: $\Delta(u) = u_{(1)} \otimes u_{(2)}$. This also means that we may also write something like $u_{(3)}$ without further explanation as $U_q(\mathfrak{sl}_2)$ is coassociative and so $\left( \operatorname{id} \otimes \Delta \right) \circ \Delta = \left( \Delta \otimes \operatorname{id} \right) \circ \Delta$.

Let $M$ and $N$ be $U_q(\mathfrak{sl}_2)$-modules, $u \in U_q(\mathfrak{sl}_2)$, $m \in M$, and $n \in N$. $U_q(\mathfrak{sl}_2)$ has a natural action on the tensor product, $M \otimes N$, defined using the coproduct:
$$u \cdot (m \otimes n) := \Delta(u) \cdot (m \otimes n),$$
making $M \otimes N$ a $U_q(\mathfrak{sl}_2)$-module as well.
Additionally, $M^* := \operatorname{Hom}_{\mathbb{C}}(M, \mathbb{C})$, can also be made into a (left) $U_q(\mathfrak{sl}_2)$-module by the action
$$(u \cdot f)(m) = f\left(S(u) \cdot m\right).$$

We can look at $M^*$ as the left and right dual of $M$. However, both corresponding evaluation homomorphisms can't be as simple as plugging $m \in M$ into $f \in M^*$ as $U_q(\mathfrak{sl}_2)$ is not cocommutative. In particular, the map $f \otimes m \to f(m)$ is a homomorphism of $U_q(\mathfrak{sl}_2)$-modules, however, $m \otimes f \mapsto f(m)$ is not a homomorphism in general. 

For a moment, let's suppose this is a $U_q(\mathfrak{sl}_2)$-module homomorphism and prod the situation a bit to see what happens. When we combine these actions and consider the evaluation map corresponding to the left dual, we see that
\begin{align*}
    u \cdot \operatorname{ev}\left( f \otimes m \right) &= \operatorname{ev}\left( u \cdot f \otimes m ) \right)\\
    &= \operatorname{ev}\left( u_{(1)}f \otimes u_{(2)}m \right)\\
    &= u_{(1)} f(u_{(2)} \cdot m)\\
    &= f( S(u_{(1)}) u_{(2)} m)\\
    &= f\left( (\mu \circ ( S \otimes \operatorname{id} ) \circ \Delta)(u) \cdot m \right).
\end{align*}
However, when we repeat this for the right dual we get
\begin{align*}
    u \cdot \operatorname{ev}\left( m \otimes f \right) &= \operatorname{ev}\left( u \cdot m \otimes f ) \right)\\
    &= \operatorname{ev}\left( u_{(1)}m \otimes u_{(2)}f \right)\\
    &= u_{(2)} f(u_{(1)} \cdot m)\\
    &= f( S(u_{(2)}) u_{(1)} m)\\
    &= f\left( (\mu \circ ( S \otimes \operatorname{id} ) \circ \tau \circ \Delta)(u) \cdot m \right).
\end{align*}

Since $\tau \circ \Delta \neq \Delta$, we encounter a contradiction. Furthermore, the uniqueness of $S$ means that the complications arising from the lack of cocommutativity in $U_q(\mathfrak{sl}_2)$ cannot be easily remedied. However, this should be viewed as a feature rather than a bug, as quantum groups were intentionally created to exhibit this defect. Moreover, if we instead define our evaluation map to be $m \otimes f \mapsto f( K^{-1} m)$, this does work as a proper evaluation map compatible with the $U_{q}(\mathfrak{sl}_2)$ action. Although the map $\tau : M^* \otimes M \to M \otimes M^*$ is not a proper $U_q(\mathfrak{sl}_2)$-module homomorphism, there must exist an isomorphism when $M$ is finite dimensional by a simple dimension counting argument. This observation indicates that $U_q(\mathfrak{sl}_2)$-Rep must be braided but not symmetric.

\begin{definition}
    A \textit{quasitriangular Hopf algebra} is a is a pair $(A, R)$ where $A$ is a Hopf algebra and $R \in A \otimes A$ such that $R$ is invertible and
    \begin{align*}
        (\Delta \otimes \operatorname{id}) R &= R_{13} R_{23}\\
        (\operatorname{id} \otimes \Delta) R &= R_{13} R_{12}\\
        (\tau \circ \Delta)(a) &= R (\Delta(a)) R^{-1}
    \end{align*}
    for all $a \in A$ where $\tau$ denotes the transposition map $\tau: a \otimes b \mapsto b \otimes a$ and $R_{ij}$ is the tensor triple with $R$ in the $i$th and $j$th factors and $1$ in the other.
\end{definition}

A Hopf algebra being quasitriangular corresponds to the ``braiding" in the notion of a braided monoidal category. In particular, $R$ is one of the isomorphisms that provides this braiding. While $R$ is not unique in general, we will demonstrate how to compute a family of examples in $U_q(\mathfrak{sl}_2)$.

First define $\theta_n = (-1)^n q^{-n(n-1)/2} \frac{(q-q^{-1})^n}{[n]!} F^n \otimes E^n$, where $[n]! = \left( \frac{q^n - q^{-n}}{q - q^{-1}} \right) \cdots \left( \frac{q - q^{-1}}{q - q^{-1}} \right)$ and suppose $M$ and $N$ are finite dimensional $U_q(\mathfrak{sl}_2)$-modules. Define the linear transformation $\theta : M \otimes N \to M \otimes N$ as the finite sum\footnote{Although this is not technically a finite sum, we consider it as such as $E$ and $F$ must act nilpotently on finite dimensional modules.} of operators $\theta = \sum_{n \geq 0} \theta_n$.
For example, if $M = N = L(1)$, then $M \otimes N$ is $4$ dimensional. Since $E^2 M = F^2 M = 0$, we get that $\theta = 1 \otimes 1 - (q - q^{-1}) F \otimes E$.
Furthermore, we can express $\theta$ as the following matrix with respect to the basis described in theorem \ref{theorem:Uqsl2FinModClassification}, $\left\{ m_0 \otimes m_0, m_0 \otimes m_1, m_1 \otimes m_0, m_1 \otimes m_1 \right\}$.
$$\theta = \begin{bmatrix} 1 & 0 & 0 & 0\\ 0 & 1 & 0 & 0\\ 0 & q^{-1} - q & 1 & 0\\ 0 & 0 & 0 & 1 \end{bmatrix}.$$

All weights of finite dimensional $U_q(\mathfrak{sl}_2)$-modules are of the form $\pm q^{n}$ for $n \in \mathbb{Z}$. Therefore, all weights for all objects in $U_q(\mathfrak{sl}_2)$-Rep lie in $\Tilde{\Lambda} = \{ \pm q^{n} \, \mid n \in \mathbb{Z} \}$. Define a map $f : \Tilde{\Lambda} \times \Tilde{\Lambda} \to \mathbb{C}^\times$ such that $f(\lambda_1, \lambda_2) = \lambda_1 f(\lambda_1, q^2 \lambda_2) = \lambda_2 f(q^2 \lambda_1, \lambda_2)$ for all $\lambda_1, \lambda_2 \in \Tilde{\Lambda}$ and consider the bilinear map $\Tilde{f}: M \otimes N \to M \otimes N$ defined as $\Tilde{f}(m \otimes n) = f(\lambda_1, \lambda_2) m \otimes n$ where $K m = \lambda_1 m$ and $Kn = \lambda_2 n$ (see chapters $2$ and $3$ in \cite{MR1359532} for more details).

Suppose $f(q, q) = q^{-1}$. Then our formulas for $f$ imply that $f(q, q^{-1}) = qf(q,q) = 1 = f(q^{-1},q)$ and $f(q^{-1}, q^{-1}) = q^{-1}$. Thus, $\Tilde{f}(m_0 \otimes m_0) = \Tilde{f}(m_1 \otimes m_1) = q^{-1}$ and $\Tilde{f}(m_0 \otimes m_1) = \Tilde{f}(m_1 \otimes m_0) = 1$ and so we get a matrix representation
$$\theta \circ \Tilde{f} = \begin{bmatrix} q^{-1} & 0 & 0 & 0\\ 0 & 1 & 0 & 0\\ 0 & q^{-1} - q & 1 & 0\\ 0 & 0 & 0 & q^{-1} \end{bmatrix}.$$

The map $R = \theta \circ \Tilde{f} \circ \tau : M \otimes N \to N \otimes M$ is a nontrivial $U_q(\mathfrak{sl}_2)$-module isomorphism. This is called an $R$-matrix and is a braiding that we've been looking for, making $U_q(\mathfrak{sl}_2)$-Rep a rigid braided monoidal category.

\begin{definition}
    A \textit{ribbon Hopf algebra} is a triple $(A, R, v)$ consisting of a quasitriangular Hopf algebra $(A, R)$ and a central invertible $v \in A$ such that
    \begin{align*}
        S(v) &= v\\
        \varepsilon(v) &= 1\\
        v^2 &= \mu (S \otimes \operatorname{id})(R_{21}) \cdot S\left(\mu (S \otimes \operatorname{id})(R_{21}) \right)\\
        \Delta(v) &= (R_{21} R_{12})^{-1} (v \otimes v)
    \end{align*}
\end{definition}

We call $v$ the \textit{universal twist} of $A$.
For any $X \in A$-Rep, the twist $\theta_v : X \to X$ is defined as $\theta_v(x) = v \cdot x$ and can be thought of as multiplication by $v$. Since $a$ is invertible, $\theta_v$ is an isomorphism and thus is clearly compatible with $B_{Y, X} \circ B_{X, Y} : X \otimes Y \to X \otimes Y$ in the sense of equation \ref{eq:Twist&BraidRel}. This brings us to the following theorem

\begin{theorem}[3.2 in \cite{MR2654259}]
    Let $A$ be a ribbon Hopf algebra. Then the category of finite dimensional $A$-modules, $A$-Rep, is a ribbon category.
\end{theorem}

Once again using $U_q(\mathfrak{sl}_2)$ as our example, we have (partially) outlined that $U_q(\mathfrak{sl}_2)$-Rep possesses the following properties:
\begin{itemize}
    \item Monoidal: The tensor product of $U_q(\mathfrak{sl}_2)$-modules is itself a $U_q(\mathfrak{sl}_2)$-module.
    \item Rigid: Every object in $U_q(\mathfrak{sl}_2)$-Rep has a left and right dual equipped with proper evaluation and coevaluation maps.
    \item Braided: We can construct $R$-matrices on these tensor products.
    \item Twist: There exists an element $\theta \in U_q(\mathfrak{sl}_2)$, that acts on the irreducible representation of highest weight $\lambda$ as scaling by the constant $-q^{-\langle \lambda, \lambda \rangle/2 - \langle \lambda, \rho \rangle}$ where $\rho \in \mathfrak{h}$ such that $\left\langle\alpha_i, \rho\right\rangle = (\alpha_i, \alpha_i)/2$, where $\{\alpha_i\}$ is the usual basis for $\mathfrak{h}^\ast$.\footnote{We'll eventually only care about the scaling factor $-q^{-3/2}$, which will be reparameterized to $-q^{-3}$. Computing $\theta$ becomes messy and so this computation was left out. However, it's not too much work to find this $\theta$ using SageMath.}
\end{itemize}
Therefore, $U_q(\mathfrak{sl}_2)$-Rep is actually a ribbon category! For more in-depth information about quantum groups and their representations, see \cite{MR1359532, MR1321145, MR1381692}.

\section{Skein Modules}\label{section:SkeinMods}

Skein modules arose from the search for particular invariants in low-dimensional topology and representation theory. We construct these objects by considering embeddings of links and tangles in a manifold and imposing certain algebraic conditions on them, known as skein relations. The idea is to form a module or algebra where the basis elements are isotopy classes of knots and links inside the manifold, and the relations reflect specific topological or algebraic properties. In this way, skein modules can be thought of as a generalization of knot polynomials. They arise naturally in the study of the representation theory of quantum groups and provide a rich, unifying framework for investigating knots, character varieties, and representations.

Before we define a skein module, we will first explain the foundations of these relations and the algebraic properties we aim to better understand.

\begin{definition}
    A \textit{ribbon}, or \textit{coupon}, in a $3$-manifold, $M$, is a smooth embedding of $[0,1] \times [0,1]$ into $M$. The image of $\left\{\frac{1}{2}\right\} \times [0,1]$ is called the \textit{core} of the ribbon and the images of $[0,1] \times \{0\}$ and $[0,1] \times \{1\}$ are the \textit{bases} of the ribbon.
\end{definition}

We often do not draw the ``thickness'' of these ribbons since it often suffices to only draw the core. We are able to avoid losing this extra information as long as there are no half-twists in the ribbon. In particular, we are able to observe and distinguish any full twists in the tangle by using loops in their place.
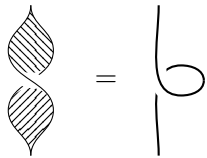
\begin{figure}[h]
    \centering
    \begin{tikzpicture}
        \draw (1, -1) to[out=90, in=270] (0.7, -0.4) to[out=90, in=270] (1.3, 0.4) to[out=90, in=270] (1, 1);
        \draw[line width=2.5mm, white] (1.3, -0.4) to[out=90, in=270] (0.7, 0.4);
        \draw (1, -1) to[out=90, in=270] (1.3, -0.4) to[out=90, in=270] (0.7, 0.4) to[out=90, in=270] (1, 1);
        \draw[pattern=north east lines, draw=none] (1, -1) to[out=90, in=270] (1.3, -0.4) to[out=90, in=0] (1, -0.1) to[out=180, in=90] (0.7, -0.4) to[out=270, in=90] (1, -1);
        \draw[pattern=north west lines, draw=none] (1, 1) to[out=270, in=90] (1.3, 0.4) to[out=270, in=0] (1, 0.1) to[out=180, in=270] (0.7, 0.4) to[out=90, in=270] (1, 1);
        \node at (2, 0) {$=$};
        \draw[thick] (2.7, -1) to[out=90, in=180] (3, 0.2) to[out=0, in=90] (3.3, 0);
        \draw[line width=2.5mm, white] (3, -0.2) to[out=180, in=270] (2.7, 1);
        \draw[thick] (3.3, 0) to[out=270, in=0] (3, -0.2) to[out=180, in=270] (2.7, 1);
    \end{tikzpicture}
    \caption[A blackboard framing of a tangle]{A blackboard framing of a tangle.}
    \label{fig:BBFraming}
\end{figure}
Thus, we will equivalently say that these tangles are equipped with a \textit{framing} in these situations. A \textit{framing} of a tangle, $T$, is a continuous assignment of a vector to each point of $T$, which is not tangent at that point. We additionally impose the condition that the framings may not include half-twists. We say a framed tangle in $\Sigma \times I$ has a \textit{blackboard framing} if the entire framing is embedded orthogonally to $\Sigma$. Every framed link in $\Sigma \times I$ is isotopic to one with a blackboard framing (see figure \ref{fig:BBFraming}). Thus, we may represent ribbons or framed links in $\Sigma \times I$ as link diagrams in $\Sigma$.

A \textit{ribbon graph} (sometimes called \textit{fat graph}) is a directed graph whose edges are ribbons and vertices are coupons. Given a category, $\mathcal{C}$, a \textit{$\mathcal{C}$-colored graph} of a ribbon graph is an assignment of objects in $\mathcal{C}$ to each ribbon and an assignment of morphisms in $\mathcal{C}$ to each coupon, compatible with the assigned objects on the adjacent ribbons.\footnote{Technically, objects are assigned to beginning and ends of ribbons, and the identity map is assigned to (untwisted) ribbons.}

Let $\mathcal{C}$ be a $\mathbb{C}$-linear ribbon category and let $M$ be the filled-in cylinder depicted in figure \ref{fig:RibbonGraphEx}. Consider the oriented $\mathcal{C}$-colored ribbon graph, $\Gamma$, where $A$, $B$, $U$, $V$, $W$, $X$, $Y$ are all objects in $\mathcal{C}$.
\begin{figure}
    \begin{center}
        \begin{tikzpicture}
            \node (u) at (-1.2, 4) {};
            \node (v) at (0, 4.2) {};
            \node (w) at (1.2, 4) {};
            \node (x) at (-0.6, -0.2) {};
            \node (y) at (0.6, -0.2) {};
            \node (f) at (0.6, 1.3) {};
            \node (g) at (1, 2.65) {};
            \node (h) at (-1.2, 3) {};
            \draw (0.5, 1.3) to[out=90, in=0] (-0.3, 3) to[out=180, in=270] (-1.2, 2) -- (h) -- (u);
            \draw[line width=2.5mm, white] (-0.6, 2.3) to[out=90, in=270] (h);
            \draw (y) to[out=90, in=270] (-0.6, 2.3) to[out=90, in=270] (h);
            \draw (h) -- (-1.2, 2.15);
            \draw[line width=2.5mm, white] (x) to[out=90, in=270] (f);
            \draw (x) to[out=90, in=270] (f);
            \draw (f) to[out=90, in=270] (g) to[out=80, in=270] (w);
            \draw (g) to[out=100, in=270] (v);
            \path[thick, tips, ->] (-0.4, 0.3) -- (0.2, 0.65);
            \path[thick, tips, ->] (0.5, 0.5) -- (-0.27, 1.45);
            \path[thick, tips, ->] (h) -- (-0.22, 3);
            \path[thick, tips, ->] (u) -- (-1.2, 3.4);
            \path[thick, tips, ->] (0.6, 1.3) -- (0.86, 2.1);
            \path[thick, tips, ->] (1.2, 2.8) -- (0.54, 3.35);
            \path[thick, tips, ->] (g) -- (1.14, 3.35);
            \draw (0,4) ellipse (2 and 0.5);
            \draw[dashed] (2,0) arc(0:180:2 and 0.5);
            \draw (2,0) arc(0:-180:2 and 0.5);
            \draw (-2,0) -- (-2,4);
            \draw (2,0) -- (2,4);
            \node[circle, inner sep=0.5pt, minimum size=3pt, fill=white] at (-1.2, 4.1) {$U$};
            \node[circle, inner sep=0.5pt, minimum size=3pt, fill=white] at (0, 4.2) {$V$};
            \node[circle, inner sep=0.5pt, minimum size=3pt, fill=white] at (1.2, 4.1) {$W$};
            \node[circle, inner sep=0.5pt, minimum size=3pt, fill=white] at (-0.6, -0.2) {$X$};
            \node[circle, inner sep=0.5pt, minimum size=3pt, fill=white] at (0.6, -0.2) {$Y$};
            \node at (0, 2.6) {$A$};
            \node at (1.1, 2) {$B$};
            \node[draw, fill=white] at (0.6, 1.4) {$f$};
            \node[draw, fill=white] at (1, 2.75) {$g$};
            \node[draw, fill=white] at (-1.2, 2.95) {$h$};
        \end{tikzpicture}
        \caption{An example of a $\mathcal{C}$-colored ribbon graph embedded into a cylinder}
        \label{fig:RibbonGraphEx}
    \end{center}
\end{figure}
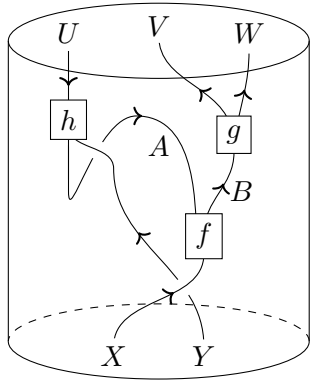
As we introduce maps in $\mathcal{C}$ diagrammatically, we will use the convention that everything is read as moving upwards. Note that we could have just as easily chosen downwards instead, but we need to make some choice so that we may talk about composition of Homs as stacking cylinders. Since we're using the convention of moving upwards, composition in $\mathcal{C}$ is seen as stacking compatible cylinders vertically on top of one another.
At every introduced vertex/coupon, we have a morphism, which we've labeled as $f$, $g$, and $h$, that are, once again, moving upwards. The domains of these morphisms are based on the objects immediately below their vertices and codomains as objects immediately above, both read left to right (the induced orientation of the coupon's bases, $[0,1] \times \{0\}$ and $[0,1] \times \{1\}$). If a ribbon's orientation doesn't flow in the upwards direction, then we consider the object coloring the ribbon as the corresponding dual object instead. This is well-defined as every ribbon category is also rigid. Therefore, the morphisms in our example are defined as

\begin{align*}
    f &: X \to A^\ast \otimes B\\
    g &: B \to V \otimes W\\
    h &: A^\ast \otimes Y \to U^\ast.
\end{align*}

With this particular labeling, we are identifying our ribbon graph $\Gamma$ as an element of $\operatorname{Hom}\left( X \otimes Y, U^\ast \otimes V \otimes W \right)$ via the following theorem. The point of this is to allow us to use a diagrammatic approach to better understand $\mathcal{C}$. In particular, diagrams allow us to express some very complicated formulas in a way that can be intuitively easier to understand.
\begin{theorem}[2.5 in \cite{MR1036112}]\label{theorem:RTFunctor}
    Let $A = (A, R, v)$ be a ribbon Hopf algebra over a field of characteristic $0$ and let $\mathcal{E}$ be the corresponding category of ($A$-Rep)-colored ribbon graphs, where $A$-Rep is the category of finite dimensional representations of $A$. There exists a unique covariant functor
    $$RT: \mathcal{E} \to A\text{-Rep}$$
    such that
    \begin{enumerate}
        \item $RT$ transforms any decoration in $\mathcal{E}$ into the corresponding $A$-module,
        \item $RT$ transforms each graph $\Gamma$ into the corresponding homomorphism,
        \item $RT$ preserves tensor products, i.e. $RT(\Gamma_1 \otimes \Gamma_2) = RT(\Gamma_1) \otimes RT(\Gamma_2)$,
        \item $RT$ maps
        \begin{tikzpicture}[baseline=3]
            \draw[thick, ->] (0,0) to[out=90, in=180] (0.5, 0.5) to[out=0, in=90] (1, 0);
            \node at (0.4, 0.2) {$V$};
        \end{tikzpicture} to the homomorphism $V^\ast \otimes V \to \mathbb{C}$ defined by $f \otimes x \mapsto f(x)$,
        \item $RT$ maps
        \begin{tikzpicture}[baseline=3]
            \draw[thick, <-] (0,0) to[out=90, in=180] (0.5, 0.5) to[out=0, in=90] (1, 0);
            \node at (0.6, 0.2) {$V$};
        \end{tikzpicture} to the homomorphism $V \otimes V^\ast \to \mathbb{C}$ defined by $x \otimes f \mapsto f\left(v^{-1} \mu (S \otimes \operatorname{id})(R_{21})x\right)$,
        \item $RT$ maps
        \begin{tikzpicture}[baseline=10]
            \draw[thick, ->] (1, 0) -- (0, 1);
            \draw[white, line width=2mm] (0.3, 0.3) -- (0.7, 0.7);
            \draw[thick, ->] (0, 0) -- (1, 1);
        \end{tikzpicture} to the $R$-matrix homomorphism.
    \end{enumerate}
    See \cite{MR1036112} for more details.
\end{theorem}

The Reshetikhin-Turaev functor, $RT$, was originally defined only for the case where $M = I \times I \times I$, a cylinder. However, we can attempt to use this definition on an arbitrary 3-manifold by imposing the local relations between morphisms that are in the kernel of the RT functor. Specifically, let $M$ be any $3$-manifold and consider all possible $\mathcal{C}$-colored ribbon graphs that can be embedded into $M$. By restricting to an embedded cylinder within $M$, we can then realize the Reshetikhin-Turaev functor as a map from the space of $\mathbb{C}$-spanned $\mathcal{C}$-colored ribbon graphs in the embedded cylinder to the appropriate Hom set. For instance, in the previous example we have
$$RT: \mathbb{C} [\Gamma] \longrightarrow \operatorname{Hom}\left( X \otimes Y, U \otimes V \otimes W \right)$$
The fact that $RT$ is well-defined is mostly only dependent on $\mathcal{C}$ being a ribbon category. For any classical Lie algebra $\mathfrak{g}$, $U_q(\mathfrak{g})$-Rep is a ribbon category that satisfies the above requirements for Theorem \ref{theorem:RTFunctor}. Therefore, the Reshetikhin-Turaev functor says that we can understand the category $U_q(\mathfrak{g})$-Rep pictorially. In particular, maps in this category can be represented by (a linear combination of) ribbon graphs.

\begin{definition}
    Given an oriented $3$-manifold, $M$, and a Lie algebra, $\mathfrak{g}$, the corresponding \textit{skein module} is
    $$\operatorname{Sk}_{\mathfrak{g},q}(M) := \mathbb{C} \langle \text{closed ribbon graphs in $M$} \rangle / \sim_{RT},$$
    where $\sim_{RT}$ is defined locally on every embedded $I \times I \times I$ and $q \in \mathbb{C}^\times$.
\end{definition}

The equivalence relation, $\sim_{RT}$, are the relations that are picked up through the functor $RT$. This definition is purposefully vague, due to the fact that the relations are dependent on the category.

When mathematicians define skein modules in the literature, it is important to remember that the given relations are not arbitrary rules that happen to match with $A$-Rep. Rather, they are the consequences of the structure and properties of that category. For example, when we restrict ourselves to the case of $\mathfrak{g} = \mathfrak{sl}_2$, our graphical calculus becomes greatly simplified. Firstly, since the fundamental representation, $L(1)$, is a tensor generator of $U_q(\mathfrak{sl}_2)$-Rep, $L(1)$ is the only label we need to consider, allowing us to drop the labeling entirely. Secondly, thanks to quantum Schur-Weyl duality, we can remove any trivalent coupons, meaning we only need to consider links. Finally, the self-duality of $L(1)$ allows us to remove any edge orientations on our ribbon graphs.

Although the first and third simplifications may be clear, the second one is a bit more complicated. It's important to understand why we don't need to consider trivalent coupons in the $\mathfrak{sl}_2$ case, since trivalent and other $n$-valent coupons are necessary in other cases.
For example, consider the subcategory of $U_q(\mathfrak{sl}_2)$-Rep tensor generated by the adjoint representation, $L(2)$ (this ends up corresponding to $PGL_2$). Using \ref{eq:Ln}, we see that $L(0) \oplus L(2) \oplus L(4) \cong L(2) \otimes L(2)$, giving us maps $\pi_2^\prime: L(2) \otimes L(2) \to L(2)$ and $\iota_2^\prime: L(2) \to L(2) \otimes L(2)$. This implies that we need to use trivalent coupons in this scenario.
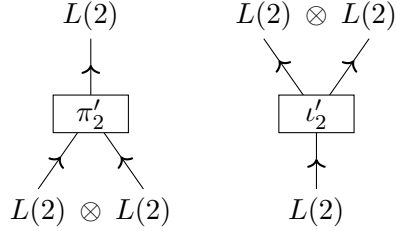
\begin{center}
    \begin{tikzpicture}
        \draw (0.3, -1) -- (1, 0) -- (1.7, -1);
        \draw (1, 0) -- (1, 1);
        \path[thick, tips, ->] (0.3, -1) -- (0.6, -0.57);
        \path[thick, tips, ->] (1.7, -1) -- (1.4, -0.57);
        \path[thick, tips, ->] (1, 0) -- (1, 0.65);
        \draw[draw=black, fill=white] (0.5, -0.25) rectangle (1.5, 0.25);
        \node at (1, 0) {$\pi_2^\prime$};
        \node at (0.3, -1.3) {$L(2)$};
        \node at (1, -1.3) {$\otimes$};
        \node at (1.7, -1.3) {$L(2)$};
        \node at (1, 1.3) {$L(2)$};
        \draw (3.3, 1) -- (4, 0) -- (4.7, 1);
        \draw (4, 0) -- (4, -1);
        \path[thick, tips, ->] (4, 0) -- (3.5, 0.71);
        \path[thick, tips, ->] (4, 0) -- (4.5, 0.71);
        \path[thick, tips, ->] (4, -1) -- (4, -0.6);
        \draw[draw=black, fill=white] (3.5, -0.25) rectangle (4.5, 0.25);
        \node at (4, 0) {$\iota_2^\prime$};
        \node at (3.3, 1.3) {$L(2)$};
        \node at (4, 1.3) {$\otimes$};
        \node at (4.7, 1.3) {$L(2)$};
        \node at (4, -1.3) {$L(2)$};
    \end{tikzpicture}
\end{center}
So why is this not needed in the $\mathfrak{sl}_2$ case?

Since $U_q(\mathfrak{sl}_2)$-Rep is tensor generated by $L(1)$, any map in this category can be understood by pre-composing and post-composing with the proper inclusion and projection maps described in (\ref{eq:JWP1}) and (\ref{eq:JWP2}). For example, if $M$, $M'$, and $N$ are finite dimensional $U_q(\mathfrak{sl}_2)$-modules, extending a map between $N$ and $M \otimes M'$ might look something like the following picture.
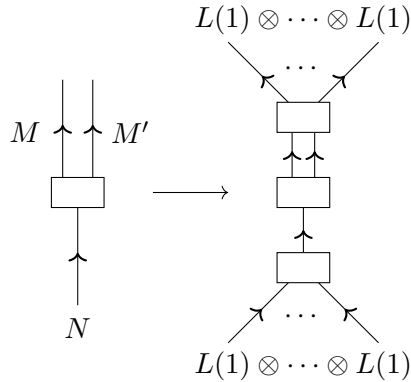
\begin{center}
    \begin{tikzpicture}
        \draw (1, -1.5) -- (1, 0);
        \draw (0.8, 0) -- (0.8, 1.5);
        \draw (1.2, 0) -- (1.2, 1.5);
        \path[thick, tips, ->] (1, -1) -- (1, -0.8);
        \path[thick, tips, ->] (0.8, 0) -- (0.8, 0.9);
        \path[thick, tips, ->] (1.2, 0) -- (1.2, 0.9);
        \draw[draw=black, fill=white] (0.65, -0.2) rectangle (1.35, 0.2);
        \node at (1, -1.8) {$N$};
        \node at (0.3, 0.8) {$M$};
        \node at (1.7, 0.8) {$M'$};
        \draw[->] (2, 0) -- (3, 0);
        \draw (3, 2) -- (4, 1) -- (5, 2);
        \node at (4, 1.65) {$\cdots$};
        \path[thick, tips, ->] (4, 1) -- (3.45, 1.55);
        \path[thick, tips, ->] (4, 1) -- (4.55, 1.55);
        \node at (4, 2.3) {$L(1) \otimes \cdots \otimes L(1)$};
        \draw (4, -1) -- (4, 0);
        \draw (3.85, 0) -- (3.85, 1);
        \draw (4.15, 0) -- (4.15, 1);
        \path[thick, tips, ->] (4, -1) -- (4, -0.5);
        \path[thick, tips, ->] (3.85, 0) -- (3.85, 0.55);
        \path[thick, tips, ->] (4.15, 0) -- (4.15, 0.55);
        \draw (3, -2) -- (4, -1) -- (5, -2);
        \node at (4, -1.65) {$\cdots$};
        \path[thick, tips, ->] (3, -2) -- (3.45, -1.55);
        \path[thick, tips, ->] (5, -2) -- (4.55, -1.55);
        \draw[draw=black, fill=white] (3.65, -0.2) rectangle (4.35, 0.2);
        \draw[draw=black, fill=white] (3.65, 0.8) rectangle (4.35, 1.2);
        \draw[draw=black, fill=white] (3.65, -1.2) rectangle (4.35, -0.8);
        \node at (4, -2.3) {$L(1) \otimes \cdots \otimes L(1)$};
    \end{tikzpicture}
\end{center}

In the classical case, Schur-Weyl duality says that the algebra of intertwining operators of $V^{\otimes m}$ that commute with the action of the algebra $U\left(\mathfrak{gl}_n\right)$ is generated by the permutation of adjacent pairs in the tensor product. However, in the quantum case of $U_q(\mathfrak{sl}_2)$, maps that permute tensor products of modules, like $\tau$, aren't necessarily module homomorphisms and so we need to upgrade to something more complex. The correct setting is to replace the $S_{m}$-action with a $\mathbf{H}(S_m)$-action, where $\mathbf{H}(S_m)$ is the corresponding Hecke algebra.

\begin{definition}
    The Hecke algebra, $\mathbf{H}(S_n)$, is the complex associative algebra with generators $T_1, \cdots, T_{n-1}$ and relations
    \begin{align*}
        T_i T_j &= T_j T_i \quad\quad \text{ if } |i-j| > 1,\\
        T_{i}T_{i+1}T_{i} &= T_{i+1}T_{i}T_{i+1},\\
        (T_i - q)(T_i + q^{-1}) &= 0.
    \end{align*}
\end{definition}
\begin{remark}
    If we specialize $q = 1$, we get back the symmetric group algebra. Therefore, the Hecke algebra, $\mathbf{H}(S_n)$, is a deformation of $\mathbb{C}[S_n]$, and is why we use this notation.
\end{remark}
\begin{remark}
    It may seem like each generator, $T_i$, does not have an inverse in $\mathbf{H}(S_n)$ a priori. However, expanding the last relation provides us with a closed form of its inverse, proving its existence.\\
    \begin{align*}
        1 &= T_i^2 - qT_i + q^{-1}T_i\\
        \Rightarrow T_i^{-1} &= T_i - (q - q^{-1})
    \end{align*}
\end{remark}
\begin{remark}
    The last two relations of the Hecke algebra are equivalent to solutions for YBE.
\end{remark}

\begin{theorem}[12.3.10 in \cite{MR1300632}]
    Suppose $m, n > 1$ and consider the module $L(n-1)^{\otimes m}$. There is a functor from $\mathbf{H}(S_m)$-Rep to the subcategory of $U_q(\mathfrak{sl}_n)$-Rep consisting of modules isomorphic to an irreducible component of $L(1)^{\otimes m}$ defined by
$$\mathcal{J}: M \mapsto M \otimes_{\mathbf{H}(S_m)} L(n-1)^{\otimes m}$$
where the $U_q(\mathfrak{sl}_n)$-module structure is the natural structure induced by $L(n-1)^{\otimes m}$.
\end{theorem}
In particular, this functor is essentially surjective and thus every endomorphism in the algebra $\operatorname{End}_{U_q(\mathfrak{sl}_n)}\left(L(n-1)^{\otimes m}\right)$ can be understood using a representation of $\mathbf{H}(S_m)$. Moreover, if $m \leq n$, then $\mathcal{J}$ is also injective, making $\mathcal{J}$ an equivalence of categories.

One would hope that the action of the Hecke algebra has a diagrammatic interpretation nice enough to not have any $n$-valent coupons for $n \geq 3$. Unfortunately, the Hecke algebra does not have to play nicely and the lack of $n$-valent coupons is not true in general. However, a particular algebra that is diagrammatic very nice and has all of the properties that we could ask for is the \emph{Temperley-Lieb algebra}.

\begin{definition}
    The Temperley-Lieb algebra, $TL_n(\delta)$, where $\delta \in \mathbb{C}^\times$, is a unital associative algebra over $\mathbb{C}$ with generators $U_1, \cdots, U_{n-1}$ and defining relations
    $$U_i^2 = \delta U_i$$
    $$U_i U_{i \pm 1} U_i = U_i$$
    $$U_i U_j = U_j U_i \quad \text{for } |i-j| > 1.$$
\end{definition}
See figure \ref{fig:TL5Ex} for a graphical example of multiplying two elements in $TL_5(\delta)$. In this example, multiplication is understood as vertical stacking and any closed loop is replaced by the distinguished scaling factor, $\delta$.

\begin{figure}
    \begin{center}
        \begin{tikzpicture}
            \draw[fill=gray!40] (-0.3, -1) -- (-0.3, 1) -- (4.3, 1) -- (4.3, -1) -- (-0.3, -1);
            \draw[line width=2mm] (0, -1) to[out=90, in=90] (1, -1);
            \draw[line width=2mm] (2, -1) to[out=90, in=90] (3, -1);
            \draw[line width=2mm] (4, -1) to[out=90, in=270] (2, 1);
            \draw[line width=2mm] (0, 1) to[out=270, in=270] (1, 1);
            \draw[line width=2mm] (3, 1) to[out=270, in=270] (4, 1);
            \node at (4.5, 0) {$\cdot$};
            \draw[fill=gray!40] (4.7, -1) -- (4.7, 1) -- (9.3, 1) -- (9.3, -1) -- (4.7, -1);
            \draw[line width=2mm] (5, -1) to[out=90, in=270] (7, 1);
            \draw[line width=2mm] (6, -1) to[out=90, in=270] (8, 1);
            \draw[line width=2mm] (9, -1) to[out=90, in=270] (9, 1);
            \draw[line width=2mm] (7, -1) to[out=90, in=90] (8, -1);
            \draw[line width=2mm] (5, 1) to[out=270, in=270] (6, 1);
        \end{tikzpicture}
    \end{center}
    \begin{center}
        \begin{tikzpicture}
            \node at (-0.7, 0) {$=$};
            \draw[fill=gray!40] (-0.3, -1.5) -- (-0.3, 1.5) -- (4.3, 1.5) -- (4.3, -1.5) -- (-0.3, -1.5);
            \draw[line width=2mm] (0, 0) to[out=90, in=90] (1, 0);
            \draw[line width=2mm] (2, 0) to[out=90, in=90] (3, 0);
            \draw[line width=2mm] (4, 0) to[out=90, in=270] (2, 1.5);
            \draw[line width=2mm] (0, 1.5) to[out=270, in=270] (1, 1.5);
            \draw[line width=2mm] (3, 1.5) to[out=270, in=270] (4, 1.5);
            \draw[line width=2mm] (0, -1.5) to[out=90, in=270] (2, 0);
            \draw[line width=2mm] (1, -1.5) to[out=90, in=270] (3, 0);
            \draw[line width=2mm] (4, -1.5) to[out=90, in=270] (4, 0);
            \draw[line width=2mm] (2, -1.5) to[out=90, in=90] (3, -1.5);
            \draw[line width=2mm] (0, 0) to[out=270, in=270] (1, 0);
        \end{tikzpicture}
    \end{center}
    \begin{center}
        \begin{tikzpicture}
            \node at (4.8, 0) {$=$};
            \node at (5.25, 0) {$\delta$};
            \draw[fill=gray!40] (5.5, -1) -- (5.5, 1) -- (10.1, 1) -- (10.1, -1) -- (5.5, -1);
            \draw[line width=2mm] (5.8, 1) to[out=270, in=270] (6.8, 1);
            \draw[line width=2mm] (8.8, 1) to[out=270, in=270] (9.8, 1);
            \draw[line width=2mm] (5.8, -1) to[out=90, in=90] (6.8, -1);
            \draw[line width=2mm] (9.8, -1) to[out=90, in=270] (7.8, 1);
            \draw[line width=2mm] (7.8, -1) to[out=90, in=90] (8.8, -1);
        \end{tikzpicture}
    \end{center}
    \caption[Example of multiplication in $TL_5(\delta)$]{Multiplication in $TL_n(\delta)$ corresponds to vertical concatenation.}
    \label{fig:TL5Ex}
\end{figure}
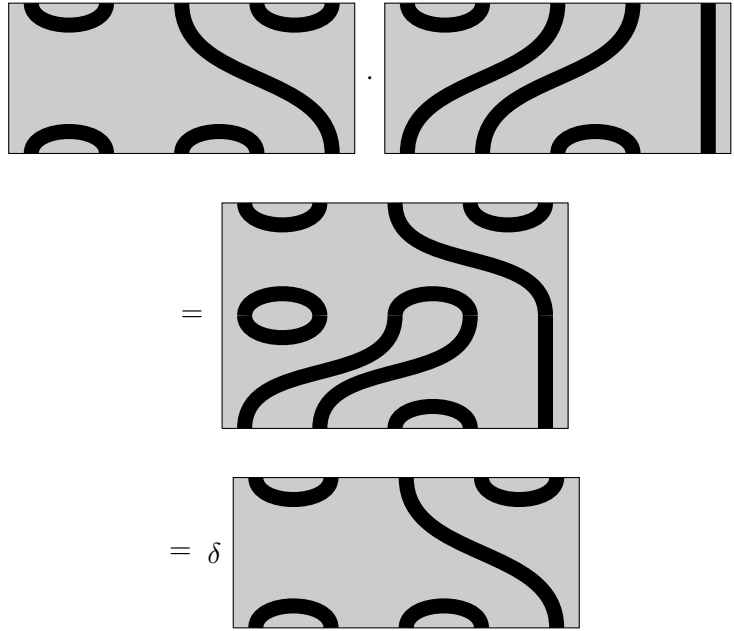

$TL_n(\delta)$ can be realized as the algebra of intertwining operators acting on the $U_q(\mathfrak{sl}_2)$-module $L(1)^{\otimes n}$, as follows:
$$U_i \mapsto \operatorname{id}^{\otimes (i-1)} \otimes
( coev \circ ev ) \otimes \operatorname{id}^{\otimes (n-i-1)},$$
where $ev: L(1) \otimes L(1) \to L(0)$ and $coev: L(0) \to L(1) \otimes L(1)$. Notice how these maps correspond to the evaluation and coevalutaion maps defined for rigid categories. Therefore, the diagrammatic interpretations of $coev$ and $ev$ correspond to cups and caps.
$$coev \circ ev = \begin{tikzpicture}[baseline=-1, scale=0.5]
    \draw[gray!40, fill=gray!40] (0,0) circle (1);
    \draw[thick] (-0.71, 0.71) to[out=-60, in=-120] (0.71, 0.71);
    \draw[thick] (-0.71, -0.71) to[out=60, in=120] (0.71, -0.71);
\end{tikzpicture}$$
In particular, $TL_n$ can be diagrammatically understood without any crossings and without any trivalent coupons (see figure \ref{fig:TL5Ex}). Moreover, the following proposition relates the Hecke algebra with a corresponding Temperley-Lieb algebra.
\begin{prop}
    There is a surjective algebra homomorphism
    \begin{align*}
        \mathbf{H}(S_m) &\twoheadrightarrow TL_m(-q-q^{-1})\\
        T_i &\mapsto q + U_i.
    \end{align*}
    When $m = 2$, this is an algebra isomorphism. 
\end{prop}
\begin{proof}
    Call this map $\varphi$ and notice that $\varphi$ is a well-defined map as
    \begin{align*}
        \varphi(T_{i} T_{i}^{-1}) &= (q + U_i)(q^{-1} + U_i) \\
        &= 1 + (q + q^{-1})U_i + U_{i}^{2} \\
        &= 1 \\[1.5ex]
        \varphi(T_i T_{i+1} T_i) &= (q  +U_i)(q + U_{i+1})(q + U_i) \\
        &= q^{3} + q^{2}U_i + q^{2}U_{i+1} + qU_{i}U_{i+1} + qU_{i+1}U_{i} + q^{2}U_{i} + qU_{i}^{2} + U_{i}U_{i+1}U_{i} \\
        &= q^{3} + q^{2}(U_i + U_{i+1}) + q(U_{i} U_{i+1} + U_{i+1}t_{i}) \\
        &= \varphi(T_{i+1} T_{i} T_{i+1}) \\[1.5ex]
        \varphi(T_i T_j) &= (q + U_i)(q + U_j) \\
        &= (q + U_j)(q + U_i) \\
        &= \varphi(T_j T_i)
    \end{align*}
    for $|i - j| > 1$.
    As this is an algebra homomorphism and $\varphi(T_i-q) = U_i$ for all $i$, $\varphi$ is surjective.

    If $n=2$, then both algebras only have one generator. Thus, the relations are greatly simplified and both algebras are two dimensional. For any $a, b \in \mathbb{C}$, if
    $$0 = \varphi(aT + b) = a(q + U) + b$$
    then $a = b = 0$ as $\{ 1, U \}$ is a basis for $TL_2\left(-q-q^{-1}\right)$ and so $\varphi$ is injective.
    This clearly fails for even just $n=3$ by a dimension counting argument.
\end{proof}
Therefore, the Hecke algebra will always surject onto the Temperley-Lieb algebra. It turns out that when working in the setting of $U_q(\mathfrak{sl}_2)$, $TL_m(-q-q^{-1})$ accounts for the entire algebra $\operatorname{End}_{U_q(\mathfrak{sl}_2)}\left( L(1)^{\otimes m} \right)$, just like the Hecke algebra. As $TL_m$ doesn't use trivalent coupons and this completely overlaps with the action of $\mathbf{H}(S_m)$, quantum Schur-Weyl duality says $U_q(\mathfrak{sl}_2)$-Rep doesn't need to consider them either.

\begin{remark}
    Note that coupons are usually used when discussing skein categories or specifically $A$-Rep itself. When focusing on particular skein modules, many works will drop the coupons altogether and just use traditional vertices instead, a convention that we'll be following for the rest of this thesis.
\end{remark}

\textbf{Note:} The previous definition that we used for the quantized universal enveloping algebra, $U_q(\mathfrak{sl}_2)$, follows the fairly standard convention found in most textbooks. However, in skein theory, mathematicians often reparameterize $q$ to $q^2$ for convenience. It's important to note that for the remainder of this work, we will adopt this reparameterization as well.
$$U_q\left(\mathfrak{sl}_2\right) = \frac{\mathbb{C} \left[E, F, K^{\pm 1}\right]}{\left( \begin{gathered} KEK^{-1} = q^4E\\ KFK^{-1} = q^{-4}F\\ [E,F] = \frac{K - K^{-1}}{q^2 - q^{-2}}\end{gathered} \right)}$$
This adjustment simplifies notation, reducing the need for fractional powers of $q$ and limiting us to using at worst $q^{1/2}$.

\section{Kauffman Bracket Skein Algebras}

When in the case of $\mathfrak{g} = \mathfrak{sl}_2$, the $R$-matrix (the braiding) has a sum decomposition that can be somewhat understood as a constant times the identity map and another constant times an evaluation map composed with a coevaluation map. This diagrammatically looks like the following.
\begin{center}\resizebox{1.2\width}{!}{
    \begin{tikzpicture}
        \draw[thick] (3,-1) -- (2,1);
        \draw[line width=3mm, white] (2,-1) -- (3,1);
        \draw[thick] (2,-1) -- (3,1);
        \draw[thick] (1.2, 1) -- (3.8, 1);
        \draw[thick] (1.2, -1) -- (3.8, -1);
        \node[draw, circle, inner sep=0pt, minimum size=4pt, fill=black] at (2,1) {};
        \node[draw, circle, inner sep=0pt, minimum size=4pt, fill=black] at (3,1) {};
        \node[draw, circle, inner sep=0pt, minimum size=4pt, fill=black] at (2,-1) {};
        \node[draw, circle, inner sep=0pt, minimum size=4pt, fill=black] at (3,-1) {};
        \node at (1.9, 1.3) {$L(1)$};
        \node at (3.1, 1.3) {$L(1)^*$};
        \node at (1.9, -1.3) {$L(1)$};
        \node at (3.1, -1.3) {$L(1)^*$};
        \node at (2.45, 1.3) {$\otimes$};
        \node at (2.45, -1.3) {$\otimes$};
        \node[text width=0.3cm] at (4.3,0) {$=$};
        \node[text width=0.3cm] at (4.9,0) {$q$};
        \draw[thick] (6,-1) -- (6,1);
        \draw[thick] (7,-1) -- (7,1);
        \draw[thick] (5.2, 1) -- (7.8, 1);
        \draw[thick] (5.2, -1) -- (7.8, -1);
        \node[draw, circle, inner sep=0pt, minimum size=4pt, fill=black] at (6,1) {};
        \node[draw, circle, inner sep=0pt, minimum size=4pt, fill=black] at (7,1) {};
        \node[draw, circle, inner sep=0pt, minimum size=4pt, fill=black] at (6,-1) {};
        \node[draw, circle, inner sep=0pt, minimum size=4pt, fill=black] at (7,-1) {};
        \node at (5.9, 1.3) {$L(1)$};
        \node at (7.1, 1.3) {$L(1)^*$};
        \node at (5.9, -1.3) {$L(1)$};
        \node at (7.1, -1.3) {$L(1)^*$};
        \node at (6.45, 1.3) {$\otimes$};
        \node at (6.45, -1.3) {$\otimes$};
        \node[text width=0.3cm] at (8.2,0) {$+$};
        \node[text width=0.5cm] at (8.8,0) {$q^{-1}$};
        \draw[thick] (10,-1) to[out=90, in=180] (10.5, -0.5) to[out=0, in=90] (11,-1);
        \draw[thick] (10,1) to[out=270, in=180] (10.5, 0.5) to[out=0, in=270] (11,1);
        \draw[thick] (9.2, 1) -- (11.8, 1);
        \draw[thick] (9.2, -1) -- (11.8, -1);
        \node[draw, circle, inner sep=0pt, minimum size=4pt, fill=black] at (10,1) {};
        \node[draw, circle, inner sep=0pt, minimum size=4pt, fill=black] at (11,1) {};
        \node[draw, circle, inner sep=0pt, minimum size=4pt, fill=black] at (10,-1) {};
        \node[draw, circle, inner sep=0pt, minimum size=4pt, fill=black] at (11,-1) {};
        \node at (9.9, 1.3) {$L(1)$};
        \node at (11.1, 1.3) {$L(1)^*$};
        \node at (9.9, -1.3) {$L(1)$};
        \node at (11.1, -1.3) {$L(1)^*$};
        \node at (10.45, 1.3) {$\otimes$};
        \node at (10.45, -1.3) {$\otimes$};
    \end{tikzpicture}}
\end{center}
Similarly, as ribbon categories are rigid categories, we can also compute $\operatorname{eval} \circ \operatorname{coeval}$ and find
\begin{center}\resizebox{1.2\width}{!}{
    \begin{tikzpicture}
        \draw (2.5,0) circle (1);
        \draw[thick] (1,1) -- (4,1);
        \draw[thick] (1,-1) -- (4,-1);
        \draw[thick, dashed] (1,0) -- (4,0);
        \node[text width=0.3cm] at (4.4,0) {$=$};
        \node[text width=1.7cm] at (5.6,0) {$-q^2 - q^{-2}$.};
    \end{tikzpicture}}
\end{center}
These two relations create what are called the Kauffman bracket skein relations.

One might consider the implications of extending these relations to $3$-manifolds beyond $\mathbb{R}^2 \times I$. Such extensions gives rise to skein modules, which notably, in the case of $\mathfrak{sl}_2$, results in the Kauffman bracket skein module. These modules have proven to be quite important in noncommutative geometry, knot theory, and, of course, quantum representation theory.

\newpage
As mentioned in section \ref{section:SkeinMods}, we will always use a blackboard framing for our curves (see figure \ref{fig:BBFraming}) and avoid using half-twists. Moreover, using the above relations, we can explicitly compute both framing relations:

\begin{align*}
    \begin{tikzpicture}[baseline=6]
        \draw[gray!40, fill=gray!40] (0.75, 0.35) circle (0.8);
        \draw[thick] (1.45, 0) -- (1.2, 0) to[out=180, in=270] (0.5, 0.5) to[out=90, in=180] (0.75, 0.8) to[out=0, in=90] (1, 0.5);
        \draw[line width=2mm, gray!40] (1, 0.5) to[out=270, in=0] (0.3, 0);
        \draw[thick] (1, 0.5) to[out=270, in=0] (0.3, 0) -- (0.05, 0);
    \end{tikzpicture} &= q \phantom{\cdot}
    \begin{tikzpicture}[baseline=6]
        \draw[gray!40, fill=gray!40] (0.75, 0.35) circle (0.8);
        \draw[thick] (0.05, 0) -- (0.5, 0) to[out=0, in=270] (0.5, 0.5) to[out=90, in=180] (0.75, 0.8) to[out=0, in=90] (1, 0.5) to[out=270, in=180] (1, 0) -- (1.45, 0);
    \end{tikzpicture} + q^{-1} \phantom{\cdot}
    \begin{tikzpicture}[baseline=6]
        \draw[gray!40, fill=gray!40] (0.75, 0.35) circle (0.8);
        \draw[thick] (0.05, 0) -- (1.45, 0);
        \draw[thick] (0.75, 0.5) circle (0.3);
    \end{tikzpicture} \\
    &= q \phantom{\cdot}
    \begin{tikzpicture}[baseline=6]
        \draw[gray!40, fill=gray!40] (0.75, 0.35) circle (0.8);
        \draw[thick] (0.05, 0) -- (1.45, 0);
    \end{tikzpicture} + q^{-1}(-q^2 - q^{-2}) \phantom{\cdot}
    \begin{tikzpicture}[baseline=-3]
        \draw[gray!40, fill=gray!40] (0.75, 0.35) circle (0.8);
        \draw[thick] (0.05, 0) -- (1.45, 0);
    \end{tikzpicture} \\
    &= -q^{-3} \phantom{\cdot}
    \begin{tikzpicture}[baseline=-3]
        \draw[gray!40, fill=gray!40] (0.75, 0.35) circle (0.8);
        \draw[thick] (0, 0) -- (1.5, 0);
    \end{tikzpicture}
\end{align*}
\begin{align*}
    \begin{tikzpicture}[baseline=6]
        \draw[gray!40, fill=gray!40] (0.75, 0.35) circle (0.8);
        \draw[thick] (0.05, 0) -- (0.3, 0) to[out=0, in=270] (1, 0.5) to[out=90, in=0] (0.75, 0.8) to[out=180, in=90] (0.5, 0.5);
        \draw[line width=2mm, gray!40] (0.5, 0.5) to[out=270, in=180] (1.25, 0);
        \draw[thick] (0.5, 0.5) to[out=270, in=180] (1.2, 0) -- (1.45, 0);
    \end{tikzpicture} &= q \phantom{\cdot}
    \begin{tikzpicture}[baseline=6]
        \draw[gray!40, fill=gray!40] (0.75, 0.35) circle (0.8);
        \draw[thick] (0.05, 0) -- (1.45, 0);
        \draw[thick] (0.75, 0.5) circle (0.3);
    \end{tikzpicture} + q^{-1} \phantom{\cdot}
    \begin{tikzpicture}[baseline=6]
        \draw[gray!40, fill=gray!40] (0.75, 0.35) circle (0.8);
        \draw[thick] (0.05, 0) -- (0.5, 0) to[out=0, in=270] (0.5, 0.5) to[out=90, in=180] (0.75, 0.8) to[out=0, in=90] (1, 0.5) to[out=270, in=180] (1, 0) -- (1.45, 0);
    \end{tikzpicture} \\
    &= q(-q^2 - q^{-2}) \phantom{\cdot}
    \begin{tikzpicture}[baseline=-3]
        \draw[gray!40, fill=gray!40] (0.75, 0.35) circle (0.8);
        \draw[thick] (0.05, 0) -- (1.45, 0);
    \end{tikzpicture} + q^{-1} \phantom{\cdot}
    \begin{tikzpicture}[baseline=-3]
        \draw[gray!40, fill=gray!40] (0.75, 0.35) circle (0.8);
        \draw[thick] (0.05, 0) -- (1.45, 0);
    \end{tikzpicture} \\
    &= -q^{3} \phantom{\cdot}
    \begin{tikzpicture}[baseline=-3]
        \draw[gray!40, fill=gray!40] (0.75, 0.35) circle (0.8);
        \draw[thick] (0.05, 0) -- (1.45, 0);
    \end{tikzpicture}
\end{align*}

\begin{definition}
   The \textit{Kauffman bracket skein module}, $K_{q}(M)$, of an oriented $3$-manifold, $M$, is the $\mathbb{C}$-module generated by isotopy classes of framed unoriented links in $M$, modulo the following two local relations.
   \begin{center}\resizebox{0.9\width}{!}{
       \begin{tikzpicture}
            \draw[gray!40, fill=gray!40] (0,0) circle (1);
            \draw[thick] (-0.71, 0.71) -- (0.71, -0.71);
            \draw[line width=3mm, gray!40] (0.70, 0.70) -- (-0.70, -0.70);
            \draw[thick] (0.71, 0.71) -- (-0.71, -0.71);
            \node[text width=1cm] at (1.75,0) {$= q$};
            \draw[gray!40, fill=gray!40] (3,0) circle (1);
            \draw[thick] (2.29, 0.71) to[out=-60, in=60] (2.29, -0.71);
            \draw[thick] (3.71, 0.71) to[out=-120, in=120] (3.71, -0.71);
            \node[text width=1.25cm] at (4.75,0) {$+$ $q^{-1}$};
            \draw[gray!40, fill=gray!40] (6.25,0) circle (1);
            \draw[thick] (5.54, 0.71) to[out=-60, in=-120] (6.96, 0.71);
            \draw[thick] (5.54, -0.71) to[out=60, in=120] (6.96, -0.71);
            \node at (3,-1.4) {$(R_1)$ Skein Relation};
            \draw[gray!40, fill=gray!40] (10,0) circle (1);
            \draw[thick] (10,0) circle (.5);
            \node[text width=2.4cm] at (12.4,0) {$= (-q^2 - q^{-2})$};
            \draw[gray!40, fill=gray!40] (14.7,0) circle (1);
            \node at (12.4,-1.4) {$(R_2)$ Trivial Knot Relation};
       \end{tikzpicture}}
   \end{center}
\end{definition}

If additionally $M = \Sigma \times I$, then we can define the multiplication of two diagrams, $\alpha \cdot \alpha^\prime$, by stacking $\alpha$ above $\alpha^\prime$ along the interval $I$.
For example, if $M$ is a thickened torus, multiplying the longitude and meridian together results in the following diagram.
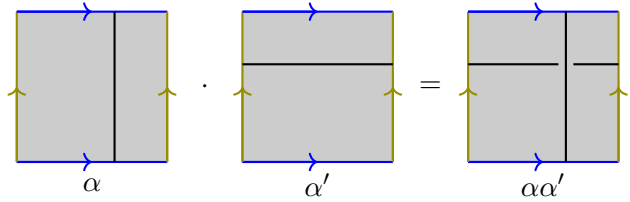
\begin{center}
    \begin{tikzpicture}
        \draw[gray!40, fill=gray!40] (-1,1) -- (1,1) -- (1,-1) -- (-1,-1) -- (-1,1);
        \draw[thick, blue] (-1,1) -- (1,1);
        \draw[thick, blue] (-1,-1) -- (1,-1);
        \draw[thick, olive] (-1,-1) -- (-1,1);
        \draw[thick, olive] (1,-1) -- (1,1);
        \draw[thick, blue, ->] (-1,1) -- (0,1);
        \draw[thick, blue, ->] (-1,-1) -- (0,-1);
        \draw[thick, olive, ->] (-1,-1) -- (-1,0);
        \draw[thick, olive, ->] (1,-1) -- (1,0);
        \draw[thick] (0.3,1) -- (0.3,-1);
        \node at (0, -1.3) {$\alpha$};
        \node at (1.5,0) {$\cdot$};
        \draw[gray!40, fill=gray!40] (2,1) -- (4,1) -- (4,-1) -- (2,-1) -- (2,1);
        \draw[thick, blue] (2,1) -- (4,1);
        \draw[thick, blue] (2,-1) -- (4,-1);
        \draw[thick, olive] (2,-1) -- (2,1);
        \draw[thick, olive] (4,-1) -- (4,1);
        \draw[thick, blue, ->] (2,1) -- (3,1);
        \draw[thick, blue, ->] (2,-1) -- (3,-1);
        \draw[thick, olive, ->] (2,-1) -- (2,0);
        \draw[thick, olive, ->] (4,-1) -- (4,0);
        \draw[thick] (2,0.3) -- (4,0.3);
        \node at (3, -1.3) {$\alpha'$};
        \node at (4.5,0) {$=$};
        \draw[gray!40, fill=gray!40] (5,1) -- (7,1) -- (7,-1) -- (5,-1) -- (5,1);
        \draw[thick, blue] (5,1) -- (7,1);
        \draw[thick, blue] (5,-1) -- (7,-1);
        \draw[thick, olive] (5,-1) -- (5,1);
        \draw[thick, olive] (7,-1) -- (7,1);
        \draw[thick, blue, ->] (5,1) -- (6,1);
        \draw[thick, blue, ->] (5,-1) -- (6,-1);
        \draw[thick, olive, ->] (5,-1) -- (5,0);
        \draw[thick, olive, ->] (7,-1) -- (7,0);
        \draw[thick] (5,0.3) -- (6.2,0.3);
        \draw[thick] (6.4,0.3) -- (7,0.3);
        \draw[thick] (6.3,1) -- (6.3,-1);
        \node at (6, -1.3) {$\alpha\alpha^\prime$};
    \end{tikzpicture}
\end{center}

While Kauffman bracket skein modules are technically only defined for $3$-manifolds, we often use the shorthand $K_q(\Sigma)$ to refer to $K_q(\Sigma \times [0,1])$ for a surface $\Sigma$. This notation also conveys to the reader that there is more than just the module structure—it includes the additional algebra structure defined by stacking.

One particular example that we're interested in is the Kauffman bracket skein algebra of the torus, $K_q(T^2)$. In general, the simple curves on $T^2$ can be classified up to homotopy by their ``slope.'' More specifically, when $r,s \in \mathbb{Z}$ are relatively prime, we define $(r,s)$ to be the unoriented simple closed curve in $T^2$ that is the image of the line $y = \frac{r}{s} x$ under the natural projection from $\mathbb{R}^2$ to $T^2$.

The algebra $K_q(T^2)$ has been studied quite heavily and thus has an explicit basis. Let $T_n(x)$ be the $n$th Chebyshev polynomial, defined recursively where $T_0(x) = 2$, $T_1(x) = x$, and $T_{n+1} = xT_{n} - T_{n-1}$. Define $(r,s)_T$ to be the evaluation of $T_d$ on the $(r/d,s/d)$-curve where $r,s \in \mathbb{Z}$ and $d = \gcd(r,s)$.

\begin{theorem}[\cite{MR1675190}]\label{theorem:KqT2Basis}
    The set $\{ (r,s)_T \}_{r,s \in \mathbb{Z}} / \sim$ where $(r,s)_T \sim (-r, -s)_T$ is a basis for $K_q(T^2)$.
\end{theorem}

\section{Stated Skein Algebras}

The main restriction of Kauffman bracket skein modules is that they are only defined using links. Whenever one works with links, it's natural to ask whether the framework can be extended to tangles as well. Moreover, in the definition of $K_q(M)$, the boundary of $M$ doesn't play any crucial part, and hence we have $K_q(M) \cong K_q(\mathring{M})$. In \cite{MR3827810}, Thang L\^{e} reinterpreted the quantum trace maps in \cite{MR2851072} by introducing \textit{stated skein modules}, which addressed how to incorporate both tangles and boundary components to establish an excision property for these modules.

\begin{definition}
    A \textit{marked $3$-manifold} is a pair $(M, \mathcal{N})$ where $M$ is a compact oriented $3$-manifold with (possibly empty) boundary $\partial M$, and $\mathcal{N} \subset \partial M$ are oriented arcs called \textit{markings}.
\end{definition}

\begin{remark}
    The orientation of the marking provides the points on this arc with a natural ordering.
\end{remark}

\begin{definition}
    A \textit{marked surface} is a pair $(\Sigma, \mathcal{P})$ where $\Sigma$ is a compact oriented surface with (possibly empty) boundary $\partial \Sigma$, and $\mathcal{P} \subset \partial \Sigma$ is a finite set, called the set of marked points.
\end{definition}
The associated marked $3$-manifold $(M, \mathcal{N})$ is defined by $M = \Sigma \times I$ and its markings $\mathcal{N} = \mathcal{P} \times I$.

\begin{definition}\label{defn:StatedTangle}
    A \textit{stated $\mathcal{N}$-tangle} is a pair ($\alpha$, $s$) where $\alpha$ is a compact $1$-dimensional unoriented submanifold with a framing such that $\partial \alpha = \alpha \cap \mathcal{N}$ and $s$ is a map $s : \partial \alpha \to \{\pm\}$.
\end{definition}

The decorations of these states at each endpoint correspond to our two basis vectors in the fundamental representation of $U_q(\mathfrak{sl}_2)$, $L(1)$. Once again, as $L(1)$ is self-dual, we don't have to worry about orientations of these tangles.

\begin{definition}
    The \textit{stated skein module} of $(M, \mathcal{N})$, denoted $\mathscr{S}\left(M, \mathcal{N} \right)$, is the quotient of the free module spanned by isotopy classes of stated $\mathcal{N}$-tangles subject to the following local relations.
    \begin{center}\resizebox{0.9\width}{!}{
       \begin{tikzpicture}
            \draw[gray!40, fill=gray!40] (0,0) circle (1);
            \draw[thick] (-0.71, 0.71) -- (0.71, -0.71);
            \draw[line width=3mm, gray!40] (0.70, 0.70) -- (-0.70, -0.70);
            \draw[thick] (0.71, 0.71) -- (-0.71, -0.71);
            \node[text width=1cm] at (1.75,0) {$= q$};
            \draw[gray!40, fill=gray!40] (3,0) circle (1);
            \draw[thick] (2.29, 0.71) to[out=-60, in=60] (2.29, -0.71);
            \draw[thick] (3.71, 0.71) to[out=-120, in=120] (3.71, -0.71);
            \node[text width=1.25cm] at (4.75,0) {$+$ $q^{-1}$};
            \draw[gray!40, fill=gray!40] (6.25,0) circle (1);
            \draw[thick] (5.54, 0.71) to[out=-60, in=-120] (6.96, 0.71);
            \draw[thick] (5.54, -0.71) to[out=60, in=120] (6.96, -0.71);
            \node at (3,-1.4) {$(R_1)$ Skein Relation};
            \draw[gray!40, fill=gray!40] (10,0) circle (1);
            \draw[thick] (10,0) circle (.5);
            \node[text width=2.4cm] at (12.4,0) {$= (-q^2 - q^{-2})$};
            \draw[gray!40, fill=gray!40] (14.7,0) circle (1);
            \node at (12.4,-1.4) {$(R_2)$ Trivial Knot Relation};
       \end{tikzpicture}}
   \end{center}
   \begin{center}\resizebox{0.9\width}{!}{
       \begin{tikzpicture}
            \draw[gray!40, thick, fill=gray!40, domain=-45:225] plot ({cos(\x)}, {sin(\x)}) to[out=45, in=130] (0.71, -0.71);
            \draw[thick] (-0.71, -0.71) to[out=45, in=135] (0.71, -0.71);
            \node[draw, circle, inner sep=0pt, minimum size=4pt, fill=black] (p1) at (0,-0.41) {};
            \draw[thick] (p1) to[out=45, in=-120] (0.31,0) to[out=60, in=0] (0,0.45) to[out=180, in=120] (-0.31,0) to[out=-60, in=150] (-0.15,-0.2);
            \node (s1) at (-0.6,-0.3) {$-$};
            \node (s2) at (0.6,-0.3) {$+$};
            \node[text width=1.4cm] at (2,0) {$= q^{-1/2}$};
            \draw[gray!40, thick, fill=gray!40, domain=-45:225] plot ({3.8+cos(\x)}, {sin(\x)}) to[out=45, in=130] (4.51, -0.71);
            \draw[thick] (3.09, -0.71) to[out=45, in=130] (4.51, -0.71);
            \node[draw, circle, inner sep=0pt, minimum size=4pt, fill=black] (p2) at (3.8,-0.41) {};
            \node at (2,-1.4) {$(R_3)$ Trivial Arc Relation 1};
            \draw[gray!40, thick, fill=gray!40, domain=-45:225] plot ({7.5+cos(\x)}, {sin(\x)}) to[out=45, in=130] (8.21, -0.71);
            \draw[thick] (6.79, -0.71) to[out=45, in=130] (8.21, -0.71);
            \node[draw, circle, inner sep=0pt, minimum size=4pt, fill=black] (p3) at (7.5,-0.41) {};
            \draw[thick] (p3) to[out=45, in=-120] (7.81,0) to[out=60, in=0] (7.5,0.45) to[out=180, in=120] (7.19,0) to[out=-60, in=150] (7.35,-0.2);
            \node (s1) at (6.9,-0.3) {$-$};
            \node (s2) at (8.1,-0.3) {$-$};
            \node[text width=1.1cm] at (9.25,0) {$= 0 =$};
            \draw[gray!40, thick, fill=gray!40, domain=-45:225] plot ({11+cos(\x)}, {sin(\x)}) to[out=45, in=130] (11.71, -0.71);
            \draw[thick] (10.29, -0.71) to[out=45, in=130] (11.71, -0.71);
            \node[draw, circle, inner sep=0pt, minimum size=4pt, fill=black] (p4) at (11,-0.41) {};
            \draw[thick] (p4) to[out=45, in=-120] (11.31,0) to[out=60, in=0] (11,0.45) to[out=180, in=120] (10.69,0) to[out=-60, in=150] (10.85,-0.2);
            \node (s1) at (10.4,-0.3) {$+$};
            \node (s2) at (11.6,-0.3) {$+$};
            \node at (9.25,-1.4) {$(R_4)$ Trivial Arc Relation 2};
       \end{tikzpicture}}
   \end{center}
   \begin{center}\resizebox{0.9\width}{!}{
       \begin{tikzpicture}
            \draw[gray!40, thick, fill=gray!40, domain=-45:225] plot ({cos(\x)}, {sin(\x)}) to[out=45, in=130] (0.71, -0.71);
            \draw[thick] (-0.71, -0.71) to[out=45, in=135] (0.71, -0.71);
            \node[draw, circle, inner sep=0pt, minimum size=4pt, fill=black] (p1) at (0,-0.41) {};
            \draw[thick] (p1) -- (-0.71,0.71);
            \draw[thick] (0.71,0.71) -- (0.1,-0.26);
            \node (s1) at (-0.6,-0.3) {$+$};
            \node (s2) at (0.6,-0.3) {$-$};
            \node[text width=1.1cm] at (1.75,0) {$= q^{-2}$};
            \draw[gray!40, thick, fill=gray!40, domain=-45:225] plot ({3.5+cos(\x)}, {sin(\x)}) to[out=45, in=130] (4.21, -0.71);
            \draw[thick] (2.79, -0.71) to[out=45, in=130] (4.21, -0.71);
            \node[draw, circle, inner sep=0pt, minimum size=4pt, fill=black] (p2) at (3.5,-0.41) {};
            \draw[thick] (p2) -- (2.79,0.71);
            \draw[thick] (3.6,-0.26) -- (4.21,0.71);
            \node (s3) at (2.9,-0.3) {$-$};
            \node (s4) at (4.1,-0.3) {$+$};
            \node[text width=1.2cm] at (5.25,0) {$+$ $q^{1/2}$};
            \draw[gray!40, thick, fill=gray!40, domain=-45:225] plot ({7+cos(\x)}, {sin(\x)}) to[out=45, in=130] (7.71, -0.71);
            \draw[thick] (6.29, -0.71) to[out=45, in=130] (7.71, -0.71);
            \node[draw, circle, inner sep=0pt, minimum size=4pt, fill=black] (p3) at (7,-0.41) {};
            \draw[thick] (6.29, 0.71) to[out=-60, in=180] (7,0) to[out=0, in=240] (7.71,0.71);
            \node at (3.5,-1.4) {$(R_5)$ State Exchange Relation};
       \end{tikzpicture}}
   \end{center}
\end{definition}

Using relations $(R_1) - (R_5)$ we can easily find the following state exchange relation, and hence get every state exchange relation.
\begin{center}
    \begin{tikzpicture}
        \draw[gray!40, thick, fill=gray!40, domain=-45:225] plot ({cos(\x)}, {sin(\x)}) to[out=45, in=130] (0.71, -0.71);
        \draw[thick] (-0.71, -0.71) to[out=45, in=135] (0.71, -0.71);
        \node[draw, circle, inner sep=0pt, minimum size=4pt, fill=black] (p1) at (0,-0.41) {};
        \draw[thick] (-0.1, -0.26) -- (-0.71,0.71);
        \draw[thick] (p1) -- (0.71,0.71);
        \node (s1) at (-0.6,-0.3) {$+$};
        \node (s2) at (0.6,-0.3) {$-$};
        \node[text width=1.1cm] at (1.9,0) {$= q^{2}$};
        \draw[gray!40, thick, fill=gray!40, domain=-45:225] plot ({3.5+cos(\x)}, {sin(\x)}) to[out=45, in=130] (4.21, -0.71);
        \draw[thick] (2.79, -0.71) to[out=45, in=130] (4.21, -0.71);
        \node[draw, circle, inner sep=0pt, minimum size=4pt, fill=black] (p2) at (3.5,-0.41) {};
        \draw[thick] (3.4,-0.26) -- (2.79,0.71);
        \draw[thick] (p2) -- (4.21,0.71);
        \node (s3) at (2.9,-0.3) {$-$};
        \node (s4) at (4.1,-0.3) {$+$};
        \node[text width=1.4cm] at (5.4,0) {$+$ $q^{-1/2}$};
        \draw[gray!40, thick, fill=gray!40, domain=-45:225] plot ({7.25+cos(\x)}, {sin(\x)}) to[out=45, in=130] (7.96, -0.71);
        \draw[thick] (6.54, -0.71) to[out=45, in=130] (7.96, -0.71);
        \node[draw, circle, inner sep=0pt, minimum size=4pt, fill=black] (p3) at (7.25,-0.41) {};
        \draw[thick] (6.54, 0.71) to[out=-60, in=180] (7.25,0) to[out=0, in=240] (7.96,0.71);
    \end{tikzpicture}
\end{center}

It's also not too hard to see that for any $\nu \in \{ \pm \}$, $(R_5)$ is equivalent to the following height exchange relations.
\begin{center}
    \begin{tikzpicture}
        \draw[gray!40, thick, fill=gray!40, domain=-45:225] plot ({cos(\x)}, {sin(\x)}) to[out=45, in=130] (0.71, -0.71);
        \draw[thick] (-0.71, -0.71) to[out=45, in=135] (0.71, -0.71);
        \node[draw, circle, inner sep=0pt, minimum size=4pt, fill=black] (p1) at (0,-0.41) {};
        \draw[thick] (p1) -- (-0.71,0.71);
        \draw[thick] (0.71,0.71) -- (0.1,-0.26);
        \node (s1) at (-0.6,-0.3) {$\nu$};
        \node (s2) at (0.6,-0.3) {$\nu$};
        \node[text width=1.1cm] at (1.72,0) {$= q^{-1}$};
        \draw[gray!40, thick, fill=gray!40, domain=-45:225] plot ({3.3+cos(\x)}, {sin(\x)}) to[out=45, in=130] (4.01, -0.71);
        \draw[thick] (2.59, -0.71) to[out=45, in=130] (4.01, -0.71);
        \node[draw, circle, inner sep=0pt, minimum size=4pt, fill=black] (p2) at (3.3,-0.41) {};
        \draw[thick] (3.2,-0.26) -- (2.59,0.71);
        \draw[thick] (p2) -- (4.01,0.71);
        \node (s3) at (2.7,-0.3) {$\nu$};
        \node (s4) at (3.9,-0.3) {$\nu$};
    \end{tikzpicture}
\end{center}
\begin{center}
    \begin{tikzpicture}
        \draw[gray!40, thick, fill=gray!40, domain=-45:225] plot ({cos(\x)}, {sin(\x)}) to[out=45, in=130] (0.71, -0.71);
        \draw[thick] (-0.71, -0.71) to[out=45, in=135] (0.71, -0.71);
        \node[draw, circle, inner sep=0pt, minimum size=4pt, fill=black] (p1) at (0,-0.41) {};
        \draw[thick] (p1) -- (-0.71,0.71);
        \draw[thick] (0.71,0.71) -- (0.1,-0.26);
        \node (s1) at (-0.6,-0.3) {$-$};
        \node (s2) at (0.6,-0.3) {$+$};
        \node[text width=1.1cm] at (1.8,0) {$= q$};
        \draw[gray!40, thick, fill=gray!40, domain=-45:225] plot ({3.3+cos(\x)}, {sin(\x)}) to[out=45, in=130] (4.01, -0.71);
        \draw[thick] (2.59, -0.71) to[out=45, in=130] (4.01, -0.71);
        \node[draw, circle, inner sep=0pt, minimum size=4pt, fill=black] (p2) at (3.3,-0.41) {};
        \draw[thick] (3.2,-0.26) -- (2.59,0.71);
        \draw[thick] (p2) -- (4.01,0.71);
        \node (s3) at (2.7,-0.3) {$-$};
        \node (s4) at (3.9,-0.3) {$+$};
    \end{tikzpicture}
\end{center}
\begin{center}
    \begin{tikzpicture}
        \draw[gray!40, thick, fill=gray!40, domain=-45:225] plot ({cos(\x)}, {sin(\x)}) to[out=45, in=130] (0.71, -0.71);
        \draw[thick] (-0.71, -0.71) to[out=45, in=135] (0.71, -0.71);
        \node[draw, circle, inner sep=0pt, minimum size=4pt, fill=black] (p1) at (0,-0.41) {};
        \draw[thick] (p1) -- (-0.71,0.71);
        \draw[thick] (0.71,0.71) -- (0.1,-0.26);
        \node (s1) at (-0.6,-0.3) {$+$};
        \node (s2) at (0.6,-0.3) {$-$};
        \node[text width=1.1cm] at (1.8,0) {$= q^{-3}$};
        \draw[gray!40, thick, fill=gray!40, domain=-45:225] plot ({3.5+cos(\x)}, {sin(\x)}) to[out=45, in=130] (4.21, -0.71);
        \draw[thick] (2.79, -0.71) to[out=45, in=130] (4.21, -0.71);
        \node[draw, circle, inner sep=0pt, minimum size=4pt, fill=black] (p2) at (3.5,-0.41) {};
        \draw[thick] (3.4,-0.26) -- (2.79,0.71);
        \draw[thick] (p2) -- (4.21,0.71);
        \node (s3) at (2.9,-0.3) {$+$};
        \node (s4) at (4.1,-0.3) {$-$};
        \node[text width=3.5cm] at (6.5,0) {$+$ $q^{-3/2} \left( q^2 - q^{-2} \right)$};
        \draw[gray!40, thick, fill=gray!40, domain=-45:225] plot ({9.25+cos(\x)}, {sin(\x)}) to[out=45, in=130] (9.96, -0.71);
        \draw[thick] (8.54, -0.71) to[out=45, in=130] (9.96, -0.71);
        \node[draw, circle, inner sep=0pt, minimum size=4pt, fill=black] (p3) at (9.25,-0.41) {};
        \draw[thick] (8.54, 0.71) to[out=-60, in=180] (9.25,0) to[out=0, in=240] (9.96,0.71);
        \node at (4.6,-1.4) {$(R_6)$ Height Exchange Relation};
    \end{tikzpicture}
\end{center}

Lastly, we can also quickly find all trivial arc relations. Below are the trivial arc relations corresponding to different states, $(R_3)$.\\
\begin{align*}
    \begin{tikzpicture}[baseline=-3]
        \draw[gray!40, thick, fill=gray!40, domain=-45:225] plot ({cos(\x)}, {sin(\x)}) to[out=45, in=130] (0.71, -0.71);
        \draw[thick] (-0.71, -0.71) to[out=45, in=135] (0.71, -0.71);
        \node[draw, circle, inner sep=0pt, minimum size=4pt, fill=black] (p1) at (0,-0.41) {};
        \draw[thick] (p1) to[out=135, in=-60] (-0.31, 0) to[out=120, in=180] (0, 0.45) to[out=0, in=60] (0.31, 0) to[out=-120, in=30] (0.15,-0.2);
        \node (s1) at (-0.6,-0.3) {$+$};
        \node (s2) at (0.6,-0.3) {$-$};
   \end{tikzpicture} &= -q^{5/2} \\
   \begin{tikzpicture}[baseline=-3]
        \draw[gray!40, thick, fill=gray!40, domain=-45:225] plot ({cos(\x)}, {sin(\x)}) to[out=45, in=130] (0.71, -0.71);
        \draw[thick] (-0.71, -0.71) to[out=45, in=135] (0.71, -0.71);
        \node[draw, circle, inner sep=0pt, minimum size=4pt, fill=black] (p1) at (0,-0.41) {};
        \draw[thick] (p1) to[out=135, in=-60] (-0.31, 0) to[out=120, in=180] (0, 0.45) to[out=0, in=60] (0.31, 0) to[out=-120, in=30] (0.15,-0.2);
        \node (s1) at (-0.6,-0.3) {$-$};
        \node (s2) at (0.6,-0.3) {$+$};
   \end{tikzpicture} &= q^{1/2} \\
   \begin{tikzpicture}[baseline=-3]
        \draw[gray!40, thick, fill=gray!40, domain=-45:225] plot ({cos(\x)}, {sin(\x)}) to[out=45, in=130] (0.71, -0.71);
        \draw[thick] (-0.71, -0.71) to[out=45, in=135] (0.71, -0.71);
        \node[draw, circle, inner sep=0pt, minimum size=4pt, fill=black] (p1) at (0,-0.41) {};
        \draw[thick] (p1) to[out=45, in=-120] (0.31,0) to[out=60, in=0] (0,0.45) to[out=180, in=120] (-0.31,0) to[out=-60, in=150] (-0.15,-0.2);
        \node (s1) at (-0.6,-0.3) {$+$};
        \node (s2) at (0.6,-0.3) {$-$};
   \end{tikzpicture} &= -q^{-5/2} \\
   \begin{tikzpicture}[baseline=-3]
        \draw[gray!40, thick, fill=gray!40, domain=-45:225] plot ({cos(\x)}, {sin(\x)}) to[out=45, in=130] (0.71, -0.71);
        \draw[thick] (-0.71, -0.71) to[out=45, in=135] (0.71, -0.71);
        \node[draw, circle, inner sep=0pt, minimum size=4pt, fill=black] (p1) at (0,-0.41) {};
        \draw[thick] (p1) to[out=45, in=-120] (0.31,0) to[out=60, in=0] (0,0.45) to[out=180, in=120] (-0.31,0) to[out=-60, in=150] (-0.15,-0.2);
        \node (s1) at (-0.6,-0.3) {$-$};
        \node (s2) at (0.6,-0.3) {$+$};
   \end{tikzpicture} &= q^{-1/2}
\end{align*}
The first and fourth diagrams, as well as the second and third diagrams, differ by a twist, which introduces a multiplicative factor of $-q^{-3}$. For consistency and convenience, throughout this thesis, we adopt L\^{e}'s notation from \cite{MR3827810} for these constants. That is,
\begin{align*}
    C_{+}^{+} &= 0, \\
    C_{+}^{-} &= -q^{-5/2}, \\
    C_{-}^{+} &= q^{-1/2}, \\
    C_{-}^{-} &= 0.
\end{align*}

Analogous to the Kauffman bracket, if $M = \Sigma \times I$ and $\mathcal{N} = \mathcal{P} \times I$ for $\mathcal{P} \subset \partial \Sigma$, then we can define an $\mathbb{C}$-algebra structure on the $\mathbb{C}$-module $\mathscr{S}(M, \mathcal{P})$ by defining products $\alpha \alpha'$ to be stacking $\alpha$ above $\alpha'$, subject to the same 5 relations.
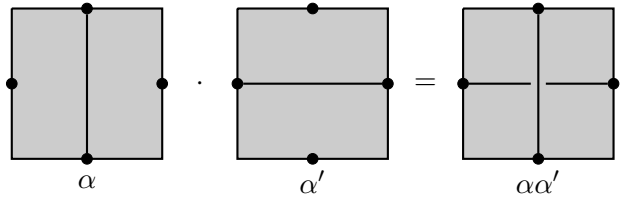
\begin{center}
    \begin{tikzpicture}
        \draw[thick, fill=gray!40] (-1,1) -- (1,1) -- (1,-1) -- (-1,-1) -- (-1,1);
        \node[draw, circle, inner sep=0pt, minimum size=4pt, fill=black] (p11) at (-1,0) {};
        \node[draw, circle, inner sep=0pt, minimum size=4pt, fill=black] (p12) at (0,1) {};
        \node[draw, circle, inner sep=0pt, minimum size=4pt, fill=black] (p13) at (1,0) {};
        \node[draw, circle, inner sep=0pt, minimum size=4pt, fill=black] (p14) at (0,-1) {};
        \draw[thick] (p12) -- (p14);
        \node at (0, -1.3) {$\alpha$};
        \node at (1.5,0) {$\cdot$};
        \draw[thick, fill=gray!40] (2,1) -- (4,1) -- (4,-1) -- (2,-1) -- (2,1);
        \node[draw, circle, inner sep=0pt, minimum size=4pt, fill=black] (p21) at (2,0) {};
        \node[draw, circle, inner sep=0pt, minimum size=4pt, fill=black] (p22) at (3,1) {};
        \node[draw, circle, inner sep=0pt, minimum size=4pt, fill=black] (p23) at (4,0) {};
        \node[draw, circle, inner sep=0pt, minimum size=4pt, fill=black] (p24) at (3,-1) {};
        \draw[thick] (p21) -- (p23);
        \node at (3, -1.3) {$\alpha'$};
        \node at (4.5,0) {$=$};
        \draw[thick, fill=gray!40] (5,1) -- (7,1) -- (7,-1) -- (5,-1) -- (5,1);
        \node[draw, circle, inner sep=0pt, minimum size=4pt, fill=black] (p31) at (5,0) {};
        \node[draw, circle, inner sep=0pt, minimum size=4pt, fill=black] (p32) at (6,1) {};
        \node[draw, circle, inner sep=0pt, minimum size=4pt, fill=black] (p33) at (7,0) {};
        \node[draw, circle, inner sep=0pt, minimum size=4pt, fill=black] (p34) at (6,-1) {};
        \draw[thick] (p31) -- (5.9,0);
        \draw[thick] (6.1,0) -- (p33);
        \draw[thick] (p32) -- (p34);
        \node at (6, -1.3) {$\alpha\alpha'$};
    \end{tikzpicture}
\end{center}

The Kauffman bracket skein module is functorial in the sense that $K_q(-)$ is a covariant functor from the category of oriented $3$-manifolds with isotopy classes of embeddings as morphisms to the category of $\mathbb{C}$-modules.
Thus, each embedding, $f: M \hookrightarrow M^\prime$, induces a module homomorphism of Kauffman bracket skein modules, $f_*: K_q(M) \to K_q(M^\prime)$, by $f_*[\alpha] = [f(\alpha)]$ for any framed link $\alpha$.
Similarly, the stated skein algebra $\mathscr{S}(-)$ is also functorial.
However, the domain of this functor is the category of marked $3$-manifolds where our embeddings preserve marking orientations.
Indeed, each embedding, $f: (M, \mathcal{N}) \hookrightarrow (M^\prime, \mathcal{N}^\prime)$, induces a similar homomorphism, $f_*: \mathscr{S}(M, \mathcal{N}) \to \mathscr{S}(M^\prime, \mathcal{N}^\prime)$, by $f_*[\alpha] = [f(\alpha)]$ for any stated tangle $\alpha$.

Since framed links exist in stated skein modules due to relations $(R_5)$ and $(R_6)$, we get that $K_q(M) \hookrightarrow \mathscr{S}(M, \mathcal{N})$.
This is consistent with the functoriality of our skein modules since $\mathcal{N}$ can be empty and $\mathscr{S}(-)$ can be restricted to the same functor on the category of oriented $3$-manifolds, making $K_q(-) \rightarrow \mathscr{S}(-)$ a natural transformation.

The main reason stated skein modules were created was due to the splitting theorem. This theorem allows us to analyze skein algebras through smaller (hopefully simpler) pieces, at the cost of more complexity. As important as this result is, the splitting theorem is not particularly relevant for this thesis, and so it will not be discussed here any further. You can read more details on this theorem in \cite{costantino2022stated, MR3827810, MR4264235}.

Since the stated skein modules can be thought of as generalizations of Kauffman bracket skein modules, there are important distinctions to consider that make the stated case more difficult to work with. For example, it is clear from relation $(R_1)$ that when $q = \pm 1$, $K_q(\Sigma \times I)$ is a commutative algebra. However, due to our additional relations, in particular $(R_6)$, $\mathscr{S}(M, \mathcal{N})$ is \emph{only} commutative at $q=1$ when $\mathcal{N}$ is nonempty. Moreover, $K_{q}(M) \cong K_{-q}(M)$ as algebras (see \cite{MR1670233}). However, this clearly can't be true in the stated case due to this lack of commutivity. Furthermore, there is additional information to keep track of when working in the stated case. Specifically, we need to keep track of not only the states but also the the heights of each tangle at each marking or marked point.

\section{The Conventional Model}\label{section:ConventionalModel}

The definitions of stated skein algebras given above differ slightly from the more common ones found in the literature. Our approach to stated tangles is somewhat closer to the notion of \textit{ideal arcs} described in \cite{costantino2022stated, MR4431131} (see also Definition \ref{defn:IdealArc}). In this section, we will first introduce the conventional definitions and then demonstrate that they are isomorphic to our definitions, showing that they can be understood in essentially the same way.

The following definitions are primarily taken from \cite{costantino2022stated} and \cite{MR3827810}. However, they have been very slightly modified to fit in our framework. In particular, given a surface, $\Sigma$, the corresponding $3$-manifold is typically taken as $\Sigma \times (0,1)$ in the literature. However, it will be important for us to consider $\Sigma \times [0,1]$ instead as this extra information will be important for us later. 

\begin{definition}\label{defn:PBS}
    Let $\Sigma^\prime$ be a (possibly punctured) oriented surface with (possibly empty) boundary, and $\mathcal{P} \subset \partial \Sigma^\prime$ be a finite nonempty set such that every connected component of $\partial \Sigma^\prime$ has at least one point in $\mathcal{P}$. Then $\Sigma = \Sigma^\prime \setminus \mathcal{P}$ is called a \textit{punctured bordered surface}.
\end{definition}

\begin{remark}
    Following the notation from definition \ref{defn:PBS}, $\Sigma^\prime$ is always uniquely determined by its punctured bordered surface, $\Sigma$.
\end{remark}

\begin{definition}\label{defn:IdealArc}
    An \textit{ideal arc} on $\Sigma$ is an immersion $ \alpha : [0,1] \to \Sigma^\prime$ such that $\alpha(0), \alpha(1) \in \mathcal{P}$ and the restriction of $\alpha$ onto $(0,1)$ is an embedding into $\Sigma$.
\end{definition}

\begin{definition}
    A connected component of $\partial \Sigma$ is called a \textit{boundary edge}.
\end{definition}

\begin{definition}
    Let $(s,t) \in \Sigma \times [0,1]$. The \textit{height} of $(s,t)$ is $t$ and we say a vector at $(s,t)$ is \textit{vertical} if it is parallel to $s \times [0,1]$.
\end{definition}

\begin{definition}
    Let $\Sigma$ be a punctured bordered surface. A \textit{stated $\partial \Sigma$-tangle} is a tuple, $(\alpha, s)$ where $\alpha \subset \Sigma \times [0,1]$ is an unoriented, framed, compact, properly embedded $1$-dimensional submanifold such that
    \begin{itemize}
        \item at every point in $\partial \alpha = \alpha \cap \left( \partial \Sigma \times [0,1] \right)$ the framing is \textit{vertical}, 
        \item for every boundary edge $b \subset \Sigma$, $\partial \alpha \cap (b \times [0,1])$ have distinct heights,
    \end{itemize}
    and $s$ is a map $s : \partial \alpha \to \{\pm\}$.
\end{definition}

Just as before, we only consider isotopy classes of these stated $\partial \Sigma$-tangles. Therefore, isotopies of stated $\partial \Sigma$-tangles are required to preserve the height order.

\begin{definition}
    The stated skein algebra of a punctured bordered surface, denoted $\mathscr{S}^{pb}(\Sigma)$ or $\mathscr{S}^{pb}(\Sigma^\prime)$, is the $\mathbb{C}$-module freely spanned by isotopy classes of stated $\partial \Sigma$-tangles modulo the following local relations.
    \begin{center}\resizebox{0.9\width}{!}{
        \begin{tikzpicture}
            \draw[gray!40, fill=gray!40] (0,0) circle (1);
            \draw[thick] (-0.71, 0.71) -- (0.71, -0.71);
            \draw[line width=3mm, gray!40] (0.70, 0.70) -- (-0.70, -0.70);
            \draw[thick] (0.71, 0.71) -- (-0.71, -0.71);
            \node[text width=1cm] at (1.75,0) {$= q$};
            \draw[gray!40, fill=gray!40] (3,0) circle (1);
            \draw[thick] (2.29, 0.71) to[out=-60, in=60] (2.29, -0.71);
            \draw[thick] (3.71, 0.71) to[out=-120, in=120] (3.71, -0.71);
            \node[text width=1.25cm] at (4.75,0) {$+$ $q^{-1}$};
            \draw[gray!40, fill=gray!40] (6.25,0) circle (1);
            \draw[thick] (5.54, 0.71) to[out=-60, in=-120] (6.96, 0.71);
            \draw[thick] (5.54, -0.71) to[out=60, in=120] (6.96, -0.71);
            \node at (3,-1.4) {$(R_1^{pb})$ Skein Relation};
            \draw[gray!40, fill=gray!40] (10,0) circle (1);
            \draw[thick] (10,0) circle (.5);
            \node[text width=2.4cm] at (12.4,0) {$= (-q^2 - q^{-2})$};
            \draw[gray!40, fill=gray!40] (14.7,0) circle (1);
            \node at (12.4,-1.4) {$(R_2^{pb})$ Trivial Knot Relation};
        \end{tikzpicture}}
    \end{center}
    \begin{center}\resizebox{0.9\width}{!}{
        \begin{tikzpicture}
            \draw[gray!40, thick, fill=gray!40, domain=-45:225] plot ({cos(\x)}, {sin(\x)}) to[out=45, in=130] (0.71, -0.71);
            \draw[thick] (-0.71, -0.71) to[out=45, in=135] (0.71, -0.71);
            \draw[thick] (-0.25, -0.45) to[out=120, in=180] (0, 0.4) to[out=0, in=60] (0.25, -0.45);
            \path[thick, tips, ->] (-0.71, -0.71) to[out=45, in=135] (0.71, -0.71);
            \node (s1) at (-0.6, -0.3) {$-$};
            \node (s2) at (0.6, -0.3) {$+$};
            \node[text width=1.4cm] at (2,0) {$= q^{-1/2}$};
            \draw[gray!40, thick, fill=gray!40, domain=-45:225] plot ({3.8+cos(\x)}, {sin(\x)}) to[out=45, in=130] (4.51, -0.71);
            \draw[thick] (3.09, -0.71) to[out=45, in=130] (4.51, -0.71);
            \path[thick, tips, ->] (3.09, -0.71) to[out=45, in=130] (4.51, -0.71);
            \node at (2, -1.4) {$(R_3^{pb})$ Trivial Arc Relation 1};
            \draw[gray!40, thick, fill=gray!40, domain=-45:225] plot ({7.5+cos(\x)}, {sin(\x)}) to[out=45, in=130] (8.21, -0.71);
            \draw[thick] (6.79, -0.71) to[out=45, in=130] (8.21, -0.71);
            \draw[thick] (7.25, -0.45) to[out=120, in=180] (7.5,0.4) to[out=0, in=60] (7.75, -0.45);
            \path[thick, tips, ->] (6.79, -0.71) to[out=45, in=130] (8.21, -0.71);
            \node (s1) at (6.9,-0.3) {$-$};
            \node (s2) at (8.1,-0.3) {$-$};
            \node[text width=1.1cm] at (9.25,0) {$= 0 =$};
            \draw[gray!40, thick, fill=gray!40, domain=-45:225] plot ({11+cos(\x)}, {sin(\x)}) to[out=45, in=130] (11.71, -0.71);
            \draw[thick] (10.29, -0.71) to[out=45, in=130] (11.71, -0.71);
            \draw[thick] (10.75, -0.45) to[out=120, in=180] (11, 0.4) to[out=0, in=60] (11.25, -0.45);
            \path[thick, tips, ->] (10.29, -0.71) to[out=45, in=130] (11.71, -0.71);
            \node (s1) at (10.4,-0.3) {$+$};
            \node (s2) at (11.6,-0.3) {$+$};
            \node at (9.25,-1.4) {$(R_4^{pb})$ Trivial Arc Relation 2};
        \end{tikzpicture}}
    \end{center}
    \begin{center}\resizebox{0.9\width}{!}{
        \begin{tikzpicture}
            \draw[gray!40, thick, fill=gray!40, domain=-45:225] plot ({cos(\x)}, {sin(\x)}) to[out=45, in=130] (0.71, -0.71);
            \draw[thick] (-0.71, -0.71) to[out=45, in=135] (0.71, -0.71);
            \draw[thick] (-0.25, -0.45) -- (-0.71, 0.71);
            \draw[thick] (0.25, -0.45) -- (0.71, 0.71);
            \path[thick, tips, ->] (-0.71, -0.71) to[out=45, in=135] (0.71, -0.71);
            \node (s1) at (-0.6,-0.3) {$+$};
            \node (s2) at (0.6,-0.3) {$-$};
            \node[text width=1.1cm] at (1.75,0) {$= q^{-2}$};
            \draw[gray!40, thick, fill=gray!40, domain=-45:225] plot ({3.5+cos(\x)}, {sin(\x)}) to[out=45, in=130] (4.21, -0.71);
            \draw[thick] (2.79, -0.71) to[out=45, in=130] (4.21, -0.71);
            \draw[thick] (3.25, -0.45) -- (2.79, 0.71);
            \draw[thick] (3.75, -0.45) -- (4.21, 0.71);
            \path[thick, tips, ->] (2.79, -0.71) to[out=45, in=130] (4.21, -0.71);
            \node (s3) at (2.9,-0.3) {$-$};
            \node (s4) at (4.1,-0.3) {$+$};
            \node[text width=1.2cm] at (5.25,0) {$+$ $q^{1/2}$};
            \draw[gray!40, thick, fill=gray!40, domain=-45:225] plot ({7+cos(\x)}, {sin(\x)}) to[out=45, in=130] (7.71, -0.71);
            \draw[thick] (6.29, -0.71) to[out=45, in=130] (7.71, -0.71);
            \path[thick, tips, ->] (6.29, -0.71) to[out=45, in=130] (7.71, -0.71);
            \draw[thick] (6.29, 0.71) to[out=-60, in=180] (7,0) to[out=0, in=240] (7.71,0.71);
            \node at (3.5,-1.4) {$(R_5^{pb})$ State Exchange Relation};
        \end{tikzpicture}}
    \end{center}
\end{definition}

\begin{definition}
    A stated $\partial \Sigma$-tangle, $\alpha$, is said to be in \textit{generic position} if the natural projection $\pi : \Sigma \times [0,1] \to \Sigma$ restricts to an embedding of $\alpha$, except for the possibility of transverse double points in the interior of $\Sigma$.
\end{definition}
Each stated $\partial \Sigma$-tangle is isotopic to one in generic position. Furthermore, we can define a $\mathbb{C}$-algebra structure on $\ms{S}^{pb}(\Sigma)$ by defining products $\alpha \cdot \alpha'$ to be stacking $\alpha$ above $\alpha'$, just as before, and isotoping the new diagram to be in generic position.

\begin{definition}
    Let $\Sigma$ be a punctured bordered surface and $\mathfrak{o}$ be an orientation of $\partial \Sigma$, which may differ from the orientation inherited from $\Sigma$. We say a $\partial \Sigma$-tangle diagram, $D$, is \textit{$\mathfrak{o}$-ordered} if for each boundary component, the points of $\partial D$ increase when traversing in the direction of $\mathfrak{o}$.
\end{definition}
Every $\partial \Sigma$-tangle can be presented, after an appropriate isotopy, by an $\mathfrak{o}$-ordered $\partial \Sigma$-tangle diagram. Figure \ref{fig:OrPBSEx} shows examples of $\mathfrak{o}$-ordered $\partial \Sigma$-tangle diagrams (without states) when our surface is a torus with boundary.

\begin{figure}[h]
    \centering
    \begin{tikzpicture}[scale=1.5, baseline=-3]
        \MarkedTorusBackground
        \node[draw, circle, inner sep=0pt, minimum size=3pt, fill=white] at (1, 0.15) {};
        \draw[thick] (0, 0.6) to[out=0, in=210] (0.88, -0.09);
        \draw[thick] (0.88, 0.09) to[out=150, in=180] (1, 0.6) -- (2, 0.6);
        \path[thick, tips, ->] (1.5, -0.15) -- (0.95, -0.15);
    \end{tikzpicture} $\cdot$ 
    \begin{tikzpicture}[scale=1.5, baseline=-3]
        \MarkedTorusBackground
        \node[draw, circle, inner sep=0pt, minimum size=3pt, fill=white] at (1, 0.15) {};
        \draw[thick] (1.4, 1) to[out=270, in=30] (1.12, 0.09);
        \draw[thick] (1.12, -0.09) to[out=330, in=90] (1.4, -1);
        \path[thick, tips, ->] (1.5, -0.15) -- (0.95, -0.15);
    \end{tikzpicture} $=$
    \begin{tikzpicture}[scale=1.5, baseline=-3]
        \MarkedTorusBackground
        \node[draw, circle, inner sep=0pt, minimum size=3pt, fill=white] at (1, 0.15) {};
        \draw[thick] (1.4, 1) to[out=270, in=30] (1.12, 0.09);
        \draw[thick] (1.12, -0.09) to[out=330, in=90] (1.4, -1);
        \draw[line width=2mm, gray!40] (1.2, 0.6) -- (1.8, 0.6);
        \draw[thick] (0, 0.6) to[out=0, in=210] (0.88, -0.09);
        \draw[thick] (0.88, 0.09) to[out=150, in=180] (1, 0.6) -- (2, 0.6);
        \path[thick, tips, ->] (1.5, -0.15) -- (0.95, -0.15);
    \end{tikzpicture}
    \caption[An $\mathfrak{o}$-ordered diagram on a punctured bordered surface]{A product of $\mathfrak{o}$-ordered (stateless) $\partial \Sigma$-diagram where $\Sigma^\prime = T^2 \setminus D^2$ and $\mathfrak{o}$ has a clockwise orientation.}
    \label{fig:OrPBSEx}
\end{figure}
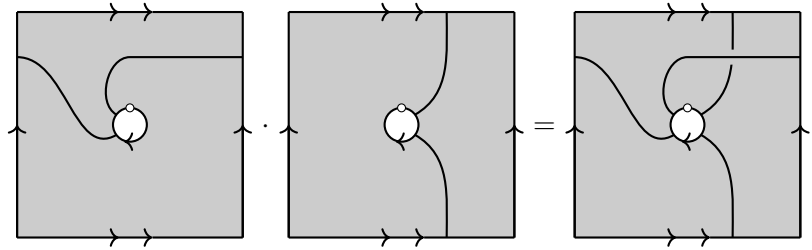

\begin{theorem}[Theorem 2.11 in \cite{MR3827810}]\label{theorem:OrBasis}
    Let $\Sigma$ be a punctured bordered surface and $\mf{o}$ an orientation of $\partial \Sigma$. Define $B(\mathfrak{o}, \Sigma)$ be the set of of all isotopy classes of increasingly stated, $\mathfrak{o}$-ordered, simple $\partial \Sigma$-tangle diagrams. Then $B(\mathfrak{o}, \Sigma)$ is a $\mathbb{C}$-basis of $\ms{S}^{pb}(\Sigma)$.
\end{theorem}

\begin{definition}
    Let $\Sigma^\prime$ be a surface with boundary and $\Sigma = \Sigma^\prime \setminus \mathcal{P}$ be the corresponding punctured bordered surface. We say a $\partial \Sigma$-orientation, $\mf{o}$, is \textit{consistent} if $\mf{o}$ can be extended to $\partial \Sigma^\prime$. That is, the direction of orientations on adjacent boundary edges agree for every boundary component.
\end{definition}

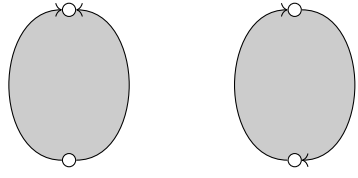
\begin{figure}[h]
    \centering
    \begin{tikzpicture}
        \draw[draw=none, fill=gray!40] (0, -1) to[out=180, in=180] (0, 1) to[out=0, in=0] (0, -1);
        \node[draw, circle, fill=white, inner sep=0pt, minimum size=5pt] (op1) at (0, -1) {};
        \node[draw, circle, fill=white, inner sep=0pt, minimum size=5pt] (op2) at (0, 1) {};
        \draw[->, fill=gray!40] (op1) to[out=180, in=270] (-0.8, 0) to[out=90, in=180] (op2);
        \draw[->, fill=gray!40] (op1) to[out=0, in=270] (0.8, 0) to[out=90, in=0] (op2);
        \draw[draw=none, fill=gray!40] (3, -1) to[out=180, in=180] (3, 1) to[out=0, in=0] (3, -1);
        \node[draw, circle, fill=white, inner sep=0pt, minimum size=5pt] (op1) at (3, -1) {};
        \node[draw, circle, fill=white, inner sep=0pt, minimum size=5pt] (op2) at (3, 1) {};
        \draw[->, fill=gray!40] (op1) to[out=180, in=270] (2.2, 0) to[out=90, in=180] (op2);
        \draw[->, fill=gray!40] (op2) to[out=0, in=90] (3.8, 0) to[out=270, in=0] (op1);
    \end{tikzpicture}
    \caption[Example of consistent boundary orientations]{The right punctured bordered surface has a consistent orientation while the left does not.}
    \label{fig:ConsistentPBSOrientations}
\end{figure}

\begin{prop}\label{prop:IsoOfModels}
    Let $\Sigma^\prime$ be a surface with boundary and $\Sigma = \Sigma^\prime \setminus \mathcal{P}$ be the corresponding punctured bordered surface. There exists an algebra isomorphism $\phi_{\mf{o}} : \ms{S}^{pb}(\Sigma) \xrightarrow{\sim} \ms{S}(\Sigma^\prime, \mathcal{P})$.
\end{prop}
\begin{proof}
    Let $\mf{o}$ be a consistent orientation on $\partial \Sigma$. By theorem \ref{theorem:OrBasis}, $B(\mathfrak{o}, \Sigma)$ is a basis for $\ms{S}^{pb}(\Sigma)$. Let $D \in B(\mathfrak{o}, \Sigma)$ be a fixed representative in its isotopy class.
    
    Define $\iota : \Sigma \times [0,1] \hookrightarrow \Sigma^\prime \times [0,1]$ to be the natural embedding and define
    \begin{align*}
        \overline{pr}_i &: \Sigma^\prime \times [0,1] \to \Sigma^\prime \times \{i\} \to \Sigma^\prime \times [0,1] \\
        pr_i &: \Sigma \times [0,1] \to \Sigma \times \{i\} \to \Sigma \times [0,1]
    \end{align*}to be the natural projection maps composed with their natural embeddings for any $i \in [0,1]$. Notice that $\left(\overline{pr}_0 \circ \iota \right)(D)$ is simple as $D$ was simple and in generic position.
    
    Endow $\mathcal{P} \times [0,1]$ with its natural orientation induced by the $[0,1]$ component. For every boundary edge, $b \subset \partial \Sigma^\prime$, isotope the endpoints of $\left(\overline{pr}_0 \circ \iota \right)(D)$ on $b \times \{0\}$ along $\mf{o}$ and along the orientation of $\mathcal{P} \times [0,1]$ so that all endpoints lie on $\mathcal{P}(0, 1)$. Call this new tangle diagram $D'$.
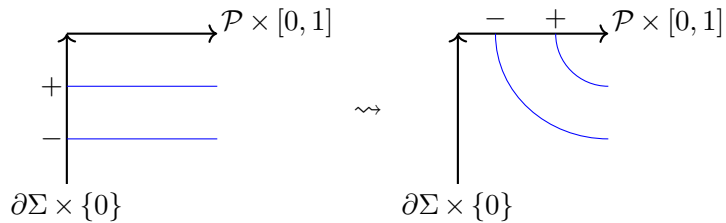
    \begin{center}
        \begin{tikzpicture}[baseline=-3]
            \draw[thick, ->] (0, -1) -- (0, 1);
            \draw[thick, ->] (0, 1) -- (2, 1);
            \draw[blue] (0, 0.3) -- (2, 0.3);
            \draw[blue] (0, -0.4) -- (2, -0.4);
            \node[text width=0.3cm] at (-0.2, 0.3) {$+$};
            \node[text width=0.3cm] at (-0.2, -0.4) {$-$};
            \node[text width=1.5cm] at (0, -1.3) {$\partial \Sigma \times \{0\}$};
            \node[text width=1.5cm] at (2.8, 1.15) {$\mathcal{P} \times [0,1]$};
        \end{tikzpicture} $\rightsquigarrow$ 
        \begin{tikzpicture}[baseline=-3]
            \draw[thick, ->] (0, -1) -- (0, 1);
            \draw[thick, ->] (0, 1) -- (2, 1);
            \draw[blue] (1.3, 1) to[out=270, in=180] (2, 0.3);
            \draw[blue] (0.5, 1) to[out=270, in=180] (2, -0.4);
            \node[text width=0.3cm] at (1.3, 1.2) {$+$};
            \node[text width=0.3cm] at (0.5, 1.2) {$-$};
            \node[text width=1.5cm] at (0, -1.3) {$\partial \Sigma \times \{0\}$};
            \node[text width=1.5cm] at (2.8, 1.15) {$\mathcal{P} \times [0,1]$};
        \end{tikzpicture}
    \end{center}
    At this point, it is not clear if $D'$ is still simple, now with respect to $\Sigma^\prime$. However, we can still canonically identify $D'$ as some element of $\ms{S}(\Sigma^\prime, \mathcal{P})$.

    Define $\phi_{\mf{o}}(D) = D'$ for each $D \in B(\mathfrak{o}, \Sigma)$ and extend linearly. Notice that $\phi_{\mf{o}}$ is well-defined as $D'$ is unique up to isotopy and $\phi_{\mf{o}}$ respects the relations $(R_1^{pb}) - (R_5^{pb})$ through the corresponding relations $(R_1) - (R_5)$. Moreover, as relative heights are preserved, this is actually an algebra homomorphism as well.
    
    Now let $\alpha' \in \ms{S}(\Sigma^\prime, \mathcal{P})$ be a fixed representative of its isotopy class and endow $\Sigma$ with the orientation $\mf{o}$. Isotope the endpoints of $\alpha$ in each summand in the reverse direction of the orientations of $\mathcal{P} \times [0,1]$ and $\partial \Sigma \times \{0\}$ until each endpoint lies entirely on $\partial \Sigma \times \{0\}$. Since $\mf{o}$ is consistent, this is well-defined.
    
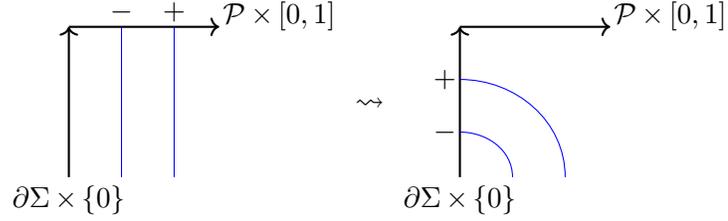
\begin{center}
        \begin{tikzpicture}[baseline=-3]
            \draw[thick, ->] (0, -1) -- (0, 1);
            \draw[thick, ->] (0, 1) -- (2, 1);
            \draw[blue] (1.4, 1) -- (1.4, -1);
            \draw[blue] (0.7, 1) -- (0.7, -1);
            \node[text width=0.3cm] at (1.4, 1.2) {$+$};
            \node[text width=0.3cm] at (0.7, 1.2) {$-$};
            \node[text width=1.5cm] at (0, -1.3) {$\partial \Sigma \times \{0\}$};
            \node[text width=1.5cm] at (2.8, 1.15) {$\mathcal{P} \times [0,1]$};
        \end{tikzpicture} $\rightsquigarrow$ 
        \begin{tikzpicture}[baseline=-3]
            \draw[thick, ->] (0, -1) -- (0, 1);
            \draw[thick, ->] (0, 1) -- (2, 1);
            \draw[blue] (0, 0.3) to[out=0, in=90] (1.4, -1);
            \draw[blue] (0, -0.4) to[out=0, in=90] (0.7, -1);
            \node[text width=0.3cm] at (-0.2, 0.3) {$+$};
            \node[text width=0.3cm] at (-0.2, -0.4) {$-$};
            \node[text width=1.5cm] at (0, -1.3) {$\partial \Sigma \times \{0\}$};
            \node[text width=1.5cm] at (2.8, 1.15) {$\mathcal{P} \times [0,1]$};
        \end{tikzpicture}
    \end{center}
    We can now view this new object, call it $\tilde{\alpha}'$, as living entirely on $\Sigma \times [0,1]$. Finally, define $\alpha := pr_{1/2}(\tilde{\alpha}')$ and isotope $\alpha$ so that it is generic and $\mf{o}$-ordered. Therefore, we can canonically identify $\alpha$ as an element in $\mathscr{S}^{pb}(\Sigma)$.

    Pointwise define $\psi_{\mf{o}}(\alpha') = \alpha$ for each $\alpha' \in \ms{S}(\Sigma^\prime, \mathcal{P})$. Clearly $(\psi_{\mf{o}} \circ \phi_{\mf{o}})(\alpha) = \alpha$ and $(\phi_{\mf{o}} \circ \psi_{\mf{o}})(\alpha') = \alpha'$ for all $\alpha \in \mathscr{S}^{pb}(\Sigma)$ and $\alpha' \in \mathscr{S}(\Sigma^\prime, \mathcal{P})$ and hence $\phi_{\mf{o}}$ is bijective.
\end{proof}

\begin{definition}
    Let $\mf{o}$ be a consistent orientation on $\partial \Sigma$. We say a $D \in \mathscr{S}(\Sigma^\prime)$ is \textit{$\mf{o}$-simple} if $D$ is simple and $\psi_{\mf{o}}(D)$ is simple.
\end{definition}

\begin{prop}\label{prop:oSimpDiag}
    Every simple diagram in $\mathscr{S}(\Sigma^\prime)$ can be written as a linear combination of increasingly stated, $\mf{o}$-simple $\partial \Sigma^\prime$-tangle diagrams.
\end{prop}
\begin{proof}
    It suffices to only check for two adjacent endpoints on a single marked point. Firstly, notice that if the diagram is simple, we only need to consider two cases, depending on the orientation, $\mf{o}$, of the adjacent boundary.
    \begin{align*}
        \psi_{\mf{o}}\left(\begin{tikzpicture}[baseline=-3]
            \draw[gray!40, thick, fill=gray!40, domain=-45:225] plot ({cos(\x)}, {sin(\x)}) to[out=45, in=130] (0.71, -0.71);
            \draw[thick] (-0.71, -0.71) to[out=45, in=135] (0.71, -0.71);
            \node[draw, circle, inner sep=0pt, minimum size=4pt, fill=black] (p1) at (0, -0.41) {};
            \path[thick, tips, ->] (-0.71, -0.71) to[out=45, in=208] (-0.35, -0.49);
            \path[thick, tips, ->] (p1) to[out=0, in=152] (0.5, -0.55);
            \draw[thick] (-0.1, -0.26) -- (-0.71, 0.71);
            \draw[thick] (p1) -- (0.71, 0.71);
            \node (s1) at (-0.6, -0.3) {$\mu$};
            \node (s2) at (0.6, -0.3) {$\nu$};
        \end{tikzpicture}\right) &= 
        \begin{tikzpicture}[baseline=-3]
            \draw[gray!40, thick, fill=gray!40, domain=-45:225] plot ({cos(\x)}, {sin(\x)}) to[out=45, in=130] (0.71, -0.71);
            \draw[thick] (-0.71, -0.71) to[out=45, in=135] (0.71, -0.71);
            \draw[thick] (-0.25, -0.45) -- (-0.71, 0.71);
            \draw[thick] (0.25, -0.45) -- (0.71, 0.71);
            \path[thick, tips, ->] (-0.71, -0.71) to[out=45, in=135] (0.71, -0.71);
            \node (s1) at (-0.6,-0.3) {$\mu$};
            \node (s2) at (0.6,-0.3) {$\nu$};
        \end{tikzpicture} & 
        \psi_{\mf{o}}\left(\begin{tikzpicture}[baseline=-3]
            \draw[gray!40, thick, fill=gray!40, domain=-45:225] plot ({cos(\x)}, {sin(\x)}) to[out=45, in=130] (0.71, -0.71);
            \draw[thick] (-0.71, -0.71) to[out=45, in=135] (0.71, -0.71);
            \node[draw, circle, inner sep=0pt, minimum size=4pt, fill=black] (p1) at (0, -0.41) {};
            \path[thick, tips, ->] (0.71, -0.71) to[out=135, in=332] (0.35, -0.49);
            \path[thick, tips, ->] (p1) to[out=0, in=28] (-0.5, -0.55);
            \draw[thick] (p1) -- (-0.71, 0.71);
            \draw[thick] (0.1, -0.26) -- (0.71, 0.71);
            \node (s1) at (-0.6, -0.3) {$\mu$};
            \node (s2) at (0.6, -0.3) {$\nu$};
        \end{tikzpicture}\right) &= 
        \begin{tikzpicture}[baseline=-3]
            \draw[gray!40, thick, fill=gray!40, domain=-45:225] plot ({cos(\x)}, {sin(\x)}) to[out=45, in=130] (0.71, -0.71);
            \draw[thick] (-0.71, -0.71) to[out=45, in=135] (0.71, -0.71);
            \draw[thick] (-0.25, -0.45) -- (-0.71, 0.71);
            \draw[thick] (0.25, -0.45) -- (0.71, 0.71);
            \path[thick, tips, ->] (0.71, -0.71) to[out=135, in=45] (-0.71, -0.71);
            \node (s1) at (-0.6,-0.3) {$\mu$};
            \node (s2) at (0.6,-0.3) {$\nu$};
        \end{tikzpicture} \\
        \psi_{\mf{o}}\left(\begin{tikzpicture}[baseline=-3]
            \draw[gray!40, thick, fill=gray!40, domain=-45:225] plot ({cos(\x)}, {sin(\x)}) to[out=45, in=130] (0.71, -0.71);
            \draw[thick] (-0.71, -0.71) to[out=45, in=135] (0.71, -0.71);
            \node[draw, circle, inner sep=0pt, minimum size=4pt, fill=black] (p1) at (0, -0.41) {};
            \path[thick, tips, ->] (-0.71, -0.71) to[out=45, in=208] (-0.35, -0.49);
            \path[thick, tips, ->] (p1) to[out=0, in=152] (0.5, -0.55);
            \draw[thick] (p1) -- (-0.71, 0.71);
            \draw[thick] (0.1, -0.26) -- (0.71, 0.71);
            \node (s1) at (-0.6, -0.3) {$\mu$};
            \node (s2) at (0.6, -0.3) {$\nu$};
        \end{tikzpicture}\right) &= 
        \begin{tikzpicture}[baseline=-3]
            \draw[gray!40, thick, fill=gray!40, domain=-45:225] plot ({cos(\x)}, {sin(\x)}) to[out=45, in=130] (0.71, -0.71);
            \draw[thick] (-0.71, -0.71) to[out=45, in=135] (0.71, -0.71);
            \draw[thick] (-0.25, -0.45) -- (0.71, 0.71);
            \draw[line width=2mm, gray!40] (0.15, -0.33) -- (-0.15, 0.033);
            \draw[thick] (0.25, -0.45) -- (-0.71, 0.71);
            \path[thick, tips, ->] (-0.71, -0.71) to[out=45, in=135] (0.71, -0.71);
            \node (s1) at (-0.6,-0.3) {$\nu$};
            \node (s2) at (0.6,-0.3) {$\mu$};
        \end{tikzpicture}  & 
        \psi_{\mf{o}}\left(\begin{tikzpicture}[baseline=-3]
            \draw[gray!40, thick, fill=gray!40, domain=-45:225] plot ({cos(\x)}, {sin(\x)}) to[out=45, in=130] (0.71, -0.71);
            \draw[thick] (-0.71, -0.71) to[out=45, in=135] (0.71, -0.71);
            \node[draw, circle, inner sep=0pt, minimum size=4pt, fill=black] (p1) at (0, -0.41) {};
            \path[thick, tips, ->] (0.71, -0.71) to[out=135, in=332] (0.35, -0.49);
            \path[thick, tips, ->] (p1) to[out=0, in=28] (-0.5, -0.55);
            \draw[thick] (-0.1, -0.26) -- (-0.71, 0.71);
            \draw[thick] (p1) -- (0.71, 0.71);
            \node (s1) at (-0.6, -0.3) {$\mu$};
            \node (s2) at (0.6, -0.3) {$\nu$};
        \end{tikzpicture}\right) &= 
        \begin{tikzpicture}[baseline=-3]
            \draw[gray!40, thick, fill=gray!40, domain=-45:225] plot ({cos(\x)}, {sin(\x)}) to[out=45, in=130] (0.71, -0.71);
            \draw[thick] (-0.71, -0.71) to[out=45, in=135] (0.71, -0.71);
            \draw[thick] (0.25, -0.45) -- (-0.71, 0.71);
            \draw[line width=2mm, gray!40] (-0.15, -0.33) -- (0.15, 0.033);
            \draw[thick] (-0.25, -0.45) -- (0.71, 0.71);
            \path[thick, tips, ->] (0.71, -0.71) to[out=135, in=45] (-0.71, -0.71);
            \node (s1) at (-0.6,-0.3) {$\nu$};
            \node (s2) at (0.6,-0.3) {$\mu$};
        \end{tikzpicture}
    \end{align*}

    When the adjacent boundary orientation is clockwise, locally we get that
    \begin{align*}
        \phi_{\mf{o}} \circ \psi_{\mf{o}}\left(\begin{tikzpicture}[baseline=-3]
            \draw[gray!40, thick, fill=gray!40, domain=-45:225] plot ({cos(\x)}, {sin(\x)}) to[out=45, in=130] (0.71, -0.71);
            \draw[thick] (-0.71, -0.71) to[out=45, in=135] (0.71, -0.71);
            \node[draw, circle, inner sep=0pt, minimum size=4pt, fill=black] (p1) at (0, -0.41) {};
            \path[thick, tips, ->] (-0.71, -0.71) to[out=45, in=208] (-0.35, -0.49);
            \path[thick, tips, ->] (p1) to[out=0, in=152] (0.5, -0.55);
            \draw[thick] (p1) -- (-0.71, 0.71);
            \draw[thick] (0.1, -0.26) -- (0.71, 0.71);
            \node (s1) at (-0.6, -0.3) {$\mu$};
            \node (s2) at (0.6, -0.3) {$\nu$};
        \end{tikzpicture}\right) &=
        \phi_{\mf{o}}\left(\begin{tikzpicture}[baseline=-3]
            \draw[gray!40, thick, fill=gray!40, domain=-45:225] plot ({cos(\x)}, {sin(\x)}) to[out=45, in=130] (0.71, -0.71);
            \draw[thick] (-0.71, -0.71) to[out=45, in=135] (0.71, -0.71);
            \draw[thick] (-0.25, -0.45) -- (0.71, 0.71);
            \draw[line width=2mm, gray!40] (0.15, -0.33) -- (-0.15, 0.033);
            \draw[thick] (0.25, -0.45) -- (-0.71, 0.71);
            \path[thick, tips, ->] (-0.71, -0.71) to[out=45, in=135] (0.71, -0.71);
            \node (s1) at (-0.6,-0.3) {$\nu$};
            \node (s2) at (0.6,-0.3) {$\mu$};
        \end{tikzpicture}\right) \\
        &= \phi_{\mf{o}}\left( q \phantom{\cdot} C_{\nu}^{\mu} \phantom{\cdot} \begin{tikzpicture}[baseline=-3]
            \draw[gray!40, thick, fill=gray!40, domain=-45:225] plot ({cos(\x)}, {sin(\x)}) to[out=45, in=130] (0.71, -0.71);
            \draw[thick] (-0.71, -0.71) to[out=45, in=135] (0.71, -0.71);
            \draw[thick] (-0.71, 0.71) to[out=300, in=180] (0, 0) to[out=0, in=240] (0.71, 0.71);
            \path[thick, tips, ->] (-0.71, -0.71) to[out=45, in=135] (0.71, -0.71);
        \end{tikzpicture} + q^{-1} \phantom{\cdot}
        \begin{tikzpicture}[baseline=-3]
            \draw[gray!40, thick, fill=gray!40, domain=-45:225] plot ({cos(\x)}, {sin(\x)}) to[out=45, in=130] (0.71, -0.71);
            \draw[thick] (-0.71, -0.71) to[out=45, in=135] (0.71, -0.71);
            \draw[thick] (-0.25, -0.45) -- (-0.71, 0.71);
            \draw[thick] (0.25, -0.45) -- (0.71, 0.71);
            \path[thick, tips, ->] (-0.71, -0.71) to[out=45, in=135] (0.71, -0.71);
            \node (s1) at (-0.6,-0.3) {$\nu$};
            \node (s2) at (0.6,-0.3) {$\mu$};
        \end{tikzpicture}
        \right) \\
        &= q \phantom{\cdot} C_{\nu}^{\mu} \phantom{\cdot} \begin{tikzpicture}[baseline=-3]
            \draw[gray!40, thick, fill=gray!40, domain=-45:225] plot ({cos(\x)}, {sin(\x)}) to[out=45, in=130] (0.71, -0.71);
            \draw[thick] (-0.71, -0.71) to[out=45, in=135] (0.71, -0.71);
            \node[draw, circle, inner sep=0pt, minimum size=4pt, fill=black] (p1) at (0, -0.41) {};
            \path[thick, tips, ->] (-0.71, -0.71) to[out=45, in=208] (-0.35, -0.49);
            \path[thick, tips, ->] (p1) to[out=0, in=152] (0.5, -0.55);
            \draw[thick] (-0.71, 0.71) to[out=300, in=180] (0, 0) to[out=0, in=240] (0.71, 0.71);
            \path[thick, tips, ->] (-0.71, -0.71) to[out=45, in=135] (0.71, -0.71);
        \end{tikzpicture} + q^{-1} \phantom{\cdot}
        \begin{tikzpicture}[baseline=-3]
            \draw[gray!40, thick, fill=gray!40, domain=-45:225] plot ({cos(\x)}, {sin(\x)}) to[out=45, in=130] (0.71, -0.71);
            \draw[thick] (-0.71, -0.71) to[out=45, in=135] (0.71, -0.71);
            \node[draw, circle, inner sep=0pt, minimum size=4pt, fill=black] (p1) at (0, -0.41) {};
            \path[thick, tips, ->] (-0.71, -0.71) to[out=45, in=208] (-0.35, -0.49);
            \path[thick, tips, ->] (p1) to[out=0, in=152] (0.5, -0.55);
            \draw[thick] (-0.1, -0.26) -- (-0.71, 0.71);
            \draw[thick] (p1) -- (0.71, 0.71);
            \node (s1) at (-0.6, -0.3) {$\nu$};
            \node (s2) at (0.6, -0.3) {$\mu$};
        \end{tikzpicture}.
    \end{align*}

    If the orientation is instead counterclockwise, then we get that
    \begin{align*}
        \phi_{\mf{o}} \circ \psi_{\mf{o}}\left(\begin{tikzpicture}[baseline=-3]
            \draw[gray!40, thick, fill=gray!40, domain=-45:225] plot ({cos(\x)}, {sin(\x)}) to[out=45, in=130] (0.71, -0.71);
            \draw[thick] (-0.71, -0.71) to[out=45, in=135] (0.71, -0.71);
            \node[draw, circle, inner sep=0pt, minimum size=4pt, fill=black] (p1) at (0, -0.41) {};
            \path[thick, tips, ->] (0.71, -0.71) to[out=135, in=332] (0.35, -0.49);
            \path[thick, tips, ->] (p1) to[out=0, in=28] (-0.5, -0.55);
            \draw[thick] (-0.1, -0.26) -- (-0.71, 0.71);
            \draw[thick] (p1) -- (0.71, 0.71);
            \node (s1) at (-0.6, -0.3) {$\mu$};
            \node (s2) at (0.6, -0.3) {$\nu$};
        \end{tikzpicture}\right) &=
        \phi_{\mf{o}}\left(\begin{tikzpicture}[baseline=-3]
            \draw[gray!40, thick, fill=gray!40, domain=-45:225] plot ({cos(\x)}, {sin(\x)}) to[out=45, in=130] (0.71, -0.71);
            \draw[thick] (-0.71, -0.71) to[out=45, in=135] (0.71, -0.71);
            \draw[thick] (0.25, -0.45) -- (-0.71, 0.71);
            \draw[line width=2mm, gray!40] (-0.15, -0.33) -- (0.15, 0.033);
            \draw[thick] (-0.25, -0.45) -- (0.71, 0.71);
            \path[thick, tips, ->] (0.71, -0.71) to[out=135, in=45] (-0.71, -0.71);
            \node (s1) at (-0.6,-0.3) {$\nu$};
            \node (s2) at (0.6,-0.3) {$\mu$};
        \end{tikzpicture}\right) \\
        &= \phi_{\mf{o}}\left( q \phantom{\cdot}
        \begin{tikzpicture}[baseline=-3]
            \draw[gray!40, thick, fill=gray!40, domain=-45:225] plot ({cos(\x)}, {sin(\x)}) to[out=45, in=130] (0.71, -0.71);
            \draw[thick] (-0.71, -0.71) to[out=45, in=135] (0.71, -0.71);
            \draw[thick] (-0.25, -0.45) -- (-0.71, 0.71);
            \draw[thick] (0.25, -0.45) -- (0.71, 0.71);
            \path[thick, tips, ->] (0.71, -0.71) to[out=135, in=45] (-0.71, -0.71);
            \node (s1) at (-0.6,-0.3) {$\nu$};
            \node (s2) at (0.6,-0.3) {$\mu$};
        \end{tikzpicture}
        + (q^{-1}) (-q^3) C_{\mu}^{\nu} \phantom{\cdot} \begin{tikzpicture}[baseline=-3]
            \draw[gray!40, thick, fill=gray!40, domain=-45:225] plot ({cos(\x)}, {sin(\x)}) to[out=45, in=130] (0.71, -0.71);
            \draw[thick] (-0.71, -0.71) to[out=45, in=135] (0.71, -0.71);
            \draw[thick] (-0.71, 0.71) to[out=300, in=180] (0, 0) to[out=0, in=240] (0.71, 0.71);
            \path[thick, tips, ->] (0.71, -0.71) to[out=135, in=45] (-0.71, -0.71);
        \end{tikzpicture} \right) \\
        &= q \phantom{\cdot} \begin{tikzpicture}[baseline=-3]
            \draw[gray!40, thick, fill=gray!40, domain=-45:225] plot ({cos(\x)}, {sin(\x)}) to[out=45, in=130] (0.71, -0.71);
            \draw[thick] (-0.71, -0.71) to[out=45, in=135] (0.71, -0.71);
            \node[draw, circle, inner sep=0pt, minimum size=4pt, fill=black] (p1) at (0, -0.41) {};
            \path[thick, tips, ->] (0.71, -0.71) to[out=135, in=332] (0.35, -0.49);
            \path[thick, tips, ->] (p1) to[out=0, in=28] (-0.5, -0.55);
            \draw[thick] (p1) -- (-0.71, 0.71);
            \draw[thick] (0.1, -0.26) -- (0.71, 0.71);
            \node (s1) at (-0.6, -0.3) {$\mu$};
            \node (s2) at (0.6, -0.3) {$\nu$};
        \end{tikzpicture} - q^{2}  C_{\mu}^{\nu} \phantom{\cdot} \phantom{\cdot}
        \begin{tikzpicture}[baseline=-3]
            \draw[gray!40, thick, fill=gray!40, domain=-45:225] plot ({cos(\x)}, {sin(\x)}) to[out=45, in=130] (0.71, -0.71);
            \draw[thick] (-0.71, -0.71) to[out=45, in=135] (0.71, -0.71);
            \node[draw, circle, inner sep=0pt, minimum size=4pt, fill=black] (p1) at (0, -0.41) {};
            \path[thick, tips, ->] (0.71, -0.71) to[out=135, in=332] (0.35, -0.49);
            \path[thick, tips, ->] (p1) to[out=0, in=28] (-0.5, -0.55);
            \draw[thick] (-0.71, 0.71) to[out=300, in=180] (0, 0) to[out=0, in=240] (0.71, 0.71);
        \end{tikzpicture}.
    \end{align*}

    From here, we can use the state exchange relations, $(R5)$, to make the diagram increasingly stated, which does not change height orders or simplicity conditions.
\end{proof}

\begin{theorem}\label{theorem:SSOrdBasis}
    The set of all isotopy classes of increasingly stated, $\mathfrak{o}$-simple $\partial \Sigma^\prime$-tangle diagrams forms a $\mathbb{C}$-basis for $\mathscr{S}(\Sigma^\prime, \mathcal{P})$.
\end{theorem}
\begin{proof}
    This directly follows from Propositions \ref{prop:IsoOfModels} and \ref{prop:oSimpDiag}.
\end{proof}

\textbf{Note:} Unless specified otherwise, if we denote the stated skein algebra without specifying the markings or marked points (e.g., $\mathscr{S}(M)$), we assume that there is exactly one marking or marked point on its boundary. Additionally, the boundary of $M$ should be obvious. I will always clarify beforehand if there is an exception to this convention, however, it should be clear from the context of which it's given.

\section{Spherical Double Affine Hecke Algebra}
As discussed in section \ref{section:SkeinMods}, we introduced the Hecke algebra, $\mathbf{H}(S_n)$, as a deformation of $\mathbb{C}[S_n]$, the symmetric group algebra. In general, Hecke algebras can be defined over any Coxeter group. More specifically, we want to think of our Hecke algebras as deformations of a Weyl group algebra, $\mathbb{C}[W]$. Given a finite dimensional semisimple Lie algebra, $\mathfrak{g}$, the jump from Hecke algebra to affine Hecke algebras (AHA) and then to double affine Hecke algebras (DAHA) can be roughly understood as deformations of
\begin{center}
    \begin{tabular}{ccc}
        \text{Hecke} & $\longleftrightarrow$ & $W$ \\
        \text{AHA} & $\longleftrightarrow$ & $W \ltimes P^\vee$ \\
        \text{DAHA} & $\longleftrightarrow$ & $W \ltimes \left( P \oplus P^\vee \right)$ \\
    \end{tabular}
\end{center}
where $P$ and $P^\vee$ are weight and coweight lattices.

The construction of a DAHA can be done by starting with a quantum torus over $P \oplus P^\vee$, where the $q$-commuting coefficients are defined through a symplectic pairing, $\omega$, between $P$ and $P^\vee$. If we think of $P$ and $P^{\vee}$ as $\mathbb{Z}^n$, then the corresponding quantum torus is generated by $X_{1}^{\pm 1}, \cdots , X_{n}^{\pm 1}, Y_{1}^{\pm 1}, \cdots, Y_{n}^{\pm 1}$ with relations $X^\lambda Y^\mu = q^{\omega(\lambda,\mu)} Y^\mu X^\lambda$, where $X^\lambda := \prod_i X_i^{\lambda_i}$ and $Y^\mu := \prod_j Y_j^{\mu_j}$ for $\lambda \in P$ and $\mu \in P^\vee$. For now, we'll denote this quantum torus as $\mathbb{T}_{\omega}^{2n}$.
In general, the Weyl group action on $P$ and $P^\vee$ corresponds to permutations of the $X_i$ and $Y_i$ in $\mathbb{T}_{\omega}^{2n}$. When $\mathfrak{g} = \mathfrak{sl}_2$, the Weyl group action corresponds to simultaneously inverting $X$ and $Y$. This action gives an embedding, $W \hookrightarrow \operatorname{Out}\left(\mathbb{T}_{\omega}^{2n}\right) \cong \operatorname{SP}_{2n}(\mathbb{Z})$, which in turn determines the extension
$$\mathbb{T}_{\omega}^{2n} \longrightarrow \mathcal{H}_{q,t=1} \longleftarrow \mathbb{C}[W].$$
As we deform $\mathcal{H}_{q,t=1}$ using the formal parameter $t$, we simultaneously deform the group algebra $\mathbb{C}[W]$ into its Hecke algebra, using $t \in \mathbb{C}^\times$ as the distinguished parameter now instead of $q$.

Let's now specialize to the case we are concerned with in this thesis. When $\mathfrak{g} = \mathfrak{sl}_2$, $P$ and $P^\vee$ are both isomorphic to $\mathbb{Z}$ and $W \cong \mathbb{Z}/2\mathbb{Z} \cong S_2$. Therefore, $\mathbf{H}(W)$ has a single generator, $T$, and single relation, $(T - t)(T + t^{-1}) = 0$, making it a $2$-dimensional algebra. Hence, the corresponding $A_1$ DAHA is precisely the following.

\begin{definition}
    Define $\mathcal{H}_{q,t}$ to be the algebra generated by $X^{\pm 1}, Y^{\pm 1}$, and $T$, subject to the relations
    $$TXT = X^{-1},
    \quad TY^{-1}T=Y,
    \quad XY=q^{2}YXT^{2},
    \quad (T-t)(T+t^{-1})=0.$$
\end{definition}

\begin{remark}
    When $t=1$, our last relation gives us $T^2 = 1$ and thus $XY = q^{2}YX$, the same relation found in the quantum torus, $\mathbb{T}_{\omega}^{2}$.
\end{remark}

\begin{definition}
    Assume $t \neq \pm i$ and let $\mathbf{e} \in \mathcal{H}_{q,t}$ be the idempotent $(T + t^{-1})/(t + t^{-1})$. The corresponding \textit{spherical subalgebra} of $\mathcal{H}_{q,t}$ is defined as the two-sided ideal $\mathcal{SH}_{q,t} := \mathbf{e}\mathcal{H}_{q,t}\mathbf{e}$.
\end{definition}

Our spherical DAHA inherits its additive and multiplicative structure from $\mathcal{H}_{q,t}$. However, $1$ is no longer the multiplicative identity but rather $\mathbf{e}$ is.

\begin{remark}
    $\mathcal{H}_{q,t}$ is not commutative, even at the limit $q=1$. However, the spherical subalgebra $\mathcal{SH}_{q,t}$ is commutative at $q=1$.
\end{remark}

\begin{remark}
    In the limit $t=1$, $\mathcal{SH}_{q,t}$ is isomorphic to the Weyl-invariant subalgebra of $\mathbb{T}_{\omega}^2$.
\end{remark}

\begin{remark}
    This particular spherical DAHA specialized to $q=t=1$ is isomorphic to the ring of functions on the moduli space of flat $SL_2$-connections on a two-torus, $T^2$.
\end{remark}

\begin{lemma}[Lemma 2.24 in \cite{MR3947640}]
    Suppose $t^{2}q^{-2}-t^{-2}q^{2}$ is invertible. Then $\mathcal{H}_{q,t}\mathbf{e}\mathcal{H}_{q,t} = \mathcal{H}_{q,t}$ and $\mathcal{SH}_{q,t}$ is Morita equivalent to $\mathcal{H}_{q,t}$ via the functors
    \begin{align*}
        \mathcal{H}_{q,t}\operatorname{-Mod} &\to \mathcal{SH}_{q,t}\operatorname{-Mod} & \mathcal{SH}_{q,t}\operatorname{-Mod} &\to \mathcal{H}_{q,t}\operatorname{-Mod}\\
        M &\mapsto \mathbf{e}M & M &\mapsto \mathcal{H}_{q,t}\mathbf{e} \otimes_{\mathcal{SH}_{q,t}} M.
    \end{align*}
\end{lemma}

Therefore, if we want to better understand the representation theory of $\mathcal{H}_{q,t}$, it is sufficient to understand the representation theory of $\mathcal{SH}_{q,t}$. Furthermore, it was shown in \cite{MR3947640} that $\mathcal{SH}_{q,t}$ is isomorphic to a slightly modified version of the Kauffman bracket skein algebra of the punctured torus.

\begin{definition}
    For any surface $\Sigma$, let $K_{q,t}(\Sigma \setminus D^2)$ be the modified Kauffman bracket skein module defined as the quotient of $K_{q}(\Sigma \setminus D^2)$ by the relation
    \begin{center}
       \begin{tikzpicture}
            \draw[thick, dashed] (2, 1) ellipse (0.5 and 0.25);
            \draw[thick, dashed] (2, -1) ellipse (0.5 and 0.25);
            \draw[thick, dashed] (1.5, 1) -- (1.5, -1);
            \draw[thick, dashed] (2.5, 1) -- (2.5, -1);
            \draw[thick] (1.5, 0.4) to[out=225, in=90] (1.3, 0) to[out=270, in=180] (2, -0.5) to[out=0, in=270] (2.7, 0) to[out=90, in=315] (2.5, 0.4);
            \node[text width=3.7cm] at (4.8,0) {$ = (-q^{2} t^{-2} - q^{-2} t^{2})$ $\cdot$};
            \draw[thick, dashed] (7.3, 1) ellipse (0.5 and 0.25);
            \draw[thick, dashed] (7.3, -1) ellipse (0.5 and 0.25);
            \draw[thick, dashed] (6.8, 1) -- (6.8, -1);
            \draw[thick, dashed] (7.8, 1) -- (7.8, -1);
       \end{tikzpicture}.
   \end{center}
\end{definition}

In other words, this is exactly the same as the usual Kauffman bracket skein algebra, $K_q(\Sigma \setminus D^2)$, except we have the additional relation where a loop around the boundary can be removed at the cost of the constant $-q^{2} t^{-2} - q^{-2} t^{2}$.

\begin{remark}
    Since the identity component of the diffeomorphism group of a surface acts transitively, choosing a different disk to remove corresponds to isomorphic algebras. Therefore, we don't need to specify where the disk we removed is located.
\end{remark}

When $\Sigma = T^2$, the torus, notice that in the limit $t=1$ we get $K_{q,t=1}(T^2 \setminus D^2) \cong K_{q}(\Sigma)$ as algebras. Therefore, setting $t=1$ can be interpreted as filling in the disk and just considering the closed torus without the resulting boundary.

\begin{theorem}[3.9 in \cite{MR3947640}]\label{theorem:PetersIso}
    The algebras $K_{q,t}(T^2 \setminus D^2)$ and $\mathcal{SH}_{q,t}$ are isomorphic.
\end{theorem}

\chapter{The Stated Skein Algebra of the Torus with Boundary}

There are three main facts from the previous sections that I would like to highlight:
\begin{enumerate}
    \item There is a natural embedding, $K_q(T^2 \setminus D^2) \hookrightarrow \ms{S}(T^2 \setminus D^2)$, from the Kauffman bracket skein algebra of $T^2 \setminus D^2$ to the stated skein algebra of $T^2 \setminus D^2$.
    \item There is an algebra isomorphism, $\mathcal{SH}_{q,t} \xrightarrow{\sim} K_{q,t}(T^2 \setminus D^2)$, from the $A_1$ spherical double affine Hecke algebra to the modified Kauffman bracket skein algebra of $T^2 \setminus D^2$.
    \item The $A_1$ double affine Hecke algebra, $\mathcal{H}_{q,t}$, is Morita equivalent to it's spherical subalgebra, $\mathcal{SH}_{q,t}$.
\end{enumerate}
Together, these facts give us the central idea and guiding principle of this work: modules over $\ms{S}(T^2 \setminus D^2)$ provide us with modules for $\mathcal{SH}_{q,t}$, and therefore for the $A_1$ DAHA, $\mathcal{H}_{q,t}$.
Consider the following diagram of algebras.
\begin{center}
    \begin{tikzcd}
        K_q(T^2 \setminus D^2) \arrow[d, two heads] \arrow[r, hook] & \ms{S}(T^2 \setminus D^2) \\
        K_{q,t}(T^2 \setminus D^2) \arrow[r, "\sim"] & \mathcal{SH}_{q,t}
    \end{tikzcd}
\end{center}
There is no obvious direct map between $\ms{S}(T^2 \setminus D^2)$ and $\mathcal{SH}_{q,t}$. However, a module over $\ms{S}(T^2 \setminus D^2)$ can be turned (by restricting the skein algbera) into a module over $K_q(T^2 \setminus D^2)$. From here we can use the functor
\begin{align}\label{eq:ExtOfScal}
    K_{q,t}(T^2 \setminus D^2) \otimes_{K_{q}(T^2 \setminus D^2)} - : K_{q}(T^2 \setminus D^2)\operatorname{-Mod} \longrightarrow K_{q,t}(T^2 \setminus D^2)\operatorname{-Mod}
\end{align}
to create $\mathcal{SH}_{q,t}$-modules from these $K_q(T^2 \setminus D^2)$-modules. This is known as \textit{extension of scalars} and is the left adjoint to the restriction functor.\footnote{Given a map of rings, $f: R \to S$, there is also a right adjoint to the corresponding restriction functor, known as \textit{coextension of scalars}, which turns $\operatorname{Hom}(S, M)$ in the category of $R$-Mod into an $S$-module by $(s \cdot g)(s') := g(s' s)$.} We will explore a few different kinds of modules over $\ms{S}(T^2 \setminus D^2)$ in later chapters, but before we can discuss this, we should first better understand the algebra $\ms{S}(T^2 \setminus D^2)$.

\section{$\ms{S}(T^2 \setminus D^2)$ Notation}

In this section we will show that the algebra $\mathscr{S}(T^2 \setminus D^2)$ with one marking is generated by the twelve elements $B = \{ X_{1,0}(\mu_1, \nu_1), X_{2,0}(\mu_2, \nu_2), X_{3,0}(\mu_3, \nu_3) \, \mid \mu_i, \nu_i \in \{ \pm \} \}$, where
$$X_{1,0}(\mu_1, \nu_1) =
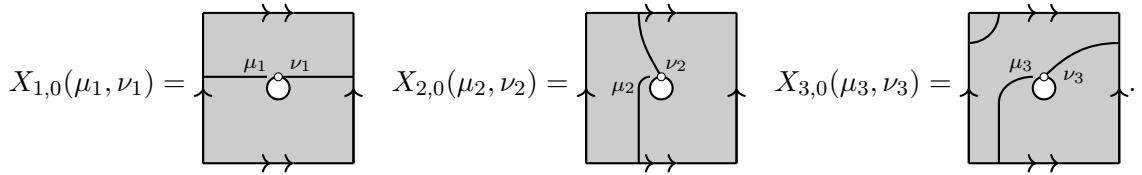
\begin{tikzpicture}[baseline=-1]
    \draw[draw=none, fill=gray!40] (0,1) -- (2,1) -- (2,-1) -- (0,-1) -- (1, -0.15) to[out=0, in=-90] (1.15, 0) to[out=90, in=0] (1, 0.15) to[out=180, in=90] (0.85, 0) to[out=-90, in=180] (1, -0.15) -- (0, -1) -- (0,1);
    \draw[thick] (0, 0.15) -- (0.85, 0.15);
    \draw[thick] (1, 0.15) -- (2, 0.15);
    \draw[thick, ->] (0,1) -- (0.9,1);
    \draw[thick, ->] (0.85,1) -- (1.2,1);
    \draw[thick] (1.2,1) -- (2,1);
    \draw[thick, ->] (0,-1) -- (0.9,-1);
    \draw[thick, ->] (0.85,-1) -- (1.2,-1);
    \draw[thick] (1.2,-1) -- (2,-1);
    \draw[thick, ->] (0,-1) -- (0,0);
    \draw[thick] (0,-1) -- (0,1);
    \draw[thick, ->] (2,-1) -- (2,0);
    \draw[thick] (2,-1) -- (2,1);
    \draw[thick, black] (1, 0) circle (0.15);
    \node[draw, circle, inner sep=0pt, minimum size=3pt, fill=white] at (1, 0.15) {};
    \node at (0.7, 0.3) {\scriptsize{$\mu_1$}};
    \node at (1.3, 0.3)  {\scriptsize{$\nu_1$}};
\end{tikzpicture} \quad
X_{2,0}(\mu_2, \nu_2) = \begin{tikzpicture}[baseline=-1]
    \draw[draw=none, fill=gray!40] (4,1) -- (6,1) -- (6,-1) -- (4,-1) -- (5, -0.15) to[out=0, in=-90] (5.15, 0) to[out=90, in=0] (5, 0.15) to[out=180, in=90] (4.85, 0) to[out=-90, in=180] (5, -0.15) -- (4, -1) -- (4,1);
    \draw[thick] (4.7, -1) to[out=90, in=270] (4.7, 0) to[out=90, in=180] (4.85, 0.15);
    \draw[thick] (5, 0.15) to[out=120, in=270] (4.7, 1);
    \draw[thick, ->] (4,1) -- (4.9,1);
    \draw[thick, ->] (4.85,1) -- (5.2,1);
    \draw[thick] (5.2,1) -- (6,1);
    \draw[thick, ->] (4,-1) -- (4.9,-1);
    \draw[thick, ->] (4.85,-1) -- (5.2,-1);
    \draw[thick] (5.2,-1) -- (6,-1);
    \draw[thick, ->] (4,-1) -- (4,0);
    \draw[thick] (4,-1) -- (4,1);
    \draw[thick, ->] (6,-1) -- (6,0);
    \draw[thick] (6,-1) -- (6,1);
    \draw[thick, black] (5, 0) circle (0.15);
    \node[draw, circle, inner sep=0pt, minimum size=3pt, fill=white] at (5, 0.15) {};
    \node at (4.5, 0) {\scriptsize{$\mu_2$}};
    \node at (5.2, 0.3)  {\scriptsize{$\nu_2$}};
\end{tikzpicture}\quad
X_{3,0}(\mu_3, \nu_3) =\begin{tikzpicture}[baseline=-1]
    \draw[draw=none, fill=gray!40] (8,1) -- (10,1) -- (10,-1) -- (8,-1) -- (9, -0.15) to[out=0, in=-90] (9.15, 0) to[out=90, in=0] (9, 0.15) to[out=180, in=90] (8.85, 0) to[out=-90, in=180] (9, -0.15) -- (8, -1) -- (8,1);
    \draw[thick] (8.4, -1) -- (8.4, -0.2) to[out=90, in=180] (8.85, 0.15);
    \draw[thick] (9, 0.15) to[out=45, in=180] (10, 0.6);
    \draw[thick] (8, 0.6) to[out=0, in=270] (8.4, 1);
    \draw[thick, ->] (8,1) -- (8.9,1);
    \draw[thick, ->] (8.85,1) -- (9.2,1);
    \draw[thick] (9.2,1) -- (10,1);
    \draw[thick, ->] (8,-1) -- (8.9,-1);
    \draw[thick, ->] (8.85,-1) -- (9.2,-1);
    \draw[thick] (9.2,-1) -- (10,-1);
    \draw[thick, ->] (8,-1) -- (8,0);
    \draw[thick] (8,-1) -- (8,1);
    \draw[thick, ->] (10,-1) -- (10,0);
    \draw[thick] (10,-1) -- (10,1);
    \draw[thick, black] (9, 0) circle (0.15);
    \node[draw, circle, inner sep=0pt, minimum size=3pt, fill=white] at (9, 0.15) {};
    \node at (8.7, 0.3) {\scriptsize{$\mu_3$}};
    \node at (9.4, 0.15)  {\scriptsize{$\nu_3$}};
\end{tikzpicture}.$$
The second subscript will be more thoroughly explained in section \ref{section:HalfTwists}.

\begin{definition}
    A curve is \textit{simple} if it does not contain any self crossings. Similarly, a $\partial \mathfrak{S}$-tangle diagram is \textit{simple} if it does not contain any self crossings on the interior of $\mathfrak{S}$ and has no trivial component.
\end{definition}
As we are working with isotopy classes of curves and isotopy classes of tangles, in general, we call $\alpha$ simple if there exists a representative of $\alpha$ that is simple.

\begin{definition}
    A closed component of a $\partial \mathfrak{S}$-tangle diagram is \textit{trivial} if it bounds an open disk in $\mathfrak{S}$. A tangle component is \textit{trivial} if it can be homotoped, relative to its endpoints, into a single marking.
\end{definition}

\begin{definition}
    A \textit{boundary curve} is a simple closed curve that is parallel to a boundary edge.
\end{definition}

\begin{definition}
    A \textit{parallel tangle} is one that can be homotoped, relative to its endpoints, to a boundary edge.
\end{definition}
An example of a parallel tangle in $\mathscr{S}(T^2 \setminus D^2)$ with one marking is
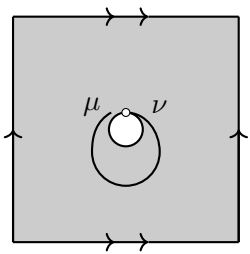
\begin{center}
    \begin{tikzpicture}[scale=1.5]
        \MarkedTorusBackground
        \draw[thick] (1, 0.15) to[out=0, in=90] (1.3, -0.2) to[out=270, in=0] (1, -0.5) to[out=180, in=270] (0.7, -0.2) to[out=90, in=210] (0.87, 0.15);
        \node[draw, circle, inner sep=0pt, minimum size=3pt, fill=white] at (1, 0.15) {};
        \node at (0.7, 0.2) {$\mu$};
        \node at (1.3, 0.2) {$\nu$};
    \end{tikzpicture}
\end{center}
as it can clearly be homotoped to the boundary.

Let $Y_1$, $Y_2$, and $Y_3$ be the meridian, longitude, and (1,1)-curve, respectively, and consider the stated skein algebra generated by $B$. Using the state-exchange relation, we can express each $Y_i$ in this algebra as
$$Y_i = q^{1/2} X_{i,0}(+,-) - q^{5/2} X_{i,0}(-,+)$$
for $i \mod 3$. Moreover, we can interchange heights using the proper height exchange relation.
\begin{figure}[h]
    \centering
    \begin{tikzpicture}[baseline=-3]
        \MarkedTorusBackground
        \draw[thick] (0, 0.4) -- (2, 0.4);
        \node[draw, circle, inner sep=0pt, minimum size=3pt, fill=white] at (1, 0.15) {};
    \end{tikzpicture} $= q^{1/2}$
    \begin{tikzpicture}[baseline=-3]
        \MarkedTorusBackground
        \draw[thick] (0, 0.15) -- (0.85, 0.15);
        \draw[thick] (1, 0.15) -- (2, 0.15);
        \node[draw, circle, inner sep=0pt, minimum size=3pt, fill=white] at (1, 0.15) {};
        \node at (0.65, 0.35) {$+$};
        \node at (1.35, 0.35) {$-$};
    \end{tikzpicture} $- q^{-5/2}$
    \begin{tikzpicture}[baseline=-3]
        \MarkedTorusBackground
        \draw[thick] (0, 0.15) -- (0.85, 0.15);
        \draw[thick] (1, 0.15) -- (2, 0.15);
        \node[draw, circle, inner sep=0pt, minimum size=3pt, fill=white] at (1, 0.15) {};
        \node at (0.65, 0.35) {$-$};
        \node at (1.35, 0.35) {$+$};
    \end{tikzpicture}
    \caption[Expressing $Y_1$ as stated $\partial (T^2 \setminus D^2)$-tangles]{Expressing $Y_1$ as stated $\partial (T^2 \setminus D^2)$-tangles using the state exchange relation}
    \label{fig:Y1StateExchangeRelation}
\end{figure}

As we're only considering $\mathscr{S}(T^2 \setminus D^2)$ with a single marking, there is only one simple parallel tangle and one boundary curve.
The boundary curve can be expressed completely in terms of $Y_i$s and constants, as a quick calculation shows that
\begin{align*}
    \begin{tikzpicture}[baseline=-1]
        \MarkedTorusBackground
        \draw[thick, fill=none] (1,0) circle (0.5);
        \node[draw, circle, inner sep=0pt, minimum size=3pt, fill=white] at (1, 0.15) {};
    \end{tikzpicture} = q Y_1 Y_2 Y_3 - q^2 Y_1^2 - q^{-2} Y_2^2 - q^2 Y_3^2 + q^2 + q^{-2},
\end{align*}
and for any $\mu, \nu \in \left\{\pm\right\}$,
\begin{align*}
    \begin{tikzpicture}[baseline=-1]
        \MarkedTorusBackground
        \draw[thick] (1, 0.15) to[out=0, in=90] (1.3, -0.2) to[out=270, in=0] (1, -0.5) to[out=180, in=270] (0.7, -0.2) to[out=90, in=210] (0.87, 0.15);
        \node[draw, circle, inner sep=0pt, minimum size=3pt, fill=white] at (1, 0.15) {};
        \node at (0.7, 0.2) {$\mu$};
        \node at (1.3, 0.2) {$\nu$};
    \end{tikzpicture} = q Y_1 Y_2 X_{3,0}\left(\mu, \nu\right) - q^2 X_{1,0}\left(\mu, \nu\right) Y_1 - q^{-2} X_{2,0}\left(\mu, \nu\right) Y_2 - q^2 X_{3,0}\left(\mu, \nu\right) Y_3 - C_{\mu}^{\nu}.
\end{align*}
The details of this calculation can be found in \ref{appendix:ParaTangCalc} of the appendix. Therefore, the rest of this section is devoted to checking that we can generate all nonparallel tangles and non-boundary curves.

There is clearly a $4$-to-$1$ correspondence between any simple stateless $\partial \left( T^2 \setminus D^2 \right)$-tangle and simple stated $\partial \left( T^2 \setminus D^2 \right)$-tangles.
Therefore, we will disregard the states in the following lemmas, as the proofs are independent of the possible states.
Additionally, we will overlook relative heights along our marking, as we can readily obtain all possible heights using $(R_6)$ alongside $Y_1$, $Y_2$, and $Y_3$.

\section{Half-Twists Around the Boundary}\label{section:HalfTwists}

Before presenting our results, it is necessary to first explain the significance of the second index, $r$, in $X_{i,r}(\mu, \nu)$.
While simple closed (non-boundary) curves on the torus (with boundary) are conventionally classified by their slopes, which are numbers in $\mathbb{Q} \cup \frac{1}{0}$, one might expect a similar classification for our tangles, considering that each tangle must start and end at the same marking.
However, this analogy breaks down when considering tangles with what we'll call ``twists'' around the boundary.

For instance, if we trace along the path of $X_{3,0}$ and introduce a twist around the boundary just before reaching the marking, we get a different element back.
$$
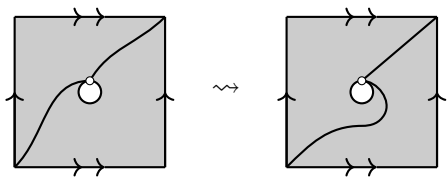
\begin{tikzpicture}[baseline=-1]
    \MarkedTorusBackground
    \draw[thick] (0, -1) to[out=45, in=180] (1, 0.15) to[out=60, in=225] (2, 1);
    \node[draw, circle, inner sep=0pt, minimum size=3pt, fill=white] at (1, 0.15) {};
\end{tikzpicture}
\quad\rightsquigarrow\quad
\begin{tikzpicture}[baseline=-1]
    \MarkedTorusBackground
    \draw[thick] (2, 1) -- (1, 0.15) to[out=0, in=60] (1.3, -0.3) to[out=240, in=0] (1,-0.45) to[out=180, in=45] (0, -1);
    \node[draw, circle, inner sep=0pt, minimum size=3pt, fill=white] at (1, 0.15) {};
\end{tikzpicture}$$
As these twists resemble half Dehn twists around the boundary, we distinguish these twists by indexing over $\frac{1}{2}\mathbb{Z}$ instead of $\mathbb{Z}$. In particular, $X_{i, r}$ is a \textit{full} Dehn twist of $X_{i, r-1}$.
The following examples illustrate $X_{1,r}$ for various choices of $r \in \frac{1}{2}\mathbb{Z}$.
$$
\begin{tikzpicture}
    \MarkedTorusBackground
    \draw[thick] (1, 0.15) to[out=180, in=90] (0.5, -0.2) to[out=270, in=180] (1, -0.75) to[out=0, in=180] (2, 0.15);
    \draw[thick] (1, 0.15) to[out=180, in=120] (0.7, -0.3) to[out=300, in=180] (1,-0.45) to[out=0, in=270] (1.5, 0.1) to[out=90, in=0] (1, 0.6) to[out=180, in=0] (0, 0.15);
    \node[draw, circle, inner sep=0pt, minimum size=3pt, fill=white] at (1, 0.15) {};
    \node at (1,-1.5) {$X_{1,-1}$};
\end{tikzpicture}
\begin{tikzpicture}
    \MarkedTorusBackground
    \draw[thick] (0, 0.15) -- (1, 0.15) to[out=180, in=120] (0.7, -0.3) to[out=300, in=180] (1,-0.45) to[out=0, in=180] (2, 0.15);
    \node[draw, circle, inner sep=0pt, minimum size=3pt, fill=white] at (1, 0.15) {};
    \node at (1,-1.5) {$X_{1,-\frac{1}{2}}$};
    \node (space) at (-1, 0) {};
\end{tikzpicture}
\begin{tikzpicture}
    \MarkedTorusBackground
    \draw[thick] (0, 0.15) -- (1, 0.15) -- (2, 0.15);
    \node[draw, circle, inner sep=0pt, minimum size=3pt, fill=white] at (1, 0.15) {};
    \node at (1,-1.5) {$X_{1,0}$};
    \node (space) at (-1, 0) {};
\end{tikzpicture}
\begin{tikzpicture}
    \MarkedTorusBackground
    \draw[thick] (2, 0.15) -- (1, 0.15) to[out=0, in=60] (1.3, -0.3) to[out=240, in=0] (1,-0.45) to[out=180, in=0] (0, 0.15);
    \node[draw, circle, inner sep=0pt, minimum size=3pt, fill=white] at (1, 0.15) {};
    \node at (1,-1.5) {$X_{1,\frac{1}{2}}$};
    \node (space) at (-1, 0) {};
\end{tikzpicture}
\begin{tikzpicture}
    \MarkedTorusBackground
    \draw[thick] (1, 0.15) to[out=0, in=90] (1.5, -0.2) to[out=270, in=0] (1, -0.75) to[out=180, in=0] (0, 0.15);
    \draw[thick] (1, 0.15) to[out=0, in=60] (1.3, -0.3) to[out=240, in=0] (1,-0.45) to[out=180, in=270] (0.5, 0.1) to[out=90, in=180] (1, 0.6) to[out=0, in=180] (2, 0.15);
    \node[draw, circle, inner sep=0pt, minimum size=3pt, fill=white] at (1, 0.15) {};
    \node at (1,-1.5) {$X_{1,1}$};
    \node (space) at (-1, 0) {};
\end{tikzpicture}
$$

\begin{remark}
    The distinction between two tangles with different numbers of twists becomes apparent when comparing these to fundamental groups with different base points.
    Let $x$ be a point on the boundary of $T^2 \setminus D^2$, $p$ be a point in the interior of $T^2 \setminus D^2$, and define $\operatorname{Sim}\left(T^2 \setminus D^2, z\right)$ as the subset of simple elements\footnote{An element of a fundamental group is simple if it has a representative that does not contain any self-intersections.} in $\pi_1\left(T^2 \setminus D^2, z\right)$.
    Although there exists an isomorphism between the fundamental groups with different base points, $\varphi: \pi_1\left(T^2 \setminus D^2, x\right) \xrightarrow{\sim} \pi_1\left(T^2 \setminus D^2, p\right)$, it is not necessarily true that $\varphi\left(\operatorname{Sim}(T^2 \setminus D^2, x)\right) = \operatorname{Sim}\left(T^2 \setminus D^2, p\right)$. In fact, we instead observe a strict inclusion $\varphi\left(\operatorname{Sim}(T^2 \setminus D^2, x)\right) \subsetneq \operatorname{Sim}\left(T^2 \setminus D^2, p\right)$, as partially demonstrated in figure \ref{fig:ProperSimpleEles}.
    The set of simple closed curves in $\pi_1(T^2 \setminus D^2)$ has already been classified in \cite{Birman_Series_1984, Cohen_Metzler_Zimmermann_1981} using cyclic reduction of words. However, this is unfortunately only under the subtle assumption that our base point is not on the boundary and so we will need to do additional work to fill in the gaps.
    \begin{figure}[h]
        \centering
        \begin{tikzpicture}
            \MarkedTorusBackground
            \draw[thick] (0.5, 1) -- (1, 0.6) -- (1.5, 1);
            \draw[thick] (0.5, -1) to[out=90, in=180] (1, 0.4) to[out=0, in=90] (1.5, -1);
            \node[draw, circle, inner sep=0pt, minimum size=3pt, fill=white] at (1, 0.6) {};
            \node at (1, -1.3) {Base point $p$};
            \node (space) at (3, 0) {};
        \end{tikzpicture}
        \begin{tikzpicture}
            \MarkedTorusBackground
            \draw[thick] (0.5, 1) -- (1, 0.15) -- (1.5, 1);
            \draw[thick] (0.5, -1) to[out=90, in=180] (1, 0.4) to[out=0, in=90] (1.5, -1);
            \node[draw, circle, inner sep=0pt, minimum size=3pt, fill=white] at (1, 0.15) {};
            \node at (1, -1.3) {Base point $x$};
            \node (space) at (-1, 0) {};
        \end{tikzpicture}
        \caption[Strict inclusion of simple elements of fundamental groups]{Left: a simple element with the base point off of the boundary. Right: after making a change of base point, it is no longer simple.}
        \label{fig:ProperSimpleEles}
    \end{figure}
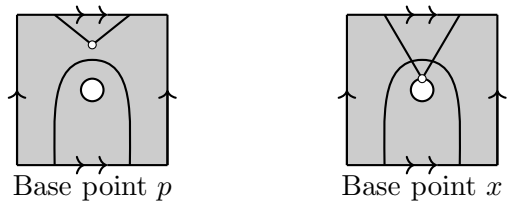
\end{remark}

\section{Generators of $\ms{S}(T^2 \setminus D^2)$}

Given any marked surface, $\mathfrak{S}$, L\^{e} shown in \cite{MR3827810} (Theorem \ref{theorem:OrBasis}) that the set of all isotopy classes of increasingly stated, simple $\partial\mathfrak{S}$-tangle diagrams forms a basis for $\mathscr{S}(\mathfrak{S})$. Therefore, our approach will be to first classify all possible simple (stateless) tangles in $T^2 \setminus D^2$ that start and end on $x \in \partial \left( T^2 \setminus D^2 \right)$ and then demonstrate that all basis elements are contained in the subalgebra generated by $\left\{ X_{1,0}(\mu_1, \nu_1), X_{2,0}(\mu_2, \nu_2), X_{3,0}(\mu_3, \nu_3)\right\}$ for all $\mu_i, \nu_j \in \{ \pm \}$.

In theorem \ref{theorem:classification}, we classify a large class of closed curves by their slope and number of twists. We construct this classification by covering a neighborhood of $\partial \left( T^2 \setminus D^2 \right)$ with an annulus and proceed to use a Seifert-Van Kampen-style approach to calculate the slope outside the annulus as well as determine the number of twists inside the annulus.

To clean things up a bit, we'll introduce some more notation. We will denote the following set of tuples as
\[ \mathcal{I} := \left.\left\{ (p,q,r) \in \mathbb{Z} \times \mathbb{Z} \times \frac{1}{2}\mathbb{Z} \, \mid \gcd(p,q) = 1 \right\} \right/ \sim \]
where $(p,q,r) \sim (-p,-q,r)$. Additionally, for any $x \in T^2 \setminus D^2$, define $\Omega_{x}$ to be the set of isotopy classes of simple unoriented closed non-parallel curves that begin and end on $x$. Any time we refer to a representative or an isotopic representative of an element in $\Omega_x$, we are always assuming that this representative begins and ends at the point $x$.

\newpage
\begin{theorem}\label{theorem:classification}
    Let $x \in \partial \left( T^2 \setminus D^2 \right)$. There exists a bijection $f_{A}: \mathcal{I} \longrightarrow \Omega_x$.
\end{theorem}
\begin{proof}
    We will construct $f_{A}$ by first constructing a map, $\widetilde{f}_{A}$, from $\mathcal{I}$ to $\widetilde{\Omega}_x$, the set of simple closed unoriented non-parallel curves that begin and end on $x$, and then projecting onto the corresponding set of isotopy classes of curves.
    \begin{center}
        \begin{tikzcd}
            \mathcal{I} \arrow[rr, hook, "\widetilde{f}_{A}"] \arrow[rrd, "f_{A}"'] & & \widetilde{\Omega}_x \arrow[d, two heads, "\pi_{\text{iso}}"] \\
            & & \Omega_x
        \end{tikzcd}
    \end{center}
    
    We'll first introduce some definitions. Let $\varepsilon > 0$ be sufficiently small. Define $B_{\varepsilon}(m,n)$ to be the open ball of radius $\varepsilon$ at point $(m,n) \in \mathbb{R}^2$, $E_{\varepsilon} := \mathbb{R}^{2} \setminus \left( \bigcup_{(m,n) \in \mathbb{Z}^2} B_{\varepsilon}(m,n) \right)$, and $\Phi_{\varepsilon}$ to be the restriction of the identity map on $\mathbb{R}^2$ to $E_{\varepsilon}$, composed with the obvious covering map.
    $$\Phi_{\varepsilon} : E_{\varepsilon} \longrightarrow T^2 \setminus D^2$$
    Let $A := \Phi_{\varepsilon}\left( \overline{E_{\epsilon} \setminus E_{2\varepsilon}} \right)$ be the closed annulus and define $\partial_A := \partial A \setminus \partial (T^2 \setminus D^2)$, the ``outer boundary''. $\partial_A$ will constantly serve as a reference point throughout the rest of this proof and is the dividing bridge between the $\{ p, q \}$ and $\{ r \}$ in the tuple.
    
    Choose any triple $(p,q,r) \in \mathcal{I}$. Let $\gamma_{p,q}$ be the graph of the function $y = \frac{p}{q}x$ in $\mathbb{R}^2$ (or $x=0$ when $q=0$) and consider $\gamma_0 := \Phi_{\varepsilon}\left( \gamma_{p,q} \right) \cap \left( T^2 \setminus A \right)$.
    Since $\gamma_{p,q}$ has constant slope, the two endpoints of $\gamma_{0}$ must lie on $\partial_A$ as antipodal points.
    
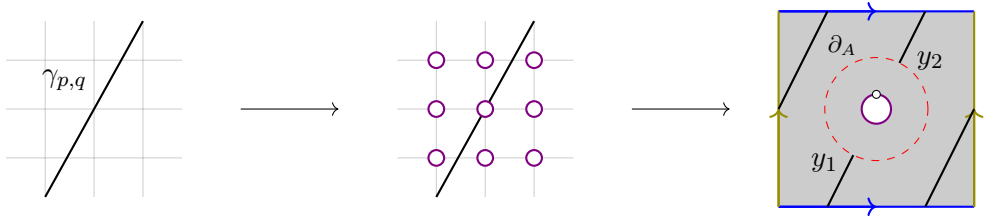
\begin{center}
        \begin{tikzpicture}[scale=1.3]
            \draw[help lines, ystep=0.5, xstep=0.5, color=gray!40] (-4.9, -0.9) grid (-3.1, 0.9);
            \draw[thick] (-4.5, -0.9) -- (-3.5, 0.9);
            \node[text width=0.6cm] at (-4.3, 0.3) {$\gamma_{p,q}$};
            \draw[->] (-2.5, 0) -- (-1.5, 0);
            \draw[help lines, ystep=0.5, xstep=0.5, color=gray!40] (-0.9, -0.9) grid (0.9, 0.9);
            \draw[thick] (-0.5, -0.9) -- (0.5, 0.9);
            \foreach \x in {-0.5, 0, 0.5}
                \foreach \y in {-0.5, 0, 0.5}
                    \draw[thick, draw=violet, fill=white] (\x, \y) circle (0.08);
            \draw[->] (1.5, 0) -- (2.5, 0);
            \draw[draw=none, fill=gray!40] (3,1) -- (5,1) -- (5,-1) -- (3,-1) -- (4, -0.15) to[out=0, in=-90] (4.15, 0) to[out=90, in=0] (4, 0.15) to[out=180, in=90] (3.85, 0) to[out=-90, in=180] (4, -0.15) -- (3, -1) -- (3,1);
            \draw[thick, blue, ->] (3,1) -- (4,1);
            \draw[thick, blue] (3,1) -- (5,1);
            \draw[thick, blue, ->] (3,-1) -- (4,-1);
            \draw[thick, blue] (3,-1) -- (5,-1);
            \draw[thick, olive, ->] (3,-1) -- (3,0);
            \draw[thick, olive] (3,-1) -- (3,1);
            \draw[thick, olive, ->] (5,-1) -- (5,0);
            \draw[thick, olive] (5,-1) -- (5,1);
            \draw[thick] (3.5, -1) -- (4.5, 1);
            \draw[thick] (4.5, -1) -- (5, 0);
            \draw[thick] (3, 0) -- (3.5, 1);
            \draw[draw=red, fill=gray!40, dashed] (4, 0) circle[radius=15pt];
            \draw[thick, violet, fill=white] (4, 0) circle (0.15);
            \node[draw, circle, inner sep=0pt, minimum size=3pt, fill=white] at (4, 0.15) {};
            \node at (3.47, -0.55) {$y_1$};
            \node at (4.55, 0.5) {$y_2$};
            \node at (3.65, 0.65) {\footnotesize{$\partial_A$}};
        \end{tikzpicture}
    \end{center}
    
    Let $x_0$ be the unique point on $\partial_A$ without a unique geodesic to $x$ through $A$.
    Let $x_1, x_2 \in \partial_A$ be the midpoints between $x_0$ and its antipodal counterpart, labeled in clockwise order from $x_0$, and let $c_1$ be the geodesic path from $x_1$ to $x$ within $A$.
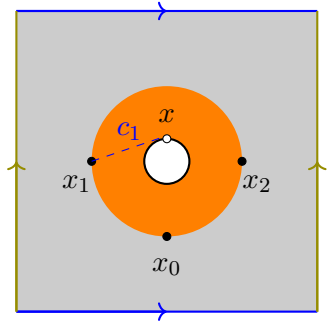
    \begin{center}
        \begin{tikzpicture}[scale=2]
            \draw[draw=none, fill=gray!40] (4,1) -- (6,1) -- (6,-1) -- (4,-1) -- (5, -0.15) to[out=0, in=-90] (5.15, 0) to[out=90, in=0] (5, 0.15) to[out=180, in=90] (4.85, 0) to[out=-90, in=180] (5, -0.15) -- (4, -1) -- (4,1);
            \draw[thick, blue, ->] (4,1) -- (5,1);
            \draw[thick, blue] (4,1) -- (6,1);
            \draw[thick, blue, ->] (4,-1) -- (5,-1);
            \draw[thick, blue] (4,-1) -- (6,-1);
            \draw[thick, olive, ->] (4,-1) -- (4,0);
            \draw[thick, olive] (4,-1) -- (4,1);
            \draw[thick, olive, ->] (6,-1) -- (6,0);
            \draw[thick, olive] (6,-1) -- (6,1);
            \draw[draw=none, fill=orange, opacity=0.7] (5,0) circle[radius=0.5cm];
            \node[draw, circle, inner sep=0pt, minimum size=3pt, fill=black] at (5.5, 0) {};
            \node[draw, circle, inner sep=0pt, minimum size=3pt, fill=black] at (4.5, 0) {};
            \node[draw, circle, inner sep=0pt, minimum size=3pt, fill=black] at (5, -0.5) {};
            \draw[thick, black, fill=white] (5, 0) circle (0.15);
            \node[draw, circle, inner sep=0pt, minimum size=3pt, fill=white] at (5, 0.15) {};
            \draw[dashed, blue] (4.5, 0) -- (5, 0.17);
            \node at (4.4, -0.15) {$x_1$};
            \node at (5.6, -0.15) {$x_2$};
            \node at (5, -0.7) {$x_0$};
            \node at (5, 0.3) {$x$};
            \node[blue] at (4.75, 0.2) {$c_1$};
        \end{tikzpicture}
    \end{center}

    We'll assume $y_1$ and $y_2$ are labeled such that $y_1$ is closer to $x_1$ on $\partial_A$. If they are both the same distance away from $x_1$, then we must be in the case some $y_i = x_0$, as $y_1$ and $y_2$ are antipodal points. If this happens, label this $y_i$ (the south pole) as $y_1$ and label the other (the north pole) as $y_2$. Let $\gamma_i$ be the geodesic paths (containing their endpoints) on $\partial_A$ from $y_i$ to $x_i$. This should correspond to a clockwise path when $\frac{p}{q}$ is a positive slope or equal to $\frac{1}{0}$, and a counterclockwise path when $\frac{p}{q}$ is a negative slope.\footnote{Our choice of clockwise rotation when the slope is $\frac{1}{0}$ was made for notational convenience outside of this proof. Alternatively, one could instead use a counterclockwise rotation by swapping the labels $y_1$ and $y_2$. The difference between these classifications results in a shift of $r$ by $1/2$ whenever $(p,q) = (1,0)$.} Define $\delta := \gamma_0 \sqcup \gamma_1 \sqcup \gamma_2$ and notice that this is simple.
    \begin{figure}[h]
    \begin{center}
        \begin{tikzpicture}[scale=1.5]
            \draw[draw=none, fill=gray!40] (0,1) -- (2,1) -- (2,-1) -- (0,-1) -- (1, -0.15) to[out=0, in=-90] (1.15, 0) to[out=90, in=0] (1, 0.15) to[out=180, in=90] (0.85, 0) to[out=-90, in=180] (1, -0.15) -- (0, -1) -- (0,1);
            \draw[thick, blue, ->] (0,1) -- (1,1);
            \draw[thick, blue] (0,1) -- (2,1);
            \draw[thick, blue, ->] (0,-1) -- (1,-1);
            \draw[thick, blue] (0,-1) -- (2,-1);
            \draw[thick, olive, ->] (0,-1) -- (0,0);
            \draw[thick, olive] (0,-1) -- (0,1);
            \draw[thick, olive, ->] (2,-1) -- (2,0);
            \draw[thick, olive] (2,-1) -- (2,1);
            \draw[draw=none, fill=orange, opacity=0.7] (1,0) circle[radius=0.5cm];
            \draw[thick] (0, 0) -- (0.5, 0);
            \draw[thick] (1.5, 0) -- (2, 0);
            \draw[thick, black, fill=white] (1, 0) circle (0.15);
            \node[draw, circle, inner sep=0pt, minimum size=3pt, fill=white] at (1, 0.15) {};
            \node[draw, circle, inner sep=0pt, minimum size=3pt, fill=black] at (1.5, 0) {};
            \node[draw, circle, inner sep=0pt, minimum size=3pt, fill=black] at (0.5, 0) {};
            \node at (0.4, -0.15) {$x_1$};
            \node at (1.6, -0.15) {$x_2$};
            \draw[draw=none, fill=gray!40] (3,1) -- (5,1) -- (5,-1) -- (3,-1) -- (4, -0.15) to[out=0, in=-90] (4.15, 0) to[out=90, in=0] (4, 0.15) to[out=180, in=90] (3.85, 0) to[out=-90, in=180] (4, -0.15) -- (3, -1) -- (3,1);
            \draw[thick, blue, ->] (3,1) -- (4,1);
            \draw[thick, blue] (3,1) -- (5,1);
            \draw[thick, blue, ->] (3,-1) -- (4,-1);
            \draw[thick, blue] (3,-1) -- (5,-1);
            \draw[thick, olive, ->] (3,-1) -- (3,0);
            \draw[thick, olive] (3,-1) -- (3,1);
            \draw[thick, olive, ->] (5,-1) -- (5,0);
            \draw[thick, olive] (5,-1) -- (5,1);
            \draw[draw=none, fill=orange, opacity=0.7] (4,0) circle[radius=0.5cm];
            \draw[thick, purple, domain=180:255] plot ({0.5*cos(\x)+4}, {0.5*sin(\x)});
            \draw[thick, purple, domain=0:53] plot ({0.5*cos(\x)+4}, {0.5*sin(\x)});
            \draw[thick] (3.61, -1) -- (3.87, -0.48);
            \draw[thick] (4.3, 0.4) -- (4.6, 1);
            \draw[thick] (4.6, -1) -- (5, -0.2);
            \draw[thick] (3, -0.2) -- (3.61, 1);
            \draw[thick, black, fill=white] (4, 0) circle (0.15);
            \node[draw, circle, inner sep=0pt, minimum size=3pt, fill=white] at (4, 0.15) {};
            \node[draw, circle, inner sep=0pt, minimum size=3pt, fill=black] at (4.5, 0) {};
            \node[draw, circle, inner sep=0pt, minimum size=3pt, fill=black] at (3.5, 0) {};
            \node at (3.4, 0.16) {$x_1$};
            \node at (4.65, -0.15) {$x_2$};
            \node at (3.69, -0.22) {$\gamma_1$};
            \node at (4.55, 0.27) {$\gamma_2$};
            \draw[draw=none, fill=gray!40] (6,1) -- (8,1) -- (8,-1) -- (6,-1) -- (7, -0.15) to[out=0, in=-90] (7.15, 0) to[out=90, in=0] (7, 0.15) to[out=180, in=90] (6.85, 0) to[out=-90, in=180] (7, -0.15) -- (6, -1) -- (6, 1);
            \draw[thick, blue, ->] (6,1) -- (7,1);
            \draw[thick, blue] (6,1) -- (8,1);
            \draw[thick, blue, ->] (6,-1) -- (7,-1);
            \draw[thick, blue] (6,-1) -- (8,-1);
            \draw[thick, olive, ->] (6,-1) -- (6,0);
            \draw[thick, olive] (6,-1) -- (6,1);
            \draw[thick, olive, ->] (8,-1) -- (8,0);
            \draw[thick, olive] (8,-1) -- (8,1);
            \draw[draw=none, fill=orange, opacity=0.7] (7,0) circle[radius=0.5cm];
            \draw[thick, purple, domain=127:180] plot ({0.5*cos(\x)+7}, {0.5*sin(\x)});
            \draw[thick, purple, domain=-37:0] plot ({0.5*cos(\x)+7}, {0.5*sin(\x)});
            \draw[thick] (6.1, 1) -- (6.7, 0.4);
            \draw[thick] (7.4, -0.3) -- (8, -0.9);
            \draw[thick] (6, -0.9) -- (6.1, -1);
            \draw[thick, black, fill=white] (7, 0) circle (0.15);
            \node[draw, circle, inner sep=0pt, minimum size=3pt, fill=white] at (7, 0.15) {};
            \node[draw, circle, inner sep=0pt, minimum size=3pt, fill=black] at (7.5, 0) {};
            \node[draw, circle, inner sep=0pt, minimum size=3pt, fill=black] at (6.5, 0) {};
            \node at (6.4, -0.15) {$x_1$};
            \node at (7.65, 0.15) {$x_2$};
            \node at (6.4, 0.27) {$\gamma_1$};
            \node at (7.6, -0.25) {$\gamma_2$};
        \end{tikzpicture}
        \caption[Constructing geodesic paths, $\gamma_i$]{Left: When $(p,q) = (0,1)$ we get $y_1 = x_1$ and $y_2 = x_2$ and so each $\gamma_i$ is trivial. Middle: When $\gamma_0$ has a positive slope, we trace out a geodesic path clockwise from $y_i$ to $x_i$. Right: When $\gamma_0$ has a negative slope we move counterclockwise instead.}
    \end{center}
    \end{figure}
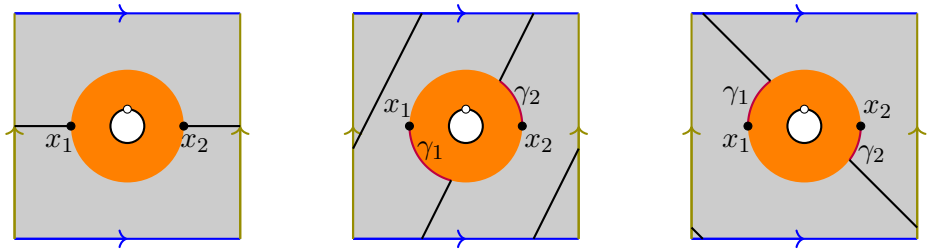

    Recall that the mapping class group, $\text{Mod}\left(\Sigma\right)$, of a surface, $\Sigma$, is the group of isotopy classes of elements of $\operatorname{Homeo}_{+}(\Sigma, \partial \Sigma)$, which pointwise-fix $\partial \Sigma$. Also recall that the mapping class group of the annulus is $\text{Mod}(A) = \langle \sigma \rangle$ where $\sigma$ is the clockwise Dehn twist along the curve $\varsigma$, parallel to the boundaries.
    
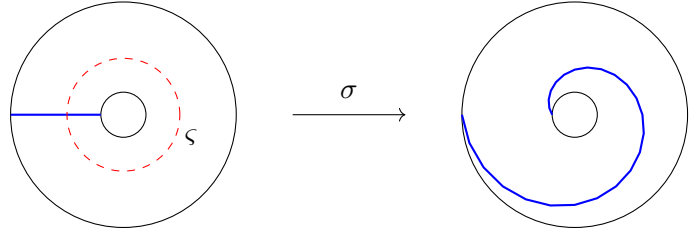
\begin{center}
        \begin{tikzpicture}[scale=1.5]
            \draw (1, 0) circle (1);
            \draw (1, 0) circle (0.2);
            \draw[thick, blue] (0, 0) -- (0.8, 0);
            \draw[dashed, red] (1, 0) circle (0.5);
            \node at (1.6, -0.2) {$\varsigma$};
            \draw[->] (2.5, 0) -- (3.5, 0);
            \node at (3, 0.2) {$\sigma$};
            \draw (5, 0) circle (1);
            \draw (5, 0) circle (0.2);
            \draw[thick, blue, domain=0:1] plot ({5+(1-0.8*\x)*cos(360*\x + 180)}, {(1-0.8*\x)*sin(360*\x + 180)});
        \end{tikzpicture}
    \end{center}
    For $r \in \frac{1}{2}\mathbb{Z}$, notice that there is a unique decomposition, $r = \frac{r_1 + r_2}{2}$, such that $r_1$,$r_2 \in \mathbb{Z}$ and $r_1 - r_2  \in \{0,1\}$.
    Let $\alpha_1 = \sigma^{r_1} c_1$ and consider this curve in $A$.
    
    Define $B$ to be the filled-in square with corners labeled $\{ v_1, v_2, v_3, v_4 \}$, indexed clockwise, and edges labeled $\{ e_1, e_2, e_3, e_{4} \}$ such that $e_i$ has endpoints $v_i$ and $v_{i+1}$ for $i \mod 4$, and let $z$ be a point on the interior of $e_{4}$.
    
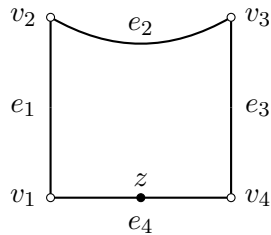
\begin{center}
        \begin{tikzpicture}[scale=1.2]
            \draw[thick] (8, 1) to[out=330, in=210] (10,1);
            \draw[thick] (8,-1) -- (10,-1);
            \draw[thick] (8,1) -- (8,0);
            \draw[thick] (8,0) -- (8,-1);
            \draw[thick] (10,1) -- (10,0);
            \draw[thick] (10,0) -- (10,-1);
            \node[draw, circle, inner sep=0pt, minimum size=3pt, fill=white] at (8, 1) {};
            \node[draw, circle, inner sep=0pt, minimum size=3pt, fill=white] at (10, 1) {};
            \node[draw, circle, inner sep=0pt, minimum size=3pt, fill=white] at (8, -1) {};
            \node[draw, circle, inner sep=0pt, minimum size=3pt, fill=white] at (10, -1) {};
            \node[draw, circle, inner sep=0pt, minimum size=3pt, fill=black] at (9, -1) {};
            \node at (7.7, -1) {$v_1$};
            \node at (7.7, 1) {$v_2$};
            \node at (10.3, 1) {$v_3$};
            \node at (10.3, -1) {$v_4$};
            \node at (7.7, 0) {$e_1$};
            \node at (9, 0.9) {$e_2$};
            \node at (10.3, 0) {$e_3$};
            \node at (9, -1.3) {$e_4$};
            \node at (9, -0.8) {$z$};
        \end{tikzpicture}
    \end{center}
    Take $g_{\alpha_1} : B \to A$ to be the quotient map such that
    \begin{itemize}
        \item $g_{\alpha_1} (e_1) = g_{\alpha_1}(e_3) = \alpha_1$,
        \item $g_{\alpha_1}(v_2) = g_{\alpha_1}(v_3) = x$,
        \item $g_{\alpha_1}(v_1) = g_{\alpha_1}(v_4) = x_1$,
        \item $g_{\alpha_1}(z) = x_2$,
        \item and canonically identifies $e_4$ to $\partial_A$ and $e_2$ to $\partial \left(T^2 \setminus D^2\right)$.
    \end{itemize}
    Using a pullback of the induced metric on $T^2 \setminus D^2$, if $r_2 = r_1 - 1$ define $\beta_2$ to be the geodesic path in $B$ from $z$ to $v_{2}$ and to be the geodesic path from $z$ to $v_{3}$ if $r_2 = r_1$. Finally, let $\alpha_2 := g_{\alpha_1}(\beta_2)$. Then the curve $\alpha := \delta \cup \alpha_1 \cup \alpha_2$ is a simple tangle with endpoints on $x$. We now finally define $\widetilde{f}_{A}(p,q,r) = \alpha$ and $f_{A}(p,q,r) = [\alpha]$.
    
    The rest of this proof is primarily devoted to showing $f_{A}$ is surjective.
    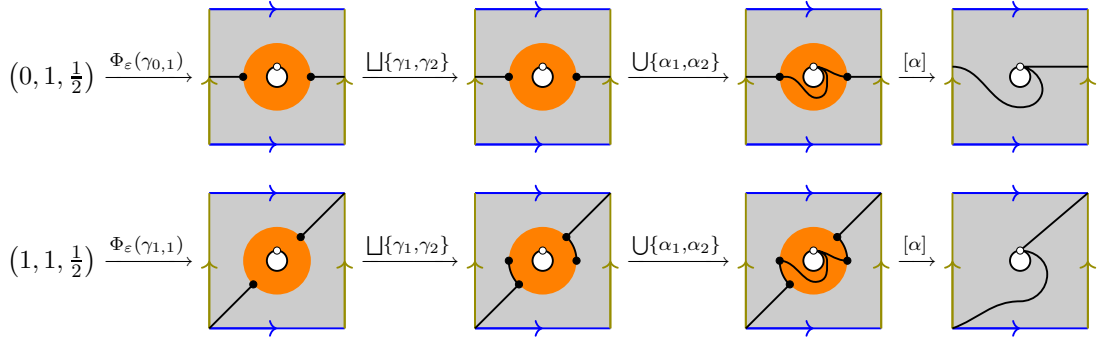
\begin{figure}[h]
        \centering
        $$\resizebox{0.9\width}{!}{$\left(0,1,\frac{1}{2}\right) \xrightarrow{\Phi_{\varepsilon}(\gamma_{0,1})}
        \begin{tikzpicture}[scale=1, baseline=-3]
            \draw[draw=none, fill=gray!40] (0,1) -- (2,1) -- (2,-1) -- (0,-1) -- (1, -0.15) to[out=0, in=-90] (1.15, 0) to[out=90, in=0] (1, 0.15) to[out=180, in=90] (0.85, 0) to[out=-90, in=180] (1, -0.15) -- (0, -1) -- (0,1);
            \draw[thick, blue, ->] (0,1) -- (1,1);
            \draw[thick, blue] (0,1) -- (2,1);
            \draw[thick, blue, ->] (0,-1) -- (1,-1);
            \draw[thick, blue] (0,-1) -- (2,-1);
            \draw[thick, olive, ->] (0,-1) -- (0,0);
            \draw[thick, olive] (0,-1) -- (0,1);
            \draw[thick, olive, ->] (2,-1) -- (2,0);
            \draw[thick, olive] (2,-1) -- (2,1);
            \draw[draw=none, fill=orange, opacity=0.7] (1,0) circle[radius=0.5cm];
            \draw[thick] (0, 0) -- (0.5, 0);
            \draw[thick] (1.5, 0) -- (2, 0);
            \draw[thick, black, fill=white] (1, 0) circle (0.15);
            \node[draw, circle, inner sep=0pt, minimum size=3pt, fill=white] at (1, 0.15) {};
            \node[draw, circle, inner sep=0pt, minimum size=3pt, fill=black] at (1.5, 0) {};
            \node[draw, circle, inner sep=0pt, minimum size=3pt, fill=black] at (0.5, 0) {};
        \end{tikzpicture} \xrightarrow{\bigsqcup \{ \gamma_1, \gamma_2 \}}
        \begin{tikzpicture}[scale=1, baseline=-3]
            \draw[draw=none, fill=gray!40] (0,1) -- (2,1) -- (2,-1) -- (0,-1) -- (1, -0.15) to[out=0, in=-90] (1.15, 0) to[out=90, in=0] (1, 0.15) to[out=180, in=90] (0.85, 0) to[out=-90, in=180] (1, -0.15) -- (0, -1) -- (0,1);
            \draw[thick, blue, ->] (0,1) -- (1,1);
            \draw[thick, blue] (0,1) -- (2,1);
            \draw[thick, blue, ->] (0,-1) -- (1,-1);
            \draw[thick, blue] (0,-1) -- (2,-1);
            \draw[thick, olive, ->] (0,-1) -- (0,0);
            \draw[thick, olive] (0,-1) -- (0,1);
            \draw[thick, olive, ->] (2,-1) -- (2,0);
            \draw[thick, olive] (2,-1) -- (2,1);
            \draw[draw=none, fill=orange, opacity=0.7] (1,0) circle[radius=0.5cm];
            \draw[thick] (0, 0) -- (0.5, 0);
            \draw[thick] (1.5, 0) -- (2, 0);
            \draw[thick, black, fill=white] (1, 0) circle (0.15);
            \node[draw, circle, inner sep=0pt, minimum size=3pt, fill=white] at (1, 0.15) {};
            \node[draw, circle, inner sep=0pt, minimum size=3pt, fill=black] at (1.5, 0) {};
            \node[draw, circle, inner sep=0pt, minimum size=3pt, fill=black] at (0.5, 0) {};
        \end{tikzpicture} \xrightarrow{\bigcup \{ \alpha_1, \alpha_2 \}}
        \begin{tikzpicture}[scale=1, baseline=-3]
            \draw[draw=none, fill=gray!40] (0,1) -- (2,1) -- (2,-1) -- (0,-1) -- (1, -0.15) to[out=0, in=-90] (1.15, 0) to[out=90, in=0] (1, 0.15) to[out=180, in=90] (0.85, 0) to[out=-90, in=180] (1, -0.15) -- (0, -1) -- (0,1);
            \draw[thick, blue, ->] (0,1) -- (1,1);
            \draw[thick, blue] (0,1) -- (2,1);
            \draw[thick, blue, ->] (0,-1) -- (1,-1);
            \draw[thick, blue] (0,-1) -- (2,-1);
            \draw[thick, olive, ->] (0,-1) -- (0,0);
            \draw[thick, olive] (0,-1) -- (0,1);
            \draw[thick, olive, ->] (2,-1) -- (2,0);
            \draw[thick, olive] (2,-1) -- (2,1);
            \draw[draw=none, fill=orange, opacity=0.7] (1,0) circle[radius=0.5cm];
            \draw[thick] (0, 0) -- (0.6, 0) to[out=0, in=200] (1.1, -0.3) to[out=20, in=270] (1.2, 0) to[out=90, in=0] (1, 0.15) to[out=0, in=180] (1.5, 0) -- (2, 0);
            \draw[thick, black, fill=white] (1, 0) circle (0.15);
            \node[draw, circle, inner sep=0pt, minimum size=3pt, fill=white] at (1, 0.15) {};
            \node[draw, circle, inner sep=0pt, minimum size=3pt, fill=black] at (1.5, 0) {};
            \node[draw, circle, inner sep=0pt, minimum size=3pt, fill=black] at (0.5, 0) {};
        \end{tikzpicture} \xrightarrow{[\alpha]}
        \begin{tikzpicture}[scale=1, baseline=-3]
            \draw[draw=none, fill=gray!40] (0,1) -- (2,1) -- (2,-1) -- (0,-1) -- (1, -0.15) to[out=0, in=-90] (1.15, 0) to[out=90, in=0] (1, 0.15) to[out=180, in=90] (0.85, 0) to[out=-90, in=180] (1, -0.15) -- (0, -1) -- (0,1);
            \draw[thick, blue, ->] (0,1) -- (1,1);
            \draw[thick, blue] (0,1) -- (2,1);
            \draw[thick, blue, ->] (0,-1) -- (1,-1);
            \draw[thick, blue] (0,-1) -- (2,-1);
            \draw[thick, olive, ->] (0,-1) -- (0,0);
            \draw[thick, olive] (0,-1) -- (0,1);
            \draw[thick, olive, ->] (2,-1) -- (2,0);
            \draw[thick, olive] (2,-1) -- (2,1);
            \draw[thick] (2, 0.15) -- (1, 0.15) to[out=0, in=60] (1.3, -0.3) to[out=240, in=0] (1,-0.45) to[out=180, in=0] (0, 0.15);
            \draw[thick, black, fill=white] (1, 0) circle (0.15);
            \node[draw, circle, inner sep=0pt, minimum size=3pt, fill=white] at (1, 0.15) {};
        \end{tikzpicture}$}$$
        $$\resizebox{0.9\width}{!}{$\left(1,1,\frac{1}{2}\right) \xrightarrow{\Phi_{\varepsilon}(\gamma_{1,1})}
        \begin{tikzpicture}[scale=1, baseline=-3]
            \draw[draw=none, fill=gray!40] (0,1) -- (2,1) -- (2,-1) -- (0,-1) -- (1, -0.15) to[out=0, in=-90] (1.15, 0) to[out=90, in=0] (1, 0.15) to[out=180, in=90] (0.85, 0) to[out=-90, in=180] (1, -0.15) -- (0, -1) -- (0,1);
            \draw[thick, blue, ->] (0,1) -- (1,1);
            \draw[thick, blue] (0,1) -- (2,1);
            \draw[thick, blue, ->] (0,-1) -- (1,-1);
            \draw[thick, blue] (0,-1) -- (2,-1);
            \draw[thick, olive, ->] (0,-1) -- (0,0);
            \draw[thick, olive] (0,-1) -- (0,1);
            \draw[thick, olive, ->] (2,-1) -- (2,0);
            \draw[thick, olive] (2,-1) -- (2,1);
            \draw[draw=none, fill=orange, opacity=0.7] (1,0) circle[radius=0.5cm];
            \draw[thick] (0, -1) -- (0.65, -0.35);
            \draw[thick] (1.35, 0.35) -- (2, 1);
            \draw[thick, black, fill=white] (1, 0) circle (0.15);
            \node[draw, circle, inner sep=0pt, minimum size=3pt, fill=white] at (1, 0.15) {};
            \node[draw, circle, inner sep=0pt, minimum size=3pt, fill=black] at (1.35, 0.35) {};
            \node[draw, circle, inner sep=0pt, minimum size=3pt, fill=black] at (0.65, -0.35) {};
        \end{tikzpicture} \xrightarrow{\bigsqcup \{ \gamma_1, \gamma_2 \}}
        \begin{tikzpicture}[scale=1, baseline=-3]
            \draw[draw=none, fill=gray!40] (0,1) -- (2,1) -- (2,-1) -- (0,-1) -- (1, -0.15) to[out=0, in=-90] (1.15, 0) to[out=90, in=0] (1, 0.15) to[out=180, in=90] (0.85, 0) to[out=-90, in=180] (1, -0.15) -- (0, -1) -- (0,1);
            \draw[thick, blue, ->] (0,1) -- (1,1);
            \draw[thick, blue] (0,1) -- (2,1);
            \draw[thick, blue, ->] (0,-1) -- (1,-1);
            \draw[thick, blue] (0,-1) -- (2,-1);
            \draw[thick, olive, ->] (0,-1) -- (0,0);
            \draw[thick, olive] (0,-1) -- (0,1);
            \draw[thick, olive, ->] (2,-1) -- (2,0);
            \draw[thick, olive] (2,-1) -- (2,1);
            \draw[draw=none, fill=orange, opacity=0.7] (1,0) circle[radius=0.5cm];
            \draw[thick] (0, -1) -- (0.65, -0.35);
            \draw[thick, domain=180:225] plot ({0.5*cos(\x)+1}, {0.5*sin(\x)});
            \draw[thick, domain=0:45] plot ({0.5*cos(\x)+1}, {0.5*sin(\x)});
            \draw[thick] (1.35, 0.35) -- (2, 1);
            \draw[thick, black, fill=white] (1, 0) circle (0.15);
            \node[draw, circle, inner sep=0pt, minimum size=3pt, fill=white] at (1, 0.15) {};
            \node[draw, circle, inner sep=0pt, minimum size=3pt, fill=black] at (1.35, 0.35) {};
            \node[draw, circle, inner sep=0pt, minimum size=3pt, fill=black] at (0.65, -0.35) {};
            \node[draw, circle, inner sep=0pt, minimum size=3pt, fill=black] at (1.5, 0) {};
            \node[draw, circle, inner sep=0pt, minimum size=3pt, fill=black] at (0.5, 0) {};
        \end{tikzpicture} \xrightarrow{\bigcup \{ \alpha_1, \alpha_2 \}}
        \begin{tikzpicture}[scale=1, baseline=-3]
            \draw[draw=none, fill=gray!40] (0,1) -- (2,1) -- (2,-1) -- (0,-1) -- (1, -0.15) to[out=0, in=-90] (1.15, 0) to[out=90, in=0] (1, 0.15) to[out=180, in=90] (0.85, 0) to[out=-90, in=180] (1, -0.15) -- (0, -1) -- (0,1);
            \draw[thick, blue, ->] (0,1) -- (1,1);
            \draw[thick, blue] (0,1) -- (2,1);
            \draw[thick, blue, ->] (0,-1) -- (1,-1);
            \draw[thick, blue] (0,-1) -- (2,-1);
            \draw[thick, olive, ->] (0,-1) -- (0,0);
            \draw[thick, olive] (0,-1) -- (0,1);
            \draw[thick, olive, ->] (2,-1) -- (2,0);
            \draw[thick, olive] (2,-1) -- (2,1);
            \draw[draw=none, fill=orange, opacity=0.7] (1,0) circle[radius=0.5cm];
            \draw[thick] (0, -1) -- (0.65, -0.35);
            \draw[thick, domain=180:225] plot ({0.5*cos(\x)+1}, {0.5*sin(\x)});
            \draw[thick, domain=0:45] plot ({0.5*cos(\x)+1}, {0.5*sin(\x)});
            \draw[thick] (1.35, 0.35) -- (2, 1);
            \draw[thick] (0.5, 0) to[out=0, in=200] (1.1, -0.3) to[out=20, in=270] (1.2, 0) to[out=90, in=0] (1, 0.15) to[out=0, in=180] (1.5, 0);
            \draw[thick, black, fill=white] (1, 0) circle (0.15);
            \node[draw, circle, inner sep=0pt, minimum size=3pt, fill=white] at (1, 0.15) {};
            \node[draw, circle, inner sep=0pt, minimum size=3pt, fill=black] at (1.35, 0.35) {};
            \node[draw, circle, inner sep=0pt, minimum size=3pt, fill=black] at (0.65, -0.35) {};
            \node[draw, circle, inner sep=0pt, minimum size=3pt, fill=black] at (1.5, 0) {};
            \node[draw, circle, inner sep=0pt, minimum size=3pt, fill=black] at (0.5, 0) {};
        \end{tikzpicture} \xrightarrow{[\alpha]}
        \begin{tikzpicture}[scale=1, baseline=-3]
            \draw[draw=none, fill=gray!40] (0,1) -- (2,1) -- (2,-1) -- (0,-1) -- (1, -0.15) to[out=0, in=-90] (1.15, 0) to[out=90, in=0] (1, 0.15) to[out=180, in=90] (0.85, 0) to[out=-90, in=180] (1, -0.15) -- (0, -1) -- (0,1);
            \draw[thick, blue, ->] (0,1) -- (1,1);
            \draw[thick, blue] (0,1) -- (2,1);
            \draw[thick, blue, ->] (0,-1) -- (1,-1);
            \draw[thick, blue] (0,-1) -- (2,-1);
            \draw[thick, olive, ->] (0,-1) -- (0,0);
            \draw[thick, olive] (0,-1) -- (0,1);
            \draw[thick, olive, ->] (2,-1) -- (2,0);
            \draw[thick, olive] (2,-1) -- (2,1);
            \draw[thick] (0, -1) to[out=20, in=180] (1, -0.6) to[out=0, in=270] (1.4, -0.2) to[out=90, in=0] (1, 0.15) -- (2, 1);
            \draw[thick, black, fill=white] (1, 0) circle (0.15);
            \node[draw, circle, inner sep=0pt, minimum size=3pt, fill=white] at (1, 0.15) {};
        \end{tikzpicture}$}$$
        \caption[Examples of $f_{A}(p,q,r)$]{Example of $f_{A}\left(0,1,\frac{1}{2}\right)$ and $f_{A}\left(1,1,\frac{1}{2}\right)$}
        \label{fig:fMapExample}
    \end{figure}
    
    Define $A$, $\partial_A$, $x_0$, $x_1$, $x_2$, and $c_1$ as before. Let $\eta$ be a smooth radial vector field on a neighborhood of $A$, denoted $N(A)$, such that
    \begin{itemize}
        \item For all $y \in N(A) \setminus A$, $\eta(y)$ is tangent to the geodesic within $N(A)$ from $y$ to $\partial_A$
        \item $\eta$ is zero outside of $N(A)$,
        \item and for all $y \in \partial_A$, the vector $\eta(y)$ is normal to $\partial_A$ and points inwards.
    \end{itemize}
    
    Define
    $$\Psi: \left(T^2 \setminus D^2 \right) \setminus A \xlongrightarrow{\sim} T^2 \setminus \{\mathfrak{p}\} \hookrightarrow T^2$$ to be the homeomorphism that flows along $\eta$, and therefore restricts to the identity on $T^2 \setminus N(A)$, followed by its canonical embedding into $T^2$.
    
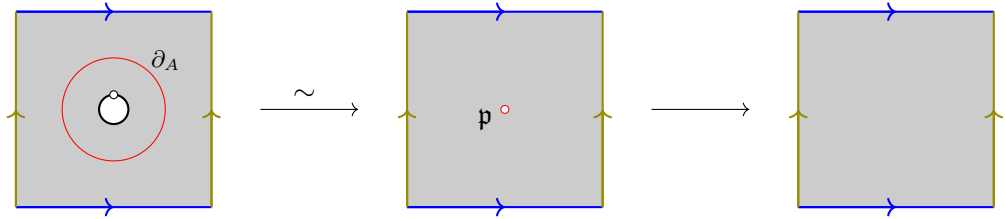
\begin{center}
        \begin{tikzpicture}[scale=1.3]
            \draw[draw=none, fill=gray!40] (-1,1) -- (1,1) -- (1,-1) -- (-1,-1);
            \draw[thick, blue, ->] (-1,1) -- (0,1);
            \draw[thick, blue] (-1,1) -- (1,1);
            \draw[thick, blue, ->] (-1,-1) -- (0,-1);
            \draw[thick, blue] (-1,-1) -- (1,-1);
            \draw[thick, olive, ->] (-1,-1) -- (-1,0);
            \draw[thick, olive] (-1,-1) -- (-1,1);
            \draw[thick, olive, ->] (1,-1) -- (1,0);
            \draw[thick, olive] (1,-1) -- (1,1);
            \draw[draw=red] (0, 0) circle[radius=15pt];
            \node at (0.53, 0.5) {\footnotesize{$\partial_A$}};
            \draw[thick, black, fill=white] (0, 0) circle (0.15);
            \node[draw, circle, inner sep=0pt, minimum size=3pt, fill=white] at (0, 0.15) {};
            \draw[->] (1.5, 0) -- (2.5, 0);
            \node at (1.95, 0.15) {$\sim$};
            \draw[draw=none, fill=gray!40] (3,1) -- (5,1) -- (5,-1) -- (3,-1) -- (3,1);
            \draw[thick, blue, ->] (3,1) -- (4,1);
            \draw[thick, blue] (3,1) -- (5,1);
            \draw[thick, blue, ->] (3,-1) -- (4,-1);
            \draw[thick, blue] (3,-1) -- (5,-1);
            \draw[thick, olive, ->] (3,-1) -- (3,0);
            \draw[thick, olive] (3,-1) -- (3,1);
            \draw[thick, olive, ->] (5,-1) -- (5,0);
            \draw[thick, olive] (5,-1) -- (5,1);
            \node[draw, circle, red, inner sep=0pt, minimum size=3pt, fill=white] at (4, 0) {};
            \node at (3.8, -0.1) {$\mathfrak{p}$};
            \draw[->] (5.5, 0) -- (6.5, 0);
            \draw[draw=none, fill=gray!40] (7,1) -- (9,1) -- (9,-1) -- (7,-1);
            \draw[thick, blue, ->] (7,1) -- (8,1);
            \draw[thick, blue] (7,1) -- (9,1);
            \draw[thick, blue, ->] (7,-1) -- (8,-1);
            \draw[thick, blue] (7,-1) -- (9,-1);
            \draw[thick, olive, ->] (7,-1) -- (7,0);
            \draw[thick, olive] (7,-1) -- (7,1);
            \draw[thick, olive, ->] (9,-1) -- (9,0);
            \draw[thick, olive] (9,-1) -- (9,1);
        \end{tikzpicture}
    \end{center}
    
    Consider any curve in $\Omega_x$ and select a simple representative, $\alpha$, such that $\left|\alpha \cap \partial_A \right| = 2$ and $\Psi\left(\alpha \setminus ( A \cap \alpha )\right)$ has constant slope. Since we required $\alpha$ to be simple, if $\alpha$ lies entirely within $A$, it must either be parallel to the boundary of $T^2 \setminus D^2$ or null-homotopic, and therefore not in $\Omega_x$.
    As all simple unoriented closed curves on the torus are classified as $(p,q)$-curves, where $(p,q) \sim (-p,-q)$ and $p$ and $q$ are coprime, we can uniquely associate a $(p,q)$ pair to the closed curve $\Psi\left(\alpha \setminus ( A \cap \alpha )\right) \cup \{\mathfrak{p}\}$.
    \begin{figure}[h]
    \begin{center}
        \begin{tikzpicture}[scale=1.5]
            \draw[draw=none, fill=gray!40] (3,1) -- (5,1) -- (5,-1) -- (3,-1) -- (4, -0.15) to[out=0, in=-90] (4.15, 0) to[out=90, in=0] (4, 0.15) to[out=180, in=90] (3.85, 0) to[out=-90, in=180] (4, -0.15) -- (3, -1) -- (3,1);
            \draw[thick, blue, ->] (3,1) -- (4,1);
            \draw[thick, blue] (3,1) -- (5,1);
            \draw[thick, blue, ->] (3,-1) -- (4,-1);
            \draw[thick, blue] (3,-1) -- (5,-1);
            \draw[thick, olive, ->] (3,-1) -- (3,0);
            \draw[thick, olive] (3,-1) -- (3,1);
            \draw[thick, olive, ->] (5,-1) -- (5,0);
            \draw[thick, olive] (5,-1) -- (5,1);
            \draw[thick] (3, -1) -- (3.4, -0.73) to[out=30, in=180] (4,-0.45) to[out=0, in=240] (4.3, -0.3) to[out=60, in=0] (4.1, 0.15);
            \draw[thick] (4, 0.15) to[out=90, in=180] (4.54, 0.03) -- (5, 0.33);
            \draw[thick] (3, 0.33) -- (4, 1);
            \draw[thick] (4, -1) -- (5, -0.33);
            \draw[thick] (3, -0.33) -- (5, 1);
            \draw[thick, black] (4, 0) circle (0.15);
            \node[draw, circle, inner sep=0pt, minimum size=3pt, fill=green] at (4, 0.15) {};
            \draw[draw=none, fill=white] (4.1,-0.26) circle[radius=15pt];
            \draw[->] (5.5, 0) -- (6.5, 0);
            \draw[draw=none, fill=gray!40] (7,1) -- (9,1) -- (9,-1) -- (7,-1);
            \draw[thick, blue, ->] (7,1) -- (8,1);
            \draw[thick, blue] (7,1) -- (9,1);
            \draw[thick, blue, ->] (7,-1) -- (8,-1);
            \draw[thick, blue] (7,-1) -- (9,-1);
            \draw[thick, olive, ->] (7,-1) -- (7,0);
            \draw[thick, olive] (7,-1) -- (7,1);
            \draw[thick, olive, ->] (9,-1) -- (9,0);
            \draw[thick, olive] (9,-1) -- (9,1);
            \draw[thick] (7, -1) -- (9, 0.33);
            \draw[thick] (7, 0.33) -- (8, 1);
            \draw[thick] (8, -1) -- (9, -0.33);
            \draw[thick] (7, -0.33) -- (9, 1);
            \node[draw, circle, inner sep=0pt, minimum size=3pt, fill=black] at (8.2, -0.2) {};
            \node at (8.25, -0.5) {$\mathfrak{p}$};
        \end{tikzpicture}\caption[Turning a tangle into a closed curve on the torus]{Turning the tangle, $\alpha$, in the complement of $A$ into a closed curve on the torus via $\Psi$. This example becomes a curve in $T^2$ with classification $(2,3)$}\label{fig:PsiClosure}
    \end{center}
    \end{figure}
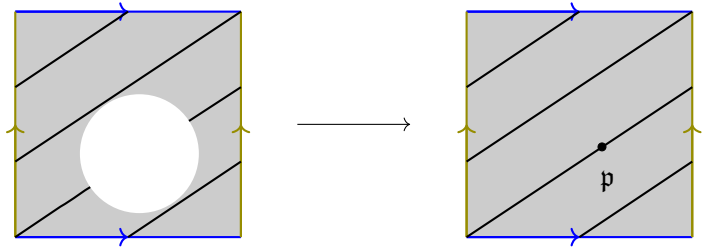

    Identify $\alpha$ with a map, $\alpha : [0,1] \to T^2 \setminus D^2$, such that $\alpha(0) = \alpha(1) = x$. Label the points $\{y_1, y_2\} = \alpha \cap \partial_A$ and their preimages as $t_i^\prime := \alpha^{-1}(y_i)$, assuming $t_1^\prime < t_2^\prime$. If $(p,q)$ is a positive slope, then let $\gamma_i$ be the geodesics on $\partial_A$ from $y_i$ to $x_i$, rotating clockwise. If these geodesics intersect each other then we can swap the labels $y_i$ by changing the orientation of $\alpha$ to avoid this. If $(p,q)$ is a negative slope then let $\gamma_i$ be the geodesics on $\partial_A$ from $y_i$ to $x_i$, rotating counterclockwise, once again changing the orientation of $\alpha$ if needed.
    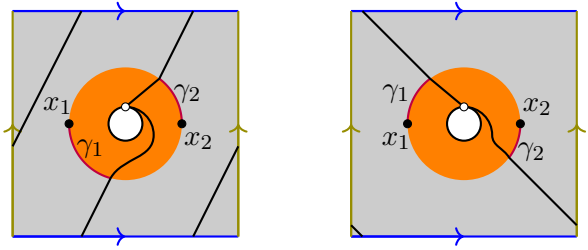
\begin{figure}
    \begin{center}
        \begin{tikzpicture}[scale=1.5]
            \draw[draw=none, fill=gray!40] (3,1) -- (5,1) -- (5,-1) -- (3,-1) -- (4, -0.15) to[out=0, in=-90] (4.15, 0) to[out=90, in=0] (4, 0.15) to[out=180, in=90] (3.85, 0) to[out=-90, in=180] (4, -0.15) -- (3, -1) -- (3,1);
            \draw[thick, blue, ->] (3,1) -- (4,1);
            \draw[thick, blue] (3,1) -- (5,1);
            \draw[thick, blue, ->] (3,-1) -- (4,-1);
            \draw[thick, blue] (3,-1) -- (5,-1);
            \draw[thick, olive, ->] (3,-1) -- (3,0);
            \draw[thick, olive] (3,-1) -- (3,1);
            \draw[thick, olive, ->] (5,-1) -- (5,0);
            \draw[thick, olive] (5,-1) -- (5,1);
            \draw[draw=none, fill=orange, opacity=0.7] (4,0) circle[radius=0.5cm];
            \draw[thick, purple, domain=180:255] plot ({0.5*cos(\x)+4}, {0.5*sin(\x)});
            \draw[thick, purple, domain=0:53] plot ({0.5*cos(\x)+4}, {0.5*sin(\x)});
            \draw[thick] (3.61, -1) -- (3.87, -0.48) to[out=60, in=270] (4.25, -0.1) to[out=90, in=0] (4, 0.15) -- (4.3, 0.4) -- (4.6, 1);
            \draw[thick] (4.6, -1) -- (5, -0.2);
            \draw[thick] (3, -0.2) -- (3.61, 1);
            \draw[thick, black, fill=white] (4, 0) circle (0.15);
            \node[draw, circle, inner sep=0pt, minimum size=3pt, fill=white] at (4, 0.15) {};
            \node[draw, circle, inner sep=0pt, minimum size=3pt, fill=black] at (4.5, 0) {};
            \node[draw, circle, inner sep=0pt, minimum size=3pt, fill=black] at (3.5, 0) {};
            \node at (3.4, 0.16) {$x_1$};
            \node at (4.65, -0.15) {$x_2$};
            \node at (3.69, -0.22) {$\gamma_1$};
            \node at (4.55, 0.27) {$\gamma_2$};
            \draw[draw=none, fill=gray!40] (6,1) -- (8,1) -- (8,-1) -- (6,-1) -- (7, -0.15) to[out=0, in=-90] (7.15, 0) to[out=90, in=0] (7, 0.15) to[out=180, in=90] (6.85, 0) to[out=-90, in=180] (7, -0.15) -- (6, -1) -- (6, 1);
            \draw[thick, blue, ->] (6,1) -- (7,1);
            \draw[thick, blue] (6,1) -- (8,1);
            \draw[thick, blue, ->] (6,-1) -- (7,-1);
            \draw[thick, blue] (6,-1) -- (8,-1);
            \draw[thick, olive, ->] (6,-1) -- (6,0);
            \draw[thick, olive] (6,-1) -- (6,1);
            \draw[thick, olive, ->] (8,-1) -- (8,0);
            \draw[thick, olive] (8,-1) -- (8,1);
            \draw[draw=none, fill=orange, opacity=0.7] (7,0) circle[radius=0.5cm];
            \draw[thick, purple, domain=127:180] plot ({0.5*cos(\x)+7}, {0.5*sin(\x)});
            \draw[thick, purple, domain=-37:0] plot ({0.5*cos(\x)+7}, {0.5*sin(\x)});
            \draw[thick] (6.1, 1) -- (6.7, 0.4) -- (7, 0.15) to[out=0, in=90] (7.25, -0.1) to[out=270, in=120] (7.4, -0.3) -- (8, -0.9);
            \draw[thick] (6, -0.9) -- (6.1, -1);
            \draw[thick, black, fill=white] (7, 0) circle (0.15);
            \node[draw, circle, inner sep=0pt, minimum size=3pt, fill=white] at (7, 0.15) {};
            \node[draw, circle, inner sep=0pt, minimum size=3pt, fill=black] at (7.5, 0) {};
            \node[draw, circle, inner sep=0pt, minimum size=3pt, fill=black] at (6.5, 0) {};
            \node at (6.4, -0.15) {$x_1$};
            \node at (7.65, 0.15) {$x_2$};
            \node at (6.4, 0.27) {$\gamma_1$};
            \node at (7.6, -0.25) {$\gamma_2$};
        \end{tikzpicture}
        \caption[Constructing geodesic paths, $\gamma_i$]{Left: Example of a curve with a positive slope. Right: Example with a negative slope.}
    \end{center}
    \end{figure}
    If $(p,q) = (1,0)$, the same slope as the longitude, then let $\gamma_i$ be the geodesics on $\partial_A$ from $y_i$ to $x_i$, rotating clockwise. Finally, if $(p,q) = (0,1)$ then $y_i = x_i$ and $\gamma_i$ is trivial for each $i$.
    
    Isotope $\alpha$ inside $A$ so that for some $0 < t_1 < t_1^\prime$ and $1 > t_2 > t_2^\prime$, $\alpha$ satisfies:
    \begin{itemize}
        \item $\alpha(t_i) = x_i$,
        \item $\alpha([t_1, t_1^\prime]) = \gamma_1$,
        \item $\alpha([t_2^\prime, t_2]) = \gamma_2$,
        \item $\alpha([t_1^\prime, t_2^\prime]\bigsqcup \{0,1\})$ remains unchanged,
        \item $\alpha$ is transverse with itself at $\alpha(0)$ and $\alpha(1)$.
    \end{itemize}
    
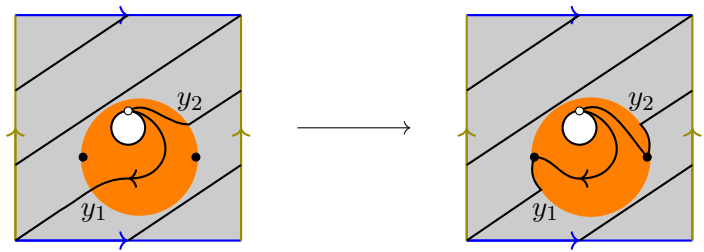
\begin{center}
        \begin{tikzpicture}[scale=1.5]
            \draw[draw=none, fill=gray!40] (4,1) -- (6,1) -- (6,-1) -- (4,-1) -- (5, -0.15) to[out=0, in=-90] (5.15, 0) to[out=90, in=0] (5, 0.15) to[out=180, in=90] (4.85, 0) to[out=-90, in=180] (5, -0.15) -- (4, -1) -- (4,1);
            \draw[thick, blue, ->] (4,1) -- (5,1);
            \draw[thick, blue] (4,1) -- (6,1);
            \draw[thick, blue, ->] (4,-1) -- (5,-1);
            \draw[thick, blue] (4,-1) -- (6,-1);
            \draw[thick, olive, ->] (4,-1) -- (4,0);
            \draw[thick, olive] (4,-1) -- (4,1);
            \draw[thick, olive, ->] (6,-1) -- (6,0);
            \draw[thick, olive] (6,-1) -- (6,1);
            \draw[draw=none, fill=orange, opacity=0.7] (5.1,-0.26) circle[radius=0.52cm];
            \draw[thick] (4, -1) -- (4.2, -0.865) -- (4.4, -0.73) to[out=30, in=180] (5,-0.45) to[out=0, in=240] (5.3, -0.3) to[out=60, in=0] (5, 0.15);
            \draw[thick] (5, 0.15) to[out=30, in=180] (5.54, 0.03) -- (6, 0.33);
            \draw[thick] (4, 0.33) -- (5, 1);
            \draw[thick] (5, -1) -- (6, -0.33);
            \draw[thick] (4, -0.33) -- (6, 1);
            \path[thick, tips, ->] (5.45,0) arc (0:-90:0.45);
            \draw[thick, black, fill=white] (5, 0) circle (0.15);
            \node[draw, circle, inner sep=0pt, minimum size=3pt, fill=white] at (5, 0.15) {};
            \node[draw, circle, inner sep=0pt, minimum size=3pt, fill=black] at (4.6, -0.26) {};
            \node[draw, circle, inner sep=0pt, minimum size=3pt, fill=black] at (5.6, -0.26) {};
            \node at (4.7, -0.75) {$y_1$};
            \node at (5.55, 0.23) {$y_2$};
            \draw[->] (6.5, 0) -- (7.5, 0);
            \draw[draw=none, fill=gray!40] (8,1) -- (10,1) -- (10,-1) -- (8,-1) -- (9, -0.15) to[out=0, in=-90] (9.15, 0) to[out=90, in=0] (9, 0.15) to[out=180, in=90] (8.85, 0) to[out=-90, in=180] (9, -0.15) -- (8, -1) -- (8,1);
            \draw[thick, blue, ->] (8,1) -- (9,1);
            \draw[thick, blue] (8,1) -- (10,1);
            \draw[thick, blue, ->] (8,-1) -- (9,-1);
            \draw[thick, blue] (8,-1) -- (10,-1);
            \draw[thick, olive, ->] (8,-1) -- (8,0);
            \draw[thick, olive] (8,-1) -- (8,1);
            \draw[thick, olive, ->] (10,-1) -- (10,0);
            \draw[thick, olive] (10,-1) -- (10,1);
            \draw[draw=none, fill=orange, opacity=0.7] (9.1,-0.26) circle[radius=0.53cm];
            \draw[thick] (8, -1) -- (8.66, -0.55) to[out=135, in=250] (8.6, -0.26) to[out=0, in=180] (9,-0.45) to[out=0, in=240] (9.3, -0.3) to[out=60, in=0] (9, 0.15);
            \draw[thick] (9, 0.15) to[out=30, in=130] (9.6, -0.26) to[out=70, in=-45] (9.54, 0.03) -- (10, 0.33);
            \draw[thick] (8, 0.33) -- (9, 1);
            \draw[thick] (9, -1) -- (10, -0.33);
            \draw[thick] (8, -0.33) -- (10, 1);
            \path[thick, tips, ->] (9.45,0) arc (0:-90:0.45);
            \draw[thick, black, fill=white] (9, 0) circle (0.15);
            \node[draw, circle, inner sep=0pt, minimum size=3pt, fill=white] at (9, 0.15) {};
            \node[draw, circle, inner sep=0pt, minimum size=3pt, fill=black] at (8.6, -0.26) {};
            \node[draw, circle, inner sep=0pt, minimum size=3pt, fill=black] at (9.6, -0.26) {};
            \node at (8.7, -0.75) {$y_1$};
            \node at (9.55, 0.23) {$y_2$};
        \end{tikzpicture}
    \end{center}
    We can now focus on the region within the annulus, $A$, and just $\alpha([0,t_1] \sqcup [t_2, 1])$.
    
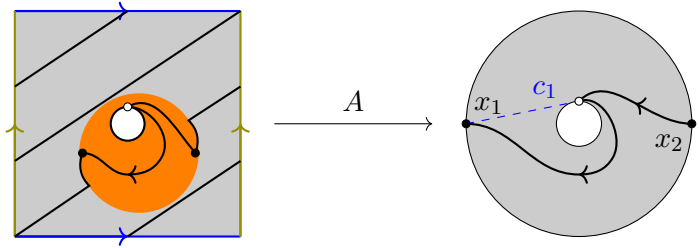
\begin{center}
        \begin{tikzpicture}[scale=1.5]
            \draw[draw=none, fill=gray!40] (0,1) -- (2,1) -- (2,-1) -- (0,-1) -- (1, -0.15) to[out=0, in=250] (1.15, 0) to[out=90, in=0] (1, 0.15) to[out=180, in=90] (0.85, 0) to[out=-90, in=180] (1, -0.15) -- (0, -1) -- (0,1);
            \draw[thick, blue, ->] (0,1) -- (1,1);
            \draw[thick, blue] (0,1) -- (2,1);
            \draw[thick, blue, ->] (0,-1) -- (1,-1);
            \draw[thick, blue] (0,-1) -- (2,-1);
            \draw[thick, olive, ->] (0,-1) -- (0,0);
            \draw[thick, olive] (0,-1) -- (0,1);
            \draw[thick, olive, ->] (2,-1) -- (2,0);
            \draw[thick, olive] (2,-1) -- (2,1);
            \draw[draw=none, fill=orange, opacity=0.7] (1.1,-0.26) circle[radius=0.53cm];
            \draw[thick] (0, -1) -- (0.66, -0.55) to[out=135, in=250] (0.6, -0.26) to[out=0, in=180] (1,-0.45) to[out=0, in=240] (1.3, -0.3) to[out=60, in=0] (1, 0.15);
            \draw[thick] (1, 0.15) to[out=30, in=130] (1.6, -0.26) to[out=70, in=-45] (1.54, 0.03) -- (2, 0.33);
            \draw[thick] (0, 0.33) -- (1, 1);
            \draw[thick] (1, -1) -- (2, -0.33);
            \draw[thick] (0, -0.33) -- (2, 1);
            \path[thick, tips, ->] (1.45,0) arc (0:-90:0.45);
            \draw[thick, black, fill=white] (1, 0) circle (0.15);
            \node[draw, circle, inner sep=0pt, minimum size=3pt, fill=white] at (1, 0.15) {};
            \node[draw, circle, inner sep=0pt, minimum size=3pt, fill=black] at (0.6, -0.26) {};
            \node[draw, circle, inner sep=0pt, minimum size=3pt, fill=black] at (1.6, -0.26) {};
            \draw[fill=gray!40] (5, 0) circle (1);
            \draw[fill=white] (5, 0) circle (0.2);
            \draw[thick] (4, 0) to[out=0, in=180] (5, -0.45) to[out=0, in=240] (5.3, -0.3) to[out=60, in=20] (5, 0.2);
            \draw[thick] (5, 0.2) to[out=45, in=180] (6, 0);
            \path[thick, tips, ->] (5.45,0) arc (0:-90:0.45);
            \path[thick, tips, ->] (6, 0) -- (5.5, 0.18);
            \draw[dashed, blue] (4, 0) -- (5, 0.2);
            \node[draw, circle, inner sep=0pt, minimum size=3pt, fill=black] at (4, 0) {};
            \node[draw, circle, inner sep=0pt, minimum size=3pt, fill=black] at (6, 0) {};
            \node[draw, circle, inner sep=0pt, minimum size=3pt, fill=white] at (5, 0.2) {};
            \node at (4.2, 0.16) {$x_1$};
            \node at (5.8, -0.17) {$x_2$};
            \node[blue] at (4.7, 0.3) {$c_1$};
            \draw[->] (2.3, 0) -- (3.7, 0);
            \node at (3, 0.2) {$A$};
        \end{tikzpicture}
    \end{center}

    Consider the corresponding restriction, $\hat{\alpha}: [0,t_1] \cup [t_2, 1] \to A$, of $\alpha$ and further restrict this into its two components, $\hat{\alpha}_1: [0,t_1] \to A$, and $\hat{\alpha}_2: [t_2,1] \to A$.
    Since the mapping class group of the annulus is isomorphic to $\mathbb{Z}$, $\hat{\alpha}_1$ must be isotopic to $\sigma^{r_{1,\alpha}} c_1$ for some $r_{1,\alpha} \in \mathbb{Z}$.
    
    If we cut $A$ along $\hat{\alpha}_1$, the resulting picture is a square.
    Since $\hat{\alpha}_2$ is a path from $\hat{\alpha}(t_2)$ to $x$ and since $\hat{\alpha}(t_2)$ lies on $\partial_A$ and is distinct from $\hat{\alpha}(t_1)$, there are exactly two possible ways $\hat{\alpha}$ can be simple up to isotopy.
    
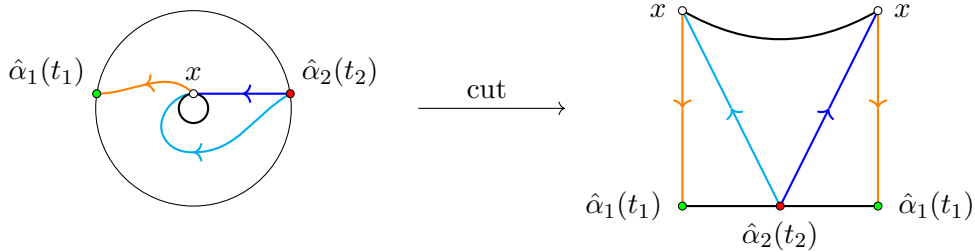
\begin{center}
        \begin{tikzpicture}[scale=1.3]
            \draw (3,0) circle (1);
            \draw[thick, orange] (2.01, 0.15) to[out=0, in=135] (3, 0.15);
            \draw[thick, cyan] (3, 0.15) to[out=180, in=120] (2.7, -0.3) to[out=300, in=180] (3,-0.45) to[out=0, in=210] (3.99, 0.15);
            \draw[thick, blue] (3, 0.15) -- (3.99, 0.15);
            \draw[thick, black] (3, 0) circle (0.15);
            \node[draw, circle, inner sep=0pt, minimum size=3pt, fill=white] at (3, 0.15) {};
            \node[draw, circle, inner sep=0pt, minimum size=3pt, fill=green] at (2.01, 0.15) {};
            \node[draw, circle, inner sep=0pt, minimum size=3pt, fill=red] at (3.99, 0.15) {};
            \path[thick, tips, orange, ->] (3, 0.35) -- (2.5, 0.24);
            \path[thick, tips, blue, ->] (4, 0.15) -- (3.5, 0.15);
            \path[thick, tips, cyan, ->] (3.45,0) arc (0:-90:0.45);
            \node at (1.5, 0.37) {$\hat{\alpha}_1(t_1)$};
            \node at (3, 0.35) {$x$};
            \node at (4.5, 0.37) {$\hat{\alpha}_2(t_2)$};
            \draw[->] (5.3, 0) -- (6.8, 0);
            \node at (6, 0.2) {cut};
            \draw[thick] (8, 1) to[out=330, in=210] (10,1);
            \draw[thick] (8,-1) -- (10,-1);
            \draw[thick, orange, ->] (8,1) -- (8,0);
            \draw[thick, orange] (8,0) -- (8,-1);
            \draw[thick, orange, ->] (10,1) -- (10,0);
            \draw[thick, orange] (10,0) -- (10,-1);
            \draw[thick, cyan, ->] (9, -1) -- (8.5, 0);
            \draw[thick, cyan] (8.5, 0) -- (8, 1);
            \draw[thick, blue, ->] (9, -1) -- (9.5, 0);
            \draw[thick, blue] (9.5, 0) -- (10, 1);
            \node[draw, circle, inner sep=0pt, minimum size=3pt, fill=white] at (8, 1) {};
            \node[draw, circle, inner sep=0pt, minimum size=3pt, fill=white] at (10, 1) {};
            \node[draw, circle, inner sep=0pt, minimum size=3pt, fill=green] at (8, -1) {};
            \node[draw, circle, inner sep=0pt, minimum size=3pt, fill=green] at (10, -1) {};
            \node[draw, circle, inner sep=0pt, minimum size=3pt, fill=red] at (9, -1) {};
            \node at (7.4, -1) {$\hat{\alpha}_1(t_1)$};
            \node at (10.6, -1) {$\hat{\alpha}_1(t_1)$};
            \node at (9, -1.3) {$\hat{\alpha}_2(t_2)$};
            \node at (7.75, 1) {$x$};
            \node at (10.25, 1) {$x$};
        \end{tikzpicture}
    \end{center}
    
    Let $v_1 = \hat{\alpha}_{1}'(0)$ and $v_2 = \hat{\alpha}_{2}'(1)$, the corresponding tangent vectors, using the induced orientation from $\alpha$'s domain.\footnote{We may assume $\alpha$ is parameterized by arclength so that these vectors are never zero.}
    Since $\alpha$ is transverse with itself at $\alpha(0)$ and $\alpha(1)$, we must have that $\operatorname{det}(v_1 v_2) \neq 0$.
    Finally, define 
    \begin{align}\label{eq:r2}
        r_{2,\alpha} = \begin{cases}
            r_{1,\alpha} & \text{if } \operatorname{det}(v_1 v_2) > 0\\
            r_{1,\alpha} - 1 & \text{if } \operatorname{det}(v_1 v_2) < 0\\
        \end{cases}
    \end{align}
    and let $r = \frac{r_{1,\alpha} + r_{2,\alpha}}{2}$ which is clearly an element of $\frac{1}{2}\mathbb{Z}$, providing us with the last entry in the tuple.\footnote{If $x_1$ were located somewhere else on $\partial_A$, it might seem more natural to define $r_{2, \alpha}$ as $r_{1,\alpha} + 1$ when the determinant is positive and as $r_{1,\alpha}$ when it's negative. However, using equation (\ref{eq:r2}) instead of this for the definition of $r_{2, \alpha}$, merely corresponds to a $\frac{1}{2}$ shift in this $\frac{1}{2}\mathbb{Z}$-torsor and does not violate any assumptions. Furthermore, if we had chosen $\sigma$ to be the counterclockwise Dehn twist instead, we would, among other things, want to define $r_{2, \alpha}$ as $r_{1,\alpha}$ when the determinant is positive and as $r_{1,\alpha}+1$ when it's negative.}
    Notice that this representative of $\alpha$ is clearly equivalent to $\widetilde{f}_{A}(p,q,r)$ outside of $A$ and isotopic inside of $A$ as the decomposition of $r$ into $r_{1,\alpha}$ and $r_{2,\alpha}$ is unique.
    Therefore, $\left(\pi_{\text{iso}} \circ \widetilde{f}_{A}\right)(p,q,r) = [\alpha]$ and so this map is surjective.

    Finally, suppose $f_{A}(p,q,r) = f_{A}(p^\prime,q^\prime,r^\prime) = [\alpha]$. First notice that $\alpha$'s geometric intersection number with a fixed meridian and fixed longitude away from $x$ are invariant under isotopy. As these intersection numbers correspond to $p$ and $q$ respectively, $p = p^\prime$ and $q = q^\prime$.
    Notice that $\widetilde{f}_{A}(p,q,r) \cap\partial_A$ must completely agree with $\widetilde{f}_{A}(p^\prime,q^\prime,r^\prime) \cap\partial_A$ as well as they are equivalent on $T^2 \setminus A$ and therefore $[\widetilde{f}_{A}(p,q,r)] = [\widetilde{f}_{A}(p^\prime,q^\prime,r^\prime)]$ on $A$.
    Because they agree up to isotopy and since $r = \frac{r_{1,\alpha} + r_{2,\alpha}}{2}$ and $r^\prime = \frac{r_{1,\alpha}^\prime + r_{2,\alpha}^\prime}{2}$ have unique decompositions, $r = r^\prime$.
    Thus $(p,q,r) = (p^\prime, q^\prime, r^\prime)$ and so $f_{A}$ is bijective.
\end{proof}

\begin{lemma}\label{lemma:IntersectionDehnTwist}
    If $\alpha \in \mathscr{S}(\mathfrak{S})$ is a tangle and $\gamma \in \mathscr{S}(\mathfrak{S})$ is a closed curve such that the geometric intersection number of $\alpha$ and $\gamma$ is $1$, then $\frac{1}{q^2 - q^{-2}}[\alpha, \gamma]_q$ resolves to a Dehn twist of $\alpha$ along $\gamma$ and $\frac{-1}{q^2 - q^{-2}}[\alpha, \gamma]_{q^{-1}} = \frac{1}{q^2 - q^{-2}}[\gamma, \alpha]_q$ is the corresponding Dehn twist in the opposite direction.
\end{lemma}
\begin{proof}
    Suppose $\alpha$ and $\gamma$ intersect exactly once.
    Then locally, we have
    \begin{align*}
        [\alpha, \gamma]_q &= q\alpha \gamma - q^{-1}\gamma \alpha\\
        &= q
        \begin{tikzpicture}[baseline=0]
            \draw[dashed, fill=gray!40] (0,0) circle (1);
            \draw[thick] (-0.71, 0.71) -- (0.71, -0.71);
            \draw[line width=3mm, gray!40] (-0.5, -0.5) -- (0.5, 0.5);
            \draw[thick] (-0.71, -0.71) -- (0.71, 0.71);
            \node at (-0.3,-0.6) {$\alpha$};
            \node at (-0.3,0.6) {$\gamma$};
            \node (space) at (-1, 0) {};
        \end{tikzpicture} - q^{-1}
        \begin{tikzpicture}[baseline=0]
            \draw[dashed, fill=gray!40] (0,0) circle (1);
            \draw[thick] (-0.71, -0.71) -- (0.71, 0.71);
            \draw[line width=3mm, gray!40] (-0.5, 0.5) -- (0.5, -0.5);
            \draw[thick] (-0.71, 0.71) -- (0.71, -0.71);
            \node at (-0.3,-0.6) {$\alpha$};
            \node at (-0.3,0.6) {$\gamma$};
            \node (space) at (-1, 0) {};
        \end{tikzpicture} \displaybreak\\
        &= q^2
        \begin{tikzpicture}[baseline=0]
            \draw[dashed, fill=gray!40] (0,0) circle (1);
            \draw[thick] (-0.71, -0.71) to[out=45, in=-45] (-0.71, 0.71);
            \draw[thick] (0.71, -0.71) to[out=135, in=-135] (0.71, 0.71);
            \node (space) at (-1, 0) {};
        \end{tikzpicture} +
        \begin{tikzpicture}[baseline=0]
            \draw[dashed, fill=gray!40] (0,0) circle (1);
            \draw[thick] (-0.71, -0.71) to[out=45, in=135] (0.71, -0.71);
            \draw[thick] (-0.71, 0.71) to[out=-45, in=-135] (0.71, 0.71);
            \node (space) at (-1, 0) {};
        \end{tikzpicture} -
        \begin{tikzpicture}[baseline=0]
            \draw[dashed, fill=gray!40] (0,0) circle (1);
            \draw[thick] (-0.71, -0.71) to[out=45, in=135] (0.71, -0.71);
            \draw[thick] (-0.71, 0.71) to[out=-45, in=-135] (0.71, 0.71);
            \node (space) at (-1, 0) {};
        \end{tikzpicture} - q^{-2}
        \begin{tikzpicture}[baseline=0]
            \draw[dashed, fill=gray!40] (0,0) circle (1);
            \draw[thick] (-0.71, -0.71) to[out=45, in=-45] (-0.71, 0.71);
            \draw[thick] (0.71, -0.71) to[out=135, in=-135] (0.71, 0.71);
            \node (space) at (-1, 0) {};
        \end{tikzpicture}\\
        &= \left( q^2 - q^{-2} \right)
        \begin{tikzpicture}[baseline=0]
            \draw[dashed, fill=gray!40] (0,0) circle (1);
            \draw[thick] (-0.71, -0.71) to[out=45, in=-45] (-0.71, 0.71);
            \draw[thick] (0.71, -0.71) to[out=135, in=-135] (0.71, 0.71);
            \node (space) at (-1, 0) {};
        \end{tikzpicture}\\
    \end{align*}
    giving us our result.
    Noticing that $[b,a]_q = -[a,b]_{q^{-1}}$, we get $\frac{1}{q^2 - q^{-2}}[\gamma, \alpha]_{q} = \frac{-1}{q^2 - q^{-2}}[\alpha,\gamma]_{q^{-1}}$ is the Dehn twist in the opposite direction.
\end{proof}

\begin{lemma}\label{lemma:rationalcurves}
    As an algebra, the set $\{ X_{1,0}(\mu_{1}, \nu_{1}), X_{2,0}(\mu_{2}, \nu_{2}), X_{3,0}(\mu_{3}, \nu_{3}) \}$ for all $\mu_i, \nu_j \in \{ \pm \}$ generates any $\left\{ \mathbb{Q} \cup \frac{1}{0} \right\}$-sloped tangle in $\mathscr{S}\left( T^2 \setminus D^2 \right)$ with $0 \in \frac{1}{2}\mathbb{Z}$ twists.
\end{lemma}
\begin{proof}
    Recall that $\left| \operatorname{det} \begin{pmatrix} a & c\\ b & d \end{pmatrix} \right| = n$ if and only if the $(a,b)$-curve and the $(c,d)$-curve (or $(c,d)$-tangle) have a geometric intersection number of $n$.
    Consider the (non-homomorphic) map of $G$-sets, $\sigma : GL_2(\mathbb{Z}) \to \mathbb{Z}^2$ defined by $A \mapsto A\cdot \left[ \begin{matrix} 1\\1\end{matrix} \right]$.
    Suppose $p$ and $q$ are coprime.
    If they are not coprime then the $(p,q,0)$-tangle intersects itself and can be resolved into curves and $\partial (T^2 \setminus D^2)$-tangles with $\mathbb{Q} \bigcup \frac{1}{0}$ slopes.
    We'll also assume that $0 < q < p$ (the proof is symmetric for negative slopes and when $q > p$).
    
    By Lemma 1 in \cite{MR1675190}, we can decompose $p$ and $q$ into $u + w = p$ and $v + z = q$ such that $\operatorname{det} \begin{pmatrix} u & w\\ v & z \end{pmatrix} = \pm 1$ with $0 < w < p$, $0 < u < p-1$, and $v, z > 0$, for $p \geq 3$.
    Thus for $\begin{pmatrix} p\\ q \end{pmatrix} \in \mathbb{Z}^2$, there exists an inverse, $\sigma^{-1} \begin{pmatrix} p\\ q \end{pmatrix} = \begin{pmatrix} u & w\\ v & z \end{pmatrix}$.
    We then find the inverse of the second column, $\sigma^{-1} \begin{pmatrix} w\\ z \end{pmatrix}$, and repeat until we get $\begin{pmatrix} p'\\ q' \end{pmatrix}$ for $q' < p' \leq 2$.
    Using Lemma \ref{lemma:IntersectionDehnTwist}, each step in the reverse process of this algorithm corresponds to a Dehn twist, $\frac{\pm 1}{q^2 - q^{-2}}\left[ Y_j, - \right]_{q^{\pm 1}}$, along some $Y_j$-curve corresponding to the first column.
    
    Finally, note that the $(2,1)$-tangle is equal to $\frac{1}{q^2 - q^{-2}} \left[ Y_1, X_{3,0}(\mu, \nu) \right]_q$.
\end{proof}

\begin{example}
    To illustrate the application of this algorithm, let's examine the (5,3)-tangle with states $\mu$ and $\nu$, which we'll denote as $\tilde{X}(\mu, \nu)$.
    The matrices of interest in this example are
    $$\sigma\begin{pmatrix} 0 & 1\\ 1 & 0 \end{pmatrix} = \begin{pmatrix} 1\\ 1 \end{pmatrix}, \quad\quad
    \sigma\begin{pmatrix} 1 & 1\\ 0 & 1 \end{pmatrix} = \begin{pmatrix} 2\\ 1 \end{pmatrix}, \quad\quad
    \sigma\begin{pmatrix} 2 & 1\\ 1 & 1 \end{pmatrix} = \begin{pmatrix} 3\\ 2 \end{pmatrix}, \quad\quad
    \sigma\begin{pmatrix} 3 & 2\\ 2 & 1 \end{pmatrix} = \begin{pmatrix} 5\\ 3 \end{pmatrix}.$$
    Thus, we get the following series of Dehn twists
    $$ \frac{-1}{\left( q^2 - q^{-2} \right)^{6}}\left[ \left[ Y_3, \left[ Y_1, \left[ Y_2, Y_1 \right]_{q^{-1}} \right]_q \right]_{q^{-1}}, \left[ Y_1, \left[ Y_2, X_{1,0}(\mu, \nu) \right]_{q^{-1}} \right]_q \right]_{q} = \tilde{X}(\mu, \nu).$$
\end{example}

\begin{thm}
    $B$ generates $\mathscr{S}\left( T^2 \setminus D^2 \right)$ as an algebra.
\end{thm}
\begin{proof}
    A short calculation shows that
    \begin{align}
        X_{i,k+1}(\mu, \nu) &= \frac{1}{(q^2 - q^{-2})^3} \left[Y_{i+1}, \left[ Y_i, \left[ Y_{i-1}, X_{i,k}(\mu, \nu) \right]_q \right]_q \right]_q \label{eq:shiftup}\\
        X_{i,k-1}(\mu, \nu) &= \frac{1}{(q^2 - q^{-2})^3} \left[ \left[ \left[ X_{i,k}(\mu, \nu), Y_{i+1} \right]_q, Y_{i} \right]_q, Y_{i-1} \right]_q \label{eq:shiftdown}
    \end{align}
    for $i \mod 3$ and $\mu, \nu \in \{ \pm \}$.
    Using this result and Lemma \ref{lemma:rationalcurves}, we can construct any $(p,q,r)$-tangle.
    By Theorem \ref{theorem:SSOrdBasis}, the set of all isotopy classes of increasingly stated, $\mathfrak{o}$-simple $\partial(T^2 \setminus D^2)$-tangle diagrams forms a basis for $\mathscr{S}\left( T^2 \setminus D^2 \right)$.
    Since every simple non-parallel tangle can be expressed as a $(p,q,r)$-tangle, every simple $\partial\left( T^2 \setminus D^2 \right)$-tangle diagram can be written as a linear combination of products of these $(p,q,r)$-tangles and powers of our single parallel tangle.
    Thus, the stated skein algebra generated by $B$ spans the entire space.\hfill
\end{proof}

\begin{remark}
    Since we have the relation
    $$X_{3,0}(\mu, \nu) = \frac{1}{q^2 - q^{-2}}\left[ X_{1,0}(\mu, \nu), q^{1/2}X_{2,0}(+,-) - q^{5/2}X_{2,0}(-,+) \right]_q,$$
    we don't technically need to include $X_{3,0}(\mu, \nu)$ to generate $\mathscr{S}(T^2 \setminus D^2)$. We only need the eight elements $\left\{ X_{i,0}(\mu, \nu) \, \mid \mu, \nu \in \{ \pm \}, i \in \{1,2\} \right\}$. However, as with $K_q(T^2 \setminus D^2)$, it is often more notationally convenient to include $X_{3,0}(\mu, \nu)$ as well.
\end{remark}

\section{Relation to Factorization Homology}

Alekseev-Grosse-Schomerus moduli algebras (AGS algebras) are deformations of representation varieties of $\Sigma \setminus D^2$ for some surface $\Sigma$, via the Fock-Rosly Poisson structure. It was shown in both \cite{faitg2022holonomy} and \cite{MR4598807} that these AGS algebras are isomorphic as $\mathcal{O}_{q}(SL_2)$-comodule algebras to $\mathscr{S}(\Sigma \setminus D^2)$, the corresponding stated skein algebra with one marking on the boundary created by removing $D^2$.

Internal skein algebras (also known as internal endomorphism algebras), denoted $A_{\Sigma \setminus D^2}$, were explored in \cite{MR3847209, MR3659493, MR4557403} using factorization homology and were also shown to be isomorphic to corresponding AGS algebras in \cite{MR3847209}. It was stated in \cite{MR4557403} and \cite{MR4264235} and made more explicit in \cite{arx220100045} that $A_{\Sigma \setminus D^2}$ should be isomorphic to $\mathscr{S}(\Sigma \setminus D^2)$ with one marking as $\mathcal{O}_q(SL_2)$-comodule algebras, tying together factorization homology and stated skein theory. Using only skein theory, Ha\"{i}oun proved in \cite{MR4437512} that $\mathscr{S}(\Sigma \setminus D^2)$ (with one marking) is isomorphic to its corresponding internal skein algebra as $\mathcal{O}_{q}(SL_2)$-comodule algebras.

Moving to our specific case, it was shown in \cite{MR3847209, MR3659493} that the algebra $A_{T^2 \setminus D^2}$, and therefore $\mathscr{S}(T^2 \setminus D^2)$, is isomorphic as an $\mathcal{O}_q(SL_2)$-comodule algebra to the algebra of quantum differential operators, $D_q(SL_2) \cong U_q(\mathfrak{sl}_2) \ltimes \mathcal{O}_{q}(SL_2)$, which they called the \emph{elliptic double} and is a subalgebra of the Heisenberg double of $U_q(\mathfrak{sl}_2)$. In \cite{MR4557403}, Gunningham, Jordan, and Safronov provided an explicit presentation for this algebra: $U_q(\mathfrak{sl}_2) \ltimes \mathcal{O}_{q}(SL_2)$ is generated by $a_{1}^{1}, a_{2}^{1}, a_{1}^{2}, a_{2}^{2}, b_{1}^{1}, b_{2}^{1}, b_{1}^{2}, b_{2}^{2}$, subject to the relations
\begin{align*}
    R_{21} A_1 R_{12} A_2 &= A_2 R_{21} A_1 R_{12} \\
    R_{21} B_1 R B_2 &= B_2 R_{21} B_1 R \\
    R_{21} B_1 R A_2 &= A_2 R_{21} B_1 R_{21}^{-1}\\
    1 &= a_1^1 a_2^2-q^2 a_2^1 a_1^2, \\
    1 &= b_1^1 b_2^2-q^2 b_2^1 b_1^2,
\end{align*}
where
\begin{align*}
    A_1 &= \begin{pmatrix} a_1^1 & a_2^1 \\ a_1^2 & a_2^2 \end{pmatrix} \otimes \operatorname{Id},
    & B_1 &= \begin{pmatrix} b_1^1 & b_2^1 \\ b_1^2 & b_2^2 \end{pmatrix} \otimes \operatorname{Id},\\
    A_2 &= \operatorname{Id} \otimes \begin{pmatrix} a_1^1 & a_2^1 \\ a_1^2 & a_2^2 \end{pmatrix},
    & B_2 &= \operatorname{Id} \otimes \begin{pmatrix} b_1^1 & b_2^1 \\ b_1^2 & b_2^2 \end{pmatrix},
\end{align*}
and $R = R_{12}$ is the $R$ matrix over $L(1) \otimes L(1)$. These generators correspond to the same generators just discussed.

\begin{align*}
    X_{1,0}(+,+) &\mapsto a_2^2 & X_{2,0}(+,+) &\mapsto b_2^2 \\
    X_{1,0}(+,-) &\mapsto a_2^1 & X_{2,0}(+,-) &\mapsto b_2^1 \\
    X_{1,0}(-,+) &\mapsto a_1^2 & X_{2,0}(-,+) &\mapsto b_1^2 \\
    X_{1,0}(-,-) &\mapsto a_1^1 & X_{2,0}(-,-) &\mapsto b_1^1
\end{align*}
However, multiplication in these two algebras are defined a bit differently and so the relations on $\mathscr{S}(T^2 \setminus D^2)$ aren't going to be the same on the nose. In particular, the product structure in $D_q(SL_2)$ is an associative braided product (called the covariantised product in \cite{MR1381692}) which corresponds to a transmutation of $\mathcal{O}_q(SL_2)$, the Hopf dual of $U_q(\mathfrak{sl}_2)$. You can read more about transmutation in \cite{MR1381692} and for great pictures and details on how it relates to stated skein algebras, see L\^{e} and Sikora's preprint \cite{arx220100045}.
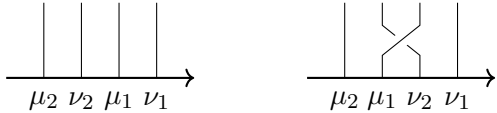
\begin{figure}[h]
    \centering
    \begin{tikzpicture}
        \draw[thick, ->] (0, 0) -- (2.5, 0);
        \draw (0.5, 0) -- (0.5, 1);
        \draw (1, 0) -- (1, 1);
        \draw (1.5, 0) -- (1.5, 1);
        \draw (2, 0) -- (2, 1);
        \node at (0.5, -0.3) {$\mu_2$};
        \node at (1, -0.3) {$\nu_2$};
        \node at (1.5, -0.3) {$\mu_1$};
        \node at (2, -0.3) {$\nu_1$};
        \draw[thick, ->] (4, 0) -- (6.5, 0);
        \draw (4.5, 0) -- (4.5, 1);
        \draw (5.5, 0) -- (5.5, 0.3) -- (5, 0.7) -- (5, 1);
        \draw[line width=2mm, white] (5, 0.3) -- (5.5, 0.7);
        \draw (5, 0) -- (5, 0.3) -- (5.5, 0.7) -- (5.5, 1);
        \draw (6, 0) -- (6, 1);
        \node at (4.5, -0.3) {$\mu_2$};
        \node at (5, -0.3) {$\mu_1$};
        \node at (5.5, -0.3) {$\nu_2$};
        \node at (6, -0.3) {$\nu_1$};
    \end{tikzpicture}
    \caption[Diagram of transmutation multiplication]{Very roughly speaking, the left is multiplication in $\mathscr{S}(T^2 \setminus D^2)$ and right is multiplication in the ``transmutated'' version.}
    \label{fig:enter-label}
\end{figure}

Note that when viewing $\mathscr{S}(T^2 \setminus D^2)$ with this braided product as our multiplication, the generators $\left\{X_{1,0}(\mu_1, \nu_1), X_{2,0}(\mu_2, \nu_2) \, \mid \mu_i, \nu_i \in \{ \pm \}\right\}$ agree with these relations under the identification above. Therefore, this is a concrete example of the relationship between factorization homology and stated skein algebras. For a more in-depth explanation of this relationship, please refer to \cite{MR3847209, MR4557403, MR4437512, arx220100045}. These details, although quite interesting, are not directly relevant to the main discussion in this work and will not be discussed here.

\section{Towards a PBW Basis}

It would be ideal and highly advantageous if we could additionally find a PBW basis for $\ms{S}\left(T^2 \setminus D^2\right)$.
Sadly, establishing the basis for this algebra has proven to be quite challenging, primarily due to the rapid escalation of calculations.
The main hurdle lies in identifying all the relations required for such a presentation, as is often the case in such endeavors.

Unfortunately, attempting to use similar techniques used in the Kauffman bracket case has not been fruitful.
While we have the relation $\frac{1}{q^2 - q^{-2}}[Y_i, Y_{i+1}]_q = Y_{i+2}$ just as the Kauffman bracket case, computing an analogous relation on tangles instead yields \ref{eq:shiftup} and \ref{eq:shiftdown}, indicating the possible need to consider these half-twists when establishing a basis.
Paradoxically, we also have the equality
\begin{align}\label{eq:NotLI}
    Y_i = q^{1/2} X_{i,r}(+,-) - q^{5/2} X_{i,r}(-,+)
\end{align}
for all $r \in \frac{1}{2}\mathbb{Z}$, further complicating things.
Note that this also extends to all closed $(p,q)$-curves and their corresponding $(p,q,r)$-tangles.

For any $r \in \frac{1}{2}\mathbb{Z}$, let $\gamma_{p,q,r}$ be a simple $\partial \left( T^2 \setminus D^2 \right)$-tangle corresponding to the tuple $(p,q,r)$ from Theorem \ref{theorem:classification}.
An interesting note is that for any $r,s \in \frac{1}{2}\mathbb{Z}$, resolving any crossings in the product of $\gamma_{p,q,r} \gamma_{p,q,s}$ will always result in a sum of simple tangles containing a summand element in $\gamma_{\frac{r+s}{2}}$ when $\frac{r+s}{2} \in \frac{1}{2}\mathbb{Z}$ and an element in $\gamma_{ \frac{\lfloor r+s \rfloor}{2} } \gamma_{ \frac{\lceil r+s \rceil}{2} }$ when $\frac{r+s}{2} \notin \frac{1}{2}\mathbb{Z}$.
\begin{conj}
    The product of stateless simple tangles, $\gamma_{p,q,r}$ and $\gamma_{p,q,s}$, always resolves to a polynomial of closed $(p,q)$-curves, parallel tangles, parallel closed curves, $\gamma_{p,q,\frac{\lfloor r+s \rfloor}{2}}$, and $\gamma_{p,q,\frac{\lceil r+s \rceil}{2}}$.
\end{conj}
In particular, if $s = r \pm \frac{1}{2}$, the product can always be drawn without any crossings.
Although this observation may give off the sense of a grading, the presence of this averaging formula precludes it from being classified as such.
Furthermore, it, of course, could not constitute a decomposition of our algebra either, as illustrated by equation \ref{eq:NotLI}, which demonstrates the absence of linear independence among these if we tried to make them direct summands.
As far as I am aware, there are no comparable examples in the existing literature of this phenomenon to serve as a reference point.

Another possible route is to employing the embedding technique described in Section \ref{section:T6Embedding} and studying the image of $\mathscr{S}(T^2 \setminus D^2)$.
Although some progress has been made, extracting useful patterns from this embedding is particularly arduous.
Notably, the image of seemingly simple tangle elements quickly becomes unwieldy in size as $r \in \frac{1}{2}\mathbb{Z}$ moves further away from $0$ (see Appendix \ref{appendix:ImageOfT6} for an example).

For convenience and further use, I have calculated, and partially verified using a computer program (Appendix \ref{appendix:qCommRel}), all $16$ commuting relations among $X_{1,k}$ and $X_{2,k}$ for all states. Note that these calculations assume both tangles share the same number of twists with respect to our classification in Theorem \ref{theorem:classification}. I have also added the full calculation for the longest calculation (in the case of $0$ twists around the boundary) in Appendex \ref{section:CommRelCalc}. Here, $\widetilde{X}_{3,k}(\mu, \nu)$ corresponds to the $(1,-1)$-tangle with $k$ twists and $\widetilde{Y}_3$ is the closed $(1,-1)$-curve.

\begin{small}
    \begin{align*}
        X_{1,k}(+,+) X_{2,k}(+,+) &= q^{2} X_{2,k}(+,+) X_{1,k}(+,+) \\
        X_{1,k}(+,+) X_{2,k}(+,-) &= q^{-2} X_{2,k}(+,-) X_{1,k}(+,+) + q^{-3/2}(q^2 - q^{-2})X_{3,k}(+,+)\\
        X_{1,k}(+,+) X_{2,k}(-,+) &= q^{-2} X_{2,k}(-,+) X_{1,k}(+,+) \\
        X_{1,k}(+,+) X_{2,k}(-,-) &= q^{-6} X_{2,k}(-,-) X_{1,k}(+,+) + q^{-3/2}(q^2 - q^{-2}) X_{3,k}(-,+) \\
        X_{1,k}(+,-) X_{2,k}(+,+) &= q^{6} X_{2,k}(+,+) X_{1,k}(+,-) - q^{7/2}(q^2 - q^{-2})X_{2,k}(+,+)Y_1\\
        &\phantom{=} - q^{5/2}(q^2 - q^{-2})\left( q^{2} \widetilde{X}_{3,k}(+,+) + q^{-2}\widetilde{X}_{3,k-\frac{1}{2}}(+,+) \right)\\
        X_{1,k}(+,-) X_{2,k}(+,-) &= q^{2} X_{2,k}(+,-) X_{1,k}(+,-) + (q^2 - q^{-2})\widetilde{Y}_3 \\
        &\phantom{=} - q^{-1/2}(q^2 - q^{-2}) \left( q \widetilde{X}_{3,k}(+,-) + X_{2}(+,-)Y_1 - q^{-1} X_{3}(+,-) \right) \\
        X_{1,k}(+,-) X_{2,k}(-,+) &= q^{2} X_{2,k}(-,+) X_{1,k}(+,-) - q^{-1/2}(q^2 - q^{-2}) \widetilde{X}_{3,k-\frac{1}{2}}(-,+)\\
        X_{1,k}(+,-) X_{2,k}(-,-) &= q^{-2} X_{2,k}(-,-) X_{1,k}(+,-) + q^{-3/2}(q^2 - q^{-2}) X_{3,k}(-,-) \\
        X_{1,k}(-,+) X_{2,k}(+,+) &= q^{6} X_{2,k}(+,+) X_{1,k}(-,+) - q^{5/2}(q^2 - q^{-2}) \widetilde{X}_{3,k-\frac{1}{2}}(+,+) \\
        X_{1,k}(-,+) X_{2,k}(+,-) &= q^{2} X_{2,k}(+,-) X_{1,k}(-,+) - q^{1/2}(q^2 - q^{-2}) \widetilde{X}_{3,k}(-,+) \\
        X_{1,k}(-,+) X_{2,k}(-,+) &= q^{2} X_{2,k}(-,+) X_{1,k}(-,+) \\
        X_{1,k}(-,+) X_{2,k}(-,-) &= q^{-2} X_{2,k}(-,-) X_{1,k}(-,+) \\
        X_{1,k}(-,-) X_{2,k}(+,+) &= \resizebox{0.95\width}{!}{$q^{10} X_{2,k}(+,+) X_{1,k}(-,-) - q^{13/2} (q^2 - q^{-2})\left( q^{4} X_{3,k}(-,+) + q^{-4} \widetilde{X}_{3,k}(+,-) \right)$}\\
        &\phantom{=} - q^{11/2}(q^2 - q^{-2}) \left( q^3 \widetilde{X}_{3,k-\frac{1}{2}}(-,+) + q^{-3} \widetilde{X}_{3,k-\frac{1}{2}}(+,-)\right)\\
        &\phantom{=} - q^{7}(q^2 - q^{-2})(q^3 + q^{-3}) \widetilde{Y}_{3,k}\\
        X_{1,k}(-,-) X_{2,k}(+,-) &= q^{6} X_{2,k}(+,-) X_{1,k}(-,-)\\
        &\phantom{=} - q^{5/2}(q^2 - q^{-2}) \left( q X_{2,k}(-,-)Y_1 + (q^{2} + q^{-2})\widetilde{X}_{3,k}(-,-) \right) \\
        X_{1,k}(-,-) X_{2,k}(-,+) &= q^{6} X_{2,k}(-,+) X_{1,k}(-,-) - q^{7/2} (q^2 - q^{-2}) \widetilde{X}_{3,k-\frac{1}{2}}(-,-) \\
        X_{1,k}(-,-) X_{2,k}(-,-) &= q^{2} X_{2,k}(-,-) X_{1,k}(-,-)
    \end{align*}
\end{small}
\vspace*{-\baselineskip}
\chapter{Representations from Quantum Tori}

\section{Embedding Into Quantum Tori}\label{section:T6Embedding}
\begin{definition}\label{defn:qtorus}
    Given an anti-symmetric integral $n \times n$ matrix $Q$, the associated \textit{quantum torus} is defined as
    $$\mathbb{T}^{n}(Q) := \frac{\mathbb{C} \left[ x_1^{\pm 1}, \cdots, x_n^{\pm 1} \right]}{\left( x_i x_j = q^{Q_{ij}} x_j x_i \right)}$$
    and the corresponding \textit{quantum plane} is
    $$\mathbb{T}_{+}^{n}(Q) := \frac{\mathbb{C} \left[ x_1, \cdots, x_n \right]}{\left( x_i x_j = q^{Q_{ij}} x_j x_i \right)}.$$
\end{definition}
We will often drop the $Q$ in $\mathbb{T}^{n}(Q)$ and $\mathbb{T}_{+}^{n}(Q)$ to shorten notation.
Clearly, $\mathbb{T}_+^n$ is an Ore domain and $\mathbb{T}^n$ is an Ore localization is $\mathbb{T}_{+}^n$.

L\^{e} and Yu proved that there exist embeddings $\mathbb{T}_{+}^r \xhookrightarrow{\psi_{\mathcal{E}}} \ms{S}(\mf{S}) \xhookrightarrow{\varphi_{\mathcal{E}}} \mathbb{T}^{r}$, where $r$ is the Gelfand-Kirillov dimension of $\ms{S}(\mf{S})$ and $\psi_{\mathcal{E}}$ and $\varphi_{\mathcal{E}}$ are algebra homomorphisms.
Both of these embeddings depend on a quasitriangulation of $\mf{S}$, denoted $\mathcal{E}$, and are discussed extensively in \cite{MR4431131}.
Bonahon and Wong constructed a similar map in \cite{MR2851072} where $\mathbb{T}^r$ would instead be rational functions in skew-commuting variables associated to the square roots of the shear coordinates of Chekhov and Fock's enhanced quantum Teichm\"{u}ller space.
Furthermore, Bonahon and Wong demonstrated that a change in quasitriangulation induces an algebra isomorphism between the respective Teichm\"{u}ller spaces, which can also be understood as a change of shear coordinates.

Let $(\Sigma, \mathcal{P})$ be a marked surface and $\mathfrak{S}$ the corresponding punctured bordered surface.
Define $r(\Sigma, \mathcal{P}) = r(\mathfrak{S})$ to be 0 if $\mathfrak{S}$ is the sphere with no or one ideal point, $1$ if $\mathfrak{S}$ is the sphere with two ideal points, $2$ if $\mathfrak{S}$ is the closed torus, and $3|\mathcal{P}| - 3 \chi(\mathfrak{S})$ otherwise, where $\chi(\mathfrak{S})$ is the Euler characteristic of $\mathfrak{S}$.
L\^{e} and Yu showed in \cite{MR4264235} that the Gelfand-Kirillov dimension of $\mathscr{S}(\mathfrak{S})$ is $r(\mathfrak{S})$.
Equivalently, $r$ can also be defined as the cardinality of the maximal collection of non-isotopic ideal arcs $\overline{\mathcal{E}} = \mathcal{E} \bigsqcup \hat{\mathcal{E}}_\partial = \left( \mathring{\mathcal{E}} \bigsqcup \mathcal{E}_\partial \right) \bigsqcup \hat{\mathcal{E}}_\partial$, where $\mathring{\mathcal{E}}$ are the ideal arcs not parallel to any boundary components, $\mathcal{E}_\partial$ are the ideal arcs that are parallel to boundary components, and $\hat{\mathcal{E}}_\partial$ is a copy of $\mathcal{E}_\partial$.
Each $e \in \overline{\mathcal{E}}$ corresponds to a generator, $x_e$, of $\mathbb{T}_{+}^{r}$.
The anti-symmetric integral matrix for $\mathbb{T}_{+}^r$ is defined using the anti-symmetric function
\begin{align*}
    Q(a,b) &= \# \locIdeal{b}{a} - \# \locIdeal{a}{b} & \text{for } a,b \in \mathcal{E}\\
    Q(a,\hat{b}) &= - \# \locIdeal{b}{a} - \# \locIdeal{a}{b} & \text{for } a \in \mathcal{E}, b \in \mathcal{E}_\partial \\
    Q(\hat{a}, \hat{b}) &= -Q(a,b) & \text{for } a,b \in \mathcal{E}_\partial.
\end{align*}
and canonically identifying the entries of $Q$ as $Q_{a,b} = Q(a,b)$.
By $\# \resizebox{0.8\width}{!}{\locIdeal{i}{j}}$ we mean the number of times a half edge of $i$ and a half edge $j$ meet at the same ideal point with $i$ following $j$ in the clockwise order.

We define $\psi_{\mathcal{E}}$ as sending each generator to the corresponding diagram (possibly scaled by some $q^{\pm 1/2}$) where the diagram has positive states if $e \in \mathcal{E}$ and is a bad arc if $e \in \hat{\mathcal{E}}_\partial$.
For example, if both of endpoints of $e$ end on the same ideal point, $\psi_{\mathcal{E}}$ sends $e$ to the diagram
$$\locIdeal{}{} \mapsto
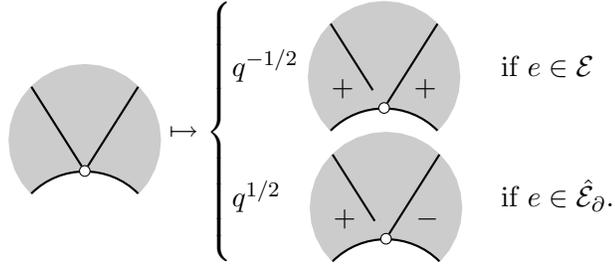
\begin{cases}
    q^{-1/2} \text{ } \locTangle{+}{+} & \text{ if } e \in \mathcal{E}\\
    q^{1/2} \text{\phantom{aa}} \locTangle{+}{-} & \text{ if } e \in \hat{\mathcal{E}}_\partial.
\end{cases}$$
where the $q^{\pm 1/2}$ scalar is introduced so that $x_e$ is reflection invariant, which follows by the $(R_6)$ relation.
Therefore, in this example, we have $\locTangle{+}{+} \mapsto q^{1/2} x_e$ and $\locTangle{+}{-} \mapsto q^{-1/2} x_{\hat{e}}$ as $\varphi_{\mathcal{E}}$ is an algebra homomorphism.

The stated skein algebra has the property that for every $\alpha \in \ms{S}(\mf{S})$, there is some monomial $m(x_1, \cdots, x_r) \in \mathbb{T}_{+}^{r}$ such that $\psi_{\mathcal{E}}\left(m(x_1, \cdots, x_r)\right) \alpha \in \operatorname{Im}(\psi_{\mathcal{E}})$.
Because these are algebra homomorphisms and since $(\varphi_{\mathcal{E}} \circ \psi_{\mathcal{E}})(x_e) = x_e$, we can explicitly find where any element in our skein algebra is mapped to using the following trick
$$\varphi_{\mathcal{E}}(\alpha) = m^{-1}(x_1, \cdots, x_r) m(x_1, \cdots, x_r) \varphi_{\mathcal{E}}(\alpha) = m^{-1}(x_1, \cdots, x_r) \varphi_{\mathcal{E}}(\psi_{\mathcal{E}}(m(x_1, \cdots, x_r)) \alpha),$$
where $\ms{S}(\mf{S})$ is viewed as the canonical $T_{+}^{r}$-module induced by $\psi_{\mathcal{E}}$, $m \cdot \alpha := \psi_{\mathcal{E}}(m)\alpha$.

\section{The Quantum $6$-Torus}\label{section:Quantum6Torus}

We will now perform this calculation for when $\mathfrak{S} = T^2 \setminus D^2$, with a single marked point on the boundary.
Let $\mathcal{E}$ be the quasitriangulation of $T^2 \setminus D^2$ shown in figure \ref{fig:quasitriangulation}.
\begin{figure}[h]
    \centering
    \resizebox{2\width}{!}{\begin{tikzpicture}
        \draw[draw=none, fill=gray!40] (0,1) -- (2,1) -- (2,-1) -- (0,-1) -- (1, -0.15) to[out=0, in=-90] (1.15, 0) to[out=90, in=0] (1, 0.15) to[out=180, in=90] (0.85, 0) to[out=-90, in=180] (1, -0.15) -- (0, -1) -- (0,1);
        \draw[thick, yellow] (1, 0) circle (0.15);
        \draw[thick, red] (1.3, -1) to[out=90, in=270] (1.3, 0) to[out=90, in=0] (1, 0.15) to[out=60, in=270] (1.3, 1);
        \draw[thick, blue] (0.4, -1) -- (0.4, -0.2) to[out=90, in=180] (1, 0.15) to[out=45, in=180] (2, 0.6);
        \draw[thick, blue] (0, 0.6) to[out=0, in=270] (0.4, 1);
        \draw[thick, green] (0.7, -1) to[out=90, in=270] (0.7, 0) to[out=90, in=180] (1, 0.15) to[out=120, in=270] (0.7, 1);
        \draw[thick, orange] (0, 0.15) -- (1, 0.15) -- (2, 0.15);
        \draw[thick, ->] (0,1) -- (0.9,1);
        \draw[thick, ->] (0.85,1) -- (1.2,1);
        \draw[thick] (1.2,1) -- (2,1);
        \draw[thick, ->] (0,-1) -- (0.9,-1);
        \draw[thick, ->] (0.85,-1) -- (1.2,-1);
        \draw[thick] (1.2,-1) -- (2,-1);
        \draw[thick, ->] (0,-1) -- (0,0);
        \draw[thick] (0,-1) -- (0,1);
        \draw[thick, ->] (2,-1) -- (2,0);
        \draw[thick] (2,-1) -- (2,1);
        \node[draw, circle, inner sep=0pt, minimum size=3pt, fill=white] at (1, 0.15) {};
    \end{tikzpicture}}
    \caption{A quasitriangulation of $T^2 \setminus D^2$}
    \label{fig:quasitriangulation}
\end{figure}
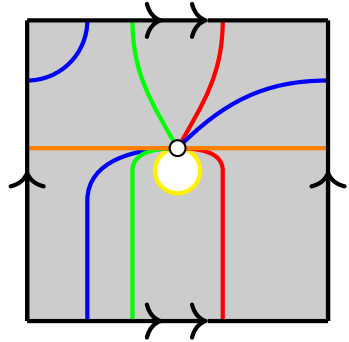
Note that as $T^2 \setminus D^2$ does not contain any punctures, $\mathcal{E}$ is a full triangulation of $T^2 \setminus D^2$. We will correspond our variables, $x_i$, to edges of $\mathcal{E}$ in the following way.
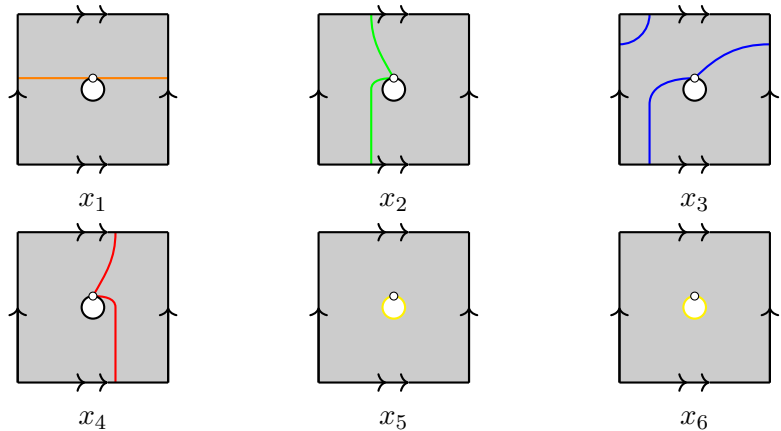
\begin{center}
    \begin{tikzpicture}
        \draw[draw=none, fill=gray!40] (0,1) -- (2,1) -- (2,-1) -- (0,-1) -- (1, -0.15) to[out=0, in=-90] (1.15, 0) to[out=90, in=0] (1, 0.15) to[out=180, in=90] (0.85, 0) to[out=-90, in=180] (1, -0.15) -- (0, -1) -- (0,1);
        \draw[thick, orange] (0, 0.15) -- (1, 0.15) -- (2, 0.15);
        \draw[thick, ->] (0,1) -- (0.9,1);
        \draw[thick, ->] (0.85,1) -- (1.2,1);
        \draw[thick] (1.2,1) -- (2,1);
        \draw[thick, ->] (0,-1) -- (0.9,-1);
        \draw[thick, ->] (0.85,-1) -- (1.2,-1);
        \draw[thick] (1.2,-1) -- (2,-1);
        \draw[thick, ->] (0,-1) -- (0,0);
        \draw[thick] (0,-1) -- (0,1);
        \draw[thick, ->] (2,-1) -- (2,0);
        \draw[thick] (2,-1) -- (2,1);
        \draw[thick, black] (1, 0) circle (0.15);
        \node[draw, circle, inner sep=0pt, minimum size=3pt, fill=white] at (1, 0.15) {};
        \node at (1,-1.5) {$x_1$};
        \draw[draw=none, fill=gray!40] (4,1) -- (6,1) -- (6,-1) -- (4,-1) -- (5, -0.15) to[out=0, in=-90] (5.15, 0) to[out=90, in=0] (5, 0.15) to[out=180, in=90] (4.85, 0) to[out=-90, in=180] (5, -0.15) -- (4, -1) -- (4,1);
        \draw[thick, green] (4.7, -1) to[out=90, in=270] (4.7, 0) to[out=90, in=180] (5, 0.15) to[out=120, in=270] (4.7, 1);
        \draw[thick, ->] (4,1) -- (4.9,1);
        \draw[thick, ->] (4.85,1) -- (5.2,1);
        \draw[thick] (5.2,1) -- (6,1);
        \draw[thick, ->] (4,-1) -- (4.9,-1);
        \draw[thick, ->] (4.85,-1) -- (5.2,-1);
        \draw[thick] (5.2,-1) -- (6,-1);
        \draw[thick, ->] (4,-1) -- (4,0);
        \draw[thick] (4,-1) -- (4,1);
        \draw[thick, ->] (6,-1) -- (6,0);
        \draw[thick] (6,-1) -- (6,1);
        \draw[thick, black] (5, 0) circle (0.15);
        \node[draw, circle, inner sep=0pt, minimum size=3pt, fill=white] at (5, 0.15) {};
        \node at (5,-1.5) {$x_2$};
        \draw[draw=none, fill=gray!40] (8,1) -- (10,1) -- (10,-1) -- (8,-1) -- (9, -0.15) to[out=0, in=-90] (9.15, 0) to[out=90, in=0] (9, 0.15) to[out=180, in=90] (8.85, 0) to[out=-90, in=180] (9, -0.15) -- (8, -1) -- (8,1);
        \draw[thick, blue] (8.4, -1) -- (8.4, -0.2) to[out=90, in=180] (9, 0.15) to[out=45, in=180] (10, 0.6);
        \draw[thick, blue] (8, 0.6) to[out=0, in=270] (8.4, 1);
        \draw[thick, ->] (8,1) -- (8.9,1);
        \draw[thick, ->] (8.85,1) -- (9.2,1);
        \draw[thick] (9.2,1) -- (10,1);
        \draw[thick, ->] (8,-1) -- (8.9,-1);
        \draw[thick, ->] (8.85,-1) -- (9.2,-1);
        \draw[thick] (9.2,-1) -- (10,-1);
        \draw[thick, ->] (8,-1) -- (8,0);
        \draw[thick] (8,-1) -- (8,1);
        \draw[thick, ->] (10,-1) -- (10,0);
        \draw[thick] (10,-1) -- (10,1);
        \draw[thick, black] (9, 0) circle (0.15);
        \node[draw, circle, inner sep=0pt, minimum size=3pt, fill=white] at (9, 0.15) {};
        \node at (9,-1.5) {$x_3$};
    \end{tikzpicture}
    \begin{tikzpicture}
        \draw[draw=none, fill=gray!40] (0,1) -- (2,1) -- (2,-1) -- (0,-1) -- (1, -0.15) to[out=0, in=-90] (1.15, 0) to[out=90, in=0] (1, 0.15) to[out=180, in=90] (0.85, 0) to[out=-90, in=180] (1, -0.15) -- (0, -1) -- (0,1);
        \draw[thick, red] (1.3, -1) to[out=90, in=270] (1.3, 0) to[out=90, in=0] (1, 0.15) to[out=60, in=270] (1.3, 1);
        \draw[thick, ->] (0,1) -- (0.9,1);
        \draw[thick, ->] (0.85,1) -- (1.2,1);
        \draw[thick] (1.2,1) -- (2,1);
        \draw[thick, ->] (0,-1) -- (0.9,-1);
        \draw[thick, ->] (0.85,-1) -- (1.2,-1);
        \draw[thick] (1.2,-1) -- (2,-1);
        \draw[thick, ->] (0,-1) -- (0,0);
        \draw[thick] (0,-1) -- (0,1);
        \draw[thick, ->] (2,-1) -- (2,0);
        \draw[thick] (2,-1) -- (2,1);
        \draw[thick, black] (1, 0) circle (0.15);
        \node[draw, circle, inner sep=0pt, minimum size=3pt, fill=white] at (1, 0.15) {};
        \node at (1,-1.5) {$x_4$};
        \draw[draw=none, fill=gray!40] (4,1) -- (6,1) -- (6,-1) -- (4,-1) -- (5, -0.15) to[out=0, in=-90] (5.15, 0) to[out=90, in=0] (5, 0.15) to[out=180, in=90] (4.85, 0) to[out=-90, in=180] (5, -0.15) -- (4, -1) -- (4,1);
        \draw[thick, ->] (4,1) -- (4.9,1);
        \draw[thick, ->] (4.85,1) -- (5.2,1);
        \draw[thick] (5.2,1) -- (6,1);
        \draw[thick, ->] (4,-1) -- (4.9,-1);
        \draw[thick, ->] (4.85,-1) -- (5.2,-1);
        \draw[thick] (5.2,-1) -- (6,-1);
        \draw[thick, ->] (4,-1) -- (4,0);
        \draw[thick] (4,-1) -- (4,1);
        \draw[thick, ->] (6,-1) -- (6,0);
        \draw[thick] (6,-1) -- (6,1);
        \draw[thick, yellow] (5, 0) circle (0.15);
        \node[draw, circle, inner sep=0pt, minimum size=3pt, fill=white] at (5, 0.15) {};
        \node at (5,-1.5) {$x_5$};
        \draw[draw=none, fill=gray!40] (8,1) -- (10,1) -- (10,-1) -- (8,-1) -- (9, -0.15) to[out=0, in=-90] (9.15, 0) to[out=90, in=0] (9, 0.15) to[out=180, in=90] (8.85, 0) to[out=-90, in=180] (9, -0.15) -- (8, -1) -- (8,1);
        \draw[thick, ->] (8,1) -- (8.9,1);
        \draw[thick, ->] (8.85,1) -- (9.2,1);
        \draw[thick] (9.2,1) -- (10,1);
        \draw[thick, ->] (8,-1) -- (8.9,-1);
        \draw[thick, ->] (8.85,-1) -- (9.2,-1);
        \draw[thick] (9.2,-1) -- (10,-1);
        \draw[thick, ->] (8,-1) -- (8,0);
        \draw[thick] (8,-1) -- (8,1);
        \draw[thick, ->] (10,-1) -- (10,0);
        \draw[thick] (10,-1) -- (10,1);
        \draw[thick, yellow] (9, 0) circle (0.15);
        \node[draw, circle, inner sep=0pt, minimum size=3pt, fill=white] at (9, 0.15) {};
        \node at (9,-1.5) {$x_6$};
    \end{tikzpicture}
\end{center}
Here $x_5$ corresponds to the single edge in $\mathcal{E}_\partial$ and $x_6$ corresponds to our extra copy of $x_5$ in $\hat{\mathcal{E}}_\partial$.
Using the anti-symmetric function, we find our $Q$ to be
$$Q =
{\begin{pmatrix}
    0 & 2 & 2 & -2 & 0 & -4\\
    -2 & 0 & -2 & -4 & 0 & -4\\
    -2 & 2 & 0 & -2 & 0 & -4\\
    2 & 4 & 2 & 0 & 0 & -4\\
    0 & 0 & 0 & 0 & 0 & 0\\
    4 & 4 & 4 & 4 & 0 & 0\\
\end{pmatrix}}.$$

Therefore, we are embedding our stated skein algebra, $\mathscr{S}\left( T^2 \setminus D^2 \right)$, into a quantum $6$-torus. This quantum torus is not simple as it has nontrivial center (see proposition $1.3$ in \cite{MR0949080}) when viewed as a complex twisted group algebra over the free abelian group of rank $6$. However, we still need to consider the entire quantum $6$-torus as $\mathscr{S}(T^2 \setminus D^2)$ has a GK-dimension of $6$.

Let $y_1, y_2, y_3 \in \ms{S}(T^2 \setminus D^2)$ be the diagrams corresponding to the meridian, longitude, and $(1,1)$-curve as closed curves, respectively.
Then we can use this quasitriangulation, $\mathcal{E}$, to find the image of $y_1$ in $\varphi_{\mathcal{E}}$.
Note the I'm not expressing the states in these diagrams as all of the states for these calculations will be positive.
$$
\psi_{\mathcal{E}} (x_2 x_4 x_3) \cdot y_1
= \left( q^{-3/2}
\begin{tikzpicture}[baseline=-1]
    \MarkedTorusBackground
    \draw[thick] (0.7, -1) to[out=90, in=270] (0.7, 0) to[out=90, in=180] (1, 0.15) to[out=120, in=270] (0.7, 1);
    \node[draw, circle, inner sep=0pt, minimum size=3pt, fill=white] at (1, 0.15) {};
\end{tikzpicture}
\begin{tikzpicture}[baseline=-1]
    \MarkedTorusBackground
    \draw[thick] (1.3, -1) to[out=90, in=270] (1.3, 0) to[out=90, in=0] (1, 0.15) to[out=60, in=270] (1.3, 1);
    \node[draw, circle, inner sep=0pt, minimum size=3pt, fill=white] at (1, 0.15) {};
\end{tikzpicture}
\begin{tikzpicture}[baseline=-1]
    \MarkedTorusBackground
    \draw[thick] (0.4, -1) -- (0.4, -0.2) to[out=90, in=180] (1, 0.15) to[out=45, in=180] (2, 0.5);
    \draw[thick] (0, 0.5) to[out=0, in=270] (0.4, 1);
    \node[draw, circle, inner sep=0pt, minimum size=3pt, fill=white] at (1, 0.15) {};
\end{tikzpicture}
\right)
\begin{tikzpicture}[baseline=-1]
    \MarkedTorusBackground
    \draw[thick] (0, 0.7) -- (2, 0.7);
    \node[draw, circle, inner sep=0pt, minimum size=3pt, fill=white] at (1, 0.15) {};
\end{tikzpicture}$$
$$= \left( q^{-3/2}
\begin{tikzpicture}[baseline=-1]
    \MarkedTorusBackground
    \draw[thick] (0.7, -1) to[out=90, in=270] (0.7, 0) to[out=90, in=180] (1, 0.15) to[out=120, in=270] (0.7, 1);
    \node[draw, circle, inner sep=0pt, minimum size=3pt, fill=white] at (1, 0.15) {};
\end{tikzpicture}
\begin{tikzpicture}[baseline=-1]
    \MarkedTorusBackground
    \draw[thick] (1.3, -1) to[out=90, in=270] (1.3, 0) to[out=90, in=0] (1, 0.15) to[out=60, in=270] (1.3, 1);
    \node[draw, circle, inner sep=0pt, minimum size=3pt, fill=white] at (1, 0.15) {};
\end{tikzpicture}
\right)
\begin{tikzpicture}[baseline=-1]
    \MarkedTorusBackground
    \draw[thick] (0, 0.7) -- (2, 0.7);
    \draw[thick] (0.4, -1) -- (0.4, -0.2) to[out=90, in=180] (1, 0.15) to[out=45, in=180] (2, 0.5);
    \draw[line width=3mm, gray!40] (0, 0.5) to[out=0, in=270] (0.4, 1);
    \draw[thick] (0, 0.5) to[out=0, in=270] (0.4, 1);
    \node[draw, circle, inner sep=0pt, minimum size=3pt, fill=white] at (1, 0.15) {};
\end{tikzpicture}$$
$$= \left( q^{-3/2}
\begin{tikzpicture}[baseline=-1]
    \MarkedTorusBackground
    \draw[thick] (0.7, -1) to[out=90, in=270] (0.7, 0) to[out=90, in=180] (1, 0.15) to[out=120, in=270] (0.7, 1);
    \node[draw, circle, inner sep=0pt, minimum size=3pt, fill=white] at (1, 0.15) {};
\end{tikzpicture}
\begin{tikzpicture}[baseline=-1]
    \MarkedTorusBackground
    \draw[thick] (1.3, -1) to[out=90, in=270] (1.3, 0) to[out=90, in=0] (1, 0.15) to[out=60, in=270] (1.3, 1);
    \node[draw, circle, inner sep=0pt, minimum size=3pt, fill=white] at (1, 0.15) {};
\end{tikzpicture}
\right)\left(q 
\begin{tikzpicture}[baseline=-1]
    \MarkedTorusBackground
    \draw[thick] (0.7, -1) to[out=90, in=270] (0.7, 0) to[out=90, in=180] (1, 0.15) to[out=120, in=270] (0.7, 1);
    \node[draw, circle, inner sep=0pt, minimum size=3pt, fill=white] at (1, 0.15) {};
\end{tikzpicture}
+ q^{-1}
\begin{tikzpicture}[baseline=-1]
    \MarkedTorusBackground
    \draw[thick] (0, 0.7) to[out=0, in=270] (0.4, 1);
    \draw[thick] (0.4, -1) -- (0.4, -0.2) to[out=90, in=180] (1, 0.15) to[out=45, in=180] (2, 0.5);
    \draw[thick] (0, 0.5) to[out=0, in=180] (2, 0.7);
    \node[draw, circle, inner sep=0pt, minimum size=3pt, fill=white] at (1, 0.15) {};
\end{tikzpicture}
\right)$$
$$= \left( q^{-3/2}
\begin{tikzpicture}[baseline=-1]
    \MarkedTorusBackground
    \draw[thick] (0.7, -1) to[out=90, in=270] (0.7, 0) to[out=90, in=180] (1, 0.15) to[out=120, in=270] (0.7, 1);
    \node[draw, circle, inner sep=0pt, minimum size=3pt, fill=white] at (1, 0.15) {};
\end{tikzpicture}
\right)\left(q
\begin{tikzpicture}[baseline=-1]
    \MarkedTorusBackground
    \draw[thick] (1.3, -1) to[out=90, in=270] (1.3, 0) to[out=90, in=0] (1, 0.15) to[out=60, in=270] (1.3, 1);
    \draw[thick] (0.7, -1) to[out=90, in=270] (0.7, 0) to[out=90, in=180] (1, 0.15) to[out=120, in=270] (0.7, 1);
    \node[draw, circle, inner sep=0pt, minimum size=3pt, fill=white] at (1, 0.15) {};
\end{tikzpicture}
+ q^{-1}
\begin{tikzpicture}[baseline=-1]
    \MarkedTorusBackground
    \draw[thick] (0, 0.5) to[out=0, in=180] (2, 0.7);
    \draw[line width=3mm, gray!40] (1.3, -1) to[out=90, in=270] (1.3, 0) to[out=90, in=0] (1, 0.15) to[out=60, in=270] (1.3, 1);
    \draw[thick] (1.3, -1) to[out=90, in=270] (1.3, 0) to[out=90, in=0] (1, 0.15) to[out=60, in=270] (1.3, 1);
    \draw[thick] (0, 0.7) to[out=0, in=270] (0.4, 1);
    \draw[thick] (0.4, -1) -- (0.4, -0.2) to[out=90, in=180] (1, 0.15) to[out=45, in=180] (2, 0.5);
    \draw[thick, black, fill=white] (1, 0) circle (0.15);
    \node[draw, circle, inner sep=0pt, minimum size=3pt, fill=white] at (1, 0.15) {};
\end{tikzpicture}
\right)$$
$$= \left( q^{-3/2}
\begin{tikzpicture}[baseline=-1]
    \MarkedTorusBackground
    \draw[thick] (0.7, -1) to[out=90, in=270] (0.7, 0) to[out=90, in=180] (1, 0.15) to[out=120, in=270] (0.7, 1);
    \node[draw, circle, inner sep=0pt, minimum size=3pt, fill=white] at (1, 0.15) {};
\end{tikzpicture}
\right)\left( q
\begin{tikzpicture}[baseline=-1]
    \MarkedTorusBackground
    \draw[thick] (1.3, -1) to[out=90, in=270] (1.3, 0) to[out=90, in=0] (1, 0.15) to[out=60, in=270] (1.3, 1);
    \draw[thick] (0.7, -1) to[out=90, in=270] (0.7, 0) to[out=90, in=180] (1, 0.15) to[out=120, in=270] (0.7, 1);
    \node[draw, circle, inner sep=0pt, minimum size=3pt, fill=white] at (1, 0.15) {};
\end{tikzpicture}
+
\begin{tikzpicture}[baseline=-1]
    \MarkedTorusBackground
    \draw[thick] (0, 0.15) -- (1, 0.15) -- (2, 0.15);
    \draw[thick] (0, 0.5) to[out=0, in=120] (1, 0.15) to[out=60, in=180] (2, 0.5);
    \node[draw, circle, inner sep=0pt, minimum size=3pt, fill=white] at (1, 0.15) {};
\end{tikzpicture}
+ q^{-2}
\begin{tikzpicture}[baseline=-1]
    \MarkedTorusBackground
    \draw[thick] (0, 0.5) to[out=0, in=270] (0.4, 1);
    \draw[thick] (0.4, -1) to[out=90, in=270] (1.4, -0.1) to[out=90, in=0] (1, 0.15) to[out=45, in=180] (2, 0.5);
    \draw[thick] (0, -1) to[out=45, in=180] (1, 0.15) to[out=60, in=225] (2, 1);
    \node[draw, circle, inner sep=0pt, minimum size=3pt, fill=white] at (1, 0.15) {};
\end{tikzpicture}
\right)$$
$$= q^{-3/2}
\left( q
\begin{tikzpicture}[baseline=-1]
    \MarkedTorusBackground
    \draw[thick] (1.3, -1) to[out=90, in=270] (1.3, 0) to[out=90, in=0] (1, 0.15) to[out=60, in=270] (1.3, 1);
    \draw[thick] (0.7, -1) to[out=90, in=270] (0.7, 0) to[out=90, in=180] (1, 0.15) to[out=120, in=270] (0.7, 1);
    \draw[thick] (0.4, -1) to[out=90, in=270] (0.4, 0) to[out=90, in=180] (1, 0.15) to[out=120, in=270] (0.4, 1);
    \node[draw, circle, inner sep=0pt, minimum size=3pt, fill=white] at (1, 0.15) {};
\end{tikzpicture}
+
\begin{tikzpicture}[baseline=-1]
    \MarkedTorusBackground
    \draw[thick] (0, 0.15) -- (1, 0.15) -- (2, 0.15);
    \draw[thick] (0, 0.5) to[out=0, in=120] (1, 0.15) to[out=60, in=180] (2, 0.5);
    \draw[thick] (0.7, -1) to[out=90, in=270] (0.7, 0) to[out=90, in=180] (1, 0.15) to[out=120, in=270] (0.7, 1);
    \node[draw, circle, inner sep=0pt, minimum size=3pt, fill=white] at (1, 0.15) {};
\end{tikzpicture}
+ q^{-2}
\begin{tikzpicture}[baseline=-1]
    \MarkedTorusBackground
    \draw[thick] (0, 0.5) to[out=0, in=270] (0.4, 1);
    \draw[thick] (0.4, -1) to[out=90, in=270] (1.4, -0.1) to[out=90, in=0] (1, 0.15) to[out=45, in=180] (2, 0.5);
    \draw[thick] (0, -1) to[out=45, in=180] (1, 0.15) to[out=60, in=225] (2, 1);
    \draw[line width=3mm, gray!40] (0.7, -0.8) -- (0.7, -0.5);
    \draw[thick] (0.7, -1) to[out=90, in=270] (0.7, -0.3) to[out=90, in=180] (1, 0.15) to[out=120, in=270] (0.7, 1);
    \draw[thick, black, fill=white] (1, 0) circle (0.15);
    \node[draw, circle, inner sep=0pt, minimum size=3pt, fill=white] at (1, 0.15) {};
\end{tikzpicture}
\right)$$
$$= q^{-3/2}
\left( q
\begin{tikzpicture}[baseline=-1]
    \MarkedTorusBackground
    \draw[thick] (1.3, -1) to[out=90, in=270] (1.3, 0) to[out=90, in=0] (1, 0.15) to[out=60, in=270] (1.3, 1);
    \draw[thick] (0.7, -1) to[out=90, in=270] (0.7, 0) to[out=90, in=180] (1, 0.15) to[out=120, in=270] (0.7, 1);
    \draw[thick] (0.4, -1) to[out=90, in=270] (0.4, 0) to[out=90, in=180] (1, 0.15) to[out=120, in=270] (0.4, 1);
    \node[draw, circle, inner sep=0pt, minimum size=3pt, fill=white] at (1, 0.15) {};
\end{tikzpicture}
+
\begin{tikzpicture}[baseline=-1]
    \MarkedTorusBackground
    \draw[thick] (0, 0.15) -- (1, 0.15) -- (2, 0.15);
    \draw[thick] (0, 0.5) to[out=0, in=120] (1, 0.15) to[out=60, in=180] (2, 0.5);
    \draw[thick] (0.7, -1) to[out=90, in=270] (0.7, 0) to[out=90, in=180] (1, 0.15) to[out=120, in=270] (0.7, 1);
    \node[draw, circle, inner sep=0pt, minimum size=3pt, fill=white] at (1, 0.15) {};
\end{tikzpicture}
+ q^{-1}
\begin{tikzpicture}[baseline=-1]
    \MarkedTorusBackground
    \draw[thick] (0, 0.15) -- (1, 0.15) -- (2, 0.15);
    \draw[thick] (0.4, -1) -- (0.4, -0.2) to[out=90, in=180] (1, 0.15) to[out=45, in=180] (2, 0.7);
    \draw[thick] (0, 0.7) to[out=0, in=270] (0.4, 1);
    \draw[thick] (1, 0.15) to[out=0, in=90] (1.3, -0.2) to[out=270, in=0] (1, -0.5) to[out=180, in=270] (0.7, -0.2) to[out=90, in=180] (1, 0.15);
    \node[draw, circle, inner sep=0pt, minimum size=3pt, fill=white] at (1, 0.15) {};
\end{tikzpicture}
+q^{-3}
\begin{tikzpicture}[baseline=-1]
    \MarkedTorusBackground
    \draw[thick] (1.3, -1) to[out=90, in=270] (1.3, 0) to[out=90, in=0] (1, 0.15) to[out=60, in=270] (1.3, 1);
    \draw[thick] (0, -1) to[out=45, in=180] (1, 0.15) to[out=60, in=225] (2, 1);
    \draw[thick] (0.7, -1) -- (0.7, -0.2) to[out=90, in=180] (1, 0.15) to[out=45, in=180] (2, 0.5);
    \draw[thick] (0, 0.5) to[out=0, in=270] (0.7, 1);
    \node[draw, circle, inner sep=0pt, minimum size=3pt, fill=white] at (1, 0.15) {};
\end{tikzpicture}
\right)$$
$$= q^{-3/2}
\left( q
\begin{tikzpicture}[baseline=-1]
    \MarkedTorusBackground
    \draw[thick] (1.3, -1) to[out=90, in=270] (1.3, 0) to[out=90, in=0] (1, 0.15) to[out=60, in=270] (1.3, 1);
    \draw[thick] (0.7, -1) to[out=90, in=270] (0.7, 0) to[out=90, in=180] (1, 0.15) to[out=120, in=270] (0.7, 1);
    \draw[thick] (0.4, -1) to[out=90, in=270] (0.4, 0) to[out=90, in=180] (1, 0.15) to[out=120, in=270] (0.4, 1);
    \node[draw, circle, inner sep=0pt, minimum size=3pt, fill=white] at (1, 0.15) {};
\end{tikzpicture}
+ q
\begin{tikzpicture}[baseline=-1]
    \MarkedTorusBackground
    \draw[thick] (0, 0.15) -- (1, 0.15) -- (2, 0.15);
    \draw[thick] (0, 0.5) to[out=0, in=120] (1, 0.15) to[out=60, in=180] (2, 0.5);
    \draw[thick] (0.7, -1) to[out=90, in=270] (0.7, 0) to[out=90, in=180] (1, 0.15) to[out=120, in=270] (0.7, 1);
    \node[draw, circle, inner sep=0pt, minimum size=3pt, fill=white] at (1, 0.15) {};
\end{tikzpicture}
+ q^{-4}
\begin{tikzpicture}[baseline=-1]
    \MarkedTorusBackground
    \draw[thick] (0, 0.15) -- (1, 0.15) -- (2, 0.15);
    \draw[thick] (0.4, -1) -- (0.4, -0.2) to[out=90, in=180] (1, 0.15) to[out=45, in=180] (2, 0.7);
    \draw[thick] (0, 0.7) to[out=0, in=270] (0.4, 1);
    \draw[thick] (1, 0.15) to[out=0, in=90] (1.3, -0.2) to[out=270, in=0] (1, -0.5) to[out=180, in=270] (0.7, -0.2) to[out=90, in=180] (1, 0.15);
    \node[draw, circle, inner sep=0pt, minimum size=3pt, fill=white] at (1, 0.15) {};
\end{tikzpicture}
+q^{-3}
\begin{tikzpicture}[baseline=-1]
    \MarkedTorusBackground
    \draw[thick] (1.3, -1) to[out=90, in=270] (1.3, 0) to[out=90, in=0] (1, 0.15) to[out=60, in=270] (1.3, 1);
    \draw[thick] (0, -1) to[out=45, in=180] (1, 0.15) to[out=60, in=225] (2, 1);
    \draw[thick] (0.7, -1) -- (0.7, -0.2) to[out=90, in=180] (1, 0.15) to[out=45, in=180] (2, 0.5);
    \draw[thick] (0, 0.5) to[out=0, in=270] (0.7, 1);
    \node[draw, circle, inner sep=0pt, minimum size=3pt, fill=white] at (1, 0.15) {};
\end{tikzpicture}
\right)$$
where the last equality comes from the height exchange relation, $(R_6)$.
Thus,
$$\varphi_{\mathcal{E}}(y_1)
= (x_2 x_4 x_3)^{-1} (x_2 x_4 x_3) \varphi_{\mathcal{E}}(y_1)
= (x_2 x_4 x_3)^{-1} \varphi_{\mathcal{E}}(\psi_{\mathcal{E}}(x_2 x_4 x_3)y_1)$$
$$= qx_3^{-1}x_4^{-1}x_2^{-1}x_2x_4x_2 + qx_3^{-1}x_4^{-1}x_2^{-1} x_2x_1^{2} + q^{-4}x_3^{-1}x_4^{-1}x_2^{-1} x_1x_3x_5 + q^{-3}x_3^{-1}x_4^{-1}x_2^{-1} x_3^{2}x_4$$
$$= q^{-1}x_2x_3^{-1} + q^{-1}x_2^{-1}x_3 + qx_1^{2}x_3^{-1}x_4^{-1} + q^{2}x_1x_2^{-1}x_4^{-1}x_5.$$
One can similarly calculate where $y_2$, $y_3$, and the boundary curve get mapped to under $\varphi_{\mathcal{E}}$ (see \ref{appendix:ImageOfT6} in the appendix).
Hence $\varphi_{\mathcal{E}}$ maps these diagrams as follows.
\begin{prop}
    Suppose $\varphi_{\mathcal{E}} : \mathscr{S}(T^2 \setminus D^2) \hookrightarrow \mathbb{T}^{n}$ is the algebra homomorphism defined as above. Let $y_1, y_2, y_3 \in \ms{S}(T^2 \setminus D^2)$ be the meridian, longitude, and $(1,1)$-curve respectively and $\partial$ the boundary curve
    $\begin{tikzpicture}[scale=0.7, baseline=-3]
        \MarkedTorusBackground
        \draw[thick] (1, 0) circle (0.4);
        \node[draw, circle, inner sep=0pt, minimum size=2pt, fill=white] at (1, 0.15) {};
    \end{tikzpicture}$. Then
    \begin{align*}
        y_1 &\mapsto q^{-1}x_2x_3^{-1} + q^{-1}x_2^{-1}x_3 + qx_1^{2}x_3^{-1}x_4^{-1} + q^{2}x_1x_2^{-1}x_4^{-1}x_5\\
        y_2 &\mapsto qx_1x_3^{-1} + qx_1^{-1}x_3 + q^{-1}x_1^{-1}x_2x_3^{-1}x_4\\
        y_3 &\mapsto q^{-1}x_1x_4^{-1} + q^{-1}x_1^{-1}x_4 + q^{-1}x_1^{-1}x_2^{-1}x_3^{2} + x_2^{-1}x_3x_4^{-1}x_5 \\
        \partial &\mapsto q^{-2}x_2^{-1}x_4 + q^{-2}x_2x_4^{-1} + qx_1^{-1}x_3x_4^{-1}x_5 + qx_1x_3^{-1}x_4^{-1}x_5 + q^{-3}x_1^{-1}x_3^{-1}x_4x_5 \\
        &\phantom{\mapsto} + q^{3}x_1x_2^{-1}x_3^{-1}x_5 + q^{-1}x_1^{-1}x_2x_3^{-1}x_5 + q^{-1}x_1^{-1}x_2^{-1}x_3x_5 + q^{2}x_2^{-1}x_4^{-1}x_5^{2}
    \end{align*}
    In particular, the images of $y_i$ satisfy the commutation relations
    $$\left(q^2 - q^{-2}\right)^{-1}[\varphi_{\mathcal{E}}(y_i), \varphi_{\mathcal{E}}(y_{i+1})]_q = \varphi_{\mathcal{E}}(y_{i+2})$$
    for $i\mod 3$ and
    $$\varphi_{\mathcal{E}}(\partial) = q \varphi_{\mathcal{E}}(y_1) \varphi_{\mathcal{E}}(y_2) \varphi_{\mathcal{E}}(y_3) - q^2 \varphi_{\mathcal{E}}(y_1)^2 - q^{-2} \varphi_{\mathcal{E}}(y_2)^2 - q^2 \varphi_{\mathcal{E}}(y_3)^2 + q^2 + q^{-2}.$$
\end{prop}

\section{A Module of Laurent Polynomials}

\begin{prop}
    Let $\mathbb{T}^n(Q)$ be the quantum torus of $n$ generators, $\frac{\mathbb{K} \langle x_1^{\pm 1}, \cdots, x_n^{\pm 1} \rangle }{x_i x_j = q^{Q_{i,j}} x_j x_i}$, where $Q$ is the corresponding skew-symmetric integral matrix.
    If $k$ is the number of non-central generators of $\mathbb{T}^n$ and $q^{\frac{Q_{i,j}}{2}} \in \mathbb{K}$ for all $i,j$, then the commutative ring $\mathbb{K}[y_1^{\pm 1}, \cdots, y_{k-1}^{\pm 1}]$ has a well defined $\mathbb{T}^n$-module structure.
    In particular, if the first $k$ generators, $\{ x_1, \cdots, x_k \}$, are our non-commutative generators, then for each $i \in \{1, \cdots, k-1\}$ and $m \in \{ k+1, \cdots, n \}$, the operators
    $$x_i \cdot f(y_1, y_2, \cdots, y_{k-1}) := y_i f(q^{Q_{i, 1}/2}y_1, q^{Q_{i, 2}/2}y_2, \cdots, q^{Q_{i, {k-1}}/2}y_{k-1})$$
    $$x_{k} \cdot f(y_1, y_2, \cdots, y_{k-1}) := f(q^{Q_{k, 1}}y_1, q^{Q_{k, 2}}y_2, \cdots, q^{Q_{k, {k-1}}}y_{k-1})$$
    $$x_m \cdot f(y_1, y_2, \cdots, y_{k-1}) := f(y_1, y_2, \cdots, y_{k-1}).$$
    define a $\mathbb{T}^n$-module.
\end{prop}
\begin{proof}
    Our relations $x_i x_j = q^{Q_{i,j}} x_j x_i$ hold as for all $i,j \in \{ 1, \cdots, k-1 \}$
    \begin{align*}
        x_i x_j \cdot f(y_1, \cdots, y_{k-1}) &= q^{\frac{Q_{i, j}}{2}} y_i y_j f(q^{\frac{Q_{j, 1} + Q_{i, 1}}{2}}y_1, \cdots, q^{\frac{Q_{j, k-1} + Q_{i, k-1}}{2}}y_{k-1})\\
        &= q^{Q_{i, j}} x_j x_i \cdot f(y_1, \cdots, y_{k-1})\\
        x_i x_k \cdot f(y_1, \cdots, y_{k-1}) &= y_i f(q^{\frac{Q_{i, 1}}{2} + Q_{k, 1}}y_1, \cdots, q^{\frac{Q_{i, k-1}}{2} + Q_{k, {k-1}}}y_{k-1})\\
        &= q^{Q_{i, k}} x_k x_i \cdot f(y_1, \cdots, y_{k-1}).
    \end{align*}
    Clearly, for any $m, m^{\prime} \in \{ k+1, \cdots, n \}$ we have $x_i x_m \cdot f(y_1, \cdots, y_{k-1}) = x_m x_i \cdot f(y_1, \cdots, y_{k-1})$ and $x_m x_{m^{\prime}} \cdot f(y_1, \cdots, y_{k-1}) = x_m x_{m^{\prime}} \cdot f(y_1, \cdots, y_{k-1})$.
\end{proof}

By this proposition, the ring of Laurent polynomials in $4$ variables is a module over our quantum $6$-torus, $\mathbb{T}^6$. Since $\mathscr{S}\left( T^2 \setminus D^2 \right) \hookrightarrow \mathbb{T}^6$, the following actions endow $\mathbb{C}[x^{\pm 1}, y^{\pm 1}, z^{\pm 1}, w^{\pm 1}]$ with the structure of a $\mathscr{S}\left( T^2 \setminus D^2 \right)$-module.
\begin{alignat*}{2}
    x_1 \cdot f(x,y,z,w) &= x f(x,qy,qz,q^{-1}w) &\qquad x_2 \cdot f(x,y,z,w) &= y f(q^{-1}x,y,q^{-1}z,q^{-2}w)\\
    x_3 \cdot f(x,y,z,w) &= z f(q^{-1}x,qy,z,q^{-1}w)  &\qquad x_4 \cdot f(x,y,z,w) &= w f(qx,q^{2}y,qz,w)\\
    x_5 \cdot f(x,y,z,w) &= f(x,y,z,w)  &\qquad x_6 \cdot f(x,y,z,w) &= f(q^{4}x,q^{4}y,q^{4}z,q^{4}w)
\end{alignat*}
Thus, the actions of $y_1$, $y_2$, $y_3$, and $\partial$ on this module are (and have been verified via Python as detailed in \ref{appendix:T6Operators})

\begin{align*}
    y_1 \cdot f(x,y,z,w) &= yz^{-1} f(x, q^{-1}y, q^{-1}z, q^{-1}w) + y^{-1}z f(x, qy, qz, qw)\\
    &\quad + x^2z^{-1}w^{-1} f(x, q^{-1}y, qz, q^{-1}w) + xy^{-1}w^{-1} f(x, q^{-1}y, qz, qw)\\
    y_2 \cdot f(x,y,z,w) &= \resizebox{0.95\width}{!}{$xz^{-1} f(qx, y, qz, w) + x^{-1}z f(q^{-1}x, y, q^{-1}z, w) + x^{-1}yz^{-1}w f(qx, y, q^{-1}z, w)$}\\
    y_3 \cdot f(x,y,z,w) &= xw^{-1} f(q^{-1}x, q^{-1}y, z, q^{-1}w) + x^{-1}w f(qx, qy, z, qw)\\
    &\quad + x^{-1}y^{-1}z^2 f(q^{-1}x, qy, z, qw) + y^{-1}zw^{-1} f(q^{-1}x, q^{-1}y, z, qw).\\
    \dee \cdot f(x,y,z,w) &= \frac{z f{\left(\frac{x}{q^{2}},\frac{y}{q^{2}},\frac{z}{q^{2}},w \right)}}{x w} + \frac{z f{\left(x,y, z, q^{2}w \right)}}{x y} + \frac{w f{\left(q^{2} x,q^{2} y,q^{2} z,q^{2} w \right)}}{y} \\
    &\quad + \frac{w f{\left(q^{2} x,y, z, q^{2}w \right)}}{z x} + \frac{y f{\left(\frac{x}{q^{2}},\frac{y}{q^{2}},\frac{z}{q^{2}},\frac{w}{q^{2}} \right)}}{w} + \frac{y f{\left(x,\frac{y}{q^{2}},\frac{z}{q^{2}},w \right)}}{z x} \\
    &\quad + \frac{x f{\left(x,\frac{y}{q^{2}},z,w \right)}}{w z} + \frac{x f{\left(q^{2} x,y,q^{2} z,q^{2} w \right)}}{y z} + \frac{f{\left(x,\frac{y}{q^{2}}, z, q^{2}w \right)}}{y w}
\end{align*}
where $\dee = q y_1 y_2 y_3 - q^2 y_1^2 - q^{-2}y_2^2 - q^2 y_3^2 + q^2 + q^{-2}$ is the boundary curve.

Unfortunately, $\partial$ does not have any eigenvalues due to grade shifts, however, it does have an invariant subspace, $\mathbb{C}\left[ \left(\frac{x}{yz}\right)^{\pm 1}, \left(\frac{y}{zx}\right)^{\pm 1}, \left(\frac{z}{xy}\right)^{\pm 1}, \left(\frac{w}{y}\right)^{\pm 1} \right]$.
To condense the notation a bit, let $\kappa = k_1 + k_2 + k_3 + k_4$ and $\mathbf{k} = (k_1,k_2,k_3,k_4)$.
Note that the following computation was done by hand and verified via python (once again see \ref{appendix:T6Operators}).

$$
\resizebox{0.93\width}{!}{$\partial \cdot \sum_{\substack{\kappa=-n \\ |k_i| \leq n }}^n c_{\mathbf{k}} \left(\frac{x}{yz}\right)^{k_1} \left(\frac{y}{zx}\right)^{k_2} \left(\frac{z}{xy}\right)^{k_3} \left(\frac{w}{y}\right)^{k_4}$}
$$
$$
= \resizebox{0.93\width}{!}{$\sum_{\substack{\kappa=-n \\ |k_i| \leq n }}^n c_{\mathbf{k}} \left[
q^{2\kappa}\left(\frac{x}{yz}\right)^{k_1} \left(\frac{y}{zx}\right)^{k_2} \left(\frac{z}{xy}\right)^{k_3+1} \left(\frac{w}{y}\right)^{k_4-1}
+ q^{2k_4}\left(\frac{x}{yz}\right)^{k_1} \left(\frac{y}{zx}\right)^{k_2} \left(\frac{z}{xy}\right)^{k_3+1} \left(\frac{w}{y}\right)^{k_4}\right.$}
$$
$$
\resizebox{0.93\width}{!}{$+ q^{2(\kappa - k_4)}\left(\frac{x}{yz}\right)^{k_1} \left(\frac{y}{zx}\right)^{k_2} \left(\frac{z}{xy}\right)^{k_3} \left(\frac{w}{y}\right)^{k_4-1}
+ q^{-2(\kappa - k_4)} \left(\frac{x}{yz}\right)^{k_1} \left(\frac{y}{zx}\right)^{k_2} \left(\frac{z}{xy}\right)^{k_3} \left(\frac{w}{y}\right)^{k_4+1}$}
$$
$$
\resizebox{0.93\width}{!}{$+ q^{2(\kappa - 2k_2 + k_4)} \left(\frac{x}{yz}\right)^{k_1+1} \left(\frac{y}{zx}\right)^{k_2} \left(\frac{z}{xy}\right)^{k_3+1} \left(\frac{w}{y}\right)^{k_4-1}
+ q^{2(\kappa - 2k_2)} \left(\frac{x}{yz}\right)^{k_1+1} \left(\frac{y}{zx}\right)^{k_2} \left(\frac{z}{xy}\right)^{k_3} \left(\frac{w}{y}\right)^{k_4-1}$}
$$
$$
\resizebox{0.93\width}{!}{$+ q^{2(k_4-2k_2)} \left(\frac{x}{yz}\right)^{k_1+1} \left(\frac{y}{zx}\right)^{k_2} \left(\frac{z}{xy}\right)^{k_3} \left(\frac{w}{y}\right)^{k_4}
+ q^{2(2k_1 + k_4)} \left(\frac{x}{yz}\right)^{k_1} \left(\frac{y}{zx}\right)^{k_2+1} \left(\frac{z}{xy}\right)^{k_3} \left(\frac{w}{y}\right)^{k_4}$}
$$
$$
\resizebox{0.93\width}{!}{$\left. + q^{2(k_1 - k_2 - k_3 + k_4)} \left(\frac{x}{yz}\right)^{k_1} \left(\frac{y}{zx}\right)^{k_2+1} \left(\frac{z}{xy}\right)^{k_3} \left(\frac{w}{y}\right)^{k_4+1} \right]
= \sum_{\substack{\kappa=-(n+1) \\ |k_i| \leq n+1 }}^{n+2} c_{\mathbf{k}}^\prime \left(\frac{x}{yz}\right)^{k_1} \left(\frac{y}{zx}\right)^{k_2} \left(\frac{z}{xy}\right)^{k_3} \left(\frac{w}{y}\right)^{k_4}$}$$

Furthermore, our operators $y_1$, $y_2$, and $y_3$ also have invariant subspaces. For example, $\mathbb{C}\left[ \left( \frac{x}{z} \right)^{\pm 1}, \left( \frac{yw}{xz}\right)^{\pm 1} \right]$ is an invariant subspace with respect to the action of $y_2$ on this module. To see this, we compute the following.
$$y_2 \cdot \sum_{\substack{k_1+k_2=-n \\ |k_i| \leq n}}^n c_{k_1, k_2} \left(\frac{x}{z}\right)^{k_1}\left(\frac{y w}{x z}\right)^{k_2}$$
$$= \sum_{\substack{k_1+k_2=-n \\ |k_i| \leq n}}^n c_{k_1, k_2} \left[ q^{2 k_2}\left(\frac{x}{z}\right)^{k_1-1}\left(\frac{y w}{x z}\right)^{k_2} + q^{-2 k_2} \left(\frac{x}{z}\right)^{k_1+1}\left(\frac{y w}{x z}\right)^{k_2}  + q^{2 k_1} \left(\frac{x}{z}\right)^{k_1}\left(\frac{y w}{x z}\right)^{k_2+1} \right]$$
$$= \sum_{\substack{k_1+k_2=-(n+1) \\ |k_i| \leq n+1}}^{n+1} c_{k_1, k_2}^{\prime} \left(\frac{x}{z}\right)^{k_1}\left(\frac{y w}{x z}\right)^{k_2}$$

It seems as though the algebra $\mathbb{C}\left[x^{\pm 1}, y^{\pm 1}, z^{\pm 1}, w^{\pm 1}\right]$ constitutes a minimal structure for this type of representation while still allowing for well-defined, nontrivial actions of the generators of $\mathbb{T}^6(Q)$. While it is straightforward to introduce additional variables into this algebra, it would be surprising if the number of variables could be reduced. For this reason, we give the following conjecture.

\begin{conj}
    Let $Q$ be the anti-symmetric matrix defined in section \ref{section:Quantum6Torus} and suppose $M$ is a Laurent polynomial algebra with a $\mathbb{T}^{6}(Q)$-module structure such that each $x_i$ acts nontrivially. Then the Gelfend-Kirillov dimension of $M$ is at least $4$.
\end{conj}

\chapter{Topologically Defined Representations}

\section{Knot Invariants}
There is another, more geometric, perspective to the story of skein modules that begins by looking for knot invariants.
Let $M$ be a $3$-manifold with boundary and consider the natural inclusion $\partial M \hookrightarrow M$.
This inclusion induces a map on fundamental groups, denoted by $\alpha: \pi_1(\partial M) \to \pi_1(M)$, known as the \textit{peripheral map}.
In \cite{Waldhausen_1968}, Waldhausen demonstrated that the peripheral map serves as a complete invariant for a significant class of 3-manifolds, namely \textit{sufficiently large} $3$-manifolds.
In particular, these kinds of manifolds include knot complements.
This result was further refined by Gordon and Luecke in \cite{MR0972070} to establish it as a knot invariant as well.

\newpage
Given a knot $K$, the \textit{knot complement} is defined as $M_K = S^3 \setminus N(K)$, where $N(K)$ is a tubular neighborhood of $K$.
Since the boundary of $M_K$ is a torus, the corresponding peripheral map is $\alpha: \pi_1(T^2) \to \pi_1(M_K)$.
Notably, $\pi_1(M_k)$ alone is not a complete knot invariant, as it cannot distinguish between mirrored images of chiral knots.
However, the additional information regarding how the torus embeds into the knot complement allows for this distinction.

We can study these groups by instead examining the induced maps on their representations.
For any reductive group, $G$, the set of all $G$-representations, denoted $\operatorname{Rep}(M,G) := \operatorname{Hom}(\pi_1(M), G)$, inherits a natural variety structure from the algebraic variety structure of $G$.
Moreover, the induced map $\operatorname{Rep}(M, G) \to \operatorname{Rep}(\partial M, G)$ descends to a map on $G$-character varieties, $\chi(M, G) \to \chi(\partial M,G)$.
Here, $\chi(M,G)$ is defined as $\operatorname{Rep}(M, G) \sslash G$, the closed algebraic quotient of the natural conjugation action of $G$ on the representation variety.
Finally, we can instead induce the map onto the coordinate rings of these character varieties: $\mathcal{O} \left( \chi(\partial M, G) \right) \to \mathcal{O} \left( \chi(M, G) \right)$.
Equivalently, we can also consider $\mathcal{O}\left( \operatorname{Rep}(\partial M, G) \right)^G \to \mathcal{O}\left( \operatorname{Rep}(M, G) \right)^G$, the corresponding map between the algebras of $G$-invariant regular functions on $\operatorname{Rep}(\partial M,G)$ and $\operatorname{Rep}(M,G)$.
Below is a more visual depiction of the process just described.
\begin{center}
    \begin{tikzcd}
        \pi_1(\partial M) \arrow[r, "\alpha"] & \pi_1(M)\\
        {\text{Rep}(\partial M,G)} \arrow[d, two heads] & {\text{Rep}(M,G)} \arrow[l, "\alpha^*"'] \arrow[d, two heads]\\
        {\chi(\partial M,G)} & {\chi(M, G)} \arrow[l, "\hat{\alpha}"']\\
        {\mathcal{O}\left(\chi(\partial M, G)\right)} \arrow[r, "\alpha_*"] & {\mathcal{O}\left(\chi(M, G)\right)}
    \end{tikzcd}
\end{center}

\newpage
So we have taken a map between not necessarily abelian groups and turned it into a map between commutative algebras, which offers several advantages.
At this point, it is natural to ask whether or not we still have the same precision of knot invariants at this level.
Initially the answer is no. However, it turns out that if we quantize these coordinate rings, then it is possible to extract this data.

Before we do this, let's take a moment to figure out what exactly is $\chi(\partial M, G)$.
As we're only considering cases when $\partial M$ is the two torus, $T^2 = S^1 \times S^1$, for any reductive group, $G$, $\chi(\partial M, G) = \operatorname{Hom}(\mathbb{Z}^2, G) \sslash G$.
Let $\mathcal{T} \subset G$ be a maximal toral subgroup of $G$, and let $W = N(\mathcal{T})/\mathcal{T}$ be the corresponding Weyl group, where $N(\mathcal{T})$ is the normalizer of $\mathcal{T}$ in $G$.
Clearly, $\operatorname{Rep}(T^2, \mathcal{T}) \cong \mathcal{T} \times \mathcal{T}$ since the fundamental group of the torus is $\mathbb{Z}^2$.
As $\mathcal{T}$ is a subgroup of $G$, there is a natural inclusion of these representation varieties, $\operatorname{Rep}(T^2, \mathcal{T}) \hookrightarrow \operatorname{Rep}(T^2, G)$.

We can surject $\operatorname{Rep}(T^2, G)$ onto its character variety and do the same for $\mathcal{T} \times \mathcal{T}$ by factoring it through the diagonal action of $W$.
It's fairly straightforward to see that this induces a map, $\Psi$, on these character varieties.
\begin{center}
    \begin{tikzcd}
        \text{Rep}(T^2, \mathcal{T}) \arrow[r, equal] & \mathcal{T} \times \mathcal{T} \arrow[r, hook] \arrow[d, two heads] & {\text{Rep}(T^2, G)} \arrow[d, two heads] \\
        & \mathcal{T} \times \mathcal{T} \sslash W \arrow[r, "\Psi"] & {\chi(T^2, G)} \\
    \end{tikzcd}
\end{center}

It's natural to expect and hope that $\Psi$ is an isomorphism; however, this is not true in general.
Nonetheless, a chain of results supports us in our specific situation.
Thaddeus proved in \cite{MR1862614} that when $G$ is a connected complex reductive algebraic group, the quotient $\mathcal{T} \times \mathcal{T} \sslash W$ is bijective with the connected component containing the trivial character, denoted $\chi_0(T^2, G)$.
Moreover, Richardson's work in \cite{MR535074} implies that if additionally $G$ is simply-connected, then the character variety $\chi(T^2, G)$ is irreducible as a moduli space, and thus $\chi(T^2, G) = \chi_0(T^2, G)$.
However, similar to normalization maps for cuspidal curves, the bijectivity of these spaces does not necessarily imply that they are isomorphic as varieties.

Nevertheless, combining these results, Sikora established in \cite{MR3205770} that whenever $G$ is a classical group, i.e. $G = GL_n(\mathbb{C})$, $SL_n(\mathbb{C})$, $SP_n(\mathbb{C})$, or $SO_n(\mathbb{C})$, $\Psi$ is indeed an isomorphism of varieties and hence the induced map on coordinate rings,
$$\Psi_* : \mathcal{O}\left(\mathcal{T} \times \mathcal{T}\right)^W \longrightarrow \mathcal{O}\left(\chi(T^2, G)\right)$$
is an isomorphism of commutative algebras.
As $\mathcal{O}\left(\mathcal{T} \times \mathcal{T}\right)^W$ has been fairly heavily studied in representation theory, it is beneficial to instead consider the map $\alpha_* : \mathcal{O}\left(\mathcal{T} \times \mathcal{T}\right)^W \to \mathcal{O}\left(\chi(M, G)\right)$.
Moreover, it has interesting non-commutative deformations: spherical DAHAs.

Since we're only working with $G = SL_2(\mathbb{C})$, the above results provide us with a more concrete way of understanding the character variety $\chi\left(T^2, SL_2(\mathbb{C})\right)$.
In particular, when $G = SL_2(\mathbb{C})$, a quick calculation shows the following.
\begin{align*}
    \mathcal{T} &= \left\{ \begin{pmatrix} a & 0\\ 0 & a^{-1}\\ \end{pmatrix} : a \in \mathbb{C}^\times \right\}\\
    N(\mathcal{T}) &= \left\{ \begin{pmatrix} a & 0\\ 0 & a^{-1}\\ \end{pmatrix}, \begin{pmatrix} 0 & -b\\ b & 0\\ \end{pmatrix} : a,b \in \mathbb{C}^\times \right\}
\end{align*}
and therefore
$$W = N(\mathcal{T}) / \mathcal{T} = \left\{ \begin{pmatrix} 1 & 0\\ 0 & 1\\ \end{pmatrix} \mathcal{T}, \begin{pmatrix} 0 & -1\\ 1 & 0\\ \end{pmatrix} \mathcal{T} \right\} \cong \mathbb{Z}/2\mathbb{Z}.$$

Since $\pi_1(T^2)$ is abelian and has two generators, we obtain the following isomorphism.
$$\mathcal{O}\left(\operatorname{Rep}(T^2, SL_2(\mathbb{C}))\right) \cong \frac{\mathbb{C}\langle X^{\pm 1}, Y^{\pm 1} \rangle}{XY - YX}$$
There is a natural $\mathbb{Z}/2\mathbb{Z}$-action, which corresponds to simultaneously inverting $X$ and $Y$. Consequently, the ring $\mathcal{O}\left( \chi(T^2, SL_2(\mathbb{C})) \right)$ is subalgebra of elements invariant under this involution.

Returning to the topic of knot invariants, we mentioned that we can recover this information through the quantization of these coordinate rings. Frohman and Gelca defined a particular quantization of $\mathcal{O}\left(\operatorname{Rep}(T^2, SL_2(\mathbb{C}))\right)$ in \cite{MR1675190}, which they call the noncommutative torus.
$$A_q := \mathbb{T}^2 \left( \begin{bmatrix} 0 & 2\\ -2 & 0\\ \end{bmatrix} \right) = \frac{\mathbb{C}\langle X^{\pm 1}, Y^{\pm 1} \rangle}{XY - q^2 YX}$$
We can apply the same $\mathbb{Z}/2\mathbb{Z}$-action on $A_q$ and look at the corresponding invariant subalgebra, denoted $A_q^{\mathbb{Z}/2\mathbb{Z}}$. They then showed that both $K_q(T^2)$ and $A_{q}^{\mathbb{Z}/2\mathbb{Z}}$ are isomorphic as algebras.

Recall that by Theorem \ref{theorem:KqT2Basis}, $\{ (r,s)_T \}_{r,s \in \mathbb{Z}} / \sim$ where $(r,s)_T \sim (-r, -s)_T$ forms a basis for $K_q(T^2)$.

\begin{theorem}[Theorem 4.3 in \cite{MR1675190}]
    The linear map $K_q(T^2) \to A_{q}^{\mathbb{Z}/2\mathbb{Z}}$ defined by $(r,s)_T \mapsto q^{-sr}\left( Y^{-s}X^{r} + Y^{s}X^{-r} \right)$ is an isomorphism of algebras.
\end{theorem}

Therefore, the meridian and longitude get mapped to $X + X^{-1}$ and $Y + Y^{-1}$, respectively. This theorem tells us that $K_q(T^2)$ can be understood as a quantum deformation of the ring of $SL_2(\mathbb{C})$-invariant regular functions over $\mathcal{O}\left(\operatorname{Rep}(T^2, SL_2(\mathbb{C}))\right)$.

Moreover, Przytycki and Sikora in \cite{MR1710996} (as well as Bullock in \cite{MR1600138}) showed that for any $3$-manifold, the specialization of $K_{q=-1}(M)$ is isomorphic to $\mathcal{O}\left( \chi(M, SL_2(\mathbb{C})) \right)$ as algebras. Therefore, $K_q(M)$ is a quantization of the commutative algebra $\mathcal{O}\left( \chi(M, SL_2(\mathbb{C})) \right)$. Furthermore, when $M$ is a surface, $K_q(M)$ is a deformation of \resizebox{0.94\width}{!}{$K_{q=-1}(M) \cong \mathcal{O}\left( \chi(M, SL_2(\mathbb{C})) \right)$} in the direction of Goldman's Poisson bracket.

Recall that the induced map $\alpha_{*} : \mathcal{O}\left(\chi(\partial M, G)\right) \to \mathcal{O}\left(\chi(M, G)\right)$ is a map between commutative algebras. Any time we have a map between algebras, we can make the codomain into a module over the domain using the map. It turns out that this module structure remains intact as we quantize both of these algebras.
Furthermore, given a knot, $K$, the colored Jones polynomial, $J_n(K, q)$, can be extracted from the $A_{q}^{\mathbb{Z}/2\mathbb{Z}}$-module structure of $K_q(M_K)$. The natural embedding $N(K) \sqcup M_K \hookrightarrow S^3$ induces the pairing
$$\langle-,-\rangle: K_q\left(S^1 \times D^2\right) \otimes_{K_q\left( T^2 \right)} K_q\left( M_K \right) \rightarrow K_q\left(S^3\right) \cong \mathbb{C}.$$
It's not too hard to see that $K_q(S^1 \times D^2) \cong \mathbb{C}[u]$ as vector spaces, and a theorem of Kirby and Melvin from \cite{MR1117149} shows
$$J_n(K, q) = \left\langle S_{n-1}(u), 1_K\right\rangle.$$
Here, $1_K$ denotes the empty link in $M_K$ and $S_{n}$ is the $n$th Chebyshev polynomials, but now with initial conditions $S_0(x) = 1$, $S_1(x) = x$, and $S_{n+1} = x S_{n} - S_{n-1}$.

As the $K_q(T^2)$-module structure of $K_q(M_K)$ comes from the peripheral map, we can geometrically interpret this as isotoping any curves in $K_q(M_K)$ away from a neighborhood of $\partial M_k$ and embedding $T^2 \times I$ into this neighborhood.
These $K_q(T^2)$-module structures, as well as their connection with DAHA representations, were examined by Gelca and Sain in \cite{MR1967240} and more thoroughly examined by Berest and Samuelson in \cite{MR3530443}. For example, when $K$ is the unknot, $M_k$ is the solid torus and so $K_q(M_k) \cong \mathbb{C}[u]1_K$. Let $Y_1$, $Y_2$, and $Y_3$ be the meridian, longitude, and $(1,1)$-curve respectively. Then the action of $K_q(T^2)$ on $K_q(M_K)$ is given by
\begin{align*}
    Y_1 \cdot f(u) 1_K &= u f(u) 1_K, \\
    Y_2 \cdot 1_K &= \left(-q^2-q^{-2}\right) 1_K, \\
    Y_3 \cdot 1_K &= -q^{-3} u 1_K.
\end{align*}

Although it is quite a bit more work, this module structure can been extended to the knot complement of any knot. In \cite{MR3530443}, Berest and Samuelson describe the natural action of the $A_1$ DAHA on these knot complements for the trefoil, the figure eight knot, all $(2, 2p+1)$-torus knots, and all $2$-bridge knots (when $q = \pm 1$).

\section{Knot Complements as $\mathscr{S}(T^2 \setminus D^2)$-Modules}\label{section:KnotCompMods}

Just as $K_q(M)$ can be understood as a quantization of $SL_2$-invariant regular functions on $\operatorname{Rep}(M, SL_2)$, a similar geometric approach applies to the theory of stated skein algebra, using the conventional model where $\Sigma$ is a punctured bordered surface, $\mathscr{S}^{pb}(\Sigma)$.

Let $\Sigma$ be an oriented surface with boundary, and fix a Riemannian metric on $\Sigma$. Denote $U\Sigma$ as the unit tangent bundle over $\Sigma$, and let $pr: U\Sigma \to \Sigma$ be the canonical projection. The preimage of any point in $\Sigma$ is a circle equipped with an orientation induced from that of $\Sigma$. For any boundary edge $e$, and any point $x \in e$, let $v \in T_x(\Sigma)$ be the unit tangent vector pointing in the direction of the orientation of $e$. Define $\theta_e$ to be the half-circle from $(x, v)$ to $(x, -v)$ in the positive direction of its orientation.\footnote{The choice of $x$ in $e$ does not actually make a difference here.}

Just as Kauffman bracket skein modules come from representation varieties over fundamental groups, stated skein algebras come from representation varieties over fundamental groupoids. Every immersion $\alpha : [0,1] \to \Sigma$ canonically lifts to the path $\left(\alpha(t), \frac{\alpha(t)}{\|\alpha(t)\|}\right)$ in $U\Sigma$ and thus any boundary edge, $e \subset \partial \Sigma$, canonically lifts to $\tilde{e} \subset \partial \left( U \Sigma \right)$. If
$$\widetilde{\partial \Sigma} := \bigcup_{e \subset \partial \Sigma} \tilde{e},$$
then $\pi_1(U\Sigma ; \widetilde{\partial \Sigma})$ is the fundamental groupoid corresponding to a lift of the fundamental groupoid, $\pi_1(\Sigma, \partial \Sigma)$.

\begin{definition}
    A \textit{flat twisted $\mathrm{SL}_2(\mathbb{C})$-bundle} on $\Sigma$ is a morphism $\rho: \pi_1(U \Sigma ; \widetilde{\partial \Sigma}) \rightarrow \mathrm{SL}_2(\mathbb{C})$ such that $\rho( \theta_{e} ) = \begin{pmatrix} 0 & -1 \\ 1 & 0 \end{pmatrix}$ for every boundary edge, $e$.
\end{definition}

Costantino and L\^{e} showed (Lemma 8.2 in \cite{costantino2022stated}) that the set of flat twisted $SL_2(\mathbb{C})$-bundles on $\Sigma$ forms an affine algebraic variety, similar to the Kauffman bracket case. In the following theorem, they established the classical limit of stated skein algebras.
\begin{theorem}[Theorem 8.12 in \cite{costantino2022stated}]
    When $q=1$, the stated skein algebra $\mathscr{S}(\Sigma)$ is naturally isomorphic to the coordinate ring of flat twisted $\mathrm{SL}_2(\mathbb{C})$-bundles on $\Sigma$.
\end{theorem}

We will now extend the action of $K_q(T^2)$ on knot complements to an action of $\mathscr{S}(T^2 \setminus D^2)$ on corresponding stated skein modules with one marking.
Just as before, the action of $\mathscr{S}(T^2 \setminus D^2)$ on $\mathscr{S}(M_K)$ for some knot $K$ should be dictated by the peripheral map, i.e. $\mathscr{S}(T^2 \setminus D^2)$ should somehow be ``glued'' to the boundary of $M_K$. In order to incorporate the data of our marking, we'll need to add a marking to the knot complement and identify both markings in such a way that aligns with a module structure. From now on, we will assume that $M_K$ has a single marking in its boundary. Define $\mathcal{N}_K: [0,1] \hookrightarrow \partial M_K$ to be this marking and consider $\mathscr{S}(M_K, \mathcal{N}_K)$.

In order to define this module structure properly, we will need to consider the ``conventional model'' for stated skein algebras, $\mathscr{S}^{pb}(T^2 \setminus D^2)$, where we identify $T^2 \setminus D^2$ as a punctured bordered surface (see section \ref{section:ConventionalModel} for more information).
Let $\mathcal{P} \in \partial \left( T^2 \setminus D^2 \right)$ be our ideal point and denote $\Sigma^\prime := \left( T^2 \setminus D^2 \right) \setminus \mathcal{P}$ as the corresponding punctured bordered surface with clockwise orientation, $\mathfrak{o}$, on our boundary edge, $\partial \Sigma^\prime$. Identify $\partial \Sigma^\prime$ with the map $b : (0,1) \to \Sigma^\prime$ such that the induced orientation from $(0,1)$ is compatible with $\mathfrak{o}$.
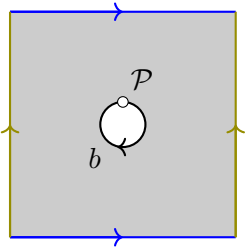
\begin{center}
    \begin{tikzpicture}[scale=1.5]
        \draw[draw=none, fill=gray!40] (0, 1) -- (2, 1) -- (2, -1) -- (0, -1) -- (1, -0.2) to[out=0, in=-90] (1.2, 0) to[out=90, in=0] (1, 0.2) to[out=180, in=90] (0.8, 0) to[out=-90, in=180] (1, -0.2) -- (0, -1) -- (0, 1);
        \draw[thick, blue, ->] (0,1) -- (1,1);
        \draw[thick, blue] (0,1) -- (2,1);
        \draw[thick, blue, ->] (0,-1) -- (1,-1);
        \draw[thick, blue] (0,-1) -- (2,-1);
        \draw[thick, olive, ->] (0,-1) -- (0,0);
        \draw[thick, olive] (0,-1) -- (0,1);
        \draw[thick, olive, ->] (2,-1) -- (2,0);
        \draw[thick, olive] (2,-1) -- (2,1);
        \draw[thick, black] (1, 0) circle (0.2);
        \path[thick, tips, ->] (2, -0.2) -- (0.95, -0.2);
        \node[draw, circle, inner sep=0pt, minimum size=4pt, fill=white] at (1, 0.2) {};
        \node at (1.17, 0.4) {$\mathcal{P}$};
        \node at (0.75, -0.3) {$b$};
    \end{tikzpicture}
\end{center}

For now, we are understanding this picture as purely topological and are not yet considering the corresponding skein algebras.
Glue $T^2 \times [0,1]$ to the boundary of $M_K$ so that $T^2 \times \{0\}$ is identified with $\partial M_K$ via the peripheral embedding, $T^2 \times \{0\} \hookrightarrow \partial M_K$. At this point, it should be the case that and $T^2 \times (0,1] \bigcap M_K = \emptyset$ and $\mathcal{N}_K$ is on the interior of $\widetilde{M}_K := M_K \bigcup \left(T^2 \times [0,1]\right)$. We can compose embeddings of our thickened punctured bordered surface
$$\iota_K : \Sigma^\prime \hookrightarrow T^2 \hookrightarrow \widetilde{M}_K$$
such that $b([\epsilon, 1-\epsilon]) \times \{0\}$ is identified with $\mathcal{N}_K$ for some small $\epsilon > 0$. We will eventually identify $b([\epsilon, 1-\epsilon]) \times \{1\}$ as the new marking.

\newpage
Let $\alpha \in \mathscr{S}^{pb}(\Sigma^\prime)$ and $D \in \mathscr{S}(M_K)$. First isotope any closed curves in $D$ away from the boundary and isotope the endpoints of any stated $\partial M_K$-tangles of $D$ down the marking to $\mathcal{N}_K\left((0, \frac{1}{2})\right)$.
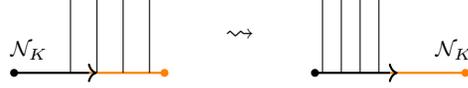
\begin{center}
    \begin{tikzpicture}
        \draw[thick] (0, 0) -- (1, 0);
        \draw[thick, orange] (1, 0) -- (2, 0);
        \node[circle, inner sep=0pt, minimum size=3pt, fill=black] at (0, 0) {};
        \node[circle, inner sep=0pt, minimum size=3pt, fill=orange] at (2, 0) {};
        \path[thick, tips, ->] (0, 0) -- (1.1, 0);
        \draw (0.75, 0) -- (0.75, 1);
        \draw (1.1, 0) -- (1.1, 1);
        \draw (1.45, 0) -- (1.45, 1);
        \draw (1.8, 0) -- (1.8, 1);
        \node at (0.2, 0.3) {$\scriptstyle{\mathcal{N}_K}$};
        \node at (3, 0.5) {$\rightsquigarrow$};
        \draw[thick] (4, 0) -- (5, 0);
        \draw[thick, orange] (5, 0) -- (6, 0);
        \node[circle, inner sep=0pt, minimum size=3pt, fill=black] at (4, 0) {};
        \node[circle, inner sep=0pt, minimum size=3pt, fill=orange] at (6, 0) {};
        \path[thick, tips, ->] (4, 0) -- (5.1, 0);
        \draw (4.1, 0) -- (4.1, 1);
        \draw (4.35, 0) -- (4.35, 1);
        \draw (4.6, 0) -- (4.6, 1);
        \draw (4.85, 0) -- (4.85, 1);
        \node at (5.85, 0.3) {$\scriptstyle{\mathcal{N}_K}$};
    \end{tikzpicture}
\end{center}
Similarly, isotope the endpoints of any stated $\partial \left(T^2 \setminus D^2\right)$-tangles in $\alpha$, remaining in generic position while doing so, up to $\mathcal{N}\left((\frac{1}{2}, 1-\epsilon)\right) \times [0,1]$.

We can extend any stated $\mathcal{N}_K$-tangles to end on $b \times \{1\}$ instead of $b \times \{0\}$. Specifically, let $(m,s)$ be a stated $\mathcal{N}_K$-tangle (using definition \ref{defn:StatedTangle}) and $e_1, e_2$ be the endpoints of $m$. In addition to lying on $\mathcal{N}_K$ in $M_K$, we can view $e_1$ and $e_2$ as lying in $b \times [0,1]$ in $\widetilde{M}_K$. Define $$\Tilde{m} = m \bigcup_{i = 1,2} \left( e_i \times [0,1] \right)$$ where $e_i \times [0,1] \subset b \times [0,1]$. Then $(\Tilde{m}, \Tilde{s})$ is a stated $\widetilde{\mathcal{N}}_K$-tangle where $\Tilde{s}(e_i \times \{1\}) = s(e_i)$ and $\widetilde{\mathcal{N}}_K := \iota_K( b([\epsilon, 1-\epsilon]) \times \{1\} )$.

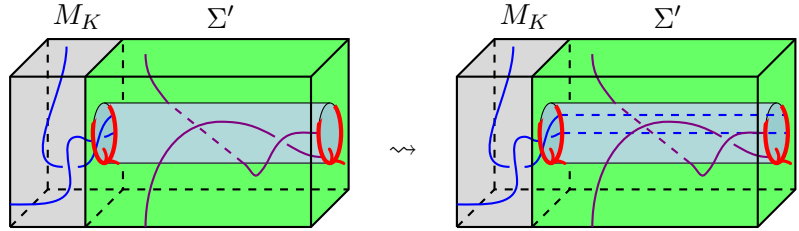
\begin{figure}[h]
    \centering
    \begin{tikzpicture}[baseline=-2]
        \path[fill=gray!30] (0, -1) -- (0, 1) -- (0.5, 1.5) -- (1.5, 1.5) -- (1, 1) -- (1, -1) -- (0, -1);
        \path[fill=green!60] (1, -1) -- (1, 1) -- (1.5, 1.5) -- (4.5, 1.5) -- (4.5, -0.5) -- (4, -1) -- (1, -1);
        \draw[thick] (0, -1) -- (0, 1) -- (0.5, 1.5) -- (4.5, 1.5) -- (4.5, -0.5) -- (4, -1) -- (0, -1);
        \draw[thick, dashed] (0, -1) -- (0.5, -0.5) -- (4.5, -0.5);
        \draw[thick, dashed] (0.5, -0.5) -- (0.5, 1.5);
        \draw[thick, dashed] (1, -1) -- (1.5, -0.5) -- (1.5, 1.5);
        \path[fill=teal!30] (1.25, 0.65) -- (4.25, 0.65) -- (4.25, -0.15) -- (1.25, -0.15) -- (1.25, 0.65);
        \draw[fill=teal!40] (1.25, 0.25) ellipse (0.15 and 0.4);
        \draw[fill=teal!40] (4.25, 0.25) ellipse (0.15 and 0.4);
        \draw[thick, blue] (1.37, 0.49) to[out=210, in=0] (0.9, -0.2);
        \draw[thick, blue] (0.7, -0.2) to[out=180, in=270] (0.75, 1.4);
        \draw[thick, blue] (1.4, 0.25) to[out=180, in=0] (1.25, 0.2);
        \draw[thick, blue] (1.1, 0.15) to[out=200, in=0] (0.9, 0.2) to[out=180, in=90] (0.8, -0.4) to[out=270, in=0] (0, -0.7);
        \draw (1.25, 0.65) -- (4.25, 0.65);
        \draw (1.25, -0.15) -- (4.25, -0.15);
        \draw[thick, violet] (4.16, -0.07) to[out=180, in=330] (3.68, 0.1);
        \draw[thick, violet] (3.52, 0.2) to[out=150, in=0] (2.8, 0.4) to[out=180, in=90] (1.8, -1);
        \draw[thick, violet] (4.1, 0.25) to[out=180, in=45] (3.6, 0.15) to[out=225, in=45] (3.3, -0.3) to[out=225, in=330] (3.1, -0.15);
        \draw[thick, dashed, violet] (3.1, -0.15) -- (2.1, 0.65);
        \draw[thick, violet] (2.1, 0.65) to[out=129, in=270] (1.8, 1.3);
        \begin{scope}
            \clip (4, 0.3) -- (4.5, 0.8) -- (4.5, -0.5) -- (4, -1) -- (4, 0.3);
            \draw[ultra thick, red] (4.25, 0.25) ellipse (0.15 and 0.4);
            \path[very thick, red, tips, ->] (4.5, 0.1) -- (4.25, -0.15);
        \end{scope}
        \begin{scope}
            \clip (1, 0.3) -- (1.5, 0.8) -- (1.5, -0.5) -- (1, -1) -- (1, 0.3);
            \draw[ultra thick, red] (1.25, 0.25) ellipse (0.15 and 0.4);
            \path[very thick, red, tips, ->] (1.5, 0.1) -- (1.25, -0.15);
        \end{scope}
        \draw[thick] (0, 1) -- (4, 1);
        \draw[thick] (1, -1) -- (1, 1) -- (1.5, 1.5);
        \draw[thick] (4, -1) -- (4, 1) -- (4.5, 1.5);
        \node at (0.9, 1.8) {$M_K$};
        \node at (2.8, 1.8) {$\Sigma^\prime$};
    \end{tikzpicture} $\quad\rightsquigarrow\quad$ 
    \begin{tikzpicture}[baseline=-2]
        \path[fill=gray!30] (0, -1) -- (0, 1) -- (0.5, 1.5) -- (1.5, 1.5) -- (1, 1) -- (1, -1) -- (0, -1);
        \path[fill=green!60] (1, -1) -- (1, 1) -- (1.5, 1.5) -- (4.5, 1.5) -- (4.5, -0.5) -- (4, -1) -- (1, -1);
        \draw[thick] (0, -1) -- (0, 1) -- (0.5, 1.5) -- (4.5, 1.5) -- (4.5, -0.5) -- (4, -1) -- (0, -1);
        \draw[thick, dashed] (0, -1) -- (0.5, -0.5) -- (4.5, -0.5);
        \draw[thick, dashed] (0.5, -0.5) -- (0.5, 1.5);
        \draw[thick, dashed] (1, -1) -- (1.5, -0.5) -- (1.5, 1.5);
        \path[fill=teal!30] (1.25, 0.65) -- (4.25, 0.65) -- (4.25, -0.15) -- (1.25, -0.15) -- (1.25, 0.65);
        \draw[fill=teal!40] (1.25, 0.25) ellipse (0.15 and 0.4);
        \draw[fill=teal!40] (4.25, 0.25) ellipse (0.15 and 0.4);
        \draw[thick, blue] (1.37, 0.49) to[out=210, in=0] (0.9, -0.2);
        \draw[thick, blue] (0.7, -0.2) to[out=180, in=270] (0.75, 1.4);
        \draw[thick, blue] (1.4, 0.25) to[out=180, in=0] (1.25, 0.2);
        \draw[thick, blue] (1.1, 0.15) to[out=200, in=0] (0.9, 0.2) to[out=180, in=90] (0.8, -0.4) to[out=270, in=0] (0, -0.7);
        \draw[thick, dashed, blue] (1.37, 0.49) -- (4.37, 0.49);
        \draw[thick, dashed, blue] (1.4, 0.25) -- (4.4, 0.25);
        \draw (1.25, 0.65) -- (4.25, 0.65);
        \draw (1.25, -0.15) -- (4.25, -0.15);
        \draw[thick, violet] (4.16, -0.07) to[out=180, in=330] (3.68, 0.1);
        \draw[thick, violet] (3.52, 0.2) to[out=150, in=0] (2.8, 0.4) to[out=180, in=90] (1.8, -1);
        \draw[thick, violet] (4.1, 0.25) to[out=180, in=45] (3.6, 0.15) to[out=225, in=45] (3.3, -0.3) to[out=225, in=330] (3.1, -0.15);
        \draw[thick, dashed, violet] (3.1, -0.15) -- (2.1, 0.65);
        \draw[thick, violet] (2.1, 0.65) to[out=129, in=270] (1.8, 1.3);
        \begin{scope}
            \clip (4, 0.3) -- (4.5, 0.8) -- (4.5, -0.5) -- (4, -1) -- (4, 0.3);
            \draw[ultra thick, red] (4.25, 0.25) ellipse (0.15 and 0.4);
            \path[very thick, red, tips, ->] (4.5, 0.1) -- (4.25, -0.15);
        \end{scope}
        \begin{scope}
            \clip (1, 0.3) -- (1.5, 0.8) -- (1.5, -0.5) -- (1, -1) -- (1, 0.3);
            \draw[ultra thick, red] (1.25, 0.25) ellipse (0.15 and 0.4);
            \path[very thick, red, tips, ->] (1.5, 0.1) -- (1.25, -0.15);
        \end{scope}
        \draw[thick] (0, 1) -- (4, 1);
        \draw[thick] (1, -1) -- (1, 1) -- (1.5, 1.5);
        \draw[thick] (4, -1) -- (4, 1) -- (4.5, 1.5);
        \node at (0.9, 1.8) {$M_K$};
        \node at (2.8, 1.8) {$\Sigma^\prime$};
    \end{tikzpicture}
    \caption[Extending stated endpoints in a knot complement]{Extending the endpoints of tangles (blue) in $M_K$ to end on the new boundary marking (blue and dashed).}
    \label{fig:ExtComplementTangles}
\end{figure}

In other words, we can identify the stated endpoints of $m$ on $\mathcal{N}_K$ as living on $b \times \{0\}$ and then ``pull'' these endpoints along the $[0,1]$ component of $\partial \Sigma^\prime \times [0,1]$ so that the endpoints are now on $b \times \{1\}$.

Define $\alpha \cdot D$ as the stated $\widetilde{\mathcal{N}}_K$-tangle diagram in $\mathscr{S}(\widetilde{M}_K, \widetilde{\mathcal{N}}_K)$ consisting of the union of $\Tilde{m}$ for each stated $\mathcal{N}_K$-tangle $m$ in $D$, along with the induced stated $\widetilde{\mathcal{N}}_K$-tangle diagram, $\iota_{K}^{\ast}(\alpha)$. Because $\alpha$ is in generic position and any of its endpoints lie on $b\left((\frac{1}{2}, 1-\epsilon)\right)$, there is no overlap in this union, ensuring that this is a well-defined element in $\mathscr{S}(\widetilde{M}_K, \widetilde{\mathcal{N}}_K)$.

It is clear that $\mathscr{S}\left(\widetilde{M}_K, \widetilde{\mathcal{N}}_K \right)$ is isomorphic to $\mathscr{S}(M_K, \mathcal{N}_K)$ as stated skein algebras. Consequently, $\alpha \cdot D$ induces a left $\mathscr{S}^{pb}(\Sigma^\prime)$-module structure on $\mathscr{S}(M_K)$, and therefore, a left $\mathscr{S}(T^2 \setminus D^2)$-module structure.

\section{The Unknot Module}

In this section, we will attempt to compute the $\mathscr{S}(T^2 \setminus D^2)$-module structure for $\mathscr{S}(M_K)$ when $K$ is the unknot. However, before we discuss this, we need to clarify a detail on how we understand $T^2 \setminus D^2$. Since we typically use the ``flat'' interpretation of the torus when performing calculations on $T^2 \setminus D^2$, we need to be careful how this interpretation is identified with the torus with boundary to ensure that the module action is well-defined. Figure \ref{fig:X3OnTorus} provides a visual aid for this identification process.

\begin{figure}[h]
    $$\begin{tikzpicture}[baseline=-3]
        \draw[draw=none] (0,1) -- (2,1) -- (2,-1) -- (0,-1) -- (1, -0.15) to[out=0, in=-90] (1.15, 0) to[out=90, in=0] (1, 0.15) to[out=180, in=90] (0.85, 0) to[out=-90, in=180] (1, -0.15) -- (0, -1) -- (0,1);
        \draw[thick, ->] (0,1) -- (0.9,1);
        \draw[thick, ->] (0.85,1) -- (1.2,1);
        \draw[thick] (1.2,1) -- (2,1);
        \draw[thick, ->] (0,-1) -- (0.9,-1);
        \draw[thick, ->] (0.85,-1) -- (1.2,-1);
        \draw[thick] (1.2,-1) -- (2,-1);
        \draw[thick, ->] (0,-1) -- (0,0);
        \draw[thick] (0,-1) -- (0,1);
        \draw[thick, ->] (2,-1) -- (2,0);
        \draw[thick] (2,-1) -- (2,1);
        \draw[thick, blue] (0, -1) to[out=45, in=180] (1, 0.15) to[out=60, in=225] (2, 1);
        \draw[thick, black] (1, 0) circle (0.15);
        \node[draw, circle, inner sep=0pt, minimum size=3pt, fill=white] at (1, 0.15) {};
        \node at (0.7, 0.2) {$\scriptstyle{\mu}$};
        \node at (1.28, 0.18) {$\scriptstyle{\nu}$};
    \end{tikzpicture}
    \quad\rightsquigarrow\quad
    \begin{tikzpicture}[baseline=-3]
        \draw[thick] (1, 0.9) ellipse (0.5 and 0.2);
        \begin{scope}
            \clip (0.5, -0.9) rectangle (1.5, -0.6);
            \draw[thick, dashed] (1, -0.9) ellipse (0.5 and 0.2);
        \end{scope}
        \begin{scope}
            \clip (0.5, -0.9) rectangle (1.5, -1.2);
            \draw[thick] (1, -0.9) ellipse (0.5 and 0.2);
        \end{scope}
        \draw[thick] (0.5, 0.9) -- (0.5, -0.9);
        \draw[thick] (1.5, 0.9) -- (1.5, -0.9);
        \path[thick, tips, ->] (0.5, 1.1) -- (0.9, 1.1);
        \path[thick, tips, ->] (0.5, 1.1) -- (1.2, 1.1);
        \path[thick, tips, ->] (0.5, -0.7) -- (0.9, -0.7);
        \path[thick, tips, ->] (0.5, -0.7) -- (1.2, -0.7);
        \draw[thick, dashed, blue] (0.5, -0.3) to[out=50, in=180] (1, 0.15) to[out=60, in=225] (1.5, 0.9);
        \draw[thick, blue] (0.5, -0.3) to[out=330, in=120] (1.5, -0.9);
        \draw[thick, dashed] (1, 0) circle (0.15);
        \node[draw, circle, inner sep=0pt, minimum size=3pt, fill=white] at (1, 0.15) {};
        \node at (0.73, 0.19) {$\scriptstyle{\mu}$};
        \node at (1.25, 0.22) {$\scriptstyle{\nu}$};
    \end{tikzpicture}
    \quad\rightsquigarrow\quad
    \begin{tikzpicture}[baseline=-3]
        \draw[thick] (1, 0) ellipse (1.2 and 1);
        \draw[thick] (0.7, 0.1) to[out=300, in=240] (1.3, 0.1);
        \draw[thick] (0.8, 0) to[out=60, in=120] (1.2, 0);
        \draw[thick, dashed, blue] (-0.1, -0.4) to[out=75, in=180] (0.3, 0.15) to[out=60, in=100] (1.2, 0);
        \draw[thick, blue] (-0.1, -0.4) to[out=345, in=250] (1.2, 0);
        \draw[thick, dashed] (0.3, 0) circle (0.15);
        \node[draw, circle, inner sep=0pt, minimum size=3pt, fill=white] at (0.3, 0.15) {};;
        \node at (0, 0.17) {$\scriptstyle{\mu}$};
        \node at (0.55, 0.14) {$\scriptstyle{\nu}$};
    \end{tikzpicture}
    \quad\rightsquigarrow\quad
    \begin{tikzpicture}[baseline=-3]
        \draw[thick] (1, 0) ellipse (1.2 and 1);
        \draw[thick] (0.7, 0.1) to[out=300, in=240] (1.3, 0.1);
        \draw[thick] (0.8, 0) to[out=60, in=120] (1.2, 0);
        \draw[thick, blue] (2.1, -0.4) to[out=105, in=0] (1.7, 0.15) to[out=120, in=80] (0.8, 0);
        \draw[thick, dashed, blue] (2.1, -0.4) to[out=195, in=290] (0.8, 0);
        \draw[thick] (1.7, 0) circle (0.15);
        \node[draw, circle, inner sep=0pt, minimum size=3pt, fill=white] at (1.7, 0.15) {};;
        \node at (2, 0.18) {$\scriptstyle{\mu}$};
        \node at (1.43, 0.14) {$\scriptstyle{\nu}$};
    \end{tikzpicture}$$
    \caption[Properly creating $T^2 \setminus D^2$ for knot complements]{Constructing $T^2 \setminus D^2$ from a square with a disk removed.}
    \label{fig:X3OnTorus}
\end{figure}
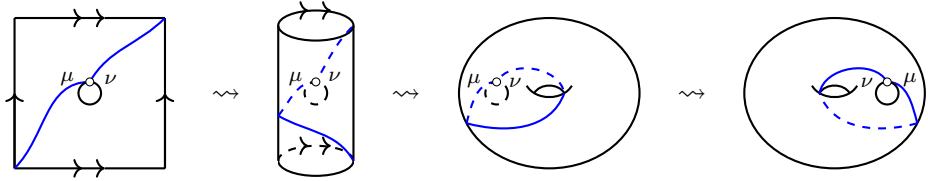
First, identify the vertical sides and fold them towards yourself to create a cylinder (with boundary). Then, identify the top and bottom boundary circles of the cylinder to form the torus, as shown in the third image. In figure \ref{fig:X3OnTorus}, the boundary and marking are on the side opposite to us, so in the final step, we flip the torus over so that the boundary faces us. Note that while the marking in the first picture has an orientation pointing towards the viewer, the orientation in the final picture points away from the viewer. Specifically, the orientation points inwards, towards the inside of the torus, rather than away from the torus. This distinction is important when identifying $T^2 \setminus D^2$ with the boundary of the knot complement.

Let $K \subset S^3$ be the unknot. Then $M_K$ is the solid torus, which is homeomorphic to the thickened annulus. Therefore, $\mathscr{S}(M_K)$ is also a $\mathbb{C}$-algebra and its stated skein algebra (with one marking) is isomorphic to the stated skein algebra of the once punctured monogon. This algebra is generated by the elements $\{ v_{\scriptscriptstyle{+},\scriptscriptstyle{+}}, v_{\scriptscriptstyle{+},\scriptscriptstyle{-}}, v_{\scriptscriptstyle{-},\scriptscriptstyle{+}}, v_{\scriptscriptstyle{-},\scriptscriptstyle{-}} \}$ where
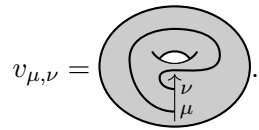
\begin{align*}
    v_{\mu,\nu} = \begin{tikzpicture}[baseline=-4]
        \draw[thick, fill=gray!40] (1, 0) ellipse (1 and 0.8);
        \draw[thick] (0.7, 0.2) to[out=300, in=240] (1.3, 0.2);
        \draw[thick, fill=white] (0.8, 0.1) to[out=60, in=120] (1.2, 0.1);
        \draw[fill=white] (0.8, 0.1) to[out=330, in=210] (1.2, 0.1);
        \draw[->] (1, -0.73) -- (1, -0.1);
        \draw[thick] (1, -0.3) to[out=180, in=270] (0.8, -0.15) to[out=90, in=180] (1, -0.05) to[out=0, in=270] (1.6, 0.15) to[out=90, in=0] (1, 0.5) to[out=180, in=90] (0.4, 0.1) to[out=270, in=180] (1, -0.6);
        \node at (1.16, -0.3) {$\scriptstyle{\nu}$};
        \node at (1.16, -0.6) {$\scriptstyle{\mu}$};
    \end{tikzpicture}.
\end{align*}
There is a small subtlety in the diagram here: if the endpoints of $v_{\mu, \nu}$ were to leave the marking in opposite directions, then we would need to consider half-twists. Since we don't want to consider half-twists, we can isotope these tangles so that they always leave the marking in the same direction.

Using section \ref{section:KnotCompMods}, we see that
\begin{align*}
    X_{1,0}(\mu, \nu) \cdot f &= v_{\mu, \nu} \phantom{\cdot} f \\
    X_{2,0}(\mu, \nu) \cdot 1_K &= C_{\mu}^{\nu} \phantom{\cdot} 1_K \\
    X_{3,0}(\mu, \nu) \cdot 1_K &= -q^{-3} \phantom{\cdot} v_{\mu, \nu}
\end{align*}
where $1_K$ is the empty link (the identity) in $\mathscr{S}(M_K)$, $f \in \mathscr{S}(M_K)$, and $C_{\mu}^{\nu} =
\resizebox{0.6\width}{!}{\begin{tikzpicture}[baseline=-3]
    \draw[gray!40, thick, fill=gray!40, domain=-45:225] plot ({cos(\x)}, {sin(\x)}) to[out=45, in=130] (0.71, -0.71);
    \draw[thick] (-0.71, -0.71) to[out=45, in=135] (0.71, -0.71);
    \node[draw, circle, inner sep=0pt, minimum size=4pt, fill=white] (p1) at (0,-0.41) {};
    \draw[thick] (p1) to[out=45, in=-120] (0.31,0) to[out=60, in=0] (0,0.45) to[out=180, in=120] (-0.31,0) to[out=-60, in=150] (-0.15,-0.2);
    \node (s1) at (-0.6,-0.3) {$\mu$};
    \node (s2) at (0.6,-0.3) {$\nu$};
\end{tikzpicture}}$. Unlike in the Kauffman bracket case (see \cite{MR3530443}), the action of $X_{2,0}(\mu, \nu)$ is not diagonalizable. For example,
$$X_{2,0}(-,-) \cdot v_{+,-} = q^{-5/2}(q^2 - q^{-2}) v_{-,-}$$
and so the action is more complicated.

Once again, let $1_K$ be the empty link diagram in $\mathscr{S}(M_K)$ and $f \in \mathscr{S}(M_K)$. Denote $\partial_K$ as the closed curve, parallel to the boundary in $M_K$. Then we have
\begin{align*}
    Y_1 \cdot f &= \partial_K \phantom{\cdot} f \\
    Y_2 \cdot 1_K &= \left( -q^2-q^{-2} \right) 1_K \\
    Y_2 \cdot v_{\mu, \nu} &= \left( -q^{4} - q^{-4} \right) v_{\mu, \nu} \\
    Y_3 \cdot 1_K &= -q^{-3} \phantom{\cdot} \partial_K \\
    \partial \cdot 1_K &= \left( -q^2-q^{-2} \right) 1_K \\
    \partial \cdot v_{\mu, \nu} &= \left( -q^6 - q^{-6} \right) v_{\mu, \nu} - (q^2 - q^{-2})^2 C_{\mu}^{\nu} \partial_K
\end{align*}

It's important to note that since $\mathscr{S}(M_K)$ is only a module over $\mathscr{S}(T^2 \setminus D^2)$, this provides only a partial description of the module. Further work is needed in order to get a full description of this module structure.

\section{A Note on the Genus Two Surface}

We've discussed how skein modules of knot complements serve as modules over the torus. However, as alluded to at the beginning of this chapter, this is just a special instance of a more broader phenomenon. For any $3$-manifold with boundary, $M$, the associated Kauffman bracket skein module, $K_q(M)$, has a natural left module structure over the Kauffman bracket skein algebra of its boundary, $K_q(\partial M)$. This structure is induced by the homeomorphism
$$\partial M \times [0,1] \bigsqcup_{\partial M \times \{0\} \sim \partial M} M \cong M.$$

When viewed as a $3$-manifold, the boundary of $\left( T^2 \setminus D^2 \right) \times [0,1]$ is the genus two surface. Therefore, this suggests that $\mathcal{SH}_{q,t}$ can be makde into a module over some skein module of the genus $2$ surface. In \cite{MR3916085}, Arthamonov and Shakirov proposed a genus $2$ generalization for the $A_1$ spherical DAHA, defined in terms of its action on a particular space. It turns out that the Kauffman bracket skein algebra of the genus two surface seems to correspond to this algebra.

Let $s \in \mathbb{C}^\times$ and let $\mathbf{H}_2$ be the genus $2$ handlebody. From now on, we will denote $\Sigma_{g,n}$ as the genus $g$ surface with $n$ punctures and define $\Sigma_{g} := \Sigma_{g,0}$. Notice that $K_s(T^2 \setminus D^2) \cong K_s(\Sigma_{1,1})$ as the Kauffman bracket skein algebras is only defined with closed curves.

\begin{definition}
    A triple, $(a, b, c)$, is called \textit{admissible} if $a,b,c \geq 0$, $a+b+c$ is even, and $|a-b| \leq c \leq a+b$.
\end{definition}

To define their algebra, Arthamonov and Shakirov use six operators, denoted $\hat{\mathcal{O}}_{B_{12}}$, $\hat{\mathcal{O}}_{B_{13}}$, $\hat{\mathcal{O}}_{B_{23}}$, $\hat{\mathcal{O}}_{A_1}$, $\hat{\mathcal{O}}_{A_2}$, and $\hat{\mathcal{O}}_{A_3}$, which correspond to the following six closed curves.
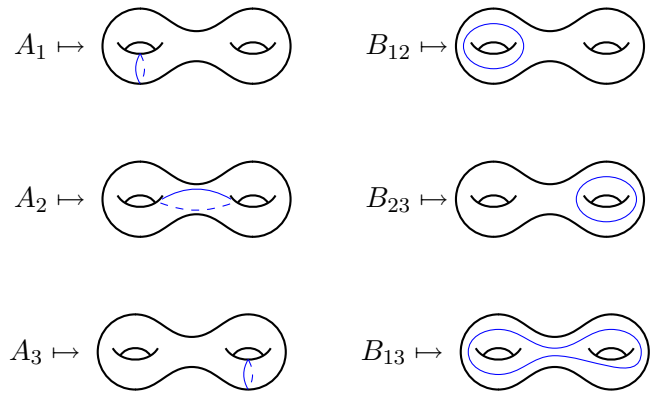
\begin{center}
    $A_1 \mapsto$ \begin{tikzpicture}[baseline=-3]
        \begin{scope}
            \clip (-0.6, -1) rectangle (0, 1);
            \draw[thick] (0, 0) circle (0.5);
        \end{scope}
        \draw[thick] (0, 0.5) to[out=0, in=180] (0.75, 0.2) to[out=0, in=180] (1.5, 0.5);
        \draw[thick] (0, -0.5) to[out=0, in=180] (0.75, -0.2) to[out=0, in=180] (1.5, -0.5);
        \begin{scope}
            \clip (1.5, -1) rectangle (2.1, 1);
            \draw[thick] (1.5, 0) circle (0.5);
        \end{scope}
        \draw[thick] (-0.3, 0.05) to[out=300, in=240] (0.3, 0.05);
        \draw[thick] (-0.2, -0.05) to[out=60, in=120] (0.2, -0.05);
        \draw[thick] (1.2, 0.05) to[out=300, in=240] (1.8, 0.05);
        \draw[thick] (1.3, -0.05) to[out=60, in=120] (1.7, -0.05);
        \draw[blue] (0, -0.1) to[out=240, in=120] (0, -0.5);
        \draw[blue, dashed] (0, -0.1) to[out=300, in=60] (0, -0.5);
    \end{tikzpicture} \quad\quad 
    $B_{12} \mapsto$\begin{tikzpicture}[baseline=-3]
        \begin{scope}
            \clip (-0.6, -1) rectangle (0, 1);
            \draw[thick] (0, 0) circle (0.5);
        \end{scope}
        \draw[thick] (0, 0.5) to[out=0, in=180] (0.75, 0.2) to[out=0, in=180] (1.5, 0.5);
        \draw[thick] (0, -0.5) to[out=0, in=180] (0.75, -0.2) to[out=0, in=180] (1.5, -0.5);
        \begin{scope}
            \clip (1.5, -1) rectangle (2.1, 1);
            \draw[thick] (1.5, 0) circle (0.5);
        \end{scope}
        \draw[thick] (-0.3, 0.05) to[out=300, in=240] (0.3, 0.05);
        \draw[thick] (-0.2, -0.05) to[out=60, in=120] (0.2, -0.05);
        \draw[thick] (1.2, 0.05) to[out=300, in=240] (1.8, 0.05);
        \draw[thick] (1.3, -0.05) to[out=60, in=120] (1.7, -0.05);
        \draw[blue] (0, 0) ellipse (0.4 and 0.3);
    \end{tikzpicture} \\
    $A_2 \mapsto$ \begin{tikzpicture}[baseline=-3]
        \begin{scope}
            \clip (-0.6, -1) rectangle (0, 1);
            \draw[thick] (0, 0) circle (0.5);
        \end{scope}
        \draw[thick] (0, 0.5) to[out=0, in=180] (0.75, 0.2) to[out=0, in=180] (1.5, 0.5);
        \draw[thick] (0, -0.5) to[out=0, in=180] (0.75, -0.2) to[out=0, in=180] (1.5, -0.5);
        \begin{scope}
            \clip (1.5, -1) rectangle (2.1, 1);
            \draw[thick] (1.5, 0) circle (0.5);
        \end{scope}
        \draw[thick] (-0.3, 0.05) to[out=300, in=240] (0.3, 0.05);
        \draw[thick] (-0.2, -0.05) to[out=60, in=120] (0.2, -0.05);
        \draw[thick] (1.2, 0.05) to[out=300, in=240] (1.8, 0.05);
        \draw[thick] (1.3, -0.05) to[out=60, in=120] (1.7, -0.05);
        \draw[blue] (0.29, 0) to[out=30, in=150] (1.21, 0);
        \draw[blue, dashed] (0.27, -0.05) to[out=340, in=200] (1.23, -0.05);
    \end{tikzpicture} \quad\quad 
    $B_{23} \mapsto$\begin{tikzpicture}[baseline=-3]
        \begin{scope}
            \clip (-0.6, -1) rectangle (0, 1);
            \draw[thick] (0, 0) circle (0.5);
        \end{scope}
        \draw[thick] (0, 0.5) to[out=0, in=180] (0.75, 0.2) to[out=0, in=180] (1.5, 0.5);
        \draw[thick] (0, -0.5) to[out=0, in=180] (0.75, -0.2) to[out=0, in=180] (1.5, -0.5);
        \begin{scope}
            \clip (1.5, -1) rectangle (2.1, 1);
            \draw[thick] (1.5, 0) circle (0.5);
        \end{scope}
        \draw[thick] (-0.3, 0.05) to[out=300, in=240] (0.3, 0.05);
        \draw[thick] (-0.2, -0.05) to[out=60, in=120] (0.2, -0.05);
        \draw[thick] (1.2, 0.05) to[out=300, in=240] (1.8, 0.05);
        \draw[thick] (1.3, -0.05) to[out=60, in=120] (1.7, -0.05);
        \draw[blue] (1.5, 0) ellipse (0.4 and 0.3);
    \end{tikzpicture} \\
    $A_3 \mapsto$ \begin{tikzpicture}[baseline=-3]
        \begin{scope}
            \clip (-0.6, -1) rectangle (0, 1);
            \draw[thick] (0, 0) circle (0.5);
        \end{scope}
        \draw[thick] (0, 0.5) to[out=0, in=180] (0.75, 0.2) to[out=0, in=180] (1.5, 0.5);
        \draw[thick] (0, -0.5) to[out=0, in=180] (0.75, -0.2) to[out=0, in=180] (1.5, -0.5);
        \begin{scope}
            \clip (1.5, -1) rectangle (2.1, 1);
            \draw[thick] (1.5, 0) circle (0.5);
        \end{scope}
        \draw[thick] (-0.3, 0.05) to[out=300, in=240] (0.3, 0.05);
        \draw[thick] (-0.2, -0.05) to[out=60, in=120] (0.2, -0.05);
        \draw[thick] (1.2, 0.05) to[out=300, in=240] (1.8, 0.05);
        \draw[thick] (1.3, -0.05) to[out=60, in=120] (1.7, -0.05);
        \draw[blue] (1.5, -0.1) to[out=240, in=120] (1.5, -0.5);
        \draw[blue, dashed] (1.5, -0.1) to[out=300, in=60] (1.5, -0.5);
    \end{tikzpicture} \quad\quad 
    $B_{13} \mapsto$ \begin{tikzpicture}[baseline=-3]
        \begin{scope}
            \clip (-0.6, -1) rectangle (0, 1);
            \draw[thick] (0, 0) circle (0.5);
        \end{scope}
        \draw[thick] (0, 0.5) to[out=0, in=180] (0.75, 0.2) to[out=0, in=180] (1.5, 0.5);
        \draw[thick] (0, -0.5) to[out=0, in=180] (0.75, -0.2) to[out=0, in=180] (1.5, -0.5);
        \begin{scope}
            \clip (1.5, -1) rectangle (2.1, 1);
            \draw[thick] (1.5, 0) circle (0.5);
        \end{scope}
        \draw[thick] (-0.3, 0.05) to[out=300, in=240] (0.3, 0.05);
        \draw[thick] (-0.2, -0.05) to[out=60, in=120] (0.2, -0.05);
        \draw[thick] (1.2, 0.05) to[out=300, in=240] (1.8, 0.05);
        \draw[thick] (1.3, -0.05) to[out=60, in=120] (1.7, -0.05);
        \draw[blue] (-0.4, 0) to[out=90, in=180] (0, 0.3) to[out=0, in=180] (0.75, 0.05) to[out=0, in=180] (1.5, 0.3) to[out=0, in=90] (1.9, 0) to[out=270, in=0] (0.75, -0.05) to[out=180, in=0] (0, -0.3) to[out=180, in=270] (-0.4, 0);
    \end{tikzpicture}
\end{center}
Specifically, these operators act on $\mathbb{C}[x_{12}^{\pm 1}, x_{13}^{\pm 1}, x_{23}^{\pm 1}]^{\mathbb{Z}/2\mathbb{Z}^3}$, the space of Laurent polynomials in $3$ variables invariant under the $\left(\mathbb{Z}/2\mathbb{Z}\right)^3$-action that simultaneously inverts $x_{12}$, $x_{13}$, and $x_{23}$. Furthermore, this space has a basis given by a family of Laurent polynomials, denoted $\left\{\Psi_{i, j, k}\right\}$, where $(i, j, k)$ are admissible triples and $\Psi_{0,0,0} = 1$.

\begin{definition}
    Let $(i, j, k)$ be an admissible triple and $a, b \in\{-1,1\}$. The \textit{Arthamonov and Shakirov coefficients} are
    $$C_{a, b}(i, j, k)=a b \frac{\left[\frac{a i+b j+k}{2}, \frac{a+b+2}{2}\right]_{q, t}\left[\frac{a i+b j-k}{2}, \frac{a+b}{2}\right]_{q, t}[i-1,2]_{q, t}[j-1,2]_{q, t}}{\left[i, \frac{a+3}{2}\right]_{q, t}\left[i-1, \frac{a+3}{2}\right]_{q, t}\left[j, \frac{b+3}{2}\right]_{q, t}\left[j-1, \frac{b+3}{2}\right]_{q, t}}$$
    where
    $$[n, m]_{q, t}:=\frac{q^{\frac{n}{2}} t^{\frac{m}{2}}-q^{-\frac{n}{2}} t^{-\frac{m}{2}}}{q^{\frac{1}{2}}-q^{-\frac{1}{2}}}.$$
\end{definition}

\begin{definition}
    The \textit{Arthamonov-Shakirov, genus $2$ spherical DAHA} is the subalgebra of the endomorphism ring of $\mathbb{C}[x_{12}^{\pm 1}, x_{13}^{\pm 1}, x_{23}^{\pm 1}]^{\mathbb{Z}/2\mathbb{Z}^3}$, generated by $\hat{\mathcal{O}}_{A_1}$, $\hat{\mathcal{O}}_{A_2}$, $\hat{\mathcal{O}}_{A_3}$, $\hat{\mathcal{O}}_{B_{12}}$, $\hat{\mathcal{O}}_{B_{13}}$, and $\hat{\mathcal{O}}_{B_{23}}$, where
    \begin{align*}
        \hat{\mathcal{O}}_{A_1} \Psi_{i, j, k} &= \left(q^{i / 2} t^{1 / 2}+q^{-i / 2} t^{-1 / 2}\right) \Psi_{i, j, k} \\
        \hat{\mathcal{O}}_{A_2} \Psi_{i, j, k}&= \left(q^{j / 2} t^{1 / 2}+q^{-j / 2} t^{-1 / 2}\right) \Psi_{i, j, k} \\
        \hat{\mathcal{O}}_{A_3} \Psi_{i, j, k} &= \left(q^{k / 2} t^{1 / 2}+q^{-k / 2} t^{-1 / 2}\right) \Psi_{i, j, k} \\
        \hat{\mathcal{O}}_{B_{12}} \Psi_{i, j, k} &= \sum_{a, b \in\{-1,1\}} C_{a, b}(i, j, k) \Psi_{i+a, j+b, k} \\
        \hat{\mathcal{O}}_{B_{13}} \Psi_{i, j, k} &= \sum_{a, b \in\{-1,1\}} C_{a, b}(i, k, j) \Psi_{i+a, j, k+b} \\
        \hat{\mathcal{O}}_{B_{23}} \Psi_{i, j, k} &= \sum_{a, b \in\{-1,1\}} C_{a, b}(j, k, i) \Psi_{i, j+a, k+b}.
    \end{align*}
See \cite{MR3916085} for details.
\end{definition}

In particular, this genus $2$ spherical DAHA depends on two parameters, $q$ and $t$. Using the action of $K_s(\Sigma_{1,1})$ and $K_s(\Sigma_{0,4})$ on $K_s(\mathbf{H}_2)$ and the fact that the action of the skein algebra of a closed surface on the skein module of a handlebody is faithful [Le21], Cooke and Samuelson showed the following theorem.
\begin{theorem}[Corollary 5.11 in \cite{MR4368676}]
    The $t=q=s^4$ specialization of the Arthamonov-Shakirov algebra is isomorphic to the skein algebra $K_s(\Sigma_2)$.
\end{theorem}

More recently, Arthamonov proved in \cite{arthamonov2023classicallimitgenusdaha} that the one-parameter deformation, the Arthamonov-Shakirov algebra, of $K_q(\Sigma_2)$ is flat.

As mentioned at the beginning of this section, there is an inclusion
$$\Sigma_2 \hookrightarrow \partial \left( (T^2 \setminus D^2) \times [0,1] \right)$$
which induces a $K_q(\Sigma_2)$-module structure for $\mathcal{SH}_{q,t} \cong K_q(T^2 \setminus D^2)$. Roughly speaking, let $a$ be a closed curve on the genus of $\Sigma_2$ that gets glued to the boundary $(T^2 \setminus D^2) \times \{1\}$, $b$ be a closed curve that lives on the other genus, and let $\alpha$ and $\beta$ be the respective curves from the induced map.
Then to see how these act on some $x_1 x_2 \in K_q(T^2 \setminus D^2)$ we get
\begin{align*}
    a \cdot (x_1 x_2) &= \alpha x_1 x_2 \\
    b \cdot (x_1 x_2) &= x_1 x_2 \beta
\end{align*}

It's important to note that the algebra structure from $K_q(T^2 \setminus D^2)$ does not carry over and and this is only a module. Although it respects the algebra's associativity, $a \cdot (x_1 x_2) = (a \cdot x_1) x_2$, it is not true that $(a \cdot x_1) x_2 = x_1 (a \cdot x_2)$.

We can construct this structure a bit more explicitly in the following way. Consider the graph defined by two loops connected by an arc and embed it into $S^3$ in the following way.
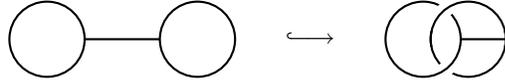
\begin{center}
    \begin{tikzpicture}
        \draw[thick] (-1, 0) circle (0.5);
        \draw[thick] (-0.5, 0) -- (0.5, 0);
        \draw[thick] (1,0) circle (0.5);
        \node at (2.5,0) {$\longhookrightarrow$};
        \begin{scope}
            \clip (3.4, -0.1) rectangle (4.6, 0.6);
            \draw[thick] (4, 0) circle (0.5);
        \end{scope}
        \begin{scope}
            \clip (4, -0.1) rectangle (5.2, 0.6);
            \draw[line width=2mm, white] (4.6, 0) circle (0.5);
            \draw[thick] (4.6, 0) circle (0.5);
        \end{scope}
        \begin{scope}
            \clip (4, 0.1) rectangle (5.2, -0.6);
            \draw[thick] (4.6, 0) circle (0.5);
        \end{scope}
        \begin{scope}
            \clip (3.4, 0.1) rectangle (4.6, -0.6);
            \draw[line width=2mm, white] (4, 0) circle (0.5);
            \draw[thick] (4, 0) circle (0.5);
        \end{scope}
        \draw[thick] (4.5, 0) -- (5.1, 0);
    \end{tikzpicture}
\end{center}
Call this embedded graph $G$ and consider an open tubular neighborhood of $G$, $N(G)$. Notice that $N(G)$ is homeomorphic to the filled in genus two surface and $S^3 \setminus N(G)$ is homeomorphic to $T^2 \setminus D^2$. Therefore, the boundary of $T^2 \setminus D^2$ is $\Sigma_2$ and so $K_q(T^2 \setminus D^2)$ is a left $K_q(\Sigma_2)$-module. However, since one of these genera corresponds to the internal boundary, $\left( T^2 \setminus D^2 \right) \times \{0\}$, the action by any tangle solely living on this ``inner genus'' corresponds to right multiplication. This is because any curves or elements in the module $K_q(T^2 \setminus D^2)$ must lie in $\left( T^2 \setminus D^2 \right) \times (0,1)$, and therefore \emph{above} this boundary.
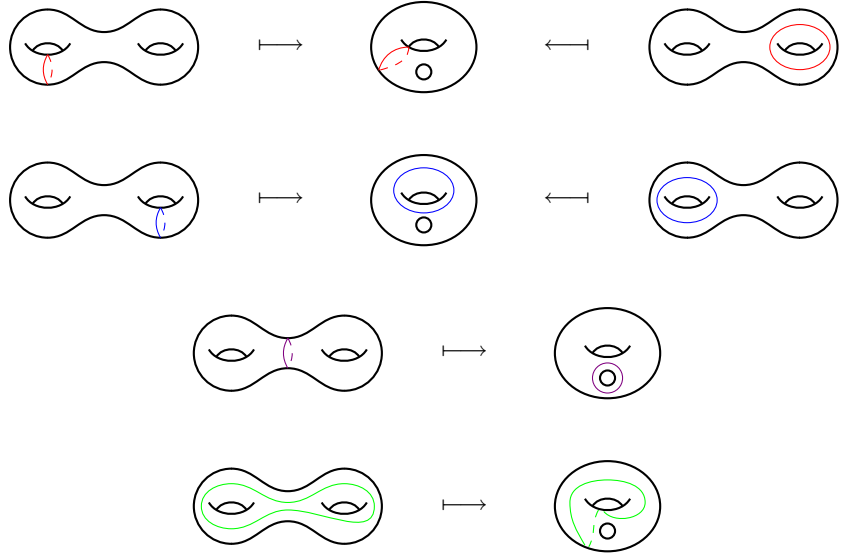
\begin{center}
    \begin{tikzpicture}
        \begin{scope}
            \clip (-0.6, -1) rectangle (0, 1);
            \draw[thick] (0, 0) circle (0.5);
        \end{scope}
        \draw[thick] (0, 0.5) to[out=0, in=180] (0.75, 0.2) to[out=0, in=180] (1.5, 0.5);
        \draw[thick] (0, -0.5) to[out=0, in=180] (0.75, -0.2) to[out=0, in=180] (1.5, -0.5);
        \begin{scope}
            \clip (1.5, -1) rectangle (2.1, 1);
            \draw[thick] (1.5, 0) circle (0.5);
        \end{scope}
        \draw[thick] (-0.3, 0.05) to[out=300, in=240] (0.3, 0.05);
        \draw[thick] (-0.2, -0.05) to[out=60, in=120] (0.2, -0.05);
        \draw[thick] (1.2, 0.05) to[out=300, in=240] (1.8, 0.05);
        \draw[thick] (1.3, -0.05) to[out=60, in=120] (1.7, -0.05);
        \draw[red] (0, -0.1) to[out=240, in=120] (0, -0.5);
        \draw[red, dashed] (0, -0.1) to[out=300, in=60] (0, -0.5);
        \node at (3.1,0) {$\longmapsto$};
        \draw[thick] (5, 0) ellipse (0.7 and 0.6);
        \draw[thick] (4.7, 0.1) to[out=300, in=240] (5.3, 0.1);
        \draw[thick] (4.8, 0) to[out=60, in=120] (5.2, 0);
        \draw[thick] (5, -0.33) circle (0.1);
        \draw[red] (4.8, 0) to[out=180, in=75] (4.4,-0.31);
        \draw[red, dashed] (4.8, 0) to[out=270, in=15] (4.4,-0.31);
        \node[xscale=-1] at (6.9,0) {$\longmapsto$};
        \begin{scope}
            \clip (7.9, -1) rectangle (8.5, 1);
            \draw[thick] (8.5, 0) circle (0.5);
        \end{scope}
        \draw[thick] (8.5, 0.5) to[out=0, in=180] (9.25, 0.2) to[out=0, in=180] (10, 0.5);
        \draw[thick] (8.5, -0.5) to[out=0, in=180] (9.25, -0.2) to[out=0, in=180] (10, -0.5);
        \begin{scope}
            \clip (10, -1) rectangle (10.6, 1);
            \draw[thick] (10, 0) circle (0.5);
        \end{scope}
        \draw[thick] (8.2, 0.05) to[out=300, in=240] (8.8, 0.05);
        \draw[thick] (8.3, -0.05) to[out=60, in=120] (8.7, -0.05);
        \draw[thick] (9.7, 0.05) to[out=300, in=240] (10.3, 0.05);
        \draw[thick] (9.8, -0.05) to[out=60, in=120] (10.2, -0.05);
        \draw[red] (10, 0) ellipse (0.4 and 0.3);
    \end{tikzpicture}\\
    \begin{tikzpicture}
        \begin{scope}
            \clip (-0.6, -1) rectangle (0, 1);
            \draw[thick] (0, 0) circle (0.5);
        \end{scope}
        \draw[thick] (0, 0.5) to[out=0, in=180] (0.75, 0.2) to[out=0, in=180] (1.5, 0.5);
        \draw[thick] (0, -0.5) to[out=0, in=180] (0.75, -0.2) to[out=0, in=180] (1.5, -0.5);
        \begin{scope}
            \clip (1.5, -1) rectangle (2.1, 1);
            \draw[thick] (1.5, 0) circle (0.5);
        \end{scope}
        \draw[thick] (-0.3, 0.05) to[out=300, in=240] (0.3, 0.05);
        \draw[thick] (-0.2, -0.05) to[out=60, in=120] (0.2, -0.05);
        \draw[thick] (1.2, 0.05) to[out=300, in=240] (1.8, 0.05);
        \draw[thick] (1.3, -0.05) to[out=60, in=120] (1.7, -0.05);
        \draw[blue] (1.5, -0.1) to[out=240, in=120] (1.5, -0.5);
        \draw[blue, dashed] (1.5, -0.1) to[out=300, in=60] (1.5, -0.5);
        \node at (3.1,0) {$\longmapsto$};
        \draw[thick] (5, 0) ellipse (0.7 and 0.6);
        \draw[thick] (4.7, 0.1) to[out=300, in=240] (5.3, 0.1);
        \draw[thick] (4.8, 0) to[out=60, in=120] (5.2, 0);
        \draw[thick] (5, -0.33) circle (0.1);
        \draw[blue] (5, 0.13) ellipse (0.4 and 0.3);
        \node[xscale=-1] at (6.9, 0) {$\longmapsto$};
        \begin{scope}
            \clip (7.9, -1) rectangle (8.5, 1);
            \draw[thick] (8.5, 0) circle (0.5);
        \end{scope}
        \draw[thick] (8.5, 0.5) to[out=0, in=180] (9.25, 0.2) to[out=0, in=180] (10, 0.5);
        \draw[thick] (8.5, -0.5) to[out=0, in=180] (9.25, -0.2) to[out=0, in=180] (10, -0.5);
        \begin{scope}
            \clip (10, -1) rectangle (10.6, 1);
            \draw[thick] (10, 0) circle (0.5);
        \end{scope}
        \draw[thick] (8.2, 0.05) to[out=300, in=240] (8.8, 0.05);
        \draw[thick] (8.3, -0.05) to[out=60, in=120] (8.7, -0.05);
        \draw[thick] (9.7, 0.05) to[out=300, in=240] (10.3, 0.05);
        \draw[thick] (9.8, -0.05) to[out=60, in=120] (10.2, -0.05);
        \draw[blue] (8.5, 0) ellipse (0.4 and 0.3);
    \end{tikzpicture}\\
    \begin{tikzpicture}
        \begin{scope}
            \clip (-0.6, -1) rectangle (0, 1);
            \draw[thick] (0, 0) circle (0.5);
        \end{scope}
        \draw[thick] (0, 0.5) to[out=0, in=180] (0.75, 0.2) to[out=0, in=180] (1.5, 0.5);
        \draw[thick] (0, -0.5) to[out=0, in=180] (0.75, -0.2) to[out=0, in=180] (1.5, -0.5);
        \begin{scope}
            \clip (1.5, -1) rectangle (2.1, 1);
            \draw[thick] (1.5, 0) circle (0.5);
        \end{scope}
        \draw[thick] (-0.3, 0.05) to[out=300, in=240] (0.3, 0.05);
        \draw[thick] (-0.2, -0.05) to[out=60, in=120] (0.2, -0.05);
        \draw[thick] (1.2, 0.05) to[out=300, in=240] (1.8, 0.05);
        \draw[thick] (1.3, -0.05) to[out=60, in=120] (1.7, -0.05);
        \draw[violet] (0.75, 0.2) to[out=240, in=120] (0.75, -0.2);
        \draw[violet, dashed] (0.75, 0.2) to[out=300, in=60] (0.75, -0.2);
        \node at (3.1,0) {$\longmapsto$};
        \draw[thick] (5, 0) ellipse (0.7 and 0.6);
        \draw[thick] (4.7, 0.1) to[out=300, in=240] (5.3, 0.1);
        \draw[thick] (4.8, 0) to[out=60, in=120] (5.2, 0);
        \draw[thick] (5, -0.33) circle (0.1);
        \draw[violet] (5, -0.33) circle (0.2);
    \end{tikzpicture}\\
    \begin{tikzpicture}
        \begin{scope}
            \clip (-0.6, -1) rectangle (0, 1);
            \draw[thick] (0, 0) circle (0.5);
        \end{scope}
        \draw[thick] (0, 0.5) to[out=0, in=180] (0.75, 0.2) to[out=0, in=180] (1.5, 0.5);
        \draw[thick] (0, -0.5) to[out=0, in=180] (0.75, -0.2) to[out=0, in=180] (1.5, -0.5);
        \begin{scope}
            \clip (1.5, -1) rectangle (2.1, 1);
            \draw[thick] (1.5, 0) circle (0.5);
        \end{scope}
        \draw[thick] (-0.3, 0.05) to[out=300, in=240] (0.3, 0.05);
        \draw[thick] (-0.2, -0.05) to[out=60, in=120] (0.2, -0.05);
        \draw[thick] (1.2, 0.05) to[out=300, in=240] (1.8, 0.05);
        \draw[thick] (1.3, -0.05) to[out=60, in=120] (1.7, -0.05);
        \draw[green] (-0.4, 0) to[out=90, in=180] (0, 0.3) to[out=0, in=180] (0.75, 0.05) to[out=0, in=180] (1.5, 0.3) to[out=0, in=90] (1.9, 0) to[out=270, in=0] (0.75, -0.05) to[out=180, in=0] (0, -0.3) to[out=180, in=270] (-0.4, 0);
        \node at (3.1,0) {$\longmapsto$};
        \draw[thick] (5, 0) ellipse (0.7 and 0.6);
        \draw[thick] (4.7, 0.1) to[out=300, in=240] (5.3, 0.1);
        \draw[thick] (4.8, 0) to[out=60, in=120] (5.2, 0);
        \draw[thick] (5, -0.33) circle (0.1);
        \draw[green] (4.95, -0.07) to[out=310, in=270] (5.5, 0.05) to[out=90, in=0] (5, 0.35) to[out=180, in=90] (4.5, 0.05) to[out=270, in=120] (4.7, -0.52);
        \draw[green, dashed] (4.88, -0.05) to[out=240, in=60] (4.75, -0.54);
    \end{tikzpicture}
\end{center}
You may have noticed that we have replaced $B_{13}$ with a different curve. This change corresponds to a different set of generators, but a set of generators nonetheless.

This module structure can alternatively be understood through the following figure.
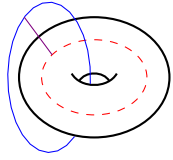
\begin{center}
    \begin{tikzpicture}
        \draw[thick] (1, 0) ellipse (1 and 0.8);
        \draw[thick] (0.7, 0.05) to[out=300, in=240] (1.3, 0.05);
        \draw[thick] (0.8, -0.05) to[out=60, in=120] (1.2, -0.05);
        \draw[red, dashed] (1, 0) ellipse (0.7 and 0.5);
        \draw[blue, domain=-5:308] plot ({0.4 + 0.55*cos(\x)}, {sin(\x)});
        \draw[violet] (0.07, 0.8) -- (0.44, 0.3);
    \end{tikzpicture}
\end{center}
The dashed red line corresponds to the inside boundary, $\left( T^2 \setminus D^2 \right) \times \{0\}$ (the ``core'' of our torus). The blue line will correspond to the outside boundary, $\left( T^2 \setminus D^2 \right) \times \{1\}$. The purple line is the connecting arc that corresponds to the removed disk.

We can factor through the module structure as follows. Take the disjoint union of two tori with boundary and embed them into the genus two surface. Then, embed the genus two surface into the boundary of $T^2 \setminus D^2$. These embeddings induce maps on corresponding Kauffman brackets skein algebras.
$$
\begin{tikzpicture}[baseline=-3]
    \begin{scope}
        \clip (-0.6, -1) rectangle (0, 1);
        \draw[thick] (0, 0) circle (0.5);
    \end{scope}
    \draw[thick] (0, 0.5) to[out=0, in=180] (0.75, 0.2);
    \draw[thick] (0, -0.5) to[out=0, in=180] (0.75, -0.2);
    \draw[thick] (0.75, 0) ellipse (0.1 and 0.2);
    \draw[thick] (-0.3, 0.05) to[out=300, in=240] (0.3, 0.05);
    \draw[thick] (-0.2, -0.05) to[out=60, in=120] (0.2, -0.05);
\end{tikzpicture} \phantom{\cdot} \bigsqcup \phantom{\cdot}
\begin{tikzpicture}[baseline=-3]
    \begin{scope}
        \clip (1.5, -1) rectangle (2.1, 1);
        \draw[thick] (1.5, 0) circle (0.5);
    \end{scope}
    \draw[thick] (0.75, 0.2) to[out=0, in=180] (1.5, 0.5);
    \draw[thick] (0.75, -0.2) to[out=0, in=180] (1.5, -0.5);
    \draw[thick] (0.75, 0) ellipse (0.1 and 0.2);
    \draw[thick] (1.2, 0.05) to[out=300, in=240] (1.8, 0.05);
    \draw[thick] (1.3, -0.05) to[out=60, in=120] (1.7, -0.05);
\end{tikzpicture} \phantom{\cdot} \hookrightarrow \phantom{\cdot}
\begin{tikzpicture}[baseline=-3]
    \begin{scope}
        \clip (-0.6, -1) rectangle (0, 1);
        \draw[thick] (0, 0) circle (0.5);
    \end{scope}
    \draw[thick] (0, 0.5) to[out=0, in=180] (0.75, 0.2) to[out=0, in=180] (1.5, 0.5);
    \draw[thick] (0, -0.5) to[out=0, in=180] (0.75, -0.2) to[out=0, in=180] (1.5, -0.5);
    \begin{scope}
        \clip (1.5, -1) rectangle (2.1, 1);
        \draw[thick] (1.5, 0) circle (0.5);
    \end{scope}
    \draw[thick] (-0.3, 0.05) to[out=300, in=240] (0.3, 0.05);
    \draw[thick] (-0.2, -0.05) to[out=60, in=120] (0.2, -0.05);
    \draw[thick] (1.2, 0.05) to[out=300, in=240] (1.8, 0.05);
    \draw[thick] (1.3, -0.05) to[out=60, in=120] (1.7, -0.05);
\end{tikzpicture} \phantom{\cdot} \hookrightarrow \phantom{\cdot}
\partial \left(\begin{tikzpicture}[baseline=-3]
    \begin{scope}
        \clip (-0.6, -0.7) rectangle (0, 0.7);
        \draw[thick] (0, 0) circle (0.5);
    \end{scope}
    \draw[thick] (0, 0.5) to[out=0, in=180] (0.75, 0.2);
    \draw[thick] (0, -0.5) to[out=0, in=180] (0.75, -0.2);
    \draw[thick] (0.75, 0) ellipse (0.1 and 0.2);
    \draw[thick] (-0.3, 0.05) to[out=300, in=240] (0.3, 0.05);
    \draw[thick] (-0.2, -0.05) to[out=60, in=120] (0.2, -0.05);
\end{tikzpicture} \times I \right)
$$
$$\mathscr{S}(T^2 \setminus D^2) \otimes \mathscr{S}(T^2 \setminus D^2) \longrightarrow \mathscr{S}(\Sigma_2) \longrightarrow \mathscr{S}(T^2 \setminus D^2)$$
Let $\alpha, \beta, \gamma \in K_q(T^2 \setminus D^2)$. If we view $\alpha$ and $\beta$ as each lying on a genus of $\Sigma_2$, then we have the left action
$$\left( \alpha \otimes \beta \right) \cdot \gamma = \alpha \gamma \beta.$$
Let $\delta$ be the loop around the boundary. It is central in $K_q(T^2 \setminus D^2)$ and it's left module action on $K_q(T^2 \setminus D^2)$ (when viewed on $\Sigma_2$) can be viewed as left or right multiplication in $K_q(T^2 \setminus D^2)$.

If we attempt to extend this action of $K_q(\Sigma_2)$ to $\mathscr{S}(T^2 \setminus D^2)$, we quickly run into a problem. By the same logic, the underlying topological structure suggests that the action of $\delta$ should correspond to multiplication by a central element. However, this is clearly not the case as $\delta$ is no longer central in $\mathscr{S}(T^2 \setminus D^2)$.
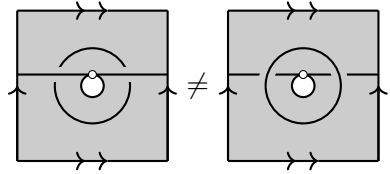
\begin{center}
    \begin{tikzpicture}[baseline=-3]
        \MarkedTorusBackground
        \draw[thick] (1, 0) circle (0.5);
        \draw[line width=2mm, gray!40] (0.2, 0.15) -- (0.8, 0.15);
        \draw[line width=2mm, gray!40] (1.2, 0.15) -- (1.8, 0.15);
        \draw[thick] (0, 0.15) -- (2, 0.15);
        \node[draw, circle, inner sep=0pt, minimum size=3pt, fill=white] at (1, 0.15) {};
    \end{tikzpicture} $\neq$
    \begin{tikzpicture}[baseline=-3]
        \MarkedTorusBackground
        \draw[thick] (0, 0.15) -- (2, 0.15);
        \draw[line width=2mm, gray!40] (1, 0) circle (0.5);
        \draw[thick] (1, 0) circle (0.5);
        \node[draw, circle, inner sep=0pt, minimum size=3pt, fill=white] at (1, 0.15) {};
    \end{tikzpicture}
\end{center}

Therefore, any module structure here would need to be extended in a different way. One possible remedy could be to introduce a boundary component to $\Sigma_2$, and shifting our focus to $\mathscr{S}(\Sigma_2 \setminus D^2)$ instead. However, there is currently no reason to believe that this corresponds to a double affine Hecke algebra.

\nocite{*}
\bibliographystyle{plain}
\bibliography{bibfile}

\begin{thebibliography}{10}

\bibitem{arthamonov2023classicallimitgenusdaha}
Semeon Arthamonov.
\newblock Classical limit of genus two daha, 2023.

\bibitem{MR3916085}
Semeon Arthamonov and Shamil Shakirov.
\newblock Genus two generalization of {$A_1$} spherical {DAHA}.
\newblock {\em Selecta Math. (N.S.)}, 25(2):Paper No. 17, 29, 2019.

\bibitem{MR1670233}
John~W. Barrett.
\newblock Skein spaces and spin structures.
\newblock {\em Math. Proc. Cambridge Philos. Soc.}, 126(2):267--275, 1999.

\bibitem{MR3847209}
David Ben-Zvi, Adrien Brochier, and David Jordan.
\newblock Integrating quantum groups over surfaces.
\newblock {\em J. Topol.}, 11(4):874--917, 2018.

\bibitem{MR3530443}
Yuri Berest and Peter Samuelson.
\newblock Double affine {H}ecke algebras and generalized {J}ones polynomials.
\newblock {\em Compos. Math.}, 152(7):1333--1384, 2016.

\bibitem{Birman_Series_1984}
Joan~S. Birman and Caroline Series.
\newblock An algorithm for simple curves on surfaces.
\newblock {\em Journal of the London Mathematical Society}, s2-29(2):331–342, 1984.

\bibitem{MR4405678}
Wade Bloomquist and Thang T.~Q. L\^{e}.
\newblock The {C}hebyshev-{F}robenius homomorphism for stated skein modules of 3-manifolds.
\newblock {\em Math. Z.}, 301(1):1063--1105, 2022.

\bibitem{MR2851072}
Francis Bonahon and Helen Wong.
\newblock Quantum traces for representations of surface groups in {${\rm SL}_2(\mathbb{C})$}.
\newblock {\em Geom. Topol.}, 15(3):1569--1615, 2011.

\bibitem{MR3659493}
Adrien Brochier and David Jordan.
\newblock Fourier transform for quantum {$D$}-modules via the punctured torus mapping class group.
\newblock {\em Quantum Topol.}, 8(2):361--379, 2017.

\bibitem{MR1600138}
Doug Bullock.
\newblock Rings of {${\rm SL}_2({\bf C})$}-characters and the {K}auffman bracket skein module.
\newblock {\em Comment. Math. Helv.}, 72(4):521--542, 1997.

\bibitem{MR1300632}
Vyjayanthi Chari and Andrew Pressley.
\newblock {\em A guide to quantum groups}.
\newblock Cambridge University Press, Cambridge, 1994.

\bibitem{MR4662169}
Bang-Yen Chen.
\newblock Geometry and topology of maximal antipodal sets and related topics.
\newblock {\em Rom. J. Math. Comput. Sci.}, 13(2):6--25, 2023.

\bibitem{MR1314036}
Ivan Cherednik.
\newblock Double affine {H}ecke algebras and {M}acdonald's conjectures.
\newblock {\em Ann. of Math. (2)}, 141(1):191--216, 1995.

\bibitem{MR2133033}
Ivan Cherednik.
\newblock {\em Double affine {H}ecke algebras}, volume 319 of {\em London Mathematical Society Lecture Note Series}.
\newblock Cambridge University Press, Cambridge, 2005.

\bibitem{Cohen_Metzler_Zimmermann_1981}
Marshall Cohen, Wolfgang Metzler, and Albert Zimmermann.
\newblock What does a basis of f(a, b) look like?
\newblock {\em Mathematische Annalen}, 257(4):435–445, 1981.

\bibitem{MR4536120}
Juliet Cooke.
\newblock Excision of skein categories and factorisation homology.
\newblock {\em Adv. Math.}, 414:Paper No. 108848, 51, 2023.

\bibitem{MR4368676}
Juliet Cooke and Peter Samuelson.
\newblock On the genus two skein algebra.
\newblock {\em J. Lond. Math. Soc. (2)}, 104(5):2260--2298, 2021.

\bibitem{MR4094715}
Benjamin Cooper and Peter Samuelson.
\newblock The {H}all algebras of surfaces {I}.
\newblock {\em J. Inst. Math. Jussieu}, 19(3):971--1028, 2020.

\bibitem{costantino2022stated}
Francesco Costantino and Thang T.~Q. L\^{e}.
\newblock Stated skein algebras of surfaces.
\newblock {\em Journal of the European Mathematical Society}, 2022.

\bibitem{10106311703773}
Freeman~J. Dyson.
\newblock {Statistical Theory of the Energy Levels of Complex Systems. I}.
\newblock {\em Journal of Mathematical Physics}, 3(1):140--156, 01 1962.

\bibitem{MR3242743}
Pavel Etingof, Shlomo Gelaki, Dmitri Nikshych, and Victor Ostrik.
\newblock {\em Tensor categories}, volume 205 of {\em Mathematical Surveys and Monographs}.
\newblock American Mathematical Society, Providence, RI, 2015.

\bibitem{faitg2022holonomy}
Matthieu Faitg.
\newblock Holonomy and (stated) skein algebras in combinatorial quantization, 2022.

\bibitem{MR2850125}
Benson Farb and Dan Margalit.
\newblock {\em A primer on mapping class groups}, volume~49 of {\em Princeton Mathematical Series}.
\newblock Princeton University Press, Princeton, NJ, 2012.

\bibitem{MR1446615}
Igor~B. Frenkel and Mikhail~G. Khovanov.
\newblock Canonical bases in tensor products and graphical calculus for {$U_q(\mathfrak{sl}_2)$}.
\newblock {\em Duke Math. J.}, 87(3):409--480, 1997.

\bibitem{MR2876932}
Igor~B. Frenkel and Hyun~Kyu Kim.
\newblock Quantum {T}eichm\"uller space from the quantum plane.
\newblock {\em Duke Math. J.}, 161(2):305--366, 2012.

\bibitem{MR1675190}
Charles Frohman and R\u{a}zvan Gelca.
\newblock Skein modules and the noncommutative torus.
\newblock {\em Trans. Amer. Math. Soc.}, 352(10):4877--4888, 2000.

\bibitem{MR1967240}
R\u{a}zvan Gelca and Jeremy Sain.
\newblock The noncommutative {A}-ideal of a {$(2,2p+1)$}-torus knot determines its {J}ones polynomial.
\newblock {\em J. Knot Theory Ramifications}, 12(2):187--201, 2003.

\bibitem{MR0972070}
Cameron~McA. Gordon and John Luecke.
\newblock Knots are determined by their complements.
\newblock {\em Bull. Amer. Math. Soc. (N.S.)}, 20(1):83--87, 1989.

\bibitem{MR4647282}
Sergei Gukov, Peter Koroteev, Satoshi Nawata, Du~Pei, and Ingmar Saberi.
\newblock {\em Branes and {DAHA} representations}, volume~48 of {\em SpringerBriefs in Mathematical Physics}.
\newblock Springer, Cham, 2023.

\bibitem{MR4557403}
Sam Gunningham, David Jordan, and Pavel Safronov.
\newblock The finiteness conjecture for skein modules.
\newblock {\em Invent. Math.}, 232(1):301--363, 2023.

\bibitem{MR4437512}
Benjamin Ha\"ioun.
\newblock Relating stated skein algebras and internal skein algebras.
\newblock {\em SIGMA Symmetry Integrability Geom. Methods Appl.}, 18:Paper No. 042, 39, 2022.

\bibitem{MR1359532}
Jens~Carsten Jantzen.
\newblock {\em Lectures on quantum groups}, volume~6 of {\em Graduate Studies in Mathematics}.
\newblock American Mathematical Society, Providence, RI, 1996.

\bibitem{MR0841713}
Michio Jimbo.
\newblock A {$q$}-analogue of {$U(\mathfrak{gl}(N+1))$}, {H}ecke algebra, and the {Y}ang-{B}axter equation.
\newblock {\em Lett. Math. Phys.}, 11(3):247--252, 1986.

\bibitem{MR1321145}
Christian Kassel.
\newblock {\em Quantum groups}, volume 155 of {\em Graduate Texts in Mathematics}.
\newblock Springer-Verlag, New York, 1995.

\bibitem{MR1117149}
Robion Kirby and Paul Melvin.
\newblock The {$3$}-manifold invariants of {W}itten and {R}eshetikhin-{T}uraev for {${\rm sl}(2,{\bf C})$}.
\newblock {\em Invent. Math.}, 105(3):473--545, 1991.

\bibitem{MR4598807}
Julien Korinman.
\newblock Finite presentations for stated skein algebras and lattice gauge field theory.
\newblock {\em Algebr. Geom. Topol.}, 23(3):1249--1302, 2023.

\bibitem{MR1403861}
Greg Kuperberg.
\newblock Spiders for rank {$2$} {L}ie algebras.
\newblock {\em Comm. Math. Phys.}, 180(1):109--151, 1996.

\bibitem{MR3827810}
Thang T.~Q. L\^{e}.
\newblock Triangular decomposition of skein algebras.
\newblock {\em Quantum Topol.}, 9(3):591--632, 2018.

\bibitem{MR3915288}
Thang T.~Q. L\^{e}.
\newblock Quantum {T}eichm\"{u}ller spaces and quantum trace map.
\newblock {\em J. Inst. Math. Jussieu}, 18(2):249--291, 2019.

\bibitem{arx220100045}
Thang T.~Q. L\^{e} and Adam~S. Sikora.
\newblock Stated {SL}(n)-skein modules and algebras, 2024.

\bibitem{MR4264235}
Thang T.~Q. L\^{e} and Tao Yu.
\newblock Stated skein modules of marked 3-manifolds/surfaces, a survey.
\newblock {\em Acta Math. Vietnam.}, 46(2):265--287, 2021.

\bibitem{MR4431131}
Thang T.~Q. L\^{e} and Tao Yu.
\newblock Quantum traces and embeddings of stated skein algebras into quantum tori.
\newblock {\em Selecta Math. (N.S.)}, 28(4):Paper No. 66, 48, 2022.

\bibitem{MR674768}
Ian~G. Macdonald.
\newblock Some conjectures for root systems.
\newblock {\em SIAM J. Math. Anal.}, 13(6):988--1007, 1982.

\bibitem{MR1381692}
Shahn Majid.
\newblock {\em Foundations of quantum group theory}.
\newblock Cambridge University Press, Cambridge, 1995.

\bibitem{MR0949080}
John~C. McConnell and J.~J. Pettit.
\newblock Crossed products and multiplicative analogues of {W}eyl algebras.
\newblock {\em J. London Math. Soc. (2)}, 38(1):47--55, 1988.

\bibitem{MR4283999}
Hugh~R. Morton and Peter Samuelson.
\newblock D{AHA}s and skein theory.
\newblock {\em Comm. Math. Phys.}, 385(3):1655--1693, 2021.

\bibitem{MR3551171}
Greg Muller.
\newblock Skein and cluster algebras of marked surfaces.
\newblock {\em Quantum Topol.}, 7(3):435--503, 2016.

\bibitem{MR2037756}
Alexei Oblomkov.
\newblock Double affine {H}ecke algebras of rank 1 and affine cubic surfaces.
\newblock {\em Int. Math. Res. Not.}, 2004(18):877--912, 2004.

\bibitem{MR4297592}
Alexander~T. Pokorny.
\newblock {\em Dubrovnik {S}kein {T}heory and {P}ower {S}um {E}lements}.
\newblock ProQuest LLC, Ann Arbor, MI, 2021.
\newblock Thesis (Ph.D.)--University of California, Riverside.

\bibitem{MR1194712}
J\'ozef~H. Przytycki.
\newblock Skein modules of {$3$}-manifolds.
\newblock {\em Bull. Polish Acad. Sci. Math.}, 39(1-2):91--100, 1991.

\bibitem{MR1710996}
J\'{o}zef~H. Przytycki and Adam~S. Sikora.
\newblock On skein algebras and {${\rm SL}_2({\mathbb{C}})$}-character varieties.
\newblock {\em Topology}, 39(1):115--148, 2000.

\bibitem{MR1036112}
Nicolai~Y. Reshetikhin and Vladimir~G. Turaev.
\newblock Ribbon graphs and their invariants derived from quantum groups.
\newblock {\em Comm. Math. Phys.}, 127(1):1--26, 1990.

\bibitem{MR535074}
Roger~W. Richardson.
\newblock Commuting varieties of semisimple {L}ie algebras and algebraic groups.
\newblock {\em Compositio Math.}, 38(3):311--327, 1979.

\bibitem{MR3947640}
Peter Samuelson.
\newblock Iterated torus knots and double affine {H}ecke algebras.
\newblock {\em Int. Math. Res. Not. IMRN}, 2019(9):2848--2893, 2017.

\bibitem{Santharoubane_2022}
Ramanujan Santharoubane.
\newblock Algebraic generators of the skein algebra of a surface.
\newblock 2018.

\bibitem{MR2893651}
Nikolai Saveliev.
\newblock {\em Lectures on the topology of 3-manifolds}.
\newblock De Gruyter Textbook. Walter de Gruyter \& Co., Berlin, revised edition, 2012.
\newblock An introduction to the Casson invariant.

\bibitem{MR3205770}
Adam~S. Sikora.
\newblock Character varieties of abelian groups.
\newblock {\em Math. Z.}, 277(1-2):241--256, 2014.

\bibitem{MR1862614}
Michael Thaddeus.
\newblock Mirror symmetry, {L}anglands duality, and commuting elements of {L}ie groups.
\newblock {\em Internat. Math. Res. Notices}, 2001(22):1169--1193, 2001.

\bibitem{MR964255}
Vladimir~G. Turaev.
\newblock The {C}onway and {K}auffman modules of a solid torus.
\newblock {\em Zap. Nauchn. Sem. Leningrad. Otdel. Mat. Inst. Steklov. (LOMI)}, 167:79--89, 190, 1988.

\bibitem{MR2654259}
Vladimir~G. Turaev.
\newblock {\em Quantum invariants of knots and 3-manifolds}, volume~18 of {\em De Gruyter Studies in Mathematics}.
\newblock Walter de Gruyter \& Co., Berlin, revised edition, 2010.

\bibitem{Waldhausen_1968}
Friedhelm Waldhausen.
\newblock On irreducible 3-manifolds which are sufficiently large.
\newblock {\em The Annals of Mathematics}, 87(1):56, January 1968.

\bibitem{wang2024stated}
Zhihao Wang.
\newblock Stated $sl_n$-skein modules, roots of unity, and tqft, 2024.

\end{thebibliography}

\newpage
\appendix

\chapter{Diagrammatic Calculations}
Throughout all of these calculation, we use positive integers placed at the bottom of the diagrams to indicate the relative height ordering of the tangle endpoints, where larger values correspond to lower heights. For each diagram, starting from the leftmost endpoint moving clockwise with respect to our marked point, we assign these integers to the endpoints. The integers are read from left to right at the bottom, corresponding to this clockwise order. For example, $X_{1,0}(-,-)X_{2,0}(+,+) = \begin{tikzpicture}[baseline=-1]
    \MarkedTorusBackground[4][2][][][3][1]
    \draw[thick] (0, 0.15) -- (1, 0.15) -- (2, 0.15);
    \draw[thick] (0.7, -1) to[out=90, in=270] (0.7, 0) to[out=90, in=180] (1, 0.15) to[out=120, in=270] (0.7, 1);
    \node[draw, circle, inner sep=0pt, minimum size=3pt, fill=white] at (1, 0.15) {};
    \node at (0.83, -0.25) {\footnotesize{$+$}};
    \node at (0.7, 0.28) {\footnotesize{$-$}};
    \node at (1.05, 0.4) {\footnotesize{$+$}};
    \node at (1.3, 0) {\footnotesize{$-$}};
\end{tikzpicture}$ where our heights correspond to
\begin{tikzpicture}[baseline=-1]
    \draw[gray!40, thick, fill=gray!40, domain=-45:225] plot ({cos(\x)}, {sin(\x)}) to[out=45, in=130] (0.71, -0.71);
    \draw[thick] (-0.71, -0.71) to[out=45, in=135] (0.71, -0.71);
    \node[draw, circle, inner sep=0pt, minimum size=4pt, fill=white] (p1) at (0,-0.41) {};
    \draw[thick] (p1) -- (-0.93, -0.37);
    \draw[thick] (p1) -- (-0.91, 0.41);
    \draw[thick] (p1) -- (-0.2, 0.98);
    \draw[thick] (p1) -- (0.88,0.48);
    \node at (-0.7, -0.2) {\footnotesize{4}};
    \node at (-0.5, 0.3) {\footnotesize{2}};
    \node at (0, 0.4) {\footnotesize{3}};
    \node at (0.6, -0.1) {\footnotesize{1}};
\end{tikzpicture}. We will also use our previous notation of $\widetilde{X}_{3,0}(\mu, \nu)$ corresponding to the the $(1,-1)$-tangle with $0$ twists and $\widetilde{Y}_3$ corresponding to the $(1,-1)$-curve.

\section{Commuting Relation for $X_{1,0}(-,-)$ and $X_{2,0}(+,+)$}\label{section:CommRelCalc}
\vspace*{-\baselineskip}
\begin{align*}
    & X_{1,0}(-,-)X_{2,0}(+,+) = \begin{tikzpicture}[baseline=-1]
        \MarkedTorusBackground[4][2][][][3][1]
        \draw[thick] (0, 0.15) -- (1, 0.15) -- (2, 0.15);
        \draw[thick] (0.7, -1) to[out=90, in=270] (0.7, 0) to[out=90, in=180] (1, 0.15) to[out=120, in=270] (0.7, 1);
        \node[draw, circle, inner sep=0pt, minimum size=3pt, fill=white] at (1, 0.15) {};
        \node at (0.83, -0.25) {\footnotesize{$+$}};
        \node at (0.7, 0.28) {\footnotesize{$-$}};
        \node at (1.05, 0.4) {\footnotesize{$+$}};
        \node at (1.3, 0) {\footnotesize{$-$}};
    \end{tikzpicture}
    = q \begin{tikzpicture}[baseline=-1]
        \MarkedTorusBackground[4][3][][][2][1]
        \draw[thick] (0, 0.15) -- (1, 0.15) -- (2, 0.15);
        \draw[thick] (0.7, -1) to[out=90, in=270] (0.7, 0) to[out=90, in=180] (1, 0.15) to[out=120, in=270] (0.7, 1);
        \node[draw, circle, inner sep=0pt, minimum size=3pt, fill=white] at (1, 0.15) {};
        \node at (0.83, -0.25) {\footnotesize{$+$}};
        \node at (0.7, 0.28) {\footnotesize{$-$}};
        \node at (1.05, 0.4) {\footnotesize{$+$}};
        \node at (1.3, 0) {\footnotesize{$-$}};
    \end{tikzpicture}\\
    &= q^{4} \begin{tikzpicture}[baseline=-1]
        \MarkedTorusBackground[3][4][][][2][1]
        \draw[thick] (0, 0.15) -- (1, 0.15) -- (2, 0.15);
        \draw[thick] (0.7, -1) to[out=90, in=270] (0.7, 0) to[out=90, in=180] (1, 0.15) to[out=120, in=270] (0.7, 1);
        \node[draw, circle, inner sep=0pt, minimum size=3pt, fill=white] at (1, 0.15) {};
        \node at (0.83, -0.25) {\footnotesize{$+$}};
        \node at (0.7, 0.28) {\footnotesize{$-$}};
        \node at (1.05, 0.4) {\footnotesize{$+$}};
        \node at (1.3, 0) {\footnotesize{$-$}};
    \end{tikzpicture} - q^{5/2}(q^2 - q^{-2}) \widetilde{X}_{3,0}(+,-)\\
    &= q^{7} \begin{tikzpicture}[baseline=-1]
        \MarkedTorusBackground[3][4][][][1][2]
        \draw[thick] (0, 0.15) -- (1, 0.15) -- (2, 0.15);
        \draw[thick] (0.7, -1) to[out=90, in=270] (0.7, 0) to[out=90, in=180] (1, 0.15) to[out=120, in=270] (0.7, 1);
        \node[draw, circle, inner sep=0pt, minimum size=3pt, fill=white] at (1, 0.15) {};
        \node at (0.83, -0.25) {\footnotesize{$+$}};
        \node at (0.7, 0.28) {\footnotesize{$-$}};
        \node at (1.05, 0.4) {\footnotesize{$+$}};
        \node at (1.3, 0) {\footnotesize{$-$}};
    \end{tikzpicture} - q^{11/2}(q^2 - q^{-2})
    \begin{tikzpicture}[baseline=-1]
        \MarkedTorusBackground[1][2]
        \draw[thick] (0.7, 1) to[out=270, in=180] (2, 0.15);
        \draw[thick] (0, 0.15) -- (1, 0.15);
        \draw[thick] (1, 0.15) to[out=180, in=90] (0.7, 0) to[in=270, out=90] (0.7, -1);
        \node[draw, circle, inner sep=0pt, minimum size=3pt, fill=white] at (1, 0.15) {};
        \node at (0.83, -0.25) {\footnotesize{$+$}};
        \node at (0.75, 0.28) {\footnotesize{$-$}};
    \end{tikzpicture} - q^{5/2}(q^2 - q^{-2}) \widetilde{X}_{3,0}(+,-)\\
    &= q^{10} X_{2,0}(+,+) X_{1,0}(-,-) - q^{17/2}(q^2 - q^{-2}) \begin{tikzpicture}[baseline=-1]
        \MarkedTorusBackground[1][2]
        \draw[thick] (0, 0.15) -- (1, 0.15);
        \draw[line width=3mm, gray!40] (0.6, -0.8) to[out=270, in=90] (0.6, 0.2);
        \draw[thick] (2, 0.15) to[out=180, in=0] (1.2, 0.5) to[out=180, in=90] (0.6, 0) to[in=270, out=90] (0.6, -1);
        \draw[line width=3mm, gray!40] (0.6, 0.9) to[out=270, in=90] (1, 0.2);
        \draw[thick] (0.7, 1) to[out=270, in=90] (1, 0.15);
        \node[draw, circle, inner sep=0pt, minimum size=3pt, fill=white] at (1, 0.15) {};
        \node at (0.8, -0.25) {\footnotesize{$-$}};
        \node at (1.2, 0.25) {\footnotesize{$+$}};
    \end{tikzpicture} - q^{11/2}(q^2 - q^{-2})
    \begin{tikzpicture}[baseline=-1]
        \MarkedTorusBackground[1][2]
        \draw[thick] (0.7, 1) to[out=270, in=180] (2, 0.15);
        \draw[thick] (0, 0.15) -- (1, 0.15);
        \draw[thick] (1, 0.15) to[out=180, in=90] (0.7, 0) to[in=270, out=90] (0.7, -1);
        \node[draw, circle, inner sep=0pt, minimum size=3pt, fill=white] at (1, 0.15) {};
        \node at (0.83, -0.25) {\footnotesize{$+$}};
        \node at (0.75, 0.28) {\footnotesize{$-$}};
    \end{tikzpicture}\\
    &\phantom{=} - q^{5/2}(q^2 - q^{-2}) \widetilde{X}_{3,0}(+,-)\\
    &= q^{10} X_{2,0}(+,+) X_{1,0}(-,-) - q^{17/2}(q^2 - q^{-2}) \left( q^{3/2} \widetilde{Y}_{3,0} + \widetilde{X}_{3,-\frac{1}{2}}(-,+) + q^{2}X_{3,0}(-,+) \right)\\
    &\phantom{=} - q^{11/2}(q^2 - q^{-2}) \left( q^{-3} \widetilde{X}_{3,-\frac{1}{2}}(+,-) + q^{-3/2}\widetilde{Y}_{3,0} \right) - q^{5/2}(q^2 - q^{-2}) \widetilde{X}_{3,0}(+,-)\\
    &= q^{10} X_{2,0}(+,+) X_{1,0}(-,-) - q^{7}(q^2 - q^{-2})(q^3 + q^{-3}) \widetilde{Y}_{3,0} - q^{5/2}(q^2 - q^{-2}) \widetilde{X}_{3,0}(+,-)\\
    &\phantom{=} - (q^2 - q^{-2}) \left(q^{17/2} \widetilde{X}_{3,-\frac{1}{2}}(-,+) + q^{21/2} X_{3,0}(-,+) + q^{5/2} \widetilde{X}_{3,-\frac{1}{2}}(+,-)\right)\\
    &= q^{10} X_{2,0}(+,+) X_{1,0}(-,-) - q^{21/2} (q^2 - q^{-2}) X_{3,0}(-,+) - q^{5/2}(q^2 - q^{-2}) \widetilde{X}_{3,0}(+,-)\\
    &\phantom{=} - q^{11/2}(q^2 - q^{-2}) \left( q^3 \widetilde{X}_{3,-\frac{1}{2}}(-,+) + q^{-3} \widetilde{X}_{3,-\frac{1}{2}}(+,-)\right) - q^{7}(q^2 - q^{-2})(q^3 + q^{-3}) \widetilde{Y}_{3,0}\\
    &= q^{10} X_{2,0}(+,+) X_{1,0}(-,-) - q^{13/2} (q^2 - q^{-2})\left( q^{4} X_{3,0}(-,+) + q^{-4} \widetilde{X}_{3,0}(+,-) \right)\\
    &\phantom{=} - q^{11/2}(q^2 - q^{-2}) \left( q^3 \widetilde{X}_{3,-\frac{1}{2}}(-,+) + q^{-3} \widetilde{X}_{3,-\frac{1}{2}}(+,-)\right) - q^{7}(q^2 - q^{-2})(q^3 + q^{-3}) \widetilde{Y}_{3,0}
\end{align*}

\section{Image of $\varphi_{\mathcal{E}}$}\label{appendix:ImageOfT6}
Following the calculations of $\varphi_{\mathcal{E}}(y_1)$, we can similarly determine the images of other tangles as well.
In particular, we left-multiply each diagram, $\alpha$, by the image of an appropriate monomial from $\mathbb{T}_{+}^6$ under $\psi_{\mathcal{E}}$, ensuring that the resulting product is expressed solely in terms of the images of elements from $\mathbb{T}_{+}^6$.
Since the composition of injections yields the identity map on $\mathbb{T}_{+}^{6}$, we then left-multiply by the inverse monomial in $T^6$ to explicitly find $\varphi_{\mathcal{E}}(\alpha)$.
Unless stated otherwise, every tangle is assumed to have positive states.

\subsection{Longitude}

\begin{align*}
    \psi_{\mathcal{E}}\left( x_1 x_3 \right) y_2 &= q^{-1}

    \right)\\
    &\Rightarrow \varphi_{\mathcal{E}} \left( y_2 \right) = (x_1 x_3)^{-1} (x_1 x_3) \varphi_{\mathcal{E}} \left( y_2 \right) = (x_1 x_3)^{-1} \varphi_{\mathcal{E}} \left( \psi_{\mathcal{E}}(x_1 x_3) y_2 \right)\\
    &= q^{3}x_3^{-1}x_1^{-1} x_3^2 + q^{-1} x_3^{-1} x_1^{-1} x_4 x_2 + q^{-1} x_3^{-1} x_1^{-1} x_1^2\\
    &= qx_1x_3^{-1} + qx_1^{-1}x_3 + q^{-1}x_1^{-1}x_2x_3^{-1}x_4.
\end{align*}

\vfill
\subsection{$(1,1)$-Curve}
\begin{align*}
    \psi_{\mathcal{E}}\left( x_4 x_2 x_1 \right) y_3 &= \left( q^{-3/2}
    %
    \right)$}\\
    &\Rightarrow \varphi_{\mathcal{E}}(y_3) = (x_4 x_2 x_1)^{-1} (x_4 x_2 x_1) \varphi_{\mathcal{E}}(y_3) = (x_4 x_2 x_1)^{-1} \varphi_{\mathcal{E}}\left( \psi_{\mathcal{E}} (x_4 x_2 x_1) y_3 \right)\\
    &= q^7x_1^{-1}x_2^{-1}x_4^{-1}x_2x_1^2 + q^{2}x_1^{-1}x_2^{-1}x_4^{-1}x_5x_1x_3 + q^{-1}x_1^{-1}x_2^{-1}x_3^{2} + q^{-1} x_1^{-1}x_4\\
    &= q^{-1}x_1x_4^{-1} + q^{-1}x_1^{-1}x_4 + q^{-1}x_1^{-1}x_2^{-1}x_3^{2} + x_2^{-1}x_3x_4^{-1}x_5.
\end{align*}

\vfill
\subsection{$(1,-1)$-Tangle}
\begin{align*}
    \psi_{\mathcal{E}}\left( x_3 \right)
    %
\\
    \Rightarrow \varphi_{\mathcal{E}}(\alpha) &= x_3^{-1} x_3 \varphi_{\mathcal{E}}(\alpha) = x_3^{-1} \varphi_{\mathcal{E}}\left( \psi_{\mathcal{E}} (x_3) \alpha \right)\\
    &= q^{5/2}x_3^{-1} x_2 x_4 + q^{-3/2} x_3^{-1} x_1^{2}\\
    &= q^{1/2}x_2x_3^{-1}x_4 + q^{5/2}x_1^{2}x_3^{-1}
\end{align*}

\subsection{$X_{1, \frac{1}{2}}(+,+)$}
\begin{align*}
    \psi_{\mathcal{E}}(x_2 x_3)
    %
\\
    \Rightarrow \varphi_{\mathcal{E}}(X_{1,\frac{1}{2}})
    &= (x_2 x_3)^{-1} (x_2 x_3) \varphi_{\mathcal{E}}\left( X_{1,\frac{1}{2}} \right)\\
    &= (x_2 x_3)^{-1} \varphi_{\mathcal{E}}\left( \psi_{\mathcal{E}} (x_2 x_3) X_{1,\frac{1}{2}} \right)\\
    &= x_3^{-1}x_2^{-1} \varphi_{\mathcal{E}} \left( q^{-2}x_2x_4x_5 + q^{-6}x_1^2x_5 + q^{-7}x_1x_3x_4 \right)\\
    &= q^{-1/2}x_3^{-1}x_4x_5 + q^{11/2}x_1^{2}x_2^{-1}x_3^{-1}x_5 + q^{1/2}x_1x_2^{-1}x_4.
\end{align*}

\subsection{$X_{3, \frac{1}{2}}(+,+)$}
\begin{align*}
    \psi_{\mathcal{E}}(x_2)
    %
\\
    \Rightarrow x_2^{-1} x_2 \varphi_{\mathcal{E}}\left( X_{3, \frac{1}{2}} \right)
    &= x_2^{-1} \varphi_{\mathcal{E}}\left( \psi_{\mathcal{E}}(x_2) X_{3, \frac{1}{2}} \right)\\
    &= x_2^{-1} \varphi_{\mathcal{E}} \left( q^{-3/2}x_1x_5 + q^{-5/2}x_3x_4 \right)\\
    &= x_2^{-1} \left( q^{-1/2}x_1x_5 + q^{-3/2}x_3x_4 \right)\\
    &= q^{3/2}x_1x_2^{-1}x_5 + q^{-3/2}x_2^{-1}x_3x_4.
\end{align*}

\subsection{$X_{1,1}(+,+)$}
\begin{align*}
    \psi_{\mathcal{E}}(x_1 x_4)&
    %
\right)^2 \\
\end{align*}

$$= q^{1/2}x_1^{-1}x_5^2$$
$$\resizebox{0.9\width}{!}{$+ q^{-7/2}x_1^{-1}x_4^{-1}x_5
\left( q^{-1/2}x_3^{-1}x_4x_5 + q^{11/2}x_1^{2}x_2^{-1}x_3^{-1}x_5 + q^{1/2}x_1x_2^{-1}x_4 \right)
\left( q^{1/2}x_2x_3^{-1}x_4 + q^{5/2}x_1^{2}x_3^{-1} \right)$}$$
$$\resizebox{0.9\width}{!}{$+ q^{-11/2}x_1^{-1}x_4^{-1}x_5
\left( q^{-1/2}x_3^{-1}x_4x_5 + q^{11/2}x_1^{2}x_2^{-1}x_3^{-1}x_5 + q^{1/2}x_1x_2^{-1}x_4 \right)
\left( q^{3/2}x_1x_2^{-1}x_5 + q^{-3/2}x_2^{-1}x_3x_4 \right)$}$$
$$\resizebox{0.9\width}{!}{$+ q^{-9/2}x_1^{-1}
\left( q^{-1/2}x_3^{-1}x_4x_5 + q^{11/2}x_1^{2}x_2^{-1}x_3^{-1}x_5 + q^{1/2}x_1x_2^{-1}x_4 \right)^2$}$$

$$= q^{1/2}x_1^{-1}x_5^2 + q^{-7/2}x_1^{-1}x_4^{-1}x_5 \left(x_2 x_3^{-2}x_4^2 x_5+q^{8} x_1^2 x_3^{-2} x_4 x_5 + q^4 x_1^2 x_3^{-2} x_4 x_5 + q^{16}\right.$$
$$\resizebox{0.9\width}{!}{$\left.x_1^4 x_2^{-1} x_3^{-2} x_5 + q^3 x_1 x_3^{-1} x_4^2 + q^{9} x_1^3 x_2^{-1} x_3^{-1} x_4\right) + q^{-11/2}x_1^{-1}x_4^{-1}x_5\left(q^3 x_1 x_2^{-1} x_3^{-1} x_4 x_5^2 + q^{-2} x_2^{-1} x_4^2 x_5 \right.$}$$
$$\left.+q^{13} x_1^3 x_2^{-2} x_3^{-1} x_5^2 + q^6 x_1^2 x_2^{-2} x_4 x_5 + q^{2} x_1^2 x_2^{-2} x_4 x_5 + q^{-3} x_1 x_2^{-2} x_3 x_4^2\right)$$
$$\resizebox{0.9\width}{!}{$+ q^{-9/2}x_1^{-1} \left(q^{-3} x_3^{-2} x_4^2 x_5^2 + q^{9} x_1^2 x_2^{-1} x_3^{-2} x_4 x_5^2 + q^{2} x_1 x_2^{-1} x_3^{-1} x_4^2 x_5 + q^5 x_1^2 x_2^{-1} x_3^{-2} x_4 x_5^2\right.$}$$
$$\resizebox{0.9\width}{!}{$\left.+ q^{21} x_1^4 x_2^{-2} x_3^{-2} x_5^2 + q^{12} x_1^3 x_2^{-2} x_3^{-1} x_4 x_5
+ q^{-2} x_1 x_2^{-1} x_3^{-1} x_4^2 x_5 + q^{8} x_1^3 x_2^{-2} x_3^{-1} x_4 x_5 + q x_1^2 x_2^{-2} x_4^2\right)$}$$

$$\resizebox{0.9\width}{!}{$= q^{1/2} x_1^{-1} x_5^2 + q^{-7/2} x_1^{-1} x_2 x_3^{-2} x_4 x_5^{2} + \left( q^{9/2} + q^{1/2}\right) x_1 x_3^{-2} x_5^2 + q^{25/2} x_1^3 x_2^{-1} x_3^{-2} x_4^{-1} x_5^2$}$$
$$\resizebox{0.9\width}{!}{$ + q^{-1/2} x_3^{-1} x_4 x_5 + q^{11/2} x_1^2 x_2^{-1} x_3^{-1} x_5 + q^{3/2} x_2^{-1} x_3^{-1} x_5^3 + q^{-7/2} x_1^{-1} x_2^{-1} x_4 x_5^2 + q^{23/2} x_1^2 x_2^{-2} x_3^{-1} x_4^{-1} x_5^3$}$$
$$\resizebox{0.9\width}{!}{$+\left(q^{1/2} + q^{-7/2}\right) x_1 x_2^{-2} x_5^2 + q^{-9/2} x_2^{-2} x_3 x_4 x_5 + q^{-15/2} x_1^{-1} x_3^{-2} x_4^{2} x_5^2 + \left(q^{9/2} + q^{1/2}\right) x_1 x_2^{-1} x_3^{-2} x_4 x_5^2$}$$
$$\resizebox{0.9\width}{!}{$+ \left(q^{-5/2} + q^{-13/2}\right) x_2^{-1} x_3^{-1} x_4^2 x_5 + q^{33/2} x_1^3 x_2^{-2} x_3^{-2} x_5^2 + \left(q^{15/2} + q^{7/2}\right) x_1^2 x_2^{-2} x_3^{-1} x_4 x_5 + q^{-3/2} x_1 x_2^{-2} x_4^2.$}$$

\vfill
\subsection{$X_{1,-\frac{1}{2}}$}
\begin{align*}
    \psi_{\mathcal{E}}(x_4) \cdot
    \begin{tikzpicture}[baseline=-1]
        \MarkedTorusBackground[2][1]
        \draw[thick] (0, 0.15) -- (1, 0.15) to[out=180, in=120] (0.7, -0.3) to[out=300, in=180] (1,-0.45) to[out=0, in=180] (2, 0.15);
        \node[draw, circle, inner sep=0pt, minimum size=3pt, fill=white] at (1, 0.15) {};
    \end{tikzpicture}
    &= q^{-1/2}
    \begin{tikzpicture}[baseline=-1]
        \MarkedTorusBackground[2][1]
        \draw[thick] (0, 0.15) -- (1, 0.15) to[out=180, in=120] (0.7, -0.3) to[out=300, in=180] (1,-0.45) to[out=0, in=180] (2, 0.15);
        \draw[line width=2mm, gray!40] (1.3, -0.5) -- (1.3, 0);
        \draw[thick] (1.3, -1) to[out=90, in=270] (1.3, 0) to[out=90, in=0] (1, 0.15) to[out=60, in=270] (1.3, 1);
        \node[draw, circle, inner sep=0pt, minimum size=3pt, fill=white] at (1, 0.15) {};
    \end{tikzpicture} \\
    &= q^{1/2}
    \begin{tikzpicture}[baseline=-1]
        \MarkedTorusBackground[4][3][][][2][1]
        \draw[thick] (0.7, -1) to[out=90, in=270] (0.7, 0) to[out=90, in=180] (1, 0.15) to[out=120, in=270] (0.7, 1);
        \draw[thick] (0, 0.15) -- (2, 0.15);
        \node[draw, circle, inner sep=0pt, minimum size=3pt, fill=white] at (1, 0.15) {};
    \end{tikzpicture} + q^{-3/2}
    \begin{tikzpicture}[baseline=-1]
        \MarkedTorusBackground[4][3][][][2][1]
        \draw[thick] (1, 0.15) to[out=0, in=90] (1.3, -0.2) to[out=270, in=0] (1, -0.5) to[out=180, in=270] (0.7, -0.2) to[out=90, in=180] (1, 0.15);
        \draw[thick] (0.5, -1) -- (0.5, -0.2) to[out=90, in=180] (1, 0.15) to[out=45, in=180] (2, 0.5);
        \draw[thick] (0, 0.5) to[out=0, in=270] (0.5, 1);
        \node[draw, circle, inner sep=0pt, minimum size=3pt, fill=white] at (1, 0.15) {};
    \end{tikzpicture} \\
    &= q^{3/2}
    \begin{tikzpicture}[baseline=-1]
        \MarkedTorusBackground[4][2][][][3][1]
        \draw[thick] (0.7, -1) to[out=90, in=270] (0.7, 0) to[out=90, in=180] (1, 0.15) to[out=120, in=270] (0.7, 1);
        \draw[thick] (0, 0.15) -- (2, 0.15);
        \node[draw, circle, inner sep=0pt, minimum size=3pt, fill=white] at (1, 0.15) {};
    \end{tikzpicture} + q^{1/2}
    \begin{tikzpicture}[baseline=-1]
        \MarkedTorusBackground[4][2][][][1][3]
        \draw[thick] (1, 0.15) to[out=0, in=90] (1.3, -0.2) to[out=270, in=0] (1, -0.5) to[out=180, in=270] (0.7, -0.2) to[out=90, in=180] (1, 0.15);
        \draw[thick] (0.5, -1) -- (0.5, -0.2) to[out=90, in=180] (1, 0.15) to[out=45, in=180] (2, 0.5);
        \draw[thick] (0, 0.5) to[out=0, in=270] (0.5, 1);
        \node[draw, circle, inner sep=0pt, minimum size=3pt, fill=white] at (1, 0.15) {};
    \end{tikzpicture} \\
    \Rightarrow \varphi_{\mathcal{E}}(X_{1,-\frac{1}{2}}(+,+))
    &= x_4^{-1}\left( q^{5/2} x_1 x_2 + q^{3/2} x_3 x_5 \right) \\
    &= q^{-7/2} x_1 x_2 x_4^{-1} + q^{-1/2} x_3 x_4^{-1} x_5
\end{align*}

\subsection{Boundary Curve}\label{appendix:BoundaryCurveCalc}
\begin{align*}
    \psi_{\mathcal{E}}(x_1) \cdot \partial
    &= q^{-1/2}
    \begin{tikzpicture}[baseline=-1]
        \MarkedTorusBackground[2][1]
        \draw[thick] (1, 0) circle (0.4);
        \draw[line width=2mm, gray!40] (0.5, 0.15) -- (0.7, 0.15);
        \draw[line width=2mm, gray!40] (1.3, 0.15) -- (1.5, 0.15);
        \draw[thick] (0, 0.15) -- (2, 0.15);
        \node[draw, circle, inner sep=0pt, minimum size=3pt, fill=white] at (1, 0.15) {};
    \end{tikzpicture} \\
    &= q^{3/2}
    \begin{tikzpicture}[baseline=-1]
        \MarkedTorusBackground[2][1]
        \draw[thick] (0, 0.15) -- (1, 0.15) to[out=180, in=120] (0.7, -0.3) to[out=300, in=180] (1,-0.45) to[out=0, in=180] (2, 0.15);
        \node[draw, circle, inner sep=0pt, minimum size=3pt, fill=white] at (1, 0.15) {};
    \end{tikzpicture} + q^{-1/2}
    \begin{tikzpicture}[baseline=-1]
        \MarkedTorusBackground[2][1]
        \draw (0, 0.5) -- (2, 0.5);
        \draw[thick] (1, 0.15) to[out=0, in=90] (1.3, -0.2) to[out=270, in=0] (1, -0.5) to[out=180, in=270] (0.7, -0.2) to[out=90, in=180] (1, 0.15);
        \node[draw, circle, inner sep=0pt, minimum size=3pt, fill=white] at (1, 0.15) {};
    \end{tikzpicture} + q^{-5/2}
    \begin{tikzpicture}[baseline=-1]
        \MarkedTorusBackground[2][1]
        \draw[thick] (2, 0.15) -- (1, 0.15) to[out=0, in=60] (1.3, -0.3) to[out=240, in=0] (1,-0.45) to[out=180, in=0] (0, 0.15);
        \node[draw, circle, inner sep=0pt, minimum size=3pt, fill=white] at (1, 0.15) {};
    \end{tikzpicture} \\
    \Rightarrow \psi_{\mathcal{E}}(\partial)
    &= \resizebox{0.95\width}{!}{$x_{1}^{-1} \left[ q^{3/2}\left( q^{-7/2} x_1 x_2 x_4^{-1} + q^{-1/2} x_3 x_4^{-1} x_5 \right) + \left( q^{-1}x_2x_3^{-1} + q^{-1}x_2^{-1}x_3 + qx_1^{2}x_3^{-1}x_4^{-1}\right.\right. $}\\
    &\phantom{=} \resizebox{0.9\width}{!}{$\left.\left. + q^{2}x_1x_2^{-1}x_4^{-1}x_5 \right)x_5 + q^{-5/2} \left( q^{-1/2}x_3^{-1}x_4x_5 + q^{11/2}x_1^{2}x_2^{-1}x_3^{-1}x_5 + q^{1/2}x_1x_2^{-1}x_4 \right) \right] $}\\
    &= q^{-2}x_2^{-1}x_4 + q^{-2}x_2x_4^{-1} + qx_1^{-1}x_3x_4^{-1}x_5 + qx_1x_3^{-1}x_4^{-1}x_5 + q^{-3}x_1^{-1}x_3^{-1}x_4x_5 \\
    &\phantom{ =} + q^{3}x_1x_2^{-1}x_3^{-1}x_5 + q^{-1}x_1^{-1}x_2x_3^{-1}x_5 + q^{-1}x_1^{-1}x_2^{-1}x_3x_5 + q^{2}x_2^{-1}x_4^{-1}x_5^{2}.
\end{align*}

\section{Parallel Tangle}\label{appendix:ParaTangCalc}
There are various ways to compute the parallel tangle corresponding to $x_5$ with states $\mu$ and $\nu$ in $\mathscr{S}(T^2 \setminus D^2)$.
$$\begin{tikzpicture}[baseline=-1]
    \MarkedTorusBackground
    \draw[thick] (1, 0.15) to[out=0, in=90] (1.3, -0.2) to[out=270, in=0] (1, -0.5) to[out=180, in=270] (0.7, -0.2) to[out=90, in=200] (0.9, 0.15);
    \node[draw, circle, inner sep=0pt, minimum size=3pt, fill=white] at (1, 0.15) {};
    \node at (0.75, 0.25) {$\mu$};
    \node at (1.25, 0.25) {$\nu$};
\end{tikzpicture}$$
Below, we present three analogous equations corresponding to calculating the parallel boundary closed curve in $K_q(T^2 \setminus D^2)$, denoted $\partial$.
When computing $\partial$ in $K_q(T^2 \setminus D^2)$, one would rewrite $Y_1Y_2Y_3$ without any crossings and rearrange the terms to obtain an explicit formula for $\partial$.
The method for our parallel tangle is similar, however, with only two points in the resulting tangle touching the marking in the boundary, we have the option to designate exactly one of the meridian, longitude, or $(1,1)$-curve to be a tangle instead (i.e. we substitute a $X_{i,k} (\mu, \nu)$ for one of the $Y_i$s) and get a slightly different equality back.

\begin{align*}
    X_{1,0}(\mu, \nu) Y_2 Y_3 &= 
    \begin{tikzpicture}[baseline=-1]
        \MarkedTorusBackground[2][1]
        \draw[thick] (0.3, -1) -- (0.3, 0) to[out=90, in=180] (1, 0.7) -- (2, 0.7);
        \draw[thick] (0, 0.7) to[out=0, in=270] (0.3, 1);
        \draw[line width=3mm, gray!40] (0.7, -1) -- (0.7, 1);
        \draw[thick] (0.7, -1) -- (0.7, 1);
        \draw[line width=3mm, gray!40] (0.2, 0.15) -- (0.8, 0.15);
        \draw[thick] (0, 0.15) -- (1, 0.15) -- (2, 0.15);
        \node[draw, circle, inner sep=0pt, minimum size=3pt, fill=white] at (1, 0.15) {};
    \end{tikzpicture}
    = q \begin{tikzpicture}[baseline=-1]
        \MarkedTorusBackground[2][1]
        \draw[thick] (0, 0.5) -- (2, 0.5);
        \draw[thick] (0, 0.15) -- (1, 0.15) -- (2, 0.15);
        \node[draw, circle, inner sep=0pt, minimum size=3pt, fill=white] at (1, 0.15) {};
    \end{tikzpicture}
    + q^{-1} \begin{tikzpicture}[baseline=-1]
        \MarkedTorusBackground[2][1]
        \draw[thick] (0.7, 1) to[out=270, in=90] (0.3, 0) -- (0.3, -1);
        \draw[thick] (0, 0.7) to[out=0, in=270] (0.3, 1);
        \draw[thick] (0.7, -1) -- (0.7, 0) to[out=90, in=180] (1, 0.7) -- (2, 0.7);
        \draw[line width=3mm, gray!40] (0.2, 0.15) -- (0.8, 0.15);
        \draw[thick] (0, 0.15) -- (1, 0.15) -- (2, 0.15);
        \node[draw, circle, inner sep=0pt, minimum size=3pt, fill=white] at (1, 0.15) {};
    \end{tikzpicture}\\
    &= q \begin{tikzpicture}[baseline=-1]
        \MarkedTorusBackground[2][1]
        \draw[thick] (0, 0.5) -- (2, 0.5);
        \draw[thick] (0, 0.15) -- (1, 0.15) -- (2, 0.15);
        \node[draw, circle, inner sep=0pt, minimum size=3pt, fill=white] at (1, 0.15) {};
    \end{tikzpicture}
    + \begin{tikzpicture}[baseline=-1]
        \MarkedTorusBackground[2][1]
        \draw[thick] (0.7, 1) to[out=270, in=0] (0, 0.15);
        \draw[thick] (0, 0.7) to[out=0, in=270] (0.3, 1);
        \draw[thick] (0.7, -1) -- (0.7, 0) to[out=90, in=180] (1, 0.7) -- (2, 0.7);
        \draw[line width=3mm, gray!40] (0.3, -1) to[out=90, in=180] (0.8, 0.05);
        \draw[thick] (0.3, -1) to[out=90, in=180] (1, 0.15) -- (2, 0.15);
        \node[draw, circle, inner sep=0pt, minimum size=3pt, fill=white] at (1, 0.15) {};
    \end{tikzpicture}
    + q^{-2} \begin{tikzpicture}[baseline=-1]
        \MarkedTorusBackground[2][1]
        \draw[thick] (0.7, -1) -- (0.7, 0) to[out=90, in=270] (1.3, 0.7) -- (1.3, 1);
        \draw[line width=3mm, gray!40] (0.7, 0.5) to[out=270, in=180] (0.88, 0.27);
        \draw[thick] (0.7, 1) to[out=270, in=180] (1, 0.15) to[out=0, in=90] (1.3, -0.7) -- (1.3, -1);
        \node[draw, circle, inner sep=0pt, minimum size=3pt, fill=white] at (1, 0.15) {};
    \end{tikzpicture}\\
    &= q \begin{tikzpicture}[baseline=-1]
        \MarkedTorusBackground[2][1]
        \draw[thick] (0, 0.5) -- (2, 0.5);
        \draw[thick] (0, 0.15) -- (1, 0.15) -- (2, 0.15);
        \node[draw, circle, inner sep=0pt, minimum size=3pt, fill=white] at (1, 0.15) {};
    \end{tikzpicture}
    + q \begin{tikzpicture}[baseline=-1]
        \MarkedTorusBackground[2][1]
        \draw[thick] (0.5, -1) -- (0.5, -0.2) to[out=90, in=180] (1, 0.15) to[out=45, in=180] (2, 0.5);
        \draw[thick] (0, 0.5) to[out=0, in=270] (0.5, 1);
        \draw[thick] (0.3, -1) -- (0.3, -0.2) to[out=90, in=180] (1.2, 0.7) -- (2, 0.7);
        \draw[thick] (0, 0.7) to[out=0, in=270] (0.3, 1);
        \node[draw, circle, inner sep=0pt, minimum size=3pt, fill=white] at (1, 0.15) {};
    \end{tikzpicture}
    + q^{-1} \begin{tikzpicture}[baseline=-1]
        \MarkedTorusBackground[2][1]
        \draw[thick] (1, 0.15) to[out=60, in=270] (1.2, 0.45) to[out=90, in=0] (1, 0.65) to[out=180, in=90] (0.8, 0.45) to[out=270, in=120] (1, 0.15);
        \node[draw, circle, inner sep=0pt, minimum size=3pt, fill=white] at (1, 0.15) {};
    \end{tikzpicture} \\
    &\phantom{=} + q^{-1}
    \begin{tikzpicture}[baseline=-1]
        \MarkedTorusBackground[2][1]
        \draw[thick] (1, 0.15) to[out=0, in=90] (1.3, -0.2) to[out=270, in=0] (1, -0.5) to[out=180, in=270] (0.7, -0.2) to[out=90, in=180] (1, 0.15);
        \node[draw, circle, inner sep=0pt, minimum size=3pt, fill=white] at (1, 0.15) {};
    \end{tikzpicture}
    + q^{-3} \begin{tikzpicture}[baseline=-1]
        \MarkedTorusBackground[2][1]
        \draw[thick] (1.3, -1) to[out=90, in=270] (1.3, 0) to[out=90, in=0] (1, 0.15) to[out=60, in=270] (1.3, 1);
        \draw[thick] (0.7, -1) -- (0.7, 1);
        \node[draw, circle, inner sep=0pt, minimum size=3pt, fill=white] at (1, 0.15) {};
    \end{tikzpicture}\\
    &= \resizebox{0.95\width}{!}{$q X_1(\mu, \nu) Y_1 + q X_3(\mu, \nu) Y_3 + q^{-1} C_{\mu}^{\nu} + q^{-1}$} \begin{tikzpicture}[baseline=-1]
        \MarkedTorusBackground[2][1]
        \draw[thick] (1, 0.15) to[out=0, in=90] (1.3, -0.2) to[out=270, in=0] (1, -0.5) to[out=180, in=270] (0.7, -0.2) to[out=90, in=180] (1, 0.15);
        \node[draw, circle, inner sep=0pt, minimum size=3pt, fill=white] at (1, 0.15) {};
    \end{tikzpicture} \resizebox{0.95\width}{!}{$+ q^{-3} X_{2, 1/2}(\mu, \nu)Y_2$}\\
    \Rightarrow \begin{tikzpicture}[baseline=-1]
        \MarkedTorusBackground[2][1]
        \draw[thick] (1, 0.15) to[out=0, in=90] (1.3, -0.2) to[out=270, in=0] (1, -0.5) to[out=180, in=270] (0.7, -0.2) to[out=90, in=180] (1, 0.15);
        \node[draw, circle, inner sep=0pt, minimum size=3pt, fill=white] at (1, 0.15) {};
    \end{tikzpicture}
    &= \resizebox{0.95\width}{!}{$qX_{1,0}(\mu, \nu)Y_2Y_3 - q^2 X_{1,0}(\mu, \nu) Y_1 - q^{-2}X_{2, \frac{1}{2}}(\mu, \nu)Y_2 - q^2 X_{3,0}(\mu, \nu) Y_3 - C_{\mu}^{\nu}$}\\
\end{align*}

\begin{align*}
    Y_1 X_{2,0}(\mu, \nu) Y_3 &= \begin{tikzpicture}[baseline=-1]
        \MarkedTorusBackground[2][1]
        \draw[thick] (0.3, -1) -- (0.3, 0) to[out=90, in=180] (1, 0.7) -- (2, 0.7);
        \draw[thick] (0, 0.7) to[out=0, in=270] (0.3, 1);
        \draw[line width=3mm, gray!40] (1, 0.3) to[out=120, in=270] (0.7, 0.85);
        \draw[thick] (0.7, -1) to[out=90, in=270] (0.7, 0) to[out=90, in=180] (1, 0.15) to[out=120, in=270] (0.7, 1);
        \draw[line width=3mm, gray!40] (0.2, 0.35) -- (1.8, 0.35);
        \draw[thick] (0, 0.35) -- (2, 0.35);
        \draw[thick] (1, 0.15) -- (0.9, 0.25);
        \node[draw, circle, inner sep=0pt, minimum size=3pt, fill=white] at (1, 0.15) {};
    \end{tikzpicture}\\
    &= q \begin{tikzpicture}[baseline=-1]
        \MarkedTorusBackground[2][1]
        \draw[thick] (0, 0.8) to[out=0, in=270] (0.3, 1);
        \draw[thick] (0, 0.5) to[out=0, in=180] (1, 0.8) -- (2, 0.8);
        \draw[line width=3mm, gray!40] (1, 0.3) -- (1, 0.85);
        \draw[thick] (1, -1) to[out=90, in=270] (0.7, 0) to[out=90, in=180] (1, 0.15) -- (1, 1);
        \draw[line width=3mm, gray!40] (0.3, -0.7) to[out=90, in=180] (1.1, 0.5);
        \draw[thick] (0.3, -1) -- (0.3, -0.7) to[out=90, in=180] (1, 0.5) -- (2, 0.5);
        \node[draw, circle, inner sep=0pt, minimum size=3pt, fill=white] at (1, 0.15) {};
    \end{tikzpicture}
    + q^{-1} \begin{tikzpicture}[baseline=-1]
        \MarkedTorusBackground[2][1]
        \draw[thick] (0, 0.65) to[out=0, in=270] (0.5, 1);
        \draw[thick] (0, 0.35) to[out=0, in=90] (0.5, -1);
        \draw[thick] (2, 0.65) -- (1.3, 0.65) to[out=180, in=90] (0.8, 0.5);
        \draw[line width=2mm, gray!40] (1, 0.3) -- (1, 0.85);
        \draw[thick] (1, -1) to[out=90, in=270] (0.7, 0) to[out=90, in=180] (1, 0.15) -- (1, 1);
        \draw[line width=1.7mm, gray!40] (0.8, 0.5) to[out=270, in=180] (1.3, 0.35) -- (2, 0.35);
        \draw[thick] (0.8, 0.5) to[out=270, in=180] (1.3, 0.35) -- (2, 0.35);
        \node[draw, circle, inner sep=0pt, minimum size=3pt, fill=white] at (1, 0.15) {};
    \end{tikzpicture}\\
    &= q^2 \begin{tikzpicture}[baseline=-1]
        \MarkedTorusBackground[2][1]
        \draw[thick] (0, 0.7) to[out=0, in=270] (0.3, 1);
        \draw[thick] (1, 1) to[out=270, in=180] (2, 0.7);
        \draw[thick] (1, -1) to[out=90, in=270] (0.7, 0) to[out=90, in=180] (1, 0.15) to[out=90, in=0] (0, 0.5);
        \draw[line width=3mm, gray!40] (0.3, -0.7) to[out=90, in=180] (1.1, 0.5);
        \draw[thick] (0.3, -1) -- (0.3, -0.7) to[out=90, in=180] (1, 0.5) -- (2, 0.5);
        \node[draw, circle, inner sep=0pt, minimum size=3pt, fill=white] at (1, 0.15) {};
    \end{tikzpicture}
    + \begin{tikzpicture}[baseline=-1]
        \MarkedTorusBackground[2][1]
        \draw[thick] (0, 0.7) to[out=0, in=270] (0.3, 1);
        \draw[thick] (0, 0.5) to[out=0, in=270] (0.5, 1);
        \draw[thick] (0.5, -1) -- (0.5, 0) to[out=90, in=180] (1, 0.15) to[out=90, in=180] (2, 0.7);
        \draw[line width=2mm, gray!40] (1, 0.5) -- (1.4, 0.5);
        \draw[thick] (0.3, -1) -- (0.3, -0.7) to[out=90, in=180] (1, 0.5) -- (2, 0.5);
        \node[draw, circle, inner sep=0pt, minimum size=3pt, fill=white] at (1, 0.15) {};
    \end{tikzpicture}
    - q^{3}\begin{tikzpicture}[baseline=-1]
        \MarkedTorusBackground[2][1]
        \draw[thick] (0, 0.6) to[out=0, in=270] (0.4, 1);
        \draw[thick] (0, 0.15) to[out=0, in=90] (0.4, -1);
        \draw[thick] (2, 0.6) to[out=180, in=270] (1, 1);
        \draw[thick] (1, -1) to[out=90, in=270] (0.7, 0) to[out=90, in=180] (1, 0.15) -- (2, 0.15);
        \node[draw, circle, inner sep=0pt, minimum size=3pt, fill=white] at (1, 0.15) {};
    \end{tikzpicture}
    + q^{-2}\begin{tikzpicture}[baseline=-1]
        \MarkedTorusBackground[2][1]
        \draw[thick] (0, 0.6) to[out=0, in=270] (0.4, 1);
        \draw[thick] (0, 0.15) to[out=0, in=90] (0.4, -1);
        \draw[thick] (1, -1) to[out=90, in=270] (0.7, 0) to[out=90, in=180] (1, 0.15) to[out=60, in=180] (2, 0.6);
        \draw[line width=3mm, gray!40] (1.8, 0.25) to[out=180, in=270] (1.1, 0.8);
        \draw[thick] (2, 0.15) to[out=180, in=270] (1, 1);
        \node[draw, circle, inner sep=0pt, minimum size=3pt, fill=white] at (1, 0.15) {};
    \end{tikzpicture}\\
    &= q \begin{tikzpicture}[baseline=-1]
        \MarkedTorusBackground[2][1]
        \draw[thick] (0, 0.5) -- (2, 0.5);
        \draw[thick] (0, 0.15) -- (1, 0.15) to[out=180, in=120] (0.7, -0.3) to[out=300, in=180] (1,-0.45) to[out=0, in=180] (2, 0.15);
        \node[draw, circle, inner sep=0pt, minimum size=3pt, fill=white] at (1, 0.15) {};
    \end{tikzpicture}
    + q \begin{tikzpicture}[baseline=-1]
        \MarkedTorusBackground[2][1]
        \draw[thick] (0.5, -1) -- (0.5, -0.2) to[out=90, in=180] (1, 0.15) to[out=45, in=180] (2, 0.5);
        \draw[thick] (0, 0.5) to[out=0, in=270] (0.5, 1);
        \draw[thick] (0.3, -1) -- (0.3, -0.2) to[out=90, in=180] (1.2, 0.7) -- (2, 0.7);
        \draw[thick] (0, 0.7) to[out=0, in=270] (0.3, 1);
        \node[draw, circle, inner sep=0pt, minimum size=3pt, fill=white] at (1, 0.15) {};
    \end{tikzpicture}
    + q^{-1}\begin{tikzpicture}[baseline=-1]
        \MarkedTorusBackground[2][1]
        \draw[thick] (1, 0.15) to[out=60, in=270] (1.2, 0.45) to[out=90, in=0] (1, 0.65) to[out=180, in=90] (0.8, 0.45) to[out=270, in=120] (1, 0.15);
        \node[draw, circle, inner sep=0pt, minimum size=3pt, fill=white] at (1, 0.15) {};
    \end{tikzpicture} \\
    &\phantom{=} + q^{-1} \begin{tikzpicture}[baseline=-1]
        \MarkedTorusBackground[2][1]
        \draw[thick] (1, 0.15) to[out=0, in=90] (1.3, -0.2) to[out=270, in=0] (1, -0.5) to[out=180, in=270] (0.7, -0.2) to[out=90, in=180] (1, 0.15);
        \node[draw, circle, inner sep=0pt, minimum size=3pt, fill=white] at (1, 0.15) {};
    \end{tikzpicture}
    + q^{-3}\begin{tikzpicture}[baseline=-1]
        \MarkedTorusBackground[2][1]
        \draw[thick] (0.4, -1) -- (0.4, 1);
        \draw[thick] (1, -1) to[out=90, in=270] (0.7, 0) to[out=90, in=180] (1, 0.15) -- (1, 1);
        \node[draw, circle, inner sep=0pt, minimum size=3pt, fill=white] at (1, 0.15) {};
    \end{tikzpicture}\\\\
    \Rightarrow \begin{tikzpicture}[baseline=-1]
        \MarkedTorusBackground[2][1]
        \draw[thick] (1, 0.15) to[out=0, in=90] (1.3, -0.2) to[out=270, in=0] (1, -0.5) to[out=180, in=270] (0.7, -0.2) to[out=90, in=180] (1, 0.15);
        \node[draw, circle, inner sep=0pt, minimum size=3pt, fill=white] at (1, 0.15) {};
    \end{tikzpicture}
    &= \resizebox{0.95\width}{!}{$qY_{1}X_{2,0}(\mu, \nu)Y_3 - q^2 X_{1, -\frac{1}{2}}(\mu, \nu) Y_1 - q^{-2}X_{2,0}(\mu, \nu)Y_2 - q^2 X_{3,0}(\mu, \nu) Y_3 - C_{\mu}^{\nu}$}\\
\end{align*}

\begin{align*}
    Y_1 Y_2 X_{3,0}(\mu, \nu) &= \begin{tikzpicture}[baseline=-1]
        \MarkedTorusBackground[2][1]
        \draw[thick] (0.3, -1) -- (0.3, -0.2) to[out=90, in=180] (1, 0.15) to[out=45, in=180] (2, 0.7);
        \draw[thick] (0, 0.7) to[out=0, in=270] (0.3, 1);
        \draw[line width=2.5mm, gray!40] (0.5, -0.8) -- (0.5, 0.8);
        \draw[thick] (0.5, -1) -- (0.5, 1);
        \draw[line width=2.5mm, gray!40] (0.2, 0.5) -- (1.8, 0.5);
        \draw[thick] (0, 0.5) -- (2, 0.5);
        \node[draw, circle, inner sep=0pt, minimum size=3pt, fill=white] at (1, 0.15) {};
    \end{tikzpicture}\\
    &= q \begin{tikzpicture}[baseline=-1]
        \MarkedTorusBackground[2][1]
        \draw[thick] (0.5, -1) -- (0.5, -0.2) to[out=90, in=180] (1, 0.15) to[out=45, in=180] (2, 0.5);
        \draw[thick] (0, 0.5) to[out=0, in=270] (0.5, 1);
        \draw[thick] (0.3, -1) -- (0.3, -0.2) to[out=90, in=180] (1.2, 0.7) -- (2, 0.7);
        \draw[thick] (0, 0.7) to[out=0, in=270] (0.3, 1);
        \node[draw, circle, inner sep=0pt, minimum size=3pt, fill=white] at (1, 0.15) {};
    \end{tikzpicture}
    + q^{-1} \begin{tikzpicture}[baseline=-1]
        \MarkedTorusBackground[2][1]
        \draw[thick] (0.3, -1) -- (0.3, -0.2) to[out=90, in=180] (1, 0.15) to[out=45, in=180] (2, 0.7);
        \draw[thick] (0, 0.7) to[out=0, in=270] (0.3, 1);
        \draw[line width=2.5mm, gray!40] (0.2, 0.2) to[out=0, in=90] (0.6, -0.8);
        \draw[line width=2.5mm, gray!40] (0.8, 0.8) to[out=270, in=180] (1.8, 0.35);
        \draw[thick] (0, 0.3) to[out=0, in=90] (0.7, -1);
        \draw[thick] (0.7, 1) to[out=270, in=180] (2, 0.3);
        \node[draw, circle, inner sep=0pt, minimum size=3pt, fill=white] at (1, 0.15) {};
    \end{tikzpicture}\\
    &= q \begin{tikzpicture}[baseline=-1]
        \MarkedTorusBackground[2][1]
        \draw[thick] (0.5, -1) -- (0.5, -0.2) to[out=90, in=180] (1, 0.15) to[out=45, in=180] (2, 0.5);
        \draw[thick] (0, 0.5) to[out=0, in=270] (0.5, 1);
        \draw[thick] (0.3, -1) -- (0.3, -0.2) to[out=90, in=180] (1.2, 0.7) -- (2, 0.7);
        \draw[thick] (0, 0.7) to[out=0, in=270] (0.3, 1);
        \node[draw, circle, inner sep=0pt, minimum size=3pt, fill=white] at (1, 0.15) {};
    \end{tikzpicture}
    + \begin{tikzpicture}[baseline=-1]
        \MarkedTorusBackground[2][1]
        \draw[thick] (0, 0.6) to[out=0, in=270] (0.5, 1);
        \draw[thick] (0.3, -1) -- (0.3, -0.2) to[out=90, in=180] (1, 0.15) -- (2, 0.15);
        \draw[thick] (0.7, 1) to[out=270, in=180] (2, 0.6);
        \draw[line width=2.5mm, gray!40] (0.2, 0.2) to[out=0, in=90] (0.8, -0.8);
        \draw[thick] (0, 0.3) to[out=0, in=90] (0.7, -1);
        \node[draw, circle, inner sep=0pt, minimum size=3pt, fill=white] at (1, 0.15) {};
    \end{tikzpicture}
    + q^{-2} \begin{tikzpicture}[baseline=-1]
        \MarkedTorusBackground[2][1]
        \draw[thick] (0.5, -1) -- (0.5, -0.2) to[out=90, in=180] (1, 0.15) -- (1, 1);
        \draw[line width=2mm, gray!40] (0.5, 0.5) to[out=270, in=90] (1, -0.7);
        \draw[thick] (0.5, 1) -- (0.5, 0.5) to[out=270, in=90] (1, -0.7) -- (1, -1);
        \node[draw, circle, inner sep=0pt, minimum size=3pt, fill=white] at (1, 0.15) {};
    \end{tikzpicture}\\
    &= q \begin{tikzpicture}[baseline=-1]
        \MarkedTorusBackground[2][1]
        \draw[thick] (0.5, -1) -- (0.5, -0.2) to[out=90, in=180] (1, 0.15) to[out=45, in=180] (2, 0.5);
        \draw[thick] (0, 0.5) to[out=0, in=270] (0.5, 1);
        \draw[thick] (0.3, -1) -- (0.3, -0.2) to[out=90, in=180] (1.2, 0.7) -- (2, 0.7);
        \draw[thick] (0, 0.7) to[out=0, in=270] (0.3, 1);
        \node[draw, circle, inner sep=0pt, minimum size=3pt, fill=white] at (1, 0.15) {};
    \end{tikzpicture}
    + q \begin{tikzpicture}[baseline=-1]
        \MarkedTorusBackground[2][1]
        \draw[thick] (0, 0.15) -- (1, 0.15) -- (2, 0.15);
        \draw[thick] (0, 0.5) -- (2, 0.5);
        \node[draw, circle, inner sep=0pt, minimum size=3pt, fill=white] at (1, 0.15) {};
    \end{tikzpicture}
    + q^{-1} \begin{tikzpicture}[baseline=-1]
        \MarkedTorusBackground[2][1]
        \draw[thick] (1, 0.15) to[out=0, in=90] (1.3, -0.2) to[out=270, in=0] (1, -0.5) to[out=180, in=270] (0.7, -0.2) to[out=90, in=180] (1, 0.15);
        \node[draw, circle, inner sep=0pt, minimum size=3pt, fill=white] at (1, 0.15) {};
    \end{tikzpicture} \\
    &\phantom{=} + q^{-1} \begin{tikzpicture}[baseline=-1]
        \MarkedTorusBackground[2][1]
        \draw[thick] (1, 0.15) to[out=60, in=270] (1.2, 0.45) to[out=90, in=0] (1, 0.65) to[out=180, in=90] (0.8, 0.45) to[out=270, in=120] (1, 0.15);
        \node[draw, circle, inner sep=0pt, minimum size=3pt, fill=white] at (1, 0.15) {};
    \end{tikzpicture}
    + q^{-3}\begin{tikzpicture}[baseline=-1]
        \MarkedTorusBackground[2][1]
        \draw[thick] (0.4, -1) -- (0.4, 1);
        \draw[thick] (1, -1) to[out=90, in=270] (0.7, 0) to[out=90, in=180] (1, 0.15) -- (1, 1);
        \node[draw, circle, inner sep=0pt, minimum size=3pt, fill=white] at (1, 0.15) {};
    \end{tikzpicture}\\\\
    \Rightarrow \begin{tikzpicture}[baseline=-1]
        \MarkedTorusBackground[2][1]
        \draw[thick] (1, 0.15) to[out=0, in=90] (1.3, -0.2) to[out=270, in=0] (1, -0.5) to[out=180, in=270] (0.7, -0.2) to[out=90, in=180] (1, 0.15);
        \node[draw, circle, inner sep=0pt, minimum size=3pt, fill=white] at (1, 0.15) {};
    \end{tikzpicture}
    &= \resizebox{0.95\width}{!}{$qY_{1}Y_{2}X_{3,0}(\mu, \nu) - q^2 X_{1, 0}(\mu, \nu) Y_1 - q^{-2}X_{2,0}(\mu, \nu)Y_2 - q^2 X_{3,0}(\mu, \nu) Y_3 - C_{\mu}^{\nu}$}\\
\end{align*}

\newpage

When resolving crossings in these calculations, the process is conducted locally and away from the boundary. As a result, it is possible to initiate the calculation with $Y_1 Y_2 X_{3,k}(\mu, \nu)$ (likewise $X_{1,k}(\mu, \nu) Y_2 Y_3$ or $Y_1 X_{2,k}(\mu, \nu) Y_3$) for any $k \in \frac{1}{2}\mathbb{Z}$ and substitute each stated tangle in the diagrams with $k$ twists (or their relative twists) around the boundary.
Since $\begin{tikzpicture}[baseline=-1]
    \MarkedTorusBackground
    \draw[thick] (1, 0.15) to[out=0, in=90] (1.3, -0.2) to[out=270, in=0] (1, -0.5) to[out=180, in=270] (0.7, -0.2) to[out=90, in=210] (0.87, 0.15);
    \node[draw, circle, inner sep=0pt, minimum size=3pt, fill=white] at (1, 0.15) {};
    \node at (0.6, 0.2) {$\mu$};
    \node at (1.4, 0.2) {$\nu$};
\end{tikzpicture}$ and
$C_{\mu}^{\nu} = \begin{tikzpicture}[baseline=-1]
    \MarkedTorusBackground
    \draw[thick] (1, 0.15) to[out=60, in=270] (1.2, 0.45) to[out=90, in=0] (1, 0.65) to[out=180, in=90] (0.8, 0.45) to[out=270, in=120] (1, 0.15);
    \node[draw, circle, inner sep=0pt, minimum size=3pt, fill=white] at (1, 0.15) {};
    \node at (0.6, 0.2) {$\mu$};
    \node at (1.4, 0.2) {$\nu$};
\end{tikzpicture}$ correspond to half twists of each other in either direction, and their coefficients are identical, interchanging their diagrams during the calculations does not affect the equations. Therefore, we get slightly more general formulas for our parallel tangle.
\vfill

Measuring a half twist by using the different equations for
$X_5(\mu, \nu) = \begin{tikzpicture}[baseline=-1, scale=0.97]
    \MarkedTorusBackground
    \draw[thick] (1, 0.15) to[out=0, in=90] (1.3, -0.2) to[out=270, in=0] (1, -0.5) to[out=180, in=270] (0.7, -0.2) to[out=90, in=210] (0.87, 0.15);
    \node[draw, circle, inner sep=0pt, minimum size=3pt, fill=white] at (1, 0.15) {};
    \node at (0.7, 0.2) {$\mu$};
    \node at (1.3, 0.2) {$\nu$};
\end{tikzpicture}$:
\begin{align*}
    X_5(\mu, \nu) &= qX_{1,k}(\mu, \nu)Y_2Y_3 - q^2 X_{1,k}(\mu, \nu) Y_1 - q^{-2}X_{2, k + \frac{1}{2}}(\mu, \nu)Y_2 - q^2 X_{3,k}(\mu, \nu) Y_3 - C_{\mu}^{\nu}\\
    &= qY_{1}X_{2,k}(\mu, \nu)Y_3 - q^2 X_{1, k-\frac{1}{2}}(\mu, \nu) Y_1 - q^{-2}X_{2,k}(\mu, \nu)Y_2 - q^2 X_{3,k}(\mu, \nu) Y_3 - C_{\mu}^{\nu}\\
    &= qY_{1}Y_{2}X_{3,k}(\mu, \nu) - q^2 X_{1, k}(\mu, \nu) Y_1 - q^{-2}X_{2,k}(\mu, \nu)Y_2 - q^2 X_{3,k}(\mu, \nu) Y_3 - C_{\mu}^{\nu}
\end{align*}
$$\Rightarrow qY_1 \left( X_{2,k} Y_3 - Y_2 X_{3,k} \right) = q^2 Y_1 \left( X_{1, k-\frac{1}{2}} - X_{1,k} \right) - q^{-2} Y_2\left( X_{2,k} - X_{2,k+\frac{1}{2}} \right)$$
\begin{align*}
    \left( X_{1,k} - X_{1,k-\frac{1}{2}} \right) Y_1 &= q^{-1} \left( Y_1 Y_2 X_{3, k} - Y_1 X_{2,k} Y_3 \right)\\
    \left( X_{2,k+\frac{1}{2}} - X_{2,k} \right) Y_2 &= q^{3} \left( Y_1 Y_2 X_{3, k} - X_{1,k} Y_2 Y_3 \right)\\
\end{align*}
\vfill

\chapter{Python Code}
\section{Quantum Commuting Relations}\label{appendix:qCommRel}
\begin{lstlisting}[language=Python]
'''
~~~~~~~~~~~~~~~~~~~~~~~~~~~~~~~~~~~~~~~~~~~~~~~~~~~~~~~~~~~~~~~~~~~~~~~~~~~~~
~~~~~~~~~~~~~~~~~~~~~~~ Program Outline & Limitations ~~~~~~~~~~~~~~~~~~~~~~~
~~~~~~~~~~~~~~~~~~~~~~~~~~~~~~~~~~~~~~~~~~~~~~~~~~~~~~~~~~~~~~~~~~~~~~~~~~~~~
I created this program to help calculate many relations pertaining to my dissertation. You can find more information about general stated skein modules/algebras in "Stated Skein Modules of Marked 3-Manifolds/Surfaces, A Survey" by Thang Le and Tao Yu. This survey paper is available on arxiv and on Dr. Le's website: https://letu.math.gatech.edu/Papers/Survey.pdf

The purpose of this program is to help calculate the q-commuting relations of nontrivial stated tangles on the once marked torus. This program doesn't support calculations done with more than one marking and major adjustments will need to be made if one would like to do this. It is, however, theoretically quite doable to use this program on any other surface with a single marking with minimal coding adjustments.
https://letu.math.gatech.edu/Papers/Survey.pdf
The logic of this program goes as follows:
	- Create a list of all possible state combinations
	- For each state combination we run through the below algorithm using height exchange relations and keep track of the commuting costs
	- Using sympy variables, each time we perform a height exchange we substitute the exchange cost into that variable
	- Finally, tell the user what the equality comes out to
We keep track of all the necessary information throughout the process using object-oriented programming by using the endpoints as dictionary-type attributes.

Note that this is not a complete calculation as there are several different variables that can appear (in particular from exchange relations from bad arc states) that are difficult for a computer to program without implementing an excessive number of edge cases and specificity, of our manifold. We instead insert new variables to replace these pictures and allow the user to figure them out by hand. Assuming that the user wants to use this program in the first place more than likely implies that there are enough calculations that need to be found to where this program will still save significant time for the user. By replacing these edge cases with new variables, we're able to incorporate more generality into this program so that it can be used in other cases. This means that the program is merely an aid to finding these relations and does not completely compute the commuting relations.
~~~~~~~~~~~~~~~~~~~~~~~~~~~~~~~~~~~~~~~~~~~~~~~~~~~~~~~~~~~~~~~~~~~~~~~~~~~~~
'''

import itertools
from sympy import *

'''
~~~~~~~~~~~~~~~~~~~~~~~~~~~~~ Program Variables ~~~~~~~~~~~~~~~~~~~~~~~~~~~~~
~~~~~~~~~~~~~~~~~~ Edit The Following Parameters If Needed ~~~~~~~~~~~~~~~~~~
This order list tells us the order of which we perform our height exchange relations
The order of each pairing helps us identify the relative positions of endpoints. For example, # [2,3] tells us that in the picture, the third endpoint (the one in front) is locally to the right of the second where as [3,2] would tell us the third endpoint is locally to the left. # Here when I say local I mean local and canonical.
                       \    /
          [2,3]  -->    \  /   
                       2 `/ 3
The only parameters you should need to edit here are the the orders of the elements within the inner-most lists (NOT the order of the lists themselves).

Example:	We want to see how a `boundary-parallel` tangle and a `meridian` tangle commute. Then the height swaps should correspond to a middle-left-right-middle swap:
 
			     1 2 3 4 --->  1 3 2 4 ---> 3 1 2 4 ---> 3 1 4 2 ---> 3 4 1 2
			             (32)          (21)         (34)         (23)
			     (left position, right position)
			A possible calculation order list could be:
				calculationOrderList = [[3, 2], [2, 1], [3, 4], [2, 3]]

			Below is a more visual interpretation of what's happening in this example. The horizontal dashed line corresponds to the marking and we're looking at it from a side perspective. The vertical lines correspond to the ends of the tangles and the numbered labelings at the bottom keeps track of them as they move past each other throughout the algorithm. Above the first vertical line are the labels correspond to the local relationships of each endpoint to each other (as explained above). The labels LL, L, R, RR correspond to leftmost, middle left, middle right, and rightmost respectively.
  L   R   LL  RR
  |   |   |   |       |   |   |   |       |   |   |   |       |   |   |   |
  |   |   |   |       |   |   |   |       |   |   |   |       |   |   |   |
--------------->    --------------->    --------------->    --------------->
  1   2   3   4       1   3   2   4       3   1   2   4       3   1   4   2

In the case of the once marked torus, you should always be able to get every commutation relation using only 4 swaps. In particular, we should always be using an order list corresponding to middle-left-right-middle (MLRM) Therefore, the only parameters you should ever need to edit in `calculationOrderList` should only be the orders of the elements within the inner-most lists and NOT the order of the lists themselves nor the entries.
~~~~~~~~~~~~~~~~~~~~~~~~~~~~~~~~~~~~~~~~~~~~~~~~~~~~~~~~~~~~~~~~~~~~~~~~~~~~~
'''
calculationOrderList = [[3, 2], [1, 2], [3, 4], [2, 3]]

# ~~~~~~~~~~~~~~~~~~~~~~~~~~~~ Do Not Edit Below ~~~~~~~~~~~~~~~~~~~~~~~~~~~~

# Symbol declaraions
global q
q = Symbol('q', commutative=True)
x = Symbol('x', commutative=False)# x will be our picture
'''
a,b,c,d are additional symbols representing occasionally arising pictures with only two states, which appear in bad arc state exchange relations. Depending where the states we're looking at are located, we can have up to 6 different variations and so the two numbers appended onto these variable names help us distinguish which variation we're dealing with.

Assuming we're dealing with the MLRM order list, should only need x12, x34, x14. Note that since this algorithm swaps the middle twice, x14 has two different variations that can appear. Feel free to add an additional variable to distinguish between the two if you'd like. I did not really need it so I didn't add one. Luckily, we only need half of them using this algorithm.
'''
# a --> states == (+, +)
a12_l = Symbol('a12_l', commutative=False)# (*, *, -, -)
a14_l = Symbol('a14_l', commutative=False)# (*, -, -, *)
a34_l = Symbol('a34_l', commutative=False)# (-, -, *, *)
a12_r = Symbol('a12_r', commutative=False)# (*, *, -, -)
a14_r = Symbol('a14_r', commutative=False)# (*, -, -, *)
a14_r = Symbol('a14_r', commutative=False)# (*, -, -, *)
a34_r = Symbol('a34_r', commutative=False)# (-, -, *, *)
# b --> states == (+, -)
b12_l = Symbol('b12_l', commutative=False)# (*, *, -, -)
b14_l = Symbol('b14_l', commutative=False)# (*, -, -, *)
b34_l = Symbol('b34_l', commutative=False)# (-, -, *, *)
b12_r = Symbol('b12_r', commutative=False)# (*, *, -, -)
b14_r = Symbol('b14_r', commutative=False)# (*, -, -, *)
b14_r = Symbol('b14_r', commutative=False)# (*, -, -, *)
b34_r = Symbol('b34_r', commutative=False)# (-, -, *, *)
# c --> states == (-, +)
c12_l = Symbol('c12_l', commutative=False)# (*, *, -, -)
c14_l = Symbol('c14_l', commutative=False)# (*, -, -, *)
c34_l = Symbol('c34_l', commutative=False)# (-, -, *, *)
c12_r = Symbol('c12_r', commutative=False)# (*, *, -, -)
c14_r = Symbol('c14_r', commutative=False)# (*, -, -, *)
c14_r = Symbol('c14_r', commutative=False)# (*, -, -, *)
c34_r = Symbol('c34_r', commutative=False)# (-, -, *, *)
# d --> states == (-, -)
d12_l = Symbol('d12_l', commutative=False)# (*, *, -, -)
d14_l = Symbol('d14_l', commutative=False)# (*, -, -, *)
d34_l = Symbol('d34_l', commutative=False)# (-, -, *, *)
d12_r = Symbol('d12_r', commutative=False)# (*, *, -, -)
d14_r = Symbol('d14_r', commutative=False)# (*, -, -, *)
d14_r = Symbol('d14_r', commutative=False)# (*, -, -, *)
d34_r = Symbol('d34_r', commutative=False)# (-, -, *, *)

# This dictionary is used only in the `findRepSym` function
additionalSymbolDictLeftFront = {
	(('+', '+'), 1, 2) : a12_l,
	(('+', '+'), 1, 4) : a14_l,
	(('+', '+'), 3, 4) : a34_l,
	(('+', '-'), 1, 2) : b12_l,
	(('+', '-'), 1, 4) : b14_l,
	(('+', '-'), 3, 4) : b34_l,
	(('-', '+'), 1, 2) : c12_l,
	(('-', '+'), 1, 4) : c14_l,
	(('-', '+'), 3, 4) : c34_l,
	(('-', '-'), 1, 2) : d12_l,
	(('-', '-'), 1, 4) : d14_l,
	(('-', '-'), 3, 4) : d34_l,
}

additionalSymbolDictRightFront = {
	(('+', '+'), 1, 2) : a12_r,
	(('+', '+'), 1, 4) : a14_r,
	(('+', '+'), 3, 4) : a34_r,
	(('+', '-'), 1, 2) : b12_r,
	(('+', '-'), 1, 4) : b14_r,
	(('+', '-'), 3, 4) : b34_r,
	(('-', '+'), 1, 2) : c12_r,
	(('-', '+'), 1, 4) : c14_r,
	(('-', '+'), 3, 4) : c34_r,
	(('-', '-'), 1, 2) : d12_r,
	(('-', '-'), 1, 4) : d14_r,
	(('-', '-'), 3, 4) : d34_r,
}



class statedEndpoint(object):
	def __init__(self, initialData):
		# Create keys for our statedEndpoint dictionaries to make calling their values easier
		for key in initialData:
			setattr(self, key, initialData[key])

class statedProduct(object):
	'''
	Each endpoint of each factor in our product has a tag identifying its starting position and a value corresponding to its state. The current position of each endpoint will be changed throughout the program by swapping its corresponding attribute data. This data structure has been configured this way to help with dynamic access throughout the algorithm.

	It is possible to instead create another attribute corresponding to its position as that's debatably cleaner organization, however, access to the position seems to get a bit more complicated due to how dynamic we need its current position to be. Using the `swapAttr` function when needed instead seemed like the easiest workaround.
	'''
	def __init__(self, startStates):
		# SEP == stated endpoint
		self.SEP1 = statedEndpoint({"identifier": "Factor 1 Left", "state": startStates[0][0]})
		self.SEP2 = statedEndpoint({"identifier": "Factor 1 Right", "state": startStates[0][1]})
		self.SEP3 = statedEndpoint({"identifier": "Factor 2 Left", "state": startStates[1][0]})
		self.SEP4 = statedEndpoint({"identifier": "Factor 2 Right", "state": startStates[1][1]})
		self.expr = x

	def findSEPAttr(self, num):
		'''
		Use this function when you want to dynamically use one of the SEP attributes. This functions takes in an integer, n in {1,2,3,4}, and returns that corresponding self.SEPn` attribute/object using a dictionary in a switch-case manner.
		'''
		switcher = {
			1: self.SEP1,
			2: self.SEP2,
			3: self.SEP3,
			4: self.SEP4
		}
		return switcher.get(num)

	def printCurrentStatus(self):
		# This function prints the current attribute/endpoint information/status to the user
		print("Position 1 \tID: " + self.SEP1.identifier + "\tState: " + self.SEP1.state)
		print("Position 2 \tID: " + self.SEP2.identifier + "\tState: " + self.SEP2.state)
		print("Position 3 \tID: " + self.SEP3.identifier + "\tState: " + self.SEP3.state)
		print("Position 4 \tID: " + self.SEP4.identifier + "\tState: " + self.SEP4.state + "\n")

	def swapAttr(self, pos1, pos2):
		# This function swaps two of the SEP attributes/objects
		tempAttr = getattr(self, 'SEP%d' % pos1)
		setattr(self, 'SEP%d' % pos1, getattr(self, 'SEP%d' % pos2))
		setattr(self, 'SEP%d' % pos2, tempAttr)
		del tempAttr
		# Returns the complement endpoints that weren't swapped for possible bad arc relations
		return list(set([1, 2, 3, 4]) - set([pos1, pos2]))

	def findRepSym(self, complementStates, leftFrontBool):
		'''
		When we must swap heights of endpoints with bad arc states, an additional term appears, which pictorially has the two endpoints connected together rather than lying on themarking. The function identifies which new term must be added based on the leftover state values and where these states are located, `complementStates`, with respect to our picture variable `x`. Although this additional term can only have 4 possible states, there are 6 = (4 choose 2) variants per combination of states along with an additional variation depending on which endpoint is locally in front (x2) => up to 48 possible variations. Rather than having a bunch of messy nested if statements, we instead use the two dictionaries `additionalSymbolDictLeftFront` and `additionalSymbolDictRightFront`, which is located outside of any loops and functions for memory efficiency.
		'''
		if leftFrontBool:
			return additionalSymbolDictLeftFront.get(((self.findSEPAttr(complementStates[0]).state, self.findSEPAttr(complementStates[1]).state), complementStates[0], complementStates[1]))
		return additionalSymbolDictRightFront.get(((self.findSEPAttr(complementStates[0]).state, self.findSEPAttr(complementStates[1]).state), complementStates[0], complementStates[1]))

	def swapHeights(self, endpointLeft, endpointRight):
		leftEndpoint = self.findSEPAttr(endpointLeft)
		rightEndpoint = self.findSEPAttr(endpointRight)
		# Is the left endpoint in front of the right endpoint?
		leftFrontBool = False if endpointLeft < endpointRight else True
		# Find the q-coefficient cost of swapping and whether or not this is a bad arc state
		qConst, badArc = self.stateExchangeCoeff(leftEndpoint.state, rightEndpoint.state, leftFrontBool)
		# Swap the endpoint attributes and return complement endpoints for possible bad arc state
		complementStates = self.swapAttr(endpointLeft, endpointRight)
		# If a bad arc state is found then we need to add a summand to our expression
		if badArc:
			summand = q**(-3/2)*(q**2 - q**-2)*self.findRepSym(complementStates, leftFrontBool) if qConst == q**-3 else -q**(3/2)*(q**2 - q**-2)*self.findRepSym(complementStates, leftFrontBool)
		else:
			summand = 0
		# Substitute updated height and costs into our expression
		self.expr = self.expr.subs(x, qConst*x + summand)

	def stateExchangeCoeff(self, stateLeft, stateRight, leftFrontBool):
		# Does this state pair correlate to a bad arc? (+,-)
		badArc = True if (stateLeft == "+" and stateRight == "-") else False
		# Find the q-coefficient based on the states
		if stateLeft == stateRight:
			qConst = q**-1 if leftFrontBool else q
		elif (stateLeft, stateRight) == ('-', '+'):
			qConst = q if leftFrontBool else q**-1
		elif (stateLeft, stateRight) == ('+', '-'):
			qConst = q**-3 if leftFrontBool else q**3
		else:
			print("Error: Unknown state arguments in function `stateExchangeCoeff`.",
				"Arguments received: stateLeft = '" + str(stateLeft) + "', stateRight = '" + str(stateRight) + "'.")
			print("Please use known states only!")
			quit()
		return qConst, badArc

	def calcCommutingRelation(self, orderList):
		for myList in orderList:
			self.swapHeights(myList[0], myList[1])
		return self.expr





# Creates a nested tuple, `productStateTuple`, of all 16 possible states
possibleStates = ["+", "-"]
stateTuple = list(itertools.product(possibleStates, possibleStates))
productStateTuple = list(itertools.product(stateTuple, stateTuple))


for statesTuple in productStateTuple:
	myProd = statedProduct(statesTuple)
	myProd.calcCommutingRelation(calculationOrderList)
	print("x1(" + str(statesTuple[1][0]) + "," + str(statesTuple[1][1]) + ")*x2(" + str(statesTuple[0][0]) + "," + str(statesTuple[0][1]) + ") = " + str(sympify(expand(myProd.expr))))

'''
Since a*b is defined as a on top of b and states area read from left to right (right on top), we swap a and b at the end print statement and keep things aligned throughout the calculations.
		-- a[0] -- a[1] -- b[0] -- b[1] -->
'''
\end{lstlisting}

\section{Quantum $6$-Torus Operators}\label{appendix:T6Operators}
\begin{lstlisting}[language=Python]
from sympy import *

def qBracket(a, b, normalized=False, operatorInput=None):
	'''
	This function takes in `sympy` elements, `a` and `b`, and returns their corresponding quantum bracket: qab - q^(-1)ba. If `a` and `b` are operators acting on 'f', then we require an input for the operator as well, `operatorInput`, and need to treat the arguments as functions to allow for the operators to act in the proper order. We also allow for bracket normalization, i.e. divide by q^(2) - q^(-2), if desired.
	Example of calling this function:
		expression = lambda f: qBracket(y1, y2, normalized=True, operatorInput=f)
	'''
	myExpr = q * a(b(operatorInput)) - q**(-1) * b(a(operatorInput)) if operatorInput else q*a*b - q**-1*b*a
	if normalized:
		myExpr = myExpr*(q**2 - q**-2)**-1
	return myExpr

# Define variables
x, y, z, w, q = symbols('x y z w q')

# Define a commutative multivariable polynomial
f = Function('f')(x, y, z, w)

# Define quantum operators
x1 = lambda f: x * f.subs({x: x, y: q * y, z: q * z, w: q**-1 * w})
x2 = lambda f: y * f.subs({x: q**-1 * x, y: y, z: q**-1 * z, w: q**-2 * w})
x3 = lambda f: z * f.subs({x: q**-1 * x, y: q * y, z: z, w: q**-1 * w})
x4 = lambda f: w * f.subs({x: q * x, y: q**2 * y, z: q * z, w: w})
x5 = lambda f: f
x6 = lambda f: f.subs({x: q**4 * x, y: q**4 * y, z: q**4 * z, w: q**4 * w})
x1inv = lambda f: x**-1 * f.subs({x: x, y: q**-1 * y, z: q**-1 * z, w: q * w})
x2inv = lambda f: y**-1 * f.subs({x: q * x, y: y, z: q * z, w: q**2 * w})
x3inv = lambda f: z**-1 * f.subs({x: q * x, y: q**-1 * y, z: z, w: q * w})
x4inv = lambda f: w**-1 * f.subs({x: q**-1 * x, y: q**-2 * y, z: q**-1 * z, w: w})
x5inv = lambda f: f
x6inv = lambda f: f.subs({x: q**-4 * x, y: q**-4 * y, z: q**-4 * z, w: q**-4 * w})
# Meridian, Longitude, and (1,1)-Curve
y1 = lambda f: y*z**-1*f.subs({x: x, y: q**-1*y, z: q**-1*z, w: q**-1*w}) + y**-1*z*f.subs({x: x, y: q*y, z: q*z, w: q*w}) + x**2*z**-1*w**-1*f.subs({x: x, y: q**-1*y, z: q*z, w: q**-1*w}) + x*y**-1*w**-1*f.subs({x: x, y: q**-1*y, z: q*z, w: q*w})
y2 = lambda f: x*z**-1*f.subs({x: q*x, y: y, z: q*z, w: w}) + x**-1*z*f.subs({x: q**-1*x, y: y, z: q**-1*z, w: w}) + x**-1*y*z**-1*w*f.subs({x: q*x, y: y, z: q**-1*z, w: w})
y3 = lambda f: x*w**-1*f.subs({x: q**-1*x, y: q**-1*y, z: z, w: q**-1*w}) + x**-1*w*f.subs({x: q*x, y: q*y, z: z, w: q*w}) + x**-1*y**-1*z**2*f.subs({x: q**-1*x, y: q*y, z: z, w: q*w}) + y**-1*z*w**-1*f.subs({x: q**-1*x, y: q**-1*y, z: z, w: q*w})
# Central \partial = q^2*y1*y2*y3 - q^2*y1^2 - q^-2*y2^2 - q^2*y3^2 + q^2 + q^-2
dee = lambda f: q*y1(y2(y3(f))) - q**2*y1(y1(f)) - q**-2*y2(y2(f)) - q**2*y3(y3(f)) + q**2*f + q**-2*f
\end{lstlisting}

\end{document}